\documentclass{iopart}

\usepackage{graphicx}

\usepackage{amsmath}
\usepackage{amsthm}
\usepackage{amssymb}

\usepackage{pdfsync}
\usepackage{multirow}

\usepackage{wrapfig}
\usepackage{empheq}

\usepackage{cite}
\usepackage{titlesec}
\usepackage{comment}
\usepackage{float}
\usepackage{color}

\usepackage{enumerate}
\usepackage[labelfont=bf]{caption}
\usepackage[space]{grffile}
\usepackage[top=2cm,bottom=2cm,left=3cm,right=3cm]{geometry}
\usepackage{caption}
\usepackage{subcaption}
\usepackage{titlesec}
\renewcommand{\thesection}{\arabic{section}}

\makeatletter
\renewcommand{\@seccntformat}[1]{\csname the#1\endcsname.\quad}
\makeatother

\renewcommand{\figurename}{Fig.}

\DeclareMathOperator*{\argmin}{arg\,min}

\newcommand{\re}{\mathbb{R}}
\newcommand{\na}{\mathbb{N}}

\newcommand{\rt}{\mathcal{R}}
\newcommand{\diverg}{{\rm div}}

\newcommand{\dom}{{\rm dom}}

\newcommand{\supp}{{\rm supp}}
\newcommand{\norm}[2]{\left\lVert#2\right\rVert_{#1}}
\newcommand{\scalprod}[3]{\left\langle #2,#3\right\rangle_{#1}}

\numberwithin{equation}{section}
\newtheorem{mydef}{Definition}[section]

\newtheorem{thrm}{Theorem}[section]

\newtheorem{lem}{Lemma}[section]

\begin{document}

\title[TV regularisation in measurement and image space for PET]{Total Variation Regularisation in Measurement and Image space for PET reconstruction}
\author{M Burger$^1$, J M\"uller$^2$, E Papoutsellis$^3$ and C B Sch{\"o}nlieb$^3$}

\address{$^1$Institute for Computational Applied Mathematics and Cells-in-Motion Cluster of Excellence (EXC 1003 – CiM), University of M\"unster, Germany}
\address{$^2$Department of Nuclear Medicine, University Hospital M\"unster}
\address{$^3$Department of Applied Mathematics and Theoretical Physics, University of Cambridge}

\eads{\mailto{martin.burger@wwu.de}, \mailto{jahn.mueller@uni-muenster.de}, \mailto{ep374@cam.ac.uk}, \mailto{cbs31@cam.ac.uk}}

\begin{abstract}
The aim of this paper is to test and analyse a novel technique for image reconstruction in positron emission tomography, which is based on (total variation) regularisation on both the image space and the projection space. We formulate our variational problem considering both total variation penalty terms on the image and on an idealised sinogram to be reconstructed from a given Poisson distributed noisy sinogram. We prove existence, uniqueness and stability results for the proposed model and provide some analytical insight into the structures favoured by joint regularisation. 

For the numerical solution of  the corresponding discretised problem we employ the split Bregman algorithm and extensively test the approach in comparison to standard total variation regularisation on the image. The numerical results show that an additional penalty on the sinogram performs better on reconstructing images with thin structures.  
\end{abstract}

\section{Introduction}

Positron emission tomography (PET) is a medical imaging technique for studying functional characteristics of the human body, used in brain imaging, neurology, oncology and recently also in cardiology. The patient is injected with a dose of radioactive tracer isotope which concentrates in tissues of interest in the body. Typically, cells in the tissue which are more active have a higher metabolism, i.e., need more energy, and hence will absorb more tracer isotope than cells which are less active. The isotope suffers radioactive decay which invokes it to emit a positron. As soon as the emitted positron meets an electron a pair of gamma rays is sent out into approximately opposite directions and is picked up by the PET-scanner. The collection of all these pairs builds the PET measurement $g$ from which the distribution $u$ of the relevant radiopharmaceutical shall be reconstructed.

As a (yet simplified) mathematical model the PET measurement can be interpreted as a sample of
\begin{equation}\label{PETmeas}
f = e^{-\int_L h\; dt} \int_L u\; dt,
\end{equation}
where the above integral is the Radon transform $\mathcal R$ of $u$ along the line $L$ connecting the emission point of the gamma rays and the detector, see Figure \ref{fig1}; the above exponential characterises the damping due to the ''attenuation`` function $h$ (which is, e.g., known from CT \cite[Chapter~7]{WA04}). The function $f(L)$ is called the sinogram of $u$. Since the attenuation can be corrected beforehand we shall ignore the attenuation term in the solution of the inverse problem (corresponding to $h\equiv 0$) in this paper. The basic mathematical problem for the reconstruction of the distribution $u$, is the inversion of the Radon transform. In PET, this inversion is complicated by the presence of undersampling and noise \cite{WA04}. The PET data usually is corrupted by Poisson noise, also called photon noise, due to the photon counting process during the PET scan.

In this paper, we propose a novel technique for reconstructing an image $u$ from noisy PET measurements $g$ by a variational regularisation approach using total variation (TV) regularisation \cite{Rudin} on both the image $u$ and the sinogram $\rt u$. More precisely, let $\Sigma^{n}=\{ (\theta,s)\in\mathcal{S}^{n-1}\times\re\}$ be the projection space (see Figure \ref{fig1}) and $\mathbb R^2$ the physical space. The Radon transform $\rt:L^{1}(\re^2)\rightarrow L^{1}(\Sigma)$ of $u\in L^1(\re^2)$ is given by 
\begin{equation}
	\rt u(\theta,s)=\int_{\re^{2}}u(x)\delta(s-x\cdot\theta)dx
	\label{m2}
\end{equation}

Given measurements $g\in L^2(\Sigma)$, we reconstruct $u\in BV(\re^2)$ by solving
\begin{equation}
\argmin_{u\in BV(\re^2), ~u\geq0\textrm{ a.e. in }\re^2}\left\{\alpha|Du|(\re^2)+\beta |D(\rt u)|(\Sigma)+\frac{1}{2}\int_{\Sigma}\frac{(g-\rt u)^{2}}{g}\right\}
\label{v1}
\end{equation}
Here $BV(\re^2)$ is the space of functions of bounded variation, see \cite{Ambrosio}, and $\alpha,\beta$ are positive parameters. The terms $|Du|(\re^2)$ and $|D(\rt u)|(\Sigma)$ are TV regularisations on the image $u$ and the sinogram $\rt u$ respectively, that is
$$
|Du|(\re^2)= \sup_{{\bf g} \in C_0^\infty(\re^2;\mathbb{R}^2), \|g\|_\infty \leq 1} \int_{\re^2}  u ~\nabla \cdot {\bf g} ~dx,\quad |D\rt u|(\Sigma)= \sup_{{\bf g} \in C_0^\infty(\Sigma;\mathbb{R}^2), \|g\|_\infty \leq 1} \int_\Sigma  \rt u ~\nabla \cdot {\bf g} ~dx.
$$
The data fidelity $\int_{\Sigma}(g-\rt u)^{2}/g$ is a weighted $L^2$ norm that constitutes a standard approximation of the Poisson noise model given by the Kullback-Leibler divergence, compare \cite[Chapter~4]{AlexPhD} for instance.

\begin{figure}[h!]
\begin{center}
\includegraphics[scale=0.9]{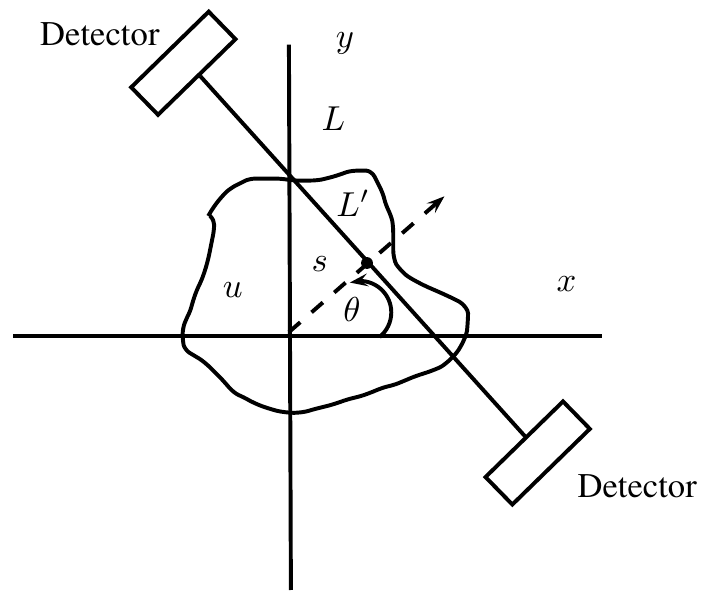}
\caption{PET scan geometry. The image function $u$ is encoded in integrals $f(L)=f(\theta,s)$ along lines $L$ from the emission point to the detectors, cf. \eqref{PETmeas}. The lines $L$ are defined by an angle $\theta$ and distance $s$ to the origin. ($L^{'}$ is the distance from the emission point through the object.) }
\label{fig1}
\end{center}
\end{figure}

PET reconstruction using TV regularisation is not new. However, typically the TV regularisation is applied to the image function $u$ only. By additionally regularising the sinogram $\rt u$ using a total variation penalty in projection space we will show that under certain conditions images of higher quality can be reconstructed. In particular, this is the case in the presence of high noise in $g$ and when aiming to preserve thin and elongated structures in $u$.

\subsection{Related methods}
Our approach \eqref{v1} is inspired by an alternating regularisation procedure for PET first introduced by Barbano et al. in \cite{Carola}. Given possible under sampled and noisy PET measurements $g\in\re^{n\times m}$ an image $u^*$ is reconstructed by solving
\begin{equation}
\min_{\{(s,u):~ s\in\re^{n\times m},u=\rt^{-1}s\}}\alpha\norm{1}{\nabla s}+\beta\norm{1}{u}+\frac{\lambda}{2}\norm{2}{g-s}^{2},
\label{f1}
\end{equation}
where $\rt^{-1}$ is the inverse Radon transform approximated by the filtered backprojection and $\alpha,\beta$ and $\lambda$ are positive weighting parameters. Note that here a regularised reconstruction $u^*$ is computed by smoothing both the image $u$ and the sinogram $s$. Indeed, the regularisation in \eqref{f1} is given by the total variation regulariser $\norm{1}{\nabla s}$, see \cite{Rudin}, that acts on the sinogram $s$ only. The image $u$ is forced to be sparse by an $\ell_1$ penalty. The main focus of \cite{Carola} is to study the effect of total variation regularisation on the sinogram, rather than the image as usually done in variational PET reconstruction \cite{Alex,Alex1,Brune,Brune1}. Therefore, in their numerical experiment the effect of the image regularisation is kept low by choosing an appropriate weighting $\alpha \gg \beta$. In \cite{Carola} it is proved that \eqref{f1} has a unique solution $(s^{*}, u^*)$. Moreover, the authors show the effect, the total variation regularisation of the sinogram $s$, has on the reconstructed image $u$ by a computational experiment on a simulated data set. 

The main idea of adding an total variation regularisation on the projection space originated in the works of Thirion \cite{Thirion}, and Prince et al. \cite{Prince}. In \cite{Thirion} the author proposes to connect edge detection of the tomographic image to finding continuous lines in the sinogram. That is, a point on a line in the sinogram corresponds to an edge in the object space with a fixed orientation and distance from the origin, see Figure \ref{fig1}. 
Moreover, in \cite{Prince}  Prince and Willsky focus on reconstructing tomographic images by using a Markov random field prior on the sinogram, in particular in the presence of data with a low signal-to-noise ratio (SNR) and limited angle or sparse-angle measurement configurations. Their approach leads to the computation of a smoothed sinogram from which the image $u$ is reconstructed using filtered back projection.

\subsection{State of the art - direct and iterative PET reconstruction} In \eqref{v1} we reconstruct an image from PET measurements by smoothing both in measurement and image space. This indeed combines the philosophies of the two main approaches for image reconstruction from PET measurements: (i) Direct methods and (ii) iterative / variational method. While in direct methods the PET measurements are smoothed by an appropriate filter and then inverted (cf. e.g. \cite{Natt1,Fokas}), iterative methods (respectively variational methods solved iteratively) are based on the standard Bayesian modelling approach in inverse problems in which prior knowledge in terms of regularity is expressed for the image function $u$ (rather than the measurements $f$). The possibility to include statistical noise models is a main advantage of iterative and variational methods, on which we shall focus in the following.

In iterative methods for PET reconstruction 
the noise distribution is accounted for by modelling the randomness in the numbers of detected gamma counts. The most popular iterative approach for PET reconstruction is the expectation-maximization (EM) algorithm. To recall, the problem of image reconstruction can be formulated as a solution of the linear and ill-conditioned operator equation: $$g=K u$$ where $g$ is the Poisson distributed data and $K$ is a finite-dimensional sampling of the Radon transform.
Typically, we may assume that the data are realizations of random variables $X_{i}$ and we consider the detected values $g_{i}$ as a realization of a random variable $X_{i}$, for $i=1,...,N$. It is reasonable to maximize the conditional probability $P(u|g)$, which by the Bayes' Law is:
$$P(u|g)=\frac{P(g|u)P(u)}{P(g)}$$ 
It is equivalent to maximize $P(g|u)P(u)$, since the denominator does not depend on $u$. Moreover, the random variables of the measured data are Poisson distributed with expected values given by $(K u)_{i}$ and 
\begin{equation}
	P(g|u)=\prod_{i=1}^{N}\frac{(Ku)_{i}^{g_{i}}}{g_{i}!}e^{-(Ku)_{i}}
	\label{poisson}
\end{equation}
The Bayesian approach allows to consider additional information to our model with an appropriate prior probability of the image $u$, see \cite{Chan}, \cite{Le}. The most frequently used prior densities are the Gibbs priors, i.e., 
\begin{equation}
	P(u)=e^{-\alpha J(u)}
	\label{Gibbs}
\end{equation}
 where $\alpha>0$ is a regularisation parameter and $J(u)$ is a convex energy functional. Instead of maximising $P(g|u)P(u)$, we minimise $-\log(P(g|u)P(u))$. Hence, we seek a minimiser of the following problem
\begin{equation}
\argmin_{u\geq0}\left\{\sum_{i=1}^{N}((Ku)_{i}-g_{i}\log(Ku)_{i})+\alpha J(u)\right\},
\label{discreteEM}
\end{equation}
where the first term is the so-called Kullback-Leibler divergence of $u$ and $g$. This often serves as a motivation to consider the continuous variational problem
\begin{equation}
\argmin_{u\geq0}\left\{\int (Ku-g\log(Ku))+\alpha J(u)\right\}
\label{m3}
\end{equation}
In the case where $J\equiv0$, the first optimality condition in \eqref{m3} yields the following iterative scheme, known as EM algorithm
$$
u^{k+1}=\frac{u^{k}}{K^{*}1}K^{*}(\frac{g}{Ku_{k}})
$$ 
Additionally imposing prior information on the solution, e.g., that the solutions has a small total variation, leads to an extension of the EM algorithm, e.g., the EM-TV algorithm \cite{bardsley2010theoretical,Brune}. See also \cite{Alex,Alex1,bardsley2010hierarchical,willett2010poisson} for related approaches and \cite{Brune,Brune1} for extensions of EM-TV to Bregmanized total variation regularization.\par
%

\subsection*{Outline} The rest of the paper is organised as follows: in the next section we prove existence, uniqueness and stability results for our variational model in the continuous setting. In section \ref{sec:numerics}, we focus on solving numerically our problem using the split Bregman method and present our numerical simulations in section 4. 


\section{TV regularisation on image and sinogram}

In this section, we will discuss the well-posedness of our minimisation problem \eqref{v1}. To do so, we first rewrite \eqref{v1} for image functions $u$ that are defined on a bounded and open domain $\Omega\subset\mathbb R^2$ including sufficiently large balls around zero.  We consider the following problem
\begin{equation}
\argmin_{u\in BV(\Omega), ~u\geq0\textrm{ a.e. in }\Omega}\left\{F(u) = \alpha|Du|(\Omega)+\beta |D(\rt u)|(\Sigma)+\frac{1}{2}\int_{\Sigma}\frac{(g-\rt u)^{2}}{g}\right\}
\label{v1_1}
\end{equation}
We enforce prior information in terms of regularisation on both the image and its sinogram. Note, that the TV regularisation on the sinogram in \eqref{v1_1} has a different effect on the reconstructed image compared to regularising in image space ($\beta=0$) only. Of course, regularisation of the image $u$ enforces a certain regularisation of the sinogram $\rt u$ as well. However, because of the nonlinear character of the total variation regularisation, TV regularisation of the sinogram is not equivalent to regularisation on the image and vice-versa. In \eqref{v1_1} the two types of TV regularisation impose different structures in the subgradients of the two terms. This is also emphasised in Section \ref{sec:bregmanerr} where the source condition \eqref{source} and the elements $p_{1},p_{2}$ are described. Indeed, beyond the topology imposed by the regulariser the structure imposed by the sub gradients of the total variation regularisers are crucial for the properties of a solution of \eqref{v1_1}. In particular, in what follows we will see that the additional TV regularisation on the sinogram can have a positive effect when reconstructing smooth, thin structures as can be observed in the myocardium for instance. This effect will be both motivated analytically by the characterisation of a solution of \eqref{v1_1} for the simple case when $g$ is the sinogram of a disc for instance, as well as experimentally verified by testing the method against some representative examples in the numerical part of the paper.

\par We start with some first observations that are crucial ingredients of the well-posedness analysis for \eqref{v1_1}. In order not to divide by zero in the weighted $L^2$ norm in \eqref{v1_1}, we first assume that there exists constant $c_{1}>0$ such that 
\begin{equation}
	0<c_{1}\leq g(\theta,s) \leq\norm{L^{\infty}(\Sigma)}{g}
	\label{positiveRadon}
\end{equation}
The constraint \eqref{positiveRadon} is not significantly restrictive in most medical experiments. Since $u$ is assumed to be nonnegative, this basically can be achieved if the lines in the Radon transform are confined to those intersecting the support of $u$, at least in a discretised setting.

Moreover, to justify the definition of $F(u)$ in \eqref{v1_1} over the admissible set $\{u\in BV(\Omega), ~u\geq0\textrm{ a.e. in }\Omega\}$ in Theorem \ref{thrm3} we show that the Radon transform of $u$ is again in $BV$. To do so it is important to assume that the object we wish to recover is compactly supported. Hence, we may assume that $\supp ~u\subset B_{r}\subset\Omega$, where $B_{r}$ is the ball with radius r centered at the origin. Consequently, \eqref{m2} implies that $\rt u(\theta,s)=0$, when $s\notin[-r,r]$ and the projection space becomes:
\begin{equation}
	\Sigma=\left\{(\theta,s):-r\leq s\leq r, 0\leq\theta<\pi\right\}
	\label{v4}
\end{equation}
If it is not stated otherwise, we will always assume that the reconstructed image is compactly supported. Note that, we allow negative values on the s variable and that we do not consider the Radon transform for $\theta=\pi$. Likewise, we may allow that $s\geq0$ and $0\leq\theta<2\pi$. Hence, we consider the Radon space $\Sigma$ as the surface of a half cylinder with radius 1. 
Moreover, since Dirac $\delta$ function is even, equation \eqref{m2} implies that the coordinates ($-s,\phi$) and ($s,\phi+\pi$) correspond to the same point in the Radon space. \par

\subsection{BV-continuity of the Radon transform}
Our first result deals with a continuity property for the Radon transform as a mapping operator for functions with bounded variation. A similar result is proved by M. Bergounioux and E. Tr\'elat \cite{Berg} in the three dimensional case and for bounded and axially symmetric objects.
In what follows we do not need this symmetry assumption, but prove that the Radon transform is BV continuous for compactly supported $u$ in two space dimensions.

\begin{thrm}
Let $u\in BV(\Omega)$ and the ball $B_{r}$ with radius r be its compact support, then $\rt u\in BV(\Sigma)$ and the Radon transform is BV continuous on the subspace of functions supported in $B_r$.
\begin{proof} 
It is well known that the Radon transform is L$^{1}$ continuous and the following estimate holds for $n\geq2$:
\begin{equation}
	\norm{L^{1}(\Sigma^{n})}{\rt u}\leq|S^{n-1}|\norm{L^{1}(\re^{n})}{u}
	\label{radon}
\end{equation}
Hence, to prove BV-continuity we need to prove that the variation of $\rt u$ over $\Sigma$ is finite and bounded by the BV norm of $u$, i.e.,
$$V(\rt u,\Sigma)=\sup\left\{\int_{\Omega}\rt u(\theta,s)\diverg g(\theta,s)d\theta ds:\mbox{ }g\in (C^{1}_{c}(\Sigma))^{2}\mbox{ , }\norm{\infty}{g}\leq1\right\}<\infty$$ The following equations can easily be derived by the geometry depicted in Figure \ref{fig1}, where $(x,y)$ is the annihilation point and $t$ runs through the line $L$:
\begin{align}
x=s\cos\theta-t\sin\theta\label{eqq1}\\
y=s\sin\theta+t\cos\theta\label{eqq2}
\end{align}
We may also assume that $t\in[-r,r]$. Therefore,
\begin{align*}
\int_{\Omega}\rt u(\theta,s)\diverg g(\theta,s)d\theta ds&=\int_{0}^{\pi}\int_{-r}^{r}\int_{-r}^{r}u(s\cos\theta-t\sin\theta,s\sin\theta+t\cos\theta)\\
                                                     &\diverg g(\theta,s)\mbox{ }dtdsd\theta\\
                                                     &=\int_{0}^{\pi}\int_{-r}^{r}\int_{-r}^{r}u(x,y)\left[\nabla g_{1}\cdot\vec{\alpha}+\nabla g_{2}\cdot\vec{\theta}\right]dxdyd\theta
\end{align*}
where $\vec{\alpha}=(-y,x)$, $\vec{\theta}=(\cos\theta,\sin\theta)$ and in the above calculations we have used equations \eqref{eqq1},\eqref{eqq2}. Define, $\vec{G}(x,y)=(G_{1}(x,y),G_{2}(x,y)))$ with 
\begin{align*}
G_{1}(x,y)&=\int_{0}^{\pi}-yg_{1}(\theta,x\cos\theta+y\sin\theta)+g_{2}(\theta,x\cos\theta+y\sin\theta)\cos\theta d\theta\\
G_{2}(x,y)&=\int_{0}^{\pi}xg_{1}(\theta,x\cos\theta+y\sin\theta)+g_{2}(\theta,x\cos\theta+y\sin\theta)\sin\theta d\theta
\end{align*}
then,
$$\diverg G(x,y)=\frac{\partial g_{1}}{\partial x}+\frac{\partial g_{2}}{\partial y}=\int_{0}^{\pi}\left(\nabla g_{1}\cdot\vec{\alpha}+\nabla g_{2}\cdot\vec{\theta}\right)d\theta$$ The function $G$ lies in $C^{1}(\re^{2})$ and if we restrict $G$ on $\Omega$ and consider $G\chi_{B_{r}}$ then $G\in C_{c}^{1}(\Omega)$. Moreover, 
$$|G_{1}(x,y)|\leq\int_{0}^{\pi}|y||g_{1}|+|g_{2}|\leq\pi\norm{\infty}{g}(1+|y|)\leq\pi\norm{\infty}{g}(1+r)=C$$
and $|G_{2}(x,y)|\leq C$. If we set 
$$A=\int_{-r}^{r}\int_{-r}^{r}u(x,y)\diverg G(x,y)dxdy=C\int_{-r}^{r}\int_{-r}^{r}u(x,y)\diverg(\frac{G(x,y)}{C})dxdy$$ and $$B=\int_{\Sigma}\rt u(\theta,s)\diverg g(\theta,s)dsd\theta$$ then, taking the supremum over all $G\in C_{c}^{1}(\Omega)$ with $\norm{\infty}{G/C}\leq1$, we have that $B=C\cdot V(u,\Omega)$. Similarly, for all $g\in (C_{c}^{1}(\Sigma))^{2}$ with $\norm{\infty}{g}\leq1$, we conclude that $$V(\rt u,\Sigma)\leq\pi(1+r)V(u,\Omega)<\infty$$ Therefore, $\rt u\in BV(\Sigma)$ and the variation coincides with the total variation $|D(\rt u)|(\Sigma)$. By the corresponding norm defined on the BV space and equation \eqref{radon}, we deduce that $$\norm{BV(\Sigma)}{\rt u}\leq\pi(1+r)\norm{BV(\Omega)}{u}.$$
\end{proof}
\label{thrm3}
\end{thrm}

\subsection{Existence and Uniqueness}

Next, we show existence and uniqueness of the minimiser for the problem \eqref{v1_1}. 

\begin{thrm}\label{exist}
Let $\alpha>0$, $\beta\geq0$ and $g\in L^{\infty}(\Sigma)$ a strictly positive function. Then the functional $F(u)$ in \eqref{v1_1} is lower semicontinuous and strictly convex and the minimisation problem \eqref{v1_1} attains a unique solution $u\in BV(\Omega)\cap L_+^1(\Omega)$.
\end{thrm}

\begin{proof}Let $(u_{n})_{n}\in BV(\Omega)$ be a minimising sequence of nonnegative functions,
then in particular there exists a constant $C_{1}>0$ such that 
\begin{equation}
F(u)=\alpha|Du_{n}|(\Omega)+\beta |D(\rt u_{n})|(\Sigma)+\frac{1}{2}\int_{\Sigma}\frac{(g-\rt u_{n})^{2}}{g}<C_{1}
\end{equation}
Let $\overline{u_{n}}=\frac{1}{|\Omega|}\int_{\Omega}u_{n}dx$, then by the Poincar\'e-Wirtinger inequality \cite{Ambrosio}, we can find a constant $C_{2}>0$ such that 
\begin{equation}
	\norm{L^{2}(\Omega)}{u_{n}-\overline{u_{n}}}\leq C_{2}|Du_{n}|(\Omega)
	\label{v5}
\end{equation}
Therefore $$\norm{L^{2}(\Omega)}{u_{n}}\leq C_{2}|Du_{n}|(\Omega)+|\int_{\Omega}u_{n}dx|$$ Following the proof of \cite{Luminita}, we set $v_{n}=u_{n}-\overline{u_{n}}$ and since 

\begin{align*}
C_{1}&\geq\int_{\Sigma}\frac{(g-\rt u_{n})^{2}}{g}\geq\frac{1}{\norm{L^{\infty}(\Sigma)}{g}}\norm{L^{2}(\Sigma)}{g-\rt u_{n}}^{2}
\end{align*}

one can prove that 

\begin{align*}
\norm{L^{2}(\Sigma)}{\rt\overline{u_{n}}}&\leq C_{1}\norm{L^{\infty}(\Sigma)}{g}+\norm{L^{2}(\Sigma)}{\rt v_{n}}+\norm{L^{2}(\Sigma)}{g}\\
                                         &\leq \widetilde{C_{1}}\norm{L^{\infty}(\Sigma)}{g}+\norm{L^{2}(\Sigma)}{\rt v_{n}}
\end{align*}

and
$$|\frac{1}{|\Omega|}\int_{\Omega}u_{n}dx|\cdot\norm{L^{2}(\Sigma)}{\rt\chi_{\Omega}}=\norm{L^{2}(\Sigma)}{\rt\overline{u_{n}}}$$ Without loss of generality, we may assume that the image domain $\Omega$ is a unit square, then $\rt\chi_{\Omega}\neq0$, see \cite[Chapter~8]{Poul} and we conclude that $|\int_{\Omega}u_{n}dx|$ is uniformly bounded. Hence, $u_{n}$ is  $L^{1}(\Omega)$ bounded ( $L^{2}(\Omega)$ bounded with $|\Omega|<\infty$). Moreover, since the Radon transform is $L^{2}$ continuous for functions with compact support (cf. \cite{Hertle},\cite{Natt1}) and using \eqref{radon}, we have the following:

Since, $(u_{n})_{n\in\na}$ is bounded in  $L^{1}(\Omega)$ and $|Du_{n}|(\Omega)<\infty$ i.e., is $BV(\Omega)$ bounded, we obtain a subsequence $(u_{n_{k}})_{k\in\na},u\in BV(\Omega)$ such that $u_{n_{k}}$ converges weakly$^{*}$ to u. Also, $u_{n_{k}}$ converges weakly to u in $L^{2}(\Omega)$. Then,
\begin{align}
&\rt u_{n_{k}}\rightarrow\rt u\mbox{  in  }L^{1}(\Sigma)\label{limit1}\\
&\rt u_{n_{k}}\rightharpoonup\rt u\mbox{  in  }L^{2}(\Sigma)\label{limit2}
\end{align}

Then, $$|D(\rt u)|(\Sigma)\leq\liminf_{k\rightarrow\infty}|D(\rt u_{n_{k}})|(\Sigma)$$ 
and the weak lower semicontinuity of the $L^{2}$ norm and the lower semicontinuity of total variation semi-norm for both the image and the projection space imply that
$$F(u)\leq\liminf_{k\rightarrow\infty}F(u_{n_{k}})$$

To prove uniqueness let $0\leq u_{1}, u_{2}\in BV(\Omega)$ be two minimisers. If $\rt u_{1}\neq\rt u_{2}$, then the strict convexity of the weighted $L^{2}$ fidelity term  together with the convexity of the total variation of $\mathcal R u$ implies that: 
$$
F\left(\frac{u_1+u_2}{2}\right)<\frac{F(u_1)}{2}+\frac{F(u_2)}{2}=\inf_{\substack{u \in BV(\Omega)\\ u \geq 0 \mbox{a.e.} }} F(u)
$$
which is a contradiction. Hence, $\rt u_{1}=\rt u_{2}$ and using the well-known \textit{Slice-Projection} theorem i.e., 

$$\mathcal{F}(\rt_{\theta} u(s)) = (2\pi)^{\frac{n-1}{2}}\mathcal{F}_{n} (u(s\theta))$$  where the right hand side denotes the n-dimensional Fourier transform, we conclude that $u_{1}=u_{2}$, see also \cite{And}, \cite{Natt1} for more details. 

\end{proof}


\subsection{Stability}

Further, we discuss the stability of problem \eqref{v1_1} in terms of a small perturbation on the data. Following the approach of Acar and Vogel in \cite{Acar}, we consider a perturbation on the projection space i.e.,
\begin{equation}
	g_{n}=g+\tau_{n}\mbox{  with  }\norm{L^{2}(\Sigma)}{\tau_{n}}\rightarrow0
\end{equation}
and define the corresponding minimisation problem on the perturbed functionals:
\begin{equation}
	\argmin_{u\geq0\mbox{ a.e, }\\u\in BV(\Omega)}\left\{F^{n}(u)=\alpha|Du|(\Omega)+\beta |D(\rt u)|(\Sigma)+\frac{1}{2}\int_{\Sigma}\frac{(g_n-\rt u)^{2}}{g_n}\right\}
	\label{v6}
\end{equation}
For \eqref{v6} to be well-defined we assume an $L^\infty$ bound on $\tau_n$ such that $g_n$ is still positive. More precisely we assume that
\begin{equation}
	0<c_{1}\leq g_{n}(\theta,s) \leq\norm{L^{\infty}(\Sigma)}{g}+\varepsilon\mbox{,  for all }n\geq1,
	\label{positiveRadon1}
\end{equation}
which is the same as assuming that the perturbations $\tau_n$ are bounded from above by a small enough constant. Then, from the previous section, we have that both $F^{n}$ and $F$ are lower semicontinuous, strictly convex with unique minimisers $u_{n}$ and $u^{*}$ respectively. In a sense, we will prove that for a small change on our data g, our solution's behaviour does not change significantly. Before, we proceed with the stability analysis we need to ensure that the functional is indeed BV-coercive. 
That is coercive with respect to the bounded variation norm $\norm{BV(\Omega)}{u}=\norm{L^{1}(\Omega)}{u}+|Du|(\Omega)$, rather than the total variation semi-norm only. 

\begin{lem} \label{stab}
Let $g\in L^{\infty}(\Sigma)$ a strictly positive and bounded function, then the functional $F$ in \eqref{v1_1} is BV coercive i.e., there exists a constant $C>0$ such that 
\begin{equation}
	F(u)\geq C\norm{BV(\Omega)}{u}
\end{equation}

\begin{proof} Let $u\geq 0$ a.e with $u\in BV(\Omega)$ and consider $v=u-\overline{u}$. Then, by H\"older and Poincar\'e inequalities, one can prove that
$$\norm{L^{p}(\Omega)}{v}\leq C_{1}|Dv|(\Omega)$$ and the corresponding estimate for the BV norm holds:
\begin{equation}
\norm{BV(\Omega)}{u}\leq\norm{L^{1}(\Omega)}{\overline{u}}+(C_{1}+1)|Dv|(\Omega)\label{v7}
\end{equation}
Note that in the above calculations we have used the fact that $|Du|(\Omega)=|Dv|(\Omega)$. Moreover, we know that there exists a constant $C_{2}>0$ such that $$\norm{L^{2}(\Sigma)}{\rt\overline{u}}=C_{2}\norm{L^{1}(\Omega)}{\overline{u}}$$ since $\rt\chi_{\Omega}\neq0$ (see proof of Theorem \ref{exist}). Hence, we can derive the following bound:
\begin{equation}
F(u)\geq\alpha|Dv|(\Omega)+\frac{C_{2}\norm{L^{1}(\Omega)}{\overline{u}}}{2\norm{L^{\infty}(\Sigma)}{g}}\left(C_{2}\norm{L^{1}(\Omega)}{\overline{u}}-2\left(\norm{}{\rt}C_{1}|Dv|(\Omega)+\norm{L^{\infty}(\Sigma)}{g}\right)\right)
\label{v8}
\end{equation}
Setting $$A=C_{2}\norm{L^{1}(\Omega)}{\overline{u}}-2\left(C_{1}\norm{}{\rt}|Dv|(\Omega)+\norm{L^{\infty}(\Sigma)}{g}\right)$$ we consider 2 cases:
\begin{enumerate}[(a)]
\item If $A\geq1$, then using \eqref{v7},\eqref{v8}, one can prove that 
\begin{equation}
	F(u)\left(\frac{C_{1}+1}{\alpha}+\frac{2\norm{L^{\infty}(\Sigma)}{g}}{C_{2}}\right)\geq\norm{BV(\Omega)}{u}
	\label{coercive1}
\end{equation}
\item If $A\leq1$, then $$\norm{L^{1}(\Omega)}{\overline{u}}\leq\frac{1+2\left(\norm{}{\rt}C_{1}|Dv|+\norm{L^{\infty}(\Omega)}{g}\right)}{C_{2}}$$ 
\end{enumerate}
and using equation \eqref{v7} we derive that:
\begin{equation}
\norm{BV(\Omega)}{u}-\frac{1+2\norm{L^{2}(\Sigma)}{g}}{C_{2}}\leq\big(\frac{2C_{1}\norm{}{\rt}}{C_{2}}+C_{1}+1\big)|Dv|\leq\frac{K}{\alpha}F(u)
\label{coercive2}
\end{equation}
where $K=\frac{2C_{1}\norm{}{\rt}}{C_{2}}+C_{1}+1$. From equations \eqref{coercive1}, \eqref{coercive2} we have that the functional $F$, is BV coercive.
\end{proof}
\end{lem}

Moreover, we can prove that given constants $C>0$ and $\varepsilon>0$, there exists $n_{0}\in\na$ such that
\begin{equation}
	|F^{n}(u)-F(u)|\leq\epsilon\mbox{  for  }n\geq n_{0}\mbox{  and  }\norm{BV(\Omega)}{u}\leq C.
	\label{cons}
\end{equation}
Indeed,
\begin{align*}
|F^{n}(u)-F(u)|&=\frac{1}{2c_{1}}\left(\norm{L^{2}(\Sigma)}{g+\tau_{n}-\rt u}^{2}-\norm{L^{2}(\Sigma)}{g-\rt u}^{2}\right)\\                
 &\leq \frac{1}{2c_{1}}\left(\norm{L^{2}(\Sigma)}{\tau_{n}}^{2}+2\scalprod{ }{\tau_{n}}{g-\rt u}\right)\\ 
        &\leq\frac{1}{2c_{1}}\norm{L^{2}(\Sigma)}{\tau_{n}}\left(\norm{L^{2}(\Sigma)}{\tau_{n}}+2\norm{L^{2}(\Sigma)}{g}+2\norm{L^{2}(\Sigma)}{\rt u}\right)
\end{align*}
The continuity of Radon transform in $L^{2}$ for functions with compact support, i.e., $$\norm{L^{2}}{\rt u}^{2}\leq|S^{n-1}|(2r)^{n-1}\norm{L^{2}}{u}^{2}$$ and $BV\hookrightarrow L^{2}$ continuously, imply that we can find an appropriate constant such that \eqref{cons} is valid. 

With these preparations we can prove the following weak stability result for minimisers of \eqref{v1_1}.

\begin{thrm} Let $0<u_{n},u^{*}\in BV(\Omega)$ be the minimisers of the functionals $F^{n}$ and $F$ defined in \eqref{v6} and \eqref{v1_1} respectively. Then 
\begin{equation}
u_{n}\rightharpoonup u^{*}\mbox{  in  }L^{2}
\label{convergence1}
\end{equation}
\begin{proof}
Observe that $F^{n}(u_{n})\leq F^{n}(u^{*})$ and using \eqref{cons} we have that $$\liminf_{n\rightarrow\infty}F^{n}(u_{n})\leq\limsup_{n\rightarrow\infty}F^{n}(u_{n})\leq F(u^{*})<\infty$$ Lemma \ref{stab} implies that $(u_{n})_{n\in\na}$ is BV bounded. Assume that \eqref{convergence1} is not true, then there exists a subsequence $u_{n_{k}}$ which converges weakly to some $u\neq u^{*}$ in $L^{2}$. Hence,
\begin{align*}
F(u)&\leq\liminf_{n\rightarrow\infty}F(u_{n_{k}})\\
        &=\lim_{k\rightarrow\infty}(F^{n_{k}}(u_{n_{k}})-F(u_{n_{k}}))+\liminf_{k\rightarrow\infty}F^{n_{k}}(u_{n_{k}})\\
        &\leq F(u^{*})
\end{align*}
which is a contradiction to the uniqueness of minimiser of $F$.
\end{proof}  
\end{thrm}

\subsection{Error analysis using the Bregman distance}\label{sec:bregmanerr}

In the following we discuss a similar approach as presented in \cite{Martin1} for deriving an error estimate for our model \eqref{v1_1} in terms of the Bregman distance. Let us note that what follows holds for the more general minimisation problem
\begin{equation}
	\argmin_{u\in X}\left\{F(u)=\alpha J(u)+\beta J(\rt u)+\frac{1}{2}\int_{\Sigma}\frac{(g-\rt u)^{2}}{g}\right\},
	\label{alternative}
\end{equation}
where $J:X\rightarrow\re$ is a convex functional and $X$ is a Banach space such that $\rt: X \rightarrow L^2(\Sigma)\cap X$ is a bounded operator. Before we proceed with proving an error estimate for \eqref{alternative}, we first recall the terminology of a minimising solution, the source-condition and the Bregman distance for a convex functional.
\begin{mydef}\label{def:minsol}
An element $\widetilde{u}\in X$ is called a minimising solution of $\rt u=g$ with respect to the functional $J:X\rightarrow\re$ if:
\begin{enumerate}[(i)]
\item $\rt \widetilde{u}=g$
\item $J(\widetilde{u})\leq J(v)\mbox{   ,}\forall v\in X\mbox{, }\rt v=g$
\end{enumerate}
\end{mydef}
We consider
the following source condition for an element $\widetilde u$
\begin{equation}
	\exists \widetilde{w}\in L^{2}(\Sigma)\mbox{  such that  }\rt^{*}\widetilde{w}\in\partial J(\widetilde{u}),
	\label{source}
\end{equation}
where $\partial J(u)$ is the subdifferential of $J$ at $u$, see \cite{Ekeland}. 

Next, we recall the Bregman distance for a convex functional $J$ together with some of its basic properties as it was introduced in \cite{Bregman}.
\begin{mydef}
Let $u,v\in X$ and $J:X\rightarrow\re$ convex functional, then the Bregman distance related to $J$, with $J(u)<\infty$, for all $u\in X$ is 
\begin{equation}
	D_{J}^{p}(u,v):=J(u)-J(v)-\scalprod{}{p}{u-v}\mbox{  ,  }p\in\partial J(v)
\end{equation}
\end{mydef}


Now, we can derive an estimate for the difference of a minimising solution $\widetilde u$ in Definition \ref{def:minsol} and a regularised solution $\widehat u$ of \eqref{alternative}. \par

Let $\alpha>0,\beta\geq0$ and the data $g$ fulfil \eqref{positiveRadon}. Then, for a minimiser $\widehat{u}$ of \eqref{alternative} and the exact solution $\widetilde{u}$ satisfying $\rt\widetilde{u}=f$ with a fixed noise bound $\norm{L^{2}(\Sigma)}{g-f}\leq\delta$ from the exact data $f$, we have
\begin{align*}
&\alpha J(\widehat{u})+ \beta J(\rt \widehat{u})+\frac{\norm{2}{g-\rt\widehat{u}}^{2}}{2\norm{L^{\infty}}{g}}\leq\alpha J(\widetilde{u})+ \beta J(\rt \widetilde{u})+\frac{\delta^{2}}{2c_{1}}\Leftrightarrow\\
&\alpha D_{J}^{p_{1}}(\widehat{u},\widetilde{u})+\alpha<p_{1},\widehat{u}-\widetilde{u}>+\beta D_{J}^{p_{2}}(\rt\widehat{u},f)+\beta<p_{2},\rt\widehat{u}-f>+\frac{\norm{2}{g-\rt\widehat{u}}^{2}}{2\norm{L^{\infty}(\Sigma)}{g}}\leq\frac{\delta^{2}}{2c_{1}}
\end{align*}
where, we have used the corresponding Bregman distances related to the functional $J$ regarding the image and the sinogram regularisation. Moreover, we require that 
\begin{equation}
	\partial( J(u)+J(\rt u))=\partial J(u) +\partial (J(\rt u))
	\label{sum_sub}
\end{equation}
holds, subject to the assumption that the related \textit{effective domains} have a common point, that is
\begin{equation}
	\dom J(u)\cap \dom J(\rt u)\neq\emptyset\mbox{ for some }u\in X
	\label{sum_sub_ass}
\end{equation}
In our case, this is valid due to Theorem \ref{thrm3}. Let
 \begin{center}
\begin{enumerate}[(i)]
\item $p_{1}\in\partial J(\widetilde{u})$ 
\item $p_{2}\in\partial (J(\rt\widetilde{u}))=\rt^{*}(\partial J(\rt \widetilde{u}))\in \rt^{*} w_{2}$
\end{enumerate}
\end{center}
Moreover, assume that the source condition \eqref{source} is satisfied with respect to $J$, that is
$$
\exists p_{1}\in\partial J(\widetilde{u})\mbox{   s.t   }p_{1}=\rt^{*}w_{1}\mbox{  ,  }w_{1}\in L^{2}(\Sigma)
$$
Then, by generalised Young's inequality, for every $\varepsilon>0$ we have
$$
ab\leq\frac{a^2}{2\varepsilon}+\frac{\varepsilon b^{2}}{2}
$$ 
and we conclude that
\begin{align*}
&\alpha D_{J}^{p_{1}}(\widehat{u},\widetilde{u})+\beta D_{J}^{p_{2}}(\rt\widehat{u},f)+<\alpha w_{1}+\beta p_{2},\rt\widehat{u}-f>+
\frac{\norm{2}{g-\rt\widehat{u}}^{2}}{2\norm{L^{\infty}(\Sigma)}{g}}\leq\frac{\delta^{2}}{2c_{1}}\Leftrightarrow\\
&\alpha D_{J}^{p_{1}}(\widehat{u},\widetilde{u})+\beta D_{J}^{p_{2}}(\rt\widehat{u},f)+\frac{\norm{2}{g-\rt\widehat{u}}^{2}}{2\norm{L^{\infty}(\Sigma)}{g}}
\leq\frac{\delta^{2}}{2c_{1}}+<\alpha w_{1}+\beta p_{2},f-\rt\widehat{u}+g-g>\Leftrightarrow\\
&\alpha D_{J}^{p_{1}}(\widehat{u},\widetilde{u})+\beta D_{J}^{p_{2}}(\rt\widehat{u},f)+\frac{\norm{2}{g-\rt\widehat{u}}^{2}}{2\norm{L^{\infty}(\Sigma)}{g}}
\leq \frac{\delta^{2}}{2c_{1}}+\frac{\norm{2}{\alpha w_{1}+\beta p_{2}}^{2}}{\varepsilon}+\frac{\varepsilon}{2}\norm{2}{g-\rt\widehat{u}}^{2}+\frac{\varepsilon\delta^{2}}{2}\Leftrightarrow\\
\end{align*}

Hence, for $\varepsilon=\norm{L^{\infty}(\Sigma)}{g}^{-1}>0$ we have
$$
D_{J}^{p_{1}}(\widehat{u},\widetilde{u})+\frac{\beta}{\alpha} D_{J}^{p_{2}}(\rt\widehat{u},f)\leq\frac{\widetilde{c_{1}}\delta^{2}}{\alpha}+\alpha\norm{L^{\infty}(\Sigma)}{g}\norm{2}{w_{1}+\frac{\beta}{\alpha}\rt^{*}w_{2}}^{2}
$$

We have proved the following theorem:
\begin{thrm}\label{thmbregstab} Let $\delta>0$ be the noise bound related to the exact data $f$ and the noise data $g$. Moreover let \eqref{sum_sub} hold. If $\widehat{u}$ is a minimiser of \eqref{alternative} and $\widetilde{u}$ the exact solution of $\rt\widetilde{u}=f$ which satisfies the source condition \eqref{source}, then for $\alpha>0,\beta\geq0$ we have the following estimate:
\begin{equation}
D_{J}^{p_{1}}(\widehat{u},\widetilde{u})+\frac{\beta}{\alpha}D_{J}^{p_{2}}(\rt\widehat{u},f)\leq\frac{\widetilde{c_{1}}\delta^{2}}{\alpha}+\alpha \norm{L^{\infty}(\Sigma)}{g}\norm{2}{w_{1}+\frac{\beta}{\alpha}\rt^{*}w_{2}}^{2}
\label{bregmanerror}
\end{equation}
where $\widetilde{c_{1}}=\frac{c_{1}+\norm{L^{\infty}(\Sigma)}{g}}{2c_{1}\norm{L^{\infty}(\Sigma)}{g}}$.
\end{thrm}
For $\beta=0$ Theorem \ref{thmbregstab} recovers the same estimates presented in \cite[Theorems~1,2]{Martin}. In the case $\beta>0$ the additional term $\rt^{*}w_{2}$, due to the source condition for total variation regularisation on the sinogram, might give room for further improvement. It is a matter of future research to improve the estimate in \eqref{bregmanerror}, where we believe that in certain cases the term related to the sinogram regularisation produces a better bound compared to no penalisation on the projection space.

\subsection{An explicit example of TV regularisation on the sinogram}\label{explicit}

Before we continue with the numerical presentation, we discuss how a regularised solution in the projection space behaves in terms of an appropriate positive parameter $\beta$. In particular, we derive an explicit solution of the weighted ROF minimisation problem for the sinogram

\begin{equation}
	\argmin_{v\geq0\mbox{ a.e }}\left\{J(v) = \beta|Dv|(\Sigma)+\frac{1}{2}\int_{\Sigma}\frac{(g-v)^2}{g}\right\},
	\label{barosROF}
\end{equation}

where we consider,
\begin{align}
  &u(x,y)=
	\begin{cases}
  1,&\mbox{ if  }x^2+y^2\leq r\\
  0,&\mbox{ otherwise}
\end{cases}\label{char1}\\
 & g(\theta,s) = \rt_{\theta} u(s)=
 \begin{cases}
 2\sqrt{r^2-s^2},&\mbox{ for }|s|<r\\
 0,              &\mbox{ otherwise}
 \label{char2}
 \end{cases}
\end{align}

In Figure \eqref{sxima_exact_solution:sin}, the given sinogram $g$ and the corresponding regularised solution $v$ of \eqref{barosROF} for $\beta=10$ is shown. We make the following Ansatz for a solution of \eqref{barosROF}
\begin{equation}
	v(s)=
	\begin{cases}
	\delta = g(\kappa),&\mbox{ for }|s|\leq \kappa,\\
	g(s), &\mbox{ for } \kappa<|s|<r,\\
	0,&\mbox{ otherwise.}
	\end{cases}
	\label{candidate}
\end{equation}


\begin{figure}[h!]
\begin{center}
\begin{subfigure}[h]{0.3\textwidth}
                \centering                                                  
                \includegraphics[scale=0.25]{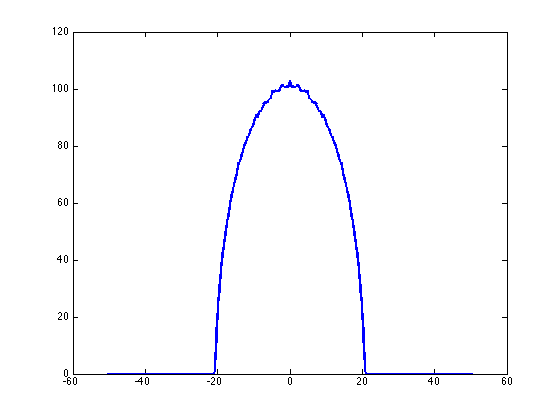}
                \caption{Given sinogram $g$} 
\end{subfigure}
\begin{subfigure}[h]{0.3\textwidth}
                \centering 
                \includegraphics[scale=0.26]{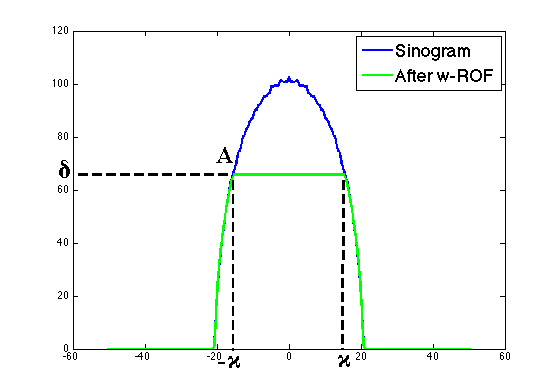}
                \caption{TV regularised sinogram $v$} 
                \label{sxima_exact_solution:sin}
\end{subfigure}          
\caption{The original sinogram $g$ with $r=50.5$, plotted at $45$ degrees in (a) and the regularised sinogram $v$ with $\beta=10$ in (b).}
\label{sxima_exact_solution}
\end{center}
\end{figure}

Note that, since $g\in C(-r,r)$, a solution $v$ of \eqref{barosROF} is in $C(-r,r)$ and hence also in $W^{1,1}(-r,r)$ \cite{caselles2007discontinuity}. Therefore, $|Dv|(\Sigma) = \int_\Sigma |\nabla v| ~ dx$. Then, if we plug-in \eqref{candidate} in \eqref{barosROF}, we obtain
\begin{align*}
	& \argmin_{v\geq0\mbox{ a.e }}\left\{\beta\norm{1}{\nabla v}+\frac{1}{2}\int_{\Sigma}\frac{(g-v)^2}{g}\right\} \\
	& = 
	\argmin_{v\geq0\mbox{ a.e }}\left\{4\beta\sqrt{r^2-\kappa^2}+\int_{0}^{\kappa}\left(2\sqrt{r^2-s^2}+\frac{v^2}{2\sqrt{r^2-s^2}}-2v\right)ds\right\}
\end{align*}


which can be simplified to
\begin{equation}	\argmin_{\kappa}\left\{(4\beta-3\kappa)\sqrt{r^2-\kappa^2}+(3r^2-2\kappa^2)\mathrm{arcsin}(\frac{\kappa}{r})\right\}
	\label{exactminimisation}
\end{equation}



Numerically solving \eqref{exactminimisation} under the constraint $0<|\kappa|<r$, we obtain a value for $\kappa$ that we can substitute in \eqref{candidate} and find the corresponding value of our solution after the regularisation. 
We solve \eqref{exactminimisation} with MATLAB's built-in routine \textit{fminbnd} in $\kappa\in[0,r)$. In Figure \ref{beta_delta}, we present how the $\beta$ parameter relates to the constant height value $\delta$ of the computed regularised solution. Clearly, for small values of $\beta$, there is no significant effect of the total variation regularisation but as we increase $\beta$ we have that $\delta$ decreases to zero,  while $\kappa$ tends to $r$. 
%
%
\begin{figure}[h!]
\begin{center}                                                  
\includegraphics[scale=0.35]{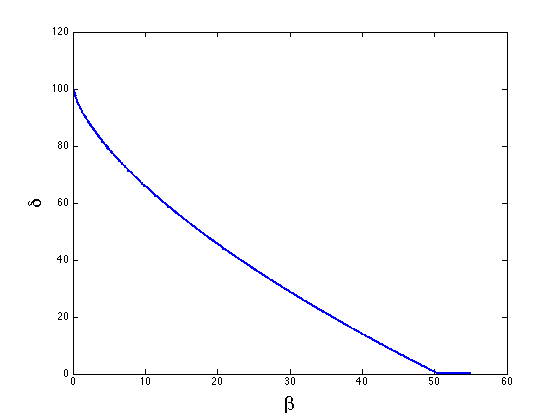}
\end{center}         
\caption{The relation  between the regularisation parameter $\beta$ and $\delta$ in \eqref{candidate}, computed using \eqref{exactminimisation} for the example in Figure \ref{sxima_exact_solution}. The parameter $\beta$ varies from $0.001$ to $55$ with step size $0.1$.}
\label{beta_delta}
\end{figure}

Before we apply the inverse Radon transform on \eqref{candidate} and find the corresponding solution in the image space, we need to verify its optimality. The following theorem ensures that the candidate solution \eqref{candidate} for the problem \eqref{barosROF} is indeed optimal.

\begin{thrm}
The unique solution of the minimisation problem \eqref{barosROF} is defined by \eqref{candidate}.
\begin{proof} The optimality condition on \eqref{barosROF} implies that:

\begin{equation}
0\in \partial J(u) \Leftrightarrow \beta q +\frac{v-g}{\sqrt{g}}=0\mbox{,  }q\in\partial |Du|(\Omega)
\label{opt_weighted}
\end{equation}

We can characterise the subdifferential of total variation, see  \cite{benning2012ground}, as

\begin{equation}
\partial |Du|(\Sigma)=\left\{\diverg p: p\in C^{\infty}_{o}(\Sigma), \norm{\infty}{p}\leq1, \scalprod{}{\diverg p}{v}=|Du|(\Sigma)\right\}
\label{tvsub}
\end{equation}

Therefore, in our case \eqref{opt_weighted} becomes 
\begin{equation}
\beta p'(s)+\frac{v(s)-g(s)}{\sqrt{g(s)}}=0\mbox{,  in  }s\in\Sigma
\end{equation}

with $-1\leq p(s)\leq 1$ and $\int_{\Sigma}p'(s)v(s)=\int_{\Sigma}|v'(s)|$. If v is either increasing or decreasing on an interval $I\subset\Sigma$, then through integration by parts one obtains $p(s)v'(s)=|v'(s)|$ which immediately implies that $p'=0$ and $v=g$ on I. However, when $v\neq g$ on an interval $I'\subset\Sigma$, then $p'\neq 0$ which is true only if $v'(s)=0$ on $I'$, i.e., v is constant.

\end{proof}
\end{thrm}

For computing the regularised image that corresponds to a solution of \eqref{barosROF} we first note that the rotational symmetry of the object in image space allows to simplify the Radon transform and its inverse. In this case the Radon transform coincides with the so-called \textit{Abel transform}, cf.  \cite[Chapter~8]{Poul}. More precisely, if $u$ is a radial function and $u(x,y)=f(\sqrt{x^{2}+y^{2}})$ we have 
\begin{equation}
	\rt_{\theta}u(s)=2\int_{s}^{\infty}\frac{f(r)r}{\sqrt{r^{2}-s^{2}}}dr
	\label{Abel}
\end{equation}

Using \eqref{Abel}, we can recover analytically the solution $u$ for a regularised sinogram \eqref{candidate}. 
%
The Abel transform and the inverse Abel transform in this case are
\begin{align}
	\mathcal{A}(u(\widetilde{r}))(x)&=2\int_{x}^{\infty}\frac{\widetilde{r}u(\widetilde{r})}{\sqrt{\widetilde{r}^{2}-x^{2}}}d\widetilde{r}\label{Abel1}\\
	u(\widetilde{r})&=-\frac{1}{\widetilde{r}\pi}\frac{d}{d\widetilde{r}}\int_{\widetilde{r}}^{\infty}\frac{r\mathcal{A}(u(\widetilde{r}))(x)}{\sqrt{x^{2}-r^{2}}}dx\label{invAbel}
\end{align}

\begin{figure}[h!]
\begin{center}                                                  
\includegraphics[scale=0.8]{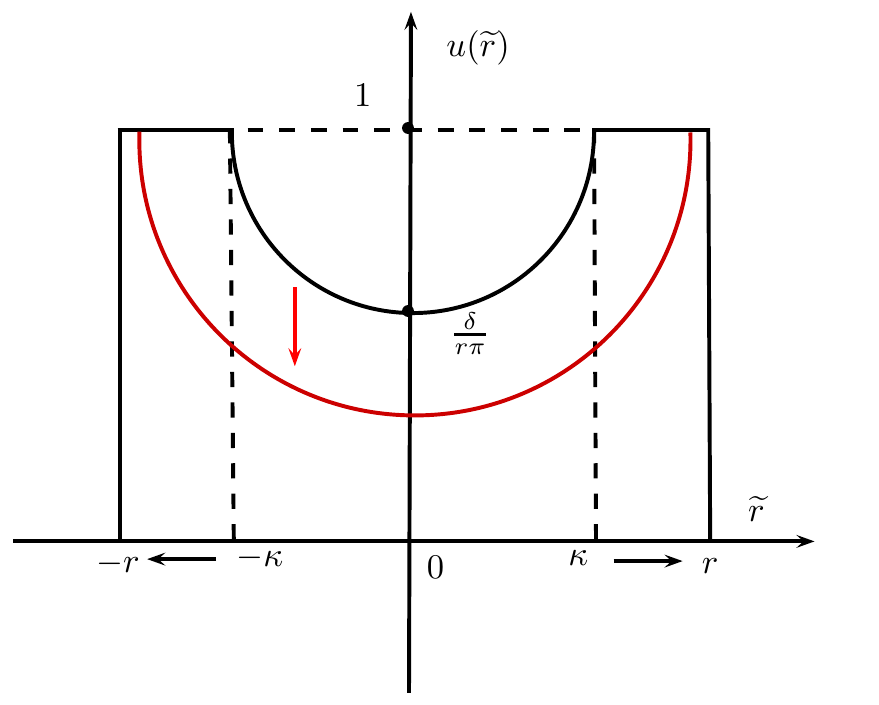}
\end{center}         
\caption{The solution $u(\widetilde{r})$ (solid line) given in \eqref{final_exact} inside the interval $[-r,r]$ and zero outside. The black and the red curve constitute the regularised solution for a smaller and larger value of $\beta$, respectively. The larger $\beta$ the more the solution concentrates around the boundaries of the disc.}
\label{u-analytic}
\end{figure}

Setting $u(\widetilde{r})=1$ and replacing the upper limit of the integral $\infty$ by $r$ in \eqref{Abel1}, the expression in \eqref{Abel1} matches the expression for the Radon transform in \eqref{char2}. Therefore, we plug-in \eqref{candidate} in \eqref{invAbel} and focus on the constant part of the sinogram for $-\kappa\leq\widetilde{r}\leq\kappa$,
\begin{align}
u(\widetilde{r})&=-\frac{1}{\widetilde{r}\pi}\frac{d}{d\widetilde{r}}\int_{\widetilde{r}}^{r}\frac{x\delta}{\sqrt{x^{2}-\widetilde{r}^{2}}}dx
                =\frac{\delta}{\pi\sqrt{r^2-\widetilde{r}^{2}}}\label{final_exact}
\end{align}
We observe that the reconstructed image is affected by the initial loss of contrast $\delta$ of the sinogram regularisation in \eqref{candidate} and depends radially on $\widetilde{r}$. In Figure \ref{u-analytic}, we present the regularised solution $u$ for two values of $\beta$. Recall, as we increase $\beta$ (red curve), we have that $\delta\rightarrow0$ and $\kappa\rightarrow r$.


\section{Numerical Implementation}\label{sec:numerics}

In this section we discuss the numerical solution of the minimisation problem \eqref{v1_1}. We employ the split Bregman technique \cite{Osher} which separates the problem into two subproblems -- one in image space and one in projection space -- that are solved iteratively in an alternating fashion. In order to present the numerical solution we start with formulating \eqref{v1_1} in a discrete setting. 

\subsection{Discrete Setting} 

Let $(u_{i,j})$, $i=1,\ldots, m$, $j=1,\ldots, n$ be the discretised image defined on a rectangular grid of size $m\times n$, $m,n>0$, and $(v_{i,j})$, $i=1,\ldots, k$, $j=1,\ldots, l$ the discretisation for an element in the sinogram space $\Sigma=[0,\pi)\times[-r,r]$ where k denotes the number of lines and l the number of angles. The values $u_{i,j}$ and $v_{i,j}$ are defined on two-dimensional grids. They are rearranged into one-dimensional vectors $u\in\re^{nm}$ and $v\in\re^{kl}$ by appending the columns of the array to each other, starting from the leftmost. Then, the discrete gradient for $u\in\re^{m\times n}$ is a matrix $\nabla \in \re^{2nm\times nm}$ which is the standard forward difference approximation of the gradient in the continuum. More precisely, applying the discrete gradient to $u$ gives $\nabla u=((\nabla u)_{1},(\nabla u)_{2})\in \re^{2nm}$ with

\begin{align*}
(\nabla u)_{1}(i,j)&=
                \begin{cases}
                 u(i,j+1)-u(i,j),&\mbox{    if   }1\leq i\leq n, 1\leq j< m,\\
                 0,             &\mbox{    if   }1\leq i\leq n, j=m.
                 \end{cases}\\                
(\nabla u)_{2}(i,j)&=
                \begin{cases}
                 u(i+1,j)-u(i,j),&\mbox{    if   }1\leq i< n, 1\leq j\leq m,\\
                 0,              &\mbox{    if   }i=n, 1\leq j\leq m.
                 \end{cases}
\end{align*}
The discrete divergence is defined as its adjoint, cf. \cite{Cham}, and is given by
\begin{align*}
\diverg :\re^{2nm}\rightarrow\re^{nm}&\mbox{   with  } \diverg(z)\cdot u=-z\cdot\nabla u
\end{align*}

Further, to approximate the Radon transform $\mathcal R$ we introduce the discrete Radon transform as a mapping $R: \re^{nm}\rightarrow \re^{kl}$ and its inverse $R^{-1}: \re^{kl} \rightarrow \re^{nm}$. In the numerical implementation the discrete Radon transform is represented by a sparse matrix $R\in\re^{kl\times nm}$ which acts on $u\in\re^{nm}$ to obtain a sinogram image $v\in\re^{kl}$. Defining $x(\theta_{\hat i},s_{\hat j})$, $\hat i=1,\ldots, k$, $\hat j = 1,\ldots, l$, the line defined by $\theta_{\hat i}, s_{\hat j}$, we can define for $i=1,\ldots, m$ and $j=1,\ldots,n$
\begin{equation}
\psi_{i,j}(\theta_{\hat i},s_{\hat j})=
\begin{cases}
1&\mbox{,  where the line } x(\theta_{\hat i},s_{\hat j}) \mbox{ goes through the pixel }  (i,j)  \\
0&\mbox{,  otherwise.} 
\end{cases}
\end{equation}
Using this notation and the linearity of the Radon transform, we define the discrete Radon transform as
\begin{equation}
R u(\theta_{\hat i},s_{\hat j})=\sum_{i=1}^{m}\sum_{j=1}^n u_{i,j} \rt\psi_{i,j}(\theta_{\hat i},s_{\hat j}) 
\label{radon_matrix}
\end{equation} 
where $\rt \psi_{i,j}(\theta_{\hat i},s_{\hat j})$ is equal to the length of the intersection of the projection line with the pixel $(i,j)$.

With these discrete quantities we define the discrete functional $F$ by
\begin{equation}
	F(u)=a\norm{1}{\nabla u}+\beta \norm{1}{(\nabla R u)}+\frac{1}{2}\sum_{k,l}\frac{(g-\rt u)^{2}}{g},
	\label{n1}
\end{equation}
and the discrete version of the minimisation problem \eqref{v1_1}
\begin{equation}
\min_{u\in \re^{m\times n}} F(u).
\label{n1a}
\end{equation}

\subsection{Split Bregman Algorithm}\label{sec:splitbreg}
To solve the problems defined in \eqref{n1} we employ the Bregman iteration \cite{Osher1} combined with a splitting technique. The resulting algorithm is called Split Bregman method which is proposed in \cite{Osher} to efficiently solve total variation and $\ell_1$ regularised image processing problems. The idea of this splitting procedure is to replace a complex and costly minimisation problem by a sequence of simple and cheaply to solve minimisation problems and to set up an iteration in which they are solved alternatingly. Note, that the Split Bregman method can be equivalently phrased in terms of an augmented Lagrange method and Douglas-Rachford splitting, cf. \cite{setzer2009split,esser2010general,setzer2011operator}. 
%
We follow \cite{Osher} to adapt the Split Bregman algorithm to the solution of \eqref{v1_1}. To do so, we consider
\begin{equation}
	\min_{\{u:\ u\geq0\mbox{ a.e.}\}}\alpha\norm{1}{\nabla u}+\beta \norm{1}{\nabla(R u)}+\frac{1}{2}\sum_{k,l}\frac{(g-R u)^{2}}{g}.
	\label{n5}
\end{equation}

We start with replacing \eqref{n5} by an equivalent constrained minimisation problem for two unknowns, the image $u\in \re^{m\times n}$ and the sinogram $v\in\re^{k\times l}$, related to each other by $v=R u$. This gives
\begin{equation}
	\min_{\{(u,v):\ u\geq0\mbox{ a.e.}\}}\alpha\norm{1}{\nabla u}+\beta \norm{1}{\nabla v}+\frac{1}{2}\sum_{k,l}\frac{(g-v)^{2}}{g}\mbox{   s.t   }v=R u.
	\label{n6}
\end{equation}
For computational efficiency reasons, we introduce three additional variables
\begin{equation}
	z=\nabla u, w=\nabla v\mbox{ and }u=\widetilde{u}
	\label{constraints}
\end{equation}
and rephrase \eqref{n6} again into
\begin{equation}
\min_{\{(u,\widetilde u,v,z,w):\ \widetilde u\geq0\mbox{ a.e.},~\mbox{satisfying \eqref{constraints}}\}}\alpha\norm{1}{z}+\beta \norm{1}{w}+\frac{1}{2}\sum_{k,l}\frac{(g-v)^{2}}{g}
\label{n6a}
\end{equation}
Then, we could iteratively solve the constrained minimisation problem \eqref{n6a} by Bregman iteration. Starting with initial conditions $b_1^0\in\re^{k\times l},b_2^0\in(\re^{k\times l})^{2},b_3^0\in(\re^{m\times n})^{2},b_4^0\in(\re^{m\times n})$ we iteratively solve for $k=0,1,\ldots$
\begin{align*}
	\argmin_{u,\widetilde{u},v,z,w}&\bigg\{\alpha\norm{1}{z}+\beta\norm{1}{w}+\sum\frac{(g-v)^{2}}{g}+\iota_{(\widetilde{u}>0)}+ \frac{\lambda_{1}}{2}\norm{2}{b_{1}^{k}+\rt u-v}^{2}+\frac{\lambda_{2}}{2}\norm{2}{b_{2}^{k}+\nabla v-w}^{2}\\
	                             &+ \frac{\lambda_{3}}{2}\norm{2}{b_{3}^{k}+\nabla u-z}^{2}+\frac{\lambda_{4}}{2}\norm{2}{b_{4}^{k}+u-\widetilde{u}}^{2}\bigg\}\\
	                             b_{1}^{k+1}&=b_{1}^{k}+\rt u^{k+1}-v^{k+1} \qquad
b_{2}^{k+1}=b_{2}^{k}+\nabla v^{k+1}-w^{k+1}\\
b_{3}^{k+1}&=b_{3}^{k}+\nabla u^{k+1}-z^{k+1} \qquad
b_{4}^{k+1}=b_{4}^{k}+u^{k+1}-\widetilde{u}^{k+1},
\end{align*}
with Lagrange multipliers $(\lambda_{i})_{i=1}^{4}>0$, $b_{1}^{k}\in\re^{k\times l}$, $b_{2}^{k}\in(\re^{k\times l})^{2}$, $b_{3}^{k}\in(\re^{m\times n})^{2}$ and $b_{4}^{k}\in(\re^{m\times n})$ and $\iota_{(\widetilde{u}>0)}$ being the characteristic function for the positivity constraint on $\widetilde u$. To progress, in each iteration above we would need to solve a minimisation problem in all $u,\widetilde{u},v,z,w$ at the same time which is numerically very involved. Instead, we use the split Bregman idea of \cite{Osher} and in each iteration solve a sequence of decoupled problems in $u,\widetilde{u},v,z,w$, that is
\begin{align}
v^{k+1}&=\argmin_{v}\bigg\{ \frac{1}{2}\sum\frac{(g-v)^{2}}{g}+\frac{\lambda_{1}}{2}\norm{2}{b_{1}^{k}+R u^{k}-v}^{2}\notag\\       
       &+\frac{\lambda_{2}}{2}\norm{2}{b_{2}^{k}+\nabla v -w^{k}}^{2}\bigg\}\label{n11}\\
u^{k+1}&=\argmin_{u}\bigg\{ \frac{\lambda_{1}}{2}\norm{2}{b_{1}^{k}+R u-v^{k+1}}^{2}+\frac{\lambda_{3}}{2}\norm{2}{b_{3}^{k}+\nabla u-z^{k}}^{2}\notag\\  
       &+\frac{\lambda_{4}}{2}\norm{2}{b_{4}^{k}+u-\widetilde{u}^{k}}^{2}\bigg\}\label{n13}\\
\widetilde{u}^{k+1}&=\argmin_{\widetilde{u}}\bigg\{\iota_{(\widetilde{u}>0)}+\frac{\lambda_{4}}{2}\norm{2}{b_{4}^{k}+u^{k+1}-\widetilde{u}}^{2}\bigg\}\label{n13a}\\
z^{k+1}&=\argmin_{z}\bigg\{\alpha\norm{1}{z}+\frac{\lambda_{3}}{2}\norm{2}{b_{3}^{k}+\nabla u^{k+1} - z}^{2}\bigg\}\label{n14}\\
w^{k+1}&=\argmin_{w}\bigg\{\beta\norm{1}{w}+\frac{\lambda_{2}}{2}\norm{2}{b_{2}^{k}+\nabla v^{k+1} - w}^{2}\bigg\}\label{n12}\\
b_{1}^{k+1}&=b_{1}^{k}+R u^{k+1}-v^{k+1}\label{b20}\\
b_{2}^{k+1}&=b_{2}^{k}+\nabla v^{k+1}-w^{k+1}\label{b21}\\
b_{3}^{k+1}&=b_{3}^{k}+\nabla u^{k+1}-z^{k+1}\label{b22}\\
b_{4}^{k+1}&=b_{4}^{k}+u^{k+1}-\widetilde{u}^{k+1}\label{b23}
\end{align} 
This procedure leads to five minimisation problems that have to be solved sequentially in each iteration. Every one of them either has an explicit solution or involves the solution of a linear system of equations that can be efficiently solved with an iterative method such as conjugate gradient. We iterate until 
$$
\frac{\norm{2}{\widetilde{u}^{K+1}-\widetilde{u}^{K}}^{2}}{\norm{2}{\widetilde{u}^{K+1}}}<10^{-4}
$$
and take $v^{K+1}$ as the regularised sinogram and $\widetilde{u}^{K+1}$ as the reconstructed image. Let us go into more detail on the solution of each minimisation problem.\\

\noindent {\bf Solution of \eqref{n11}:} To solve \eqref{n11} we derive the corresponding Euler-Lagrange equation for $v$ and obtain a linear system of equations with $k\cdot l$ unknowns $v_{i,j},\ i=1,\ldots, k, \ j=1,\ldots,l$ which reads
\begin{align}
	\eqref{n11}\Rightarrow& ((1+\lambda_{1})g -\lambda_{2}g\diverg\cdot\nabla)v=g+\lambda_{1}g(b_{1}^{k}+R u^{k})+\lambda_{2}g\diverg(b_{2}^{k}-w^{k})\label{sinogram_subproblem}
\end{align}
The system \eqref{sinogram_subproblem} is solved by a conjugate gradient method.\\

\noindent {\bf Solution of \eqref{n13}:} The Euler-Lagrange equation of \eqref{n13}  for $u$ reads
\begin{equation}
	\eqref{n13}\Rightarrow  (\lambda_{1}R^{*}R-\lambda_{3}\diverg\cdot\nabla+\lambda_{4})u=\lambda_{1}R^{*}(v^{k+1}-b_{1}^{k})+\lambda_{3}\diverg(b_{3}^{k}-z^{k})-\lambda_{4}(b_{4}-\widetilde{u}^{k})
	\label{image_subproblem}
\end{equation}
where $R^{*}$ is the adjoint of $R$, that is the discrete backprojection. As before, the system \eqref{image_subproblem} is solved by a conjugate gradient method.\\

\noindent {\bf Solution of \eqref{n13a}:} The solution of \eqref{n13a} is given by
$$
\widetilde u^{k+1} = \max\{b_4^{k+1}+u^{k+1},0\}.
$$

\noindent {\bf Solution of \eqref{n14} and \eqref{n12}:} Finally, the solution of the minimisation problems \eqref{n14},\eqref{n12} can be obtained exactly through soft shrinkage. That is,
\begin{align}
	z^{k+1}&=\mathcal{S}_{\frac{\alpha}{\lambda_3}}(b_3^k+\nabla u^{k+1}):=\max\left(\norm{2}{b_3^k+\nabla u^{k+1}}-\frac{\alpha}{\lambda_3},0\right)\frac{b_3^k+\nabla u^{k+1}}{\norm{2}{b_3^k+\nabla u^{k+1}}}\label{shrinkage}\\
	w^{k+1}&=\mathcal{S}_{\frac{\beta}{\lambda_2}}(b_2^k+\nabla v^{k+1}):=\max\left(\norm{2}{b_2^k+\nabla v^{k+1}}-\frac{\beta}{\lambda_2},0\right)\frac{b_2^k+\nabla v^{k+1}}{\norm{2}{b_2^k+\nabla v^{k+1}}} 
\end{align}


\section{Numerical Results}
\label{numerical}
In this section, we present our results on both simulated and real PET data. The Radon matrix that we described in \eqref{radon_matrix} is fixed and produces sinograms of size $192\times192$, that is the sinogram is given in 192 projection lines, $192^{o}$ degrees with $1^{o}$ degree incrementation and the corresponding reconstructed image is of size $175\times175$ pixels. We corrupt the sinograms with Poisson noise of different levels. In order to create noisy images corrupted by Poisson noise, we apply the MATLAB routine \textit{imnoise} (sinogram, \textit{poisson}). MATLAB's \textit{imnoise} function acts in the following way: for an image of double precision, the input pixel values are interpreted as means of a Poisson distribution scaled by a factor of $10^{-12}$. For example, if an input pixel has the value $5.5*10^{-12}$ then the corresponding output pixel will be generated from a Poisson distribution with mean of $5.5$ and afterwards scaled back to its original range by $10^{12}$.  The factor 10$^{12}$ is fixed to represent the maximal number of detectable photons. Our simulated sinograms are in $[0,1]$ intensity and in order to create different noise levels, we have to rescale the initial sinogram with a suitable factor before applying \textit{imnoise} and then scale it back with the same factor, i.e., \textit{Noisy Sinogram } = \textit{scale $*$ imnoise} ( $\frac{sinogram}{scale}$, $poisson$ ).


To simulate realistic sinograms with higher noise level, we use $10^{13}$ as a scaling factor, see for example Figure \ref{sxima4}.  The real data was obtained from the hardware phantom "Wilhelm", a self-built phantom modelled of the human body. Beside the activity in the heart a small source is placed in the phantom to simulate a lesion, see section \ref{petrecon} for more information.\\

Before presenting our results we give some specifics on how equations \eqref{n11}-\eqref{n12} are solved and how parameters are chosen. Both linear systems \eqref{sinogram_subproblem} and \eqref{image_subproblem} are solved using MATLAB's built-in function \textit{cg} which performs a conjugate gradient method. As a stopping criterium we either stop after at most $200$ iterations or if the relative residual is smaller than $10^{-3}$. As it is observed in \cite{Osher}, it seems optimal to apply only a few steps of an iterative solver for both subproblems \eqref{sinogram_subproblem} and \eqref{image_subproblem} since the error in the split Bregman algorithm is updated in every iteration. 

The Lagrange multiplies $(\lambda_{i})_{i=1}^{4}$ in equations \eqref{n11}-\eqref{n12} in section \ref{sec:splitbreg} are chosen following \cite{Pascal} to optimise convergence speed and well conditioning. They were fixed as $\lambda_{1}=0.001$, $\lambda_{2}=1$, and $\lambda_{3}=\lambda_{4}=100$. Note that these parameters may affect the condition number for both system matrices 
\begin{align*}
A_{img}&=\lambda_{1}R^{*}R-\lambda_{3}\diverg\cdot\nabla+\lambda_{4}I\\
A_{sin}&=(1+\lambda_{1})g -\lambda_{2}g\diverg\cdot\nabla
\end{align*}
in \eqref{sinogram_subproblem} and \eqref{image_subproblem} and hence the convergence rates of iterative solvers used to solve them are affected by this choice.

Finally, we observe that after $150$ Split Bregman iterations, there are no significant changes in the reconstructed image and therefore we choose a stopping criteria of either at most $K=400$ iterations or we stop at iteration $K$ where for the first time we have
$$
\frac{\norm{2}{\widetilde{u}^{K+1}-\widetilde{u}^{K}}}{\norm{2}{\widetilde{u}^{K+1}}}<10^{-4}
$$ 
where 
$\widetilde{u}^{K+1}$ the regularised image. To evaluate the quality of reconstructed images we choose the \emph{Signal-to-Noise Ratio} (SNR) as a quality measure. The SNR is defined as
\begin{equation}
 SNR=20\log\bigg(\frac{\norm{2}{u}}{\norm{2}{u-\widetilde{u}}}\bigg)
\end{equation}

where $u$ and $\widetilde{u}$ denote the ground truth and the reconstructed image, respectively. In what follows, we first evaluate the proposed reconstruction technique \eqref{v1_1} against pure total variation regularisation on the image ($\beta=0$) for a synthetic image of two circles and for different noise levels, as well as for a real data set acquired for the Wilhelm phantom. Then, we numerically analyse the scale space properties of pure sinogram regularisation, that is for $\alpha=0$, which will be a motivation for the final section in which we discuss the merit of the proposed reconstruction method for PET data that encodes thin image structures.

\subsection{Image reconstruction from corrupted simulated and real PET data}
\label{petrecon}

We start with a discussion of numerical results obtained for simulated PET data. Figure \ref{sxima4} shows a simulated phantom of two discs with different radius and the corresponding noiseless and noisy sinograms corrupted with low and high level Poisson noise as described above. 

\begin{figure}[h!]
\begin{center}
\begin{subfigure}[h!]{3cm}
                \centering
                \includegraphics[scale=0.18]{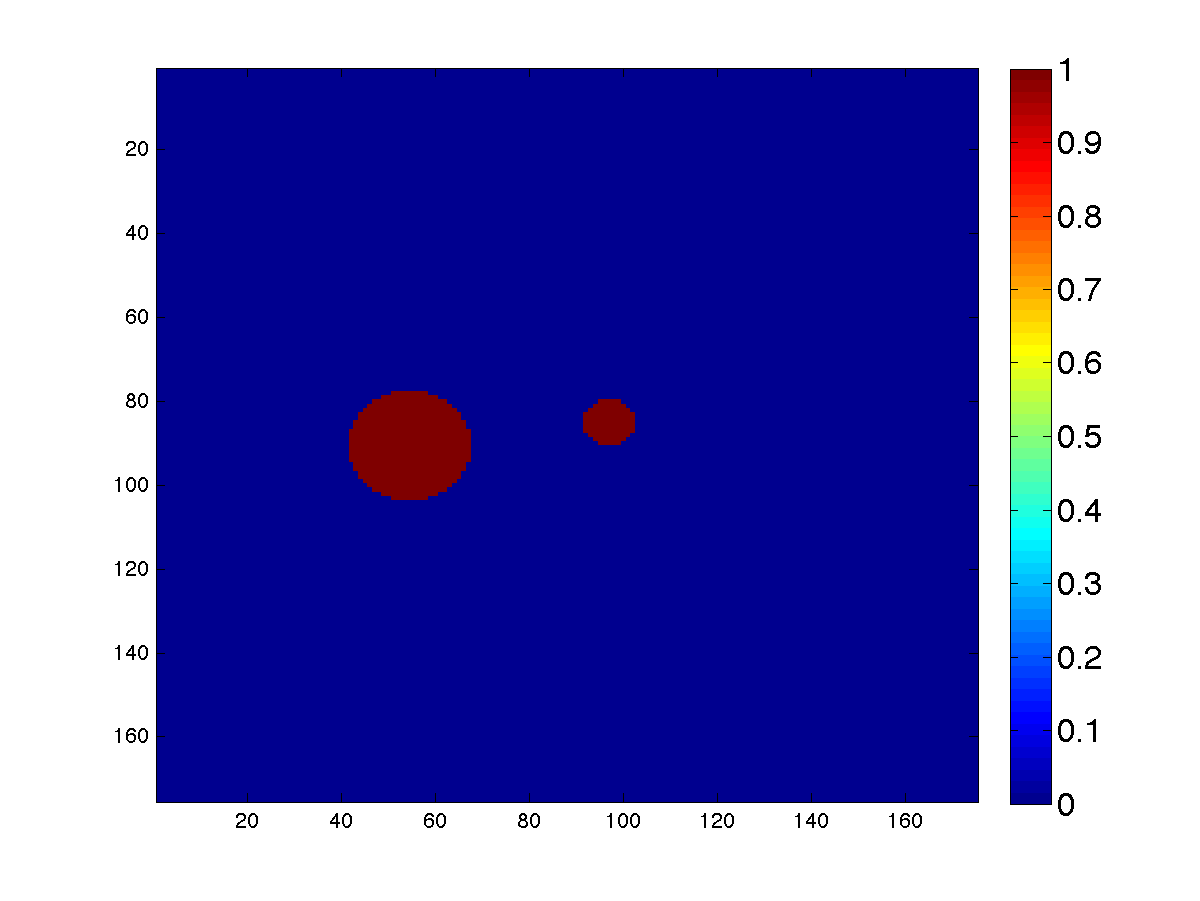}
								\caption{2 discs \\ ~}
\end{subfigure}
\begin{subfigure}[h!]{3cm}
                \centering
                \includegraphics[scale=0.18]{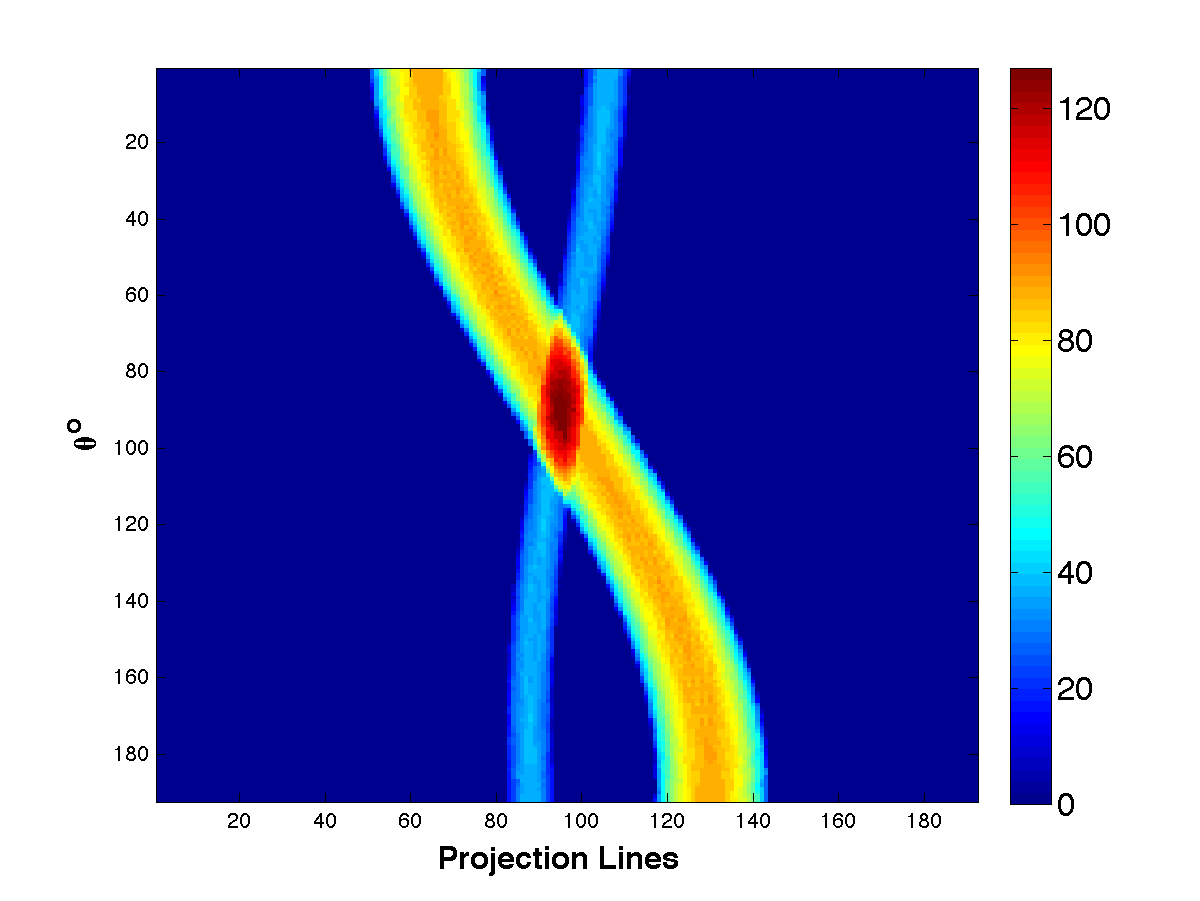}
								 \caption{Noiseless Sinogram}
\end{subfigure}
\begin{subfigure}[h!]{3cm}
                \centering
                \includegraphics[scale=0.18]{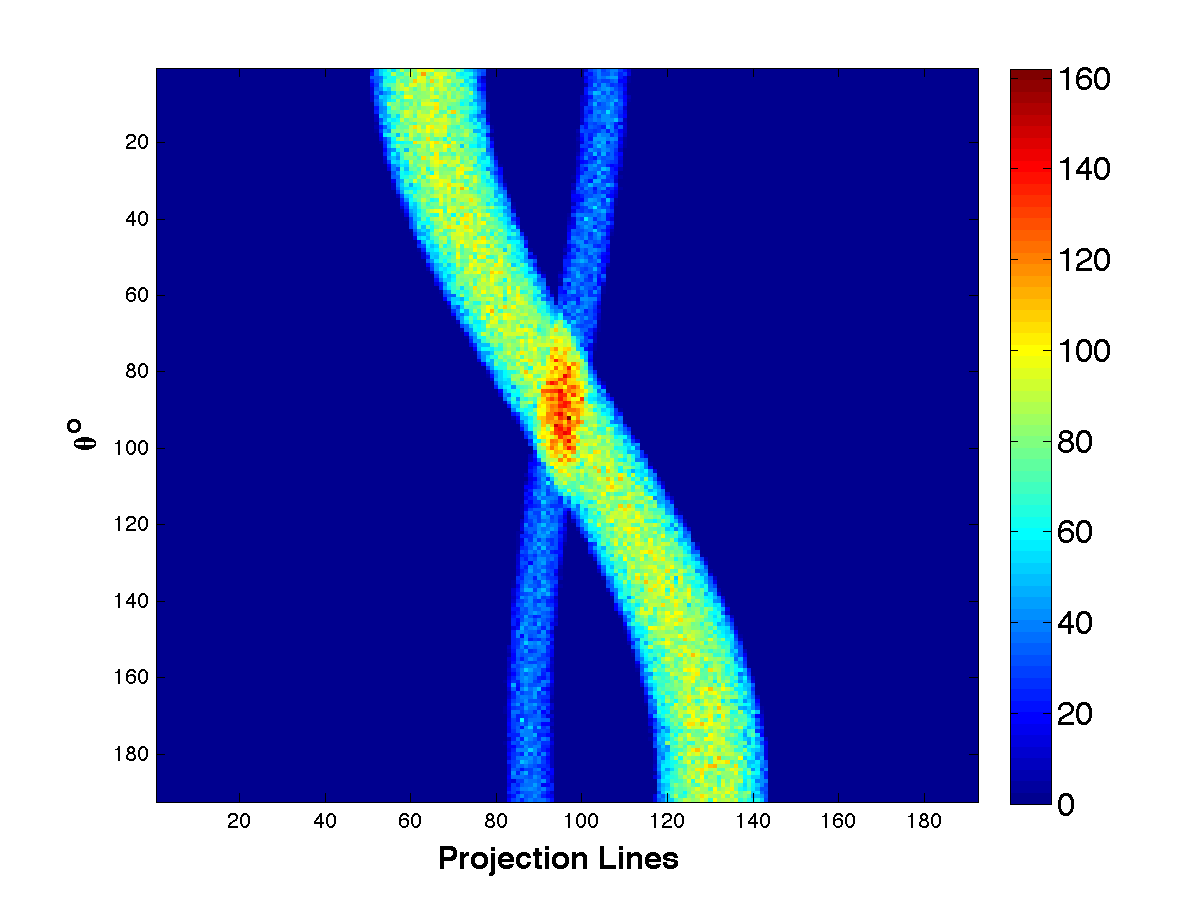}
								\caption{Low Noise: SNR=18.5246}
								\label{sxima4:c}
\end{subfigure}
\begin{subfigure}[h!]{3cm}
                \centering  
                \includegraphics[scale=0.18]{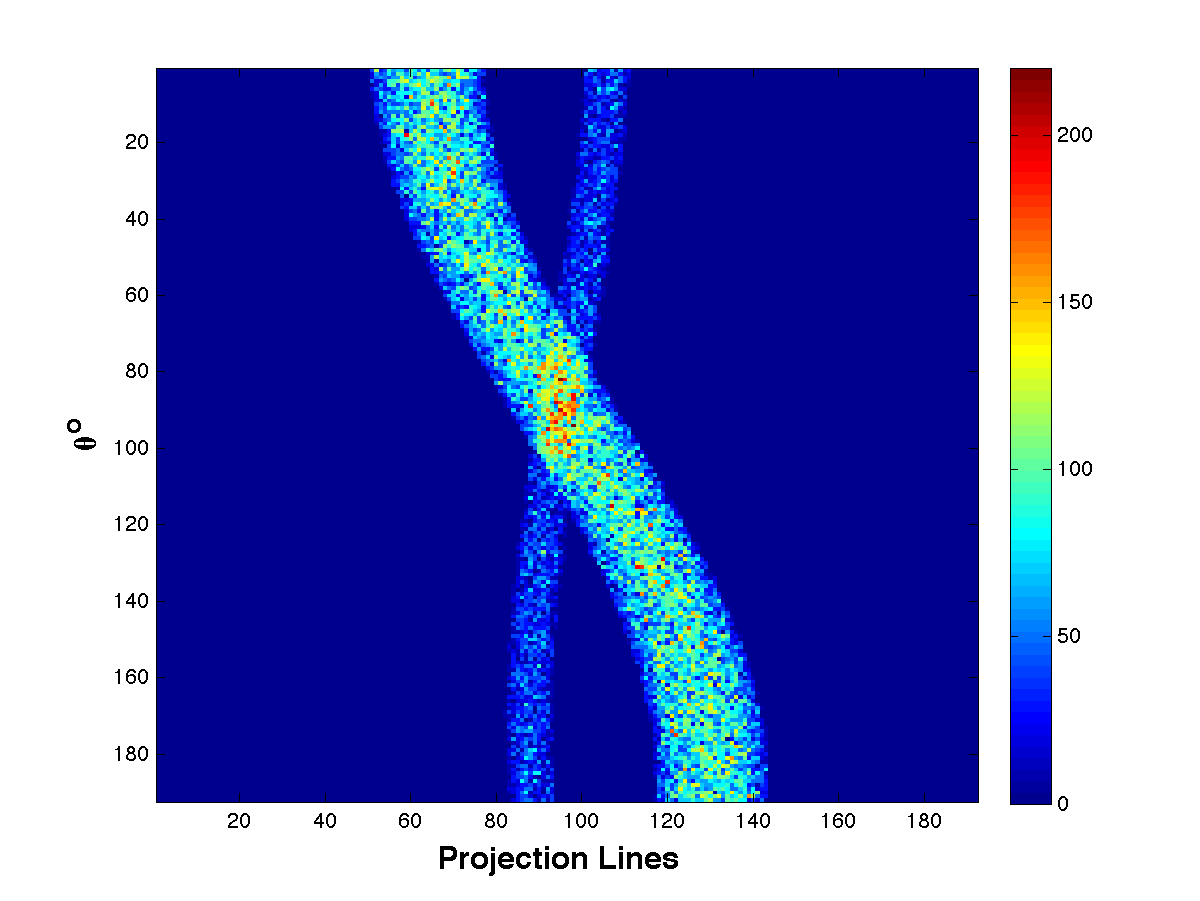}
								\caption{High Noise: SNR=8.6814}
								\label{sxima4:d}
\end{subfigure}%
\caption{The phantom image includes 2 discs of radius $r_{1}=26$ and $r_{2}=11$ pixels. Its sinogram has 192 angles and $192$ rays with low and high noise.}
\label{sxima4}
\end{center}
\end{figure}

First, we evaluate the proposed algorithm for reconstructing an image from the sinogram corrupted by low level Poisson noise with SNR=18.5246, see Figure \eqref{sxima4:c}. The proposed reconstruction algorithm with joint total variation regularisation of image and sinogram (that is $\alpha,\beta>0$) is compared with the algorithm that uses pure total variation regularisation of the image (that is $\alpha>0$ and $\beta=0$). Both reconstruction strategies are tested for a range of parameters $\alpha, \beta$ and in each case the reconstruction is found which has the highest SNR value. For $\beta=0$ we computed the reconstructed image for $\alpha=3,4,5,6,7$. The optimal reconstructed image in terms of the best SNR$=25.8589$ is obtained for $\alpha=6$, see Figure \ref{sxima6}. Then, we test the proposed reconstruction method applying total variation regularisation on both the image and the sinogram using the same range of $\alpha=3,4,5,6,7$ and $\beta=0.001,0.005,0.01,0.05$. Here, the optimal reconstruction was obtained for $\alpha=6$ and $\beta=0.001$ with SNR$=25.3127$, see Figure \ref{sxima6}. In Table \ref{Table_SNR_not_0} a full list of tested parameters and SNRs for corresponding reconstructed images is given. The results do not indicate a significant difference between the algorithm with and without total variation regularisation on the sinogram, both visually and also in terms of the SNR. Indeed, in the low noise case additional total variation regularisation on the sinogram produces even slightly worse results in terms of SNR than using no regularisation on the sinogram at all.

\begin{figure}[h!]
\begin{center}
\begin{subfigure}[h!]{5cm}
\captionsetup{width=.6\textwidth}
                \centering
                \includegraphics[scale=0.2]{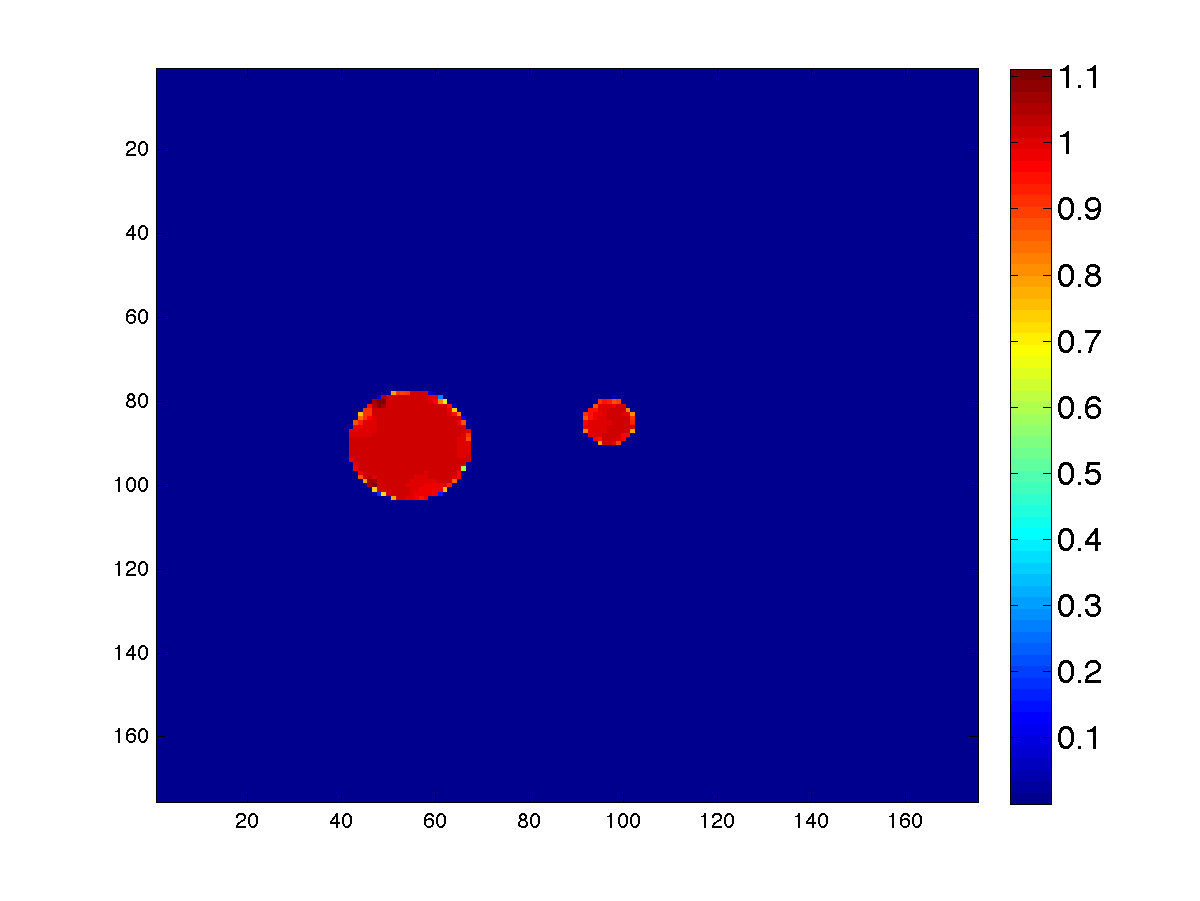}
                \caption{$\alpha=6$, $\beta=0$\\SNR=25.8589}
\end{subfigure}
\begin{subfigure}[h]{0.3\textwidth}
\captionsetup{width=.7\textwidth}
                \centering
                \includegraphics[scale=0.2]{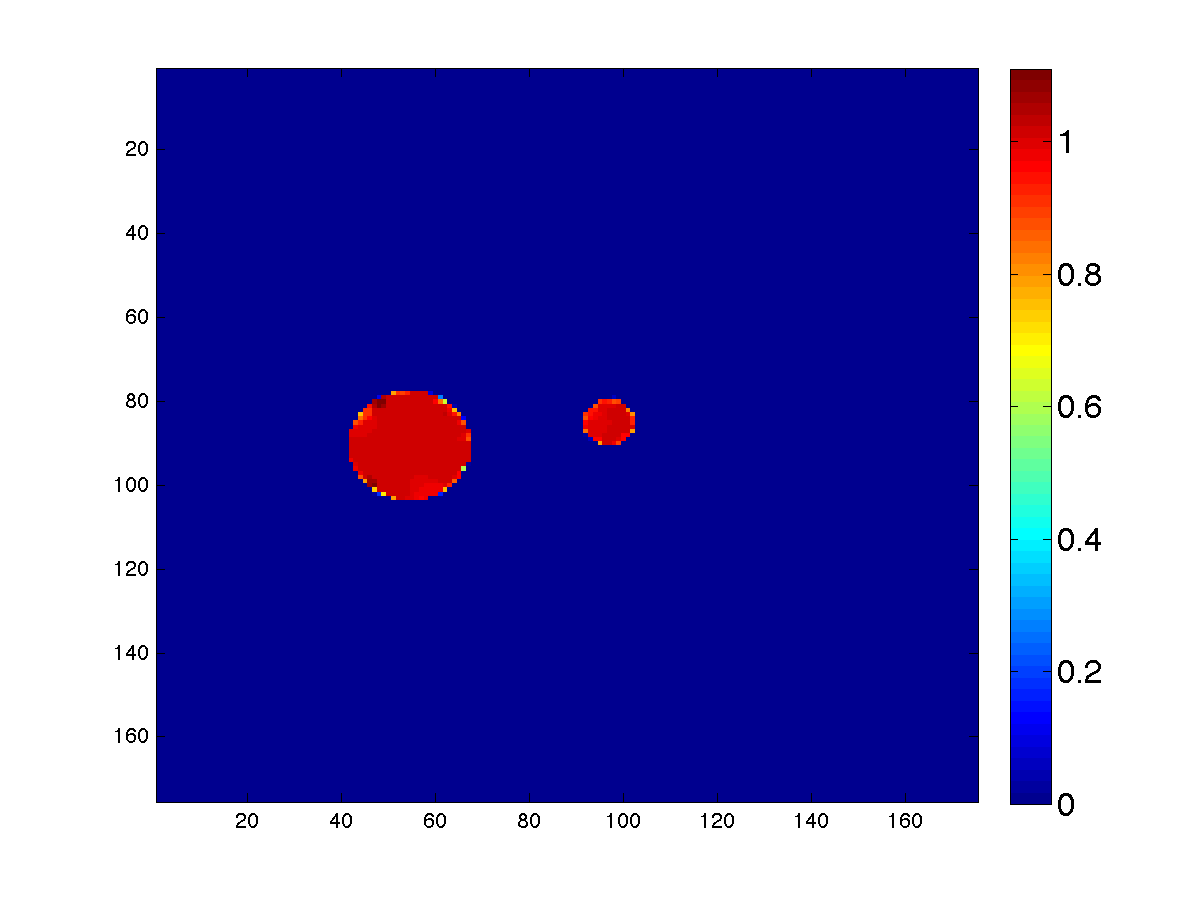}
                \caption{$\alpha=6$, $\beta=10^{-3}$\\SNR=25.3127}
\end{subfigure}
\begin{subfigure}[h]{0.3\textwidth}
                \centering
                \vspace{-0.3cm}
                \includegraphics[scale=0.2]{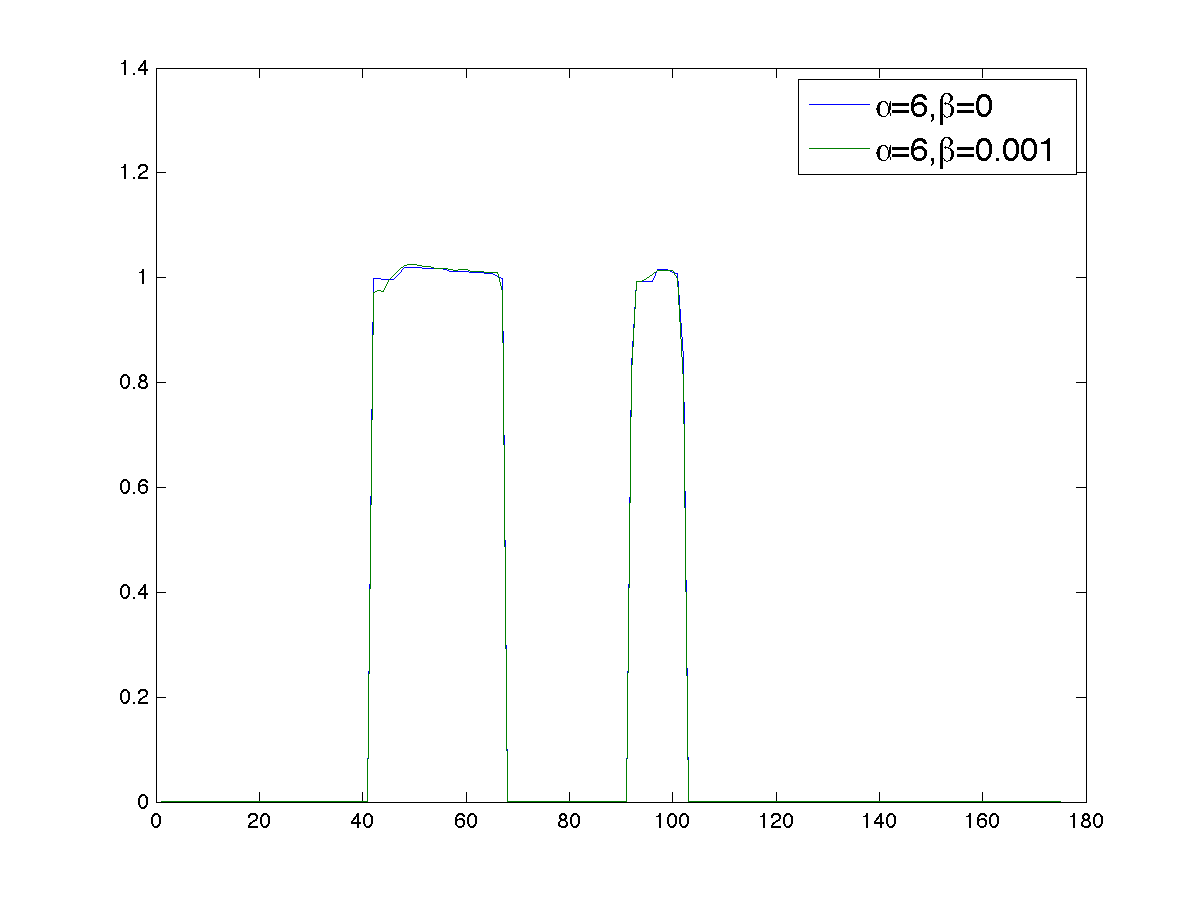}
                \caption{Middle line profiles}
\end{subfigure}
\caption{Low Level Noise: Optimal reconstruction results of the 2 discs image with sinogram shown in Figure \eqref{sxima4:c} with and without sinogram regularisation and a comparison of the line profiles for the two results.}
\label{sxima6}
\end{center}
\end{figure}

\begin{table}[h!]
\begin{center}
\begin{tabular}{|c|c|c|c|c|c|c|c|}
\cline{3-8}
\multicolumn{2}{c|}{} & \multicolumn{6}{c|}{$\beta$} \\
\cline{3-8}
\multicolumn{2}{c|}{} & 0 & 0.001 & 0.005 & 0.01 & 0.05 & 0.1  \\
\hline
\multirow{5}{*}{$\alpha$} & 3 & 24.0819 & 22.0172 & 22.4415 & 22.8894 & 23.2414 & 21.6533 \\
\cline{2-8}
& 4 & 25.3682 & 24.0926 & 24.2951 & 24.4801 & 23.6303 & 21.9382 \\
\cline{2-8}
& \textbf{5} & \textbf{25.7867} & 25.0829 & 25.0779 & 25.0469 & 23.9432 & 22.0367 \\
\cline{2-8}
& \textbf{6} & \textbf{25.8589} & \textbf{25.3127} & 24.7787 & 25.0602 & 24.0095 & 22.1034 \\
\cline{2-8}
& 7 & 25.7436 & 24.8499 & 24.8278 & 25.0148 & 23.9662 & 22.2289 \\
\cline{1-8}
\end{tabular}
\end{center}
\caption{Low Level Noise for simulated example in Figure \ref{sxima4}: SNRs of reconstructed images for different combinations of $\alpha$ and $\beta$ values.}
\label{Table_SNR_not_0}
\end{table}

The TV regularisation on the sinogram gains importance in the reconstruction algorithm when the noise in the corruption of the sinogram is increased. The sinogram with high level noise is shown in Figure \eqref{sxima4:d} and has SNR=8.6814.  We tested the proposed method for $\alpha=250, 275, 300, 325, 350$ and $\beta=0, 0.001, 0.01, 0.05, 0.1$. The results are reported in Table \ref{High_Noise_Table_SNR_not_0}. 

\begin{table}[h!]
\begin{center}
\begin{tabular}{|c|c|c|c|c|c|c|c|}
\cline{3-8}
\multicolumn{2}{c|}{} & \multicolumn{6}{c|}{$\beta$} \\
\cline{3-8}
\multicolumn{2}{c|}{} & 0 & 0.001 & 0.005 & 0.01 & 0.05 & 0.1  \\
\hline
\multirow{5}{*}{$\alpha$} & \textbf{250} & \textbf{10.9544} & \textbf{10.9665} & 10.9557 & 10.9464 & 10.8531 & 10.8058 \\
\cline{2-8}
& 275 & 10.9502 & 10.9599 & 10.9501 & 10.9381 & 10.8595 & 10.8013 \\
\cline{2-8}
& \textbf{300} & \textbf{10.9425} & 10.9543 & 10.9415 & 10.9257 & 10.8267 & 10.7777 \\
\cline{2-8}
& 325 & 10.9167 & 10.9551 & 10.9434 & 10.9283 & 10.8101 & 10.7293 \\
\cline{2-8}
& 350 & 10.8784 & 10.9289 & 10.9165 & 10.9014 & 10.7946 & 10.7104 \\
\cline{1-8}
\end{tabular}
\end{center}
\caption{High Level Noise for simulated example in Figure \ref{sxima4}: SNRs of reconstructed images for different combinations of $\alpha$ and $\beta$ values.}
\label{High_Noise_Table_SNR_not_0}
\end{table}

\begin{figure}[h!]
\begin{center}
\begin{subfigure}[h]{5cm}
\captionsetup{width=.7\textwidth}
                \centering
                \includegraphics[scale=0.2]{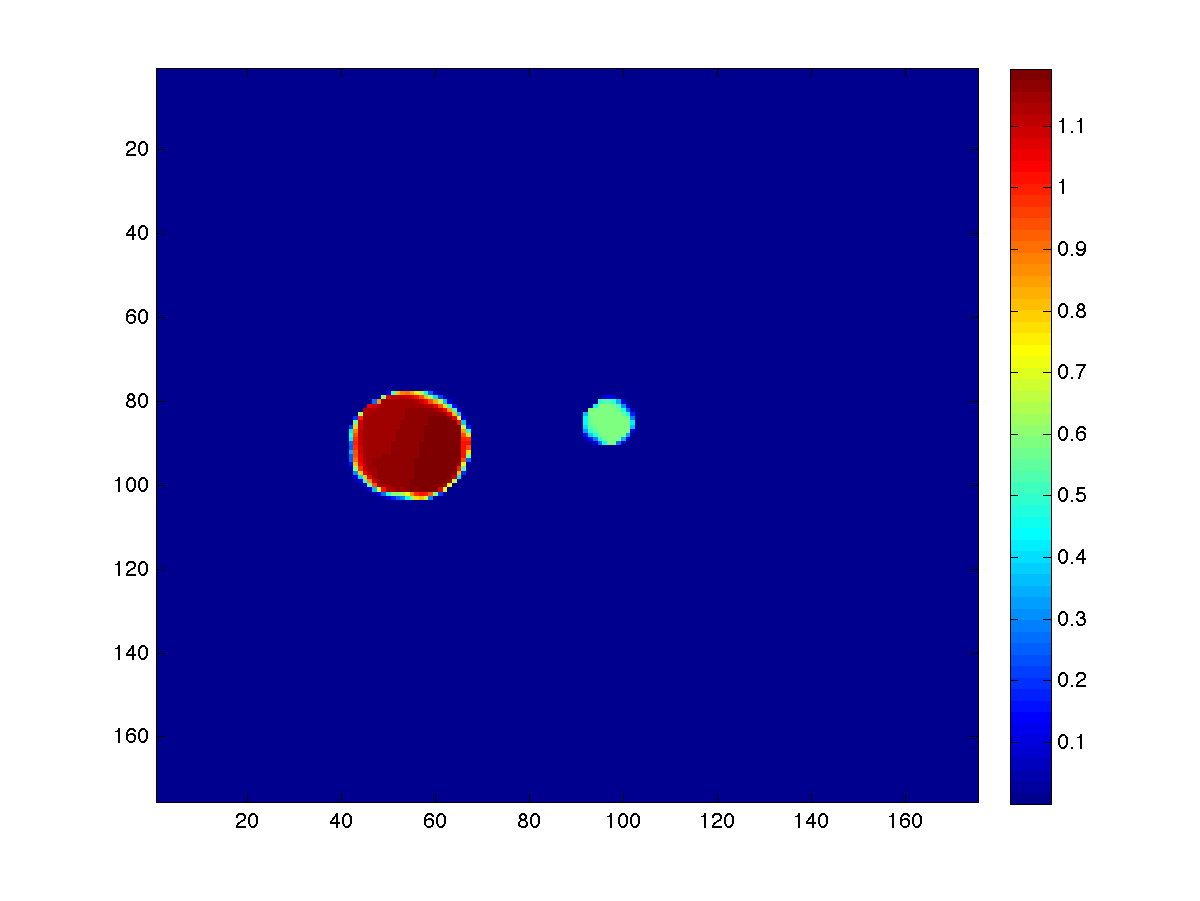}
                \caption{$\alpha=250$, $\beta=0$ \\SNR=10.9544}
\end{subfigure}
\begin{subfigure}[h]{5cm}
\captionsetup{width=.7\textwidth}
                \centering
                \includegraphics[scale=0.2]{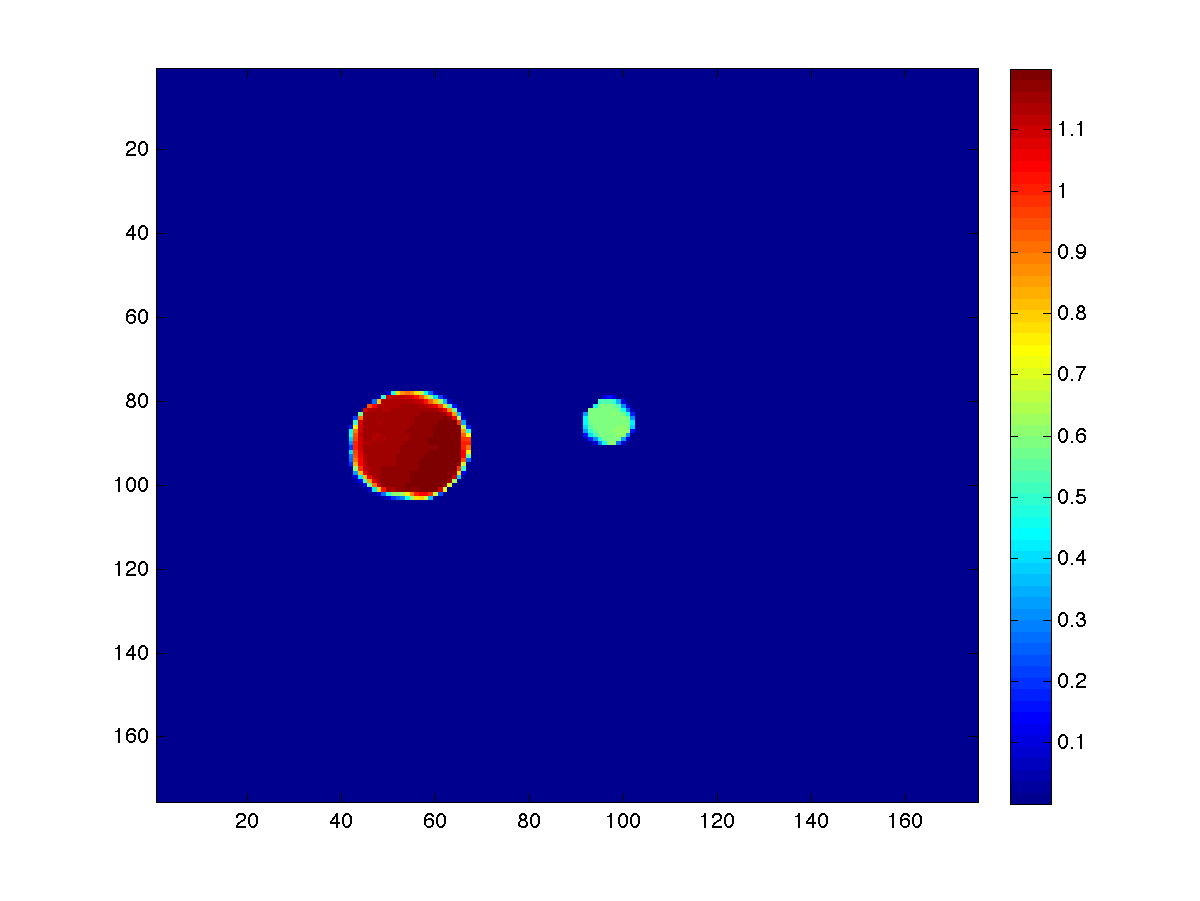}
                \caption{$\alpha=250$, $\beta=10^{-3}$ \\SNR=10.9665}
                 \label{sxima8:b}
\end{subfigure}
\begin{subfigure}[h]{5cm}
\captionsetup{width=.7\textwidth}
                \centering
                \vspace{-0.3cm}
                \includegraphics[scale=0.2]{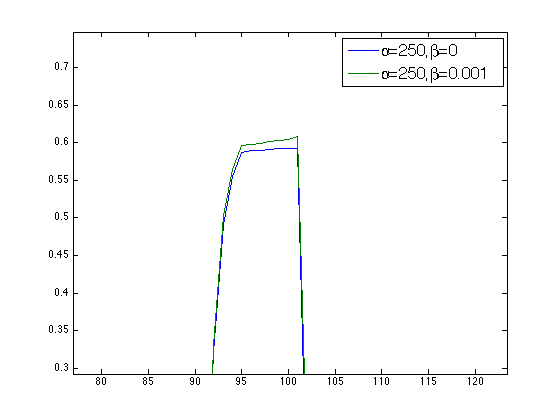}
                \caption{Zoom In: Middle line profiles}
                \label{sxima8:c}
\end{subfigure}
\caption{High Level Noise for simulated example in Figure \ref{sxima4}: Best SNRs with/without total variation regularisation on the sinogram and the middle line profiles of the reconstructed images.}
\label{sxima8}
\end{center}
\end{figure}


The highest SNR is obtained when $\alpha=250$ and $\beta=0.001$, cf. Figure \eqref{sxima8:b}. Although, it is hard to distinguish any difference between the cases of $\beta$, we observe that the extra penalisation on the sinogram produces better results in terms of the SNR value, see Figure \ref{sxima8}. The increase in SNR for $\beta>0$ can be seen when comparing the middle line profiles of the reconstructed images with and without sinogram regularisation in Figure \eqref{sxima8:c}. 


As a second example for our evaluation of the algorithm for PET reconstruction we consider real PET data obtained from scanning a self-built phantom of a human breast with a small source which simulates a lesion, compare Figure \eqref{wilhelm}. The data has been acquired with a Siemens Biograph Sensation 16 PET/CT scanner (Siemens Medical Solutions) located at the University Hospital in M\"unster. From the 3D PET data we used only one sinogram slice. The 2D sinogram dimension is $192\times192$ with a pixel size of $3.375mm^2$. The size of the reconstructed image is $175\times175$, covering a field of view of $590.625mm$ in diameter. The 2D slice of the noisy sinogram which has been used in our computations is shown in Figure \eqref{real_data:noisy}. Reconstructions obtained from the proposed algorithm, with and without sinogram regularisation, are shown in Figure \ref{real_data}. The additional regularisation of the sinogram seems to allow for smoother image structures (such as the boundary of the red lesion) and results in a slight reduction of the stair casing effect of total variation regularisation. 
  
\begin{figure}[h!]
\begin{subfigure}[h]{7.5cm}
\centering
\includegraphics[scale=0.15]{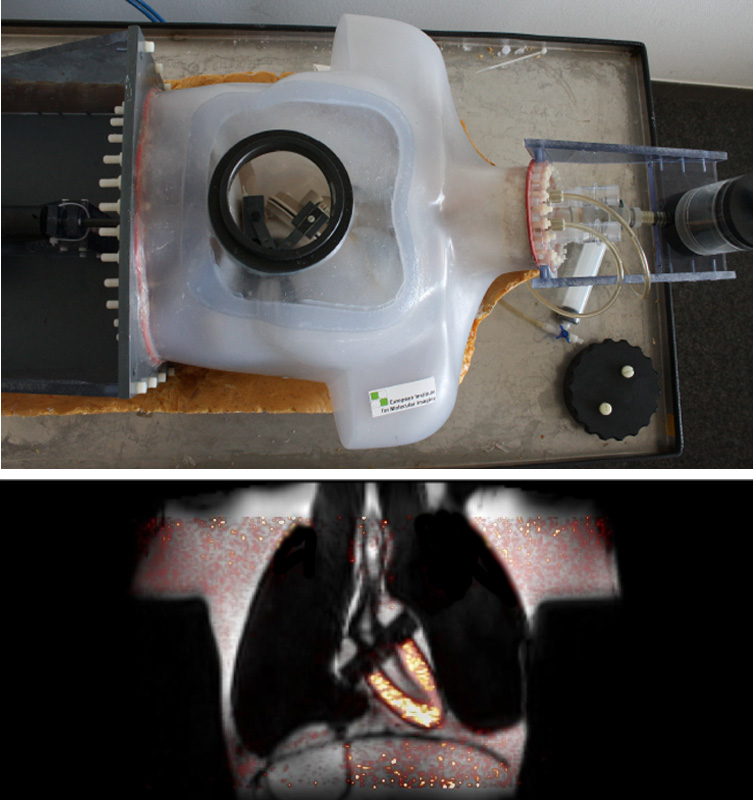}
\caption{Top: Phantom "Wilhelm", consisting of a plastic torso and inserts for the lungs, heart and liver. Bottom: Phantom reconstruction with combined PET-MRI. Data courtesy of the European Institute for Molecular Imaging (EIMI), M\"unster.}	
\label{wilhelm}	
\end{subfigure}		
\begin{subfigure}[h]{7.5cm}
                \centering  
               \includegraphics[scale=0.4]{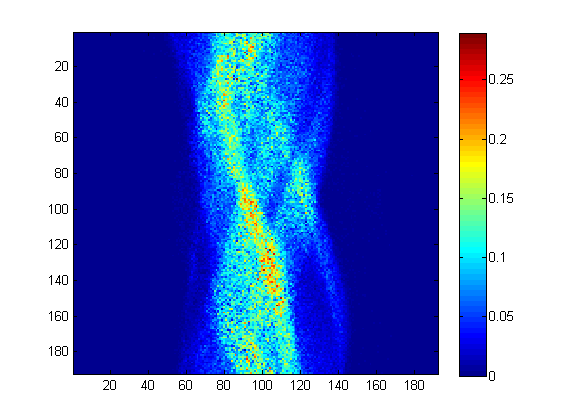}
               \caption{Noisy sinogram.}
               \label{real_data:noisy}  
               \end{subfigure}	
               \caption{Real PET data.}
\end{figure}

\begin{figure}[h!]
\begin{subfigure}[h]{7.5cm}
                \centering  
               \includegraphics[scale=0.25]{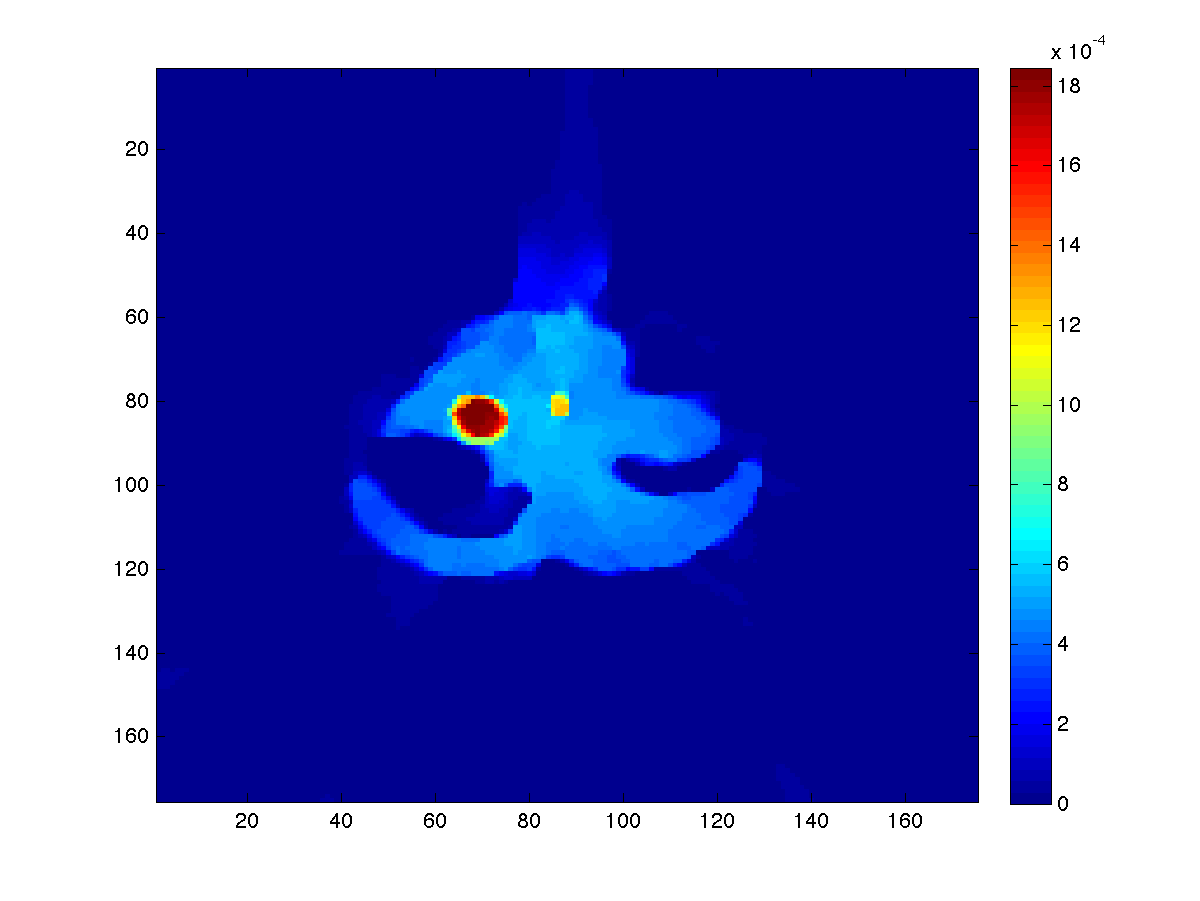}
               \caption{$\alpha=5$, $\beta=0$}
               \label{real_data:rec1}  
\end{subfigure}
\begin{subfigure}[h]{7.5cm}
               \centering  
               \includegraphics[scale=0.25]{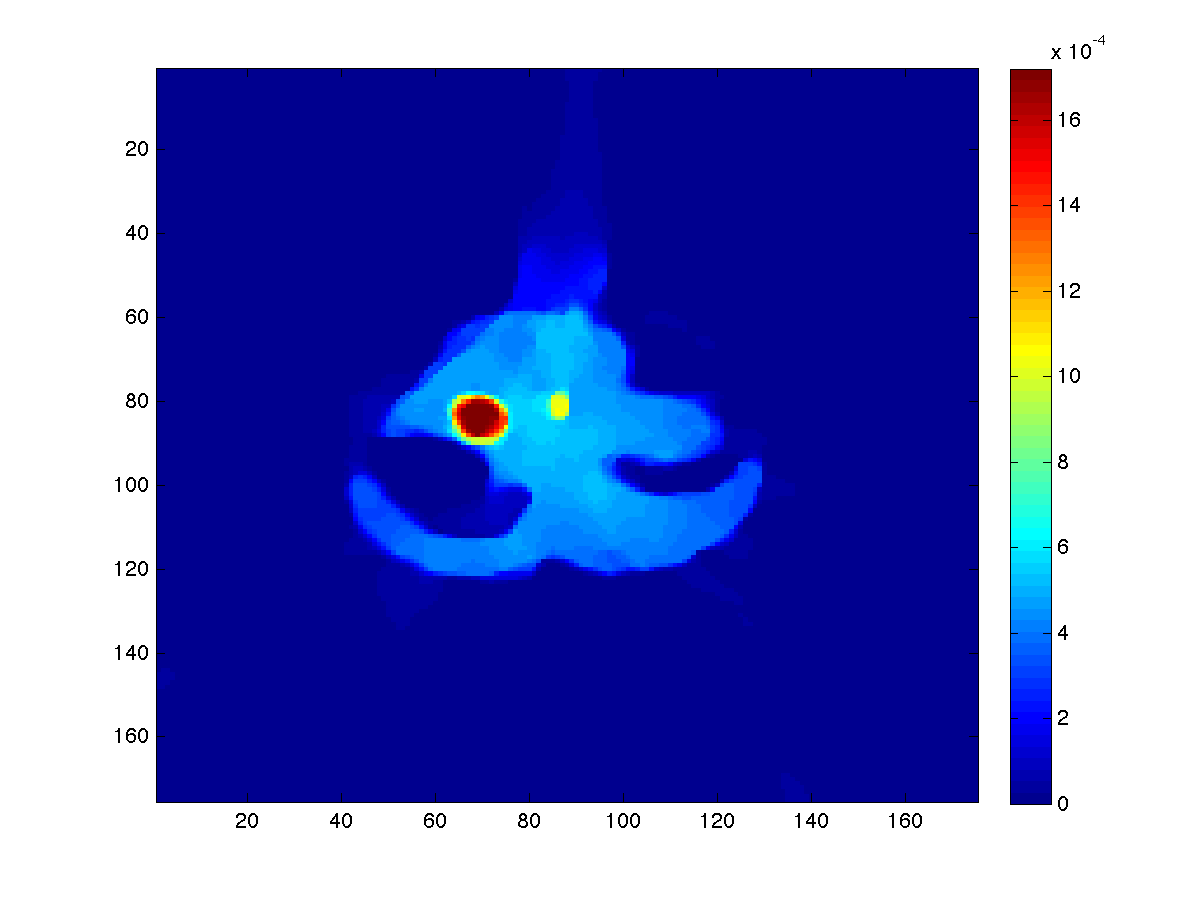}
               \caption{$\alpha=4$, $\beta=0.1$}
							 \label{real_data:rec2}  
\end{subfigure}
\caption{Real Data: Best TV regularised reconstructions for noisy slice in Figure \ref{real_data:noisy}.}
\label{real_data}
\end{figure}

In the following two sections we will aim to improve our understanding of this new sinogram regularisation, taking the analytic solution of section \ref{explicit} as a starting point. A thorough numerical discussion of this example in section \ref{scalespace} leeds us to section \ref{thinrecon} where the benefits of total variation regularisation of the sinogram for the reconstruction of thin objects are discussed.

\subsection{Scale space of sinogram regularisation}
\label{scalespace}

Following up on the computations in section \ref{explicit}, we now discuss how the regularisation on the sinogram effects the backprojected image. Let us recall that every point $(\theta,s)$ on the sinogram corresponds to a line $s=x\cos\theta+y\sin\theta$ that passes through a point $(x,y)$ on the image, with a distance $s$ from the origin and normal to the direction $(cos\theta,\sin\theta)$. Moreover (compare Thirion \cite{Thirion}), every point on an edge in the sinogram corresponds to a line in the object space which is 
tangent to the boundary of the object. To further understand how sinogram regularisation acts, we consider the effect of the regularisation when reconstructing an image from simulated noise-free Radon data. To this end we set $\alpha=0$, regularise the noise-free sinogram with different values of $\beta$, and apply FBP to the regularised sinogram to obtain the corresponding reconstructed image. We call the set of reconstructed images from regularised sinograms with varying $\beta$ regularisation, the scale space of total variation regularisation of the sinogram.

Considering the reconstruction method \eqref{v1_1} for $\alpha=0$ results in the following weighted total variation denoising problem for the sinogram $g$
\begin{equation}
	\argmin_{v\geq0\mbox{ a.e }}\beta\norm{1}{\nabla v}+\sum_{k,l}\frac{(g-v)^{2}}{g}
	\label{ROF_problem}
\end{equation}
where $\norm{1}{\cdot}$ is the discrete $l^{1}$ norm as defined before. Similar to before, we solve \eqref{ROF_problem} by a Split Bregman technique, introducing two more variables $w=\nabla v$ and $\widetilde{v}=v$. Then, starting with initial conditions $b_{1}^{0}\in\re^{2 k\times l}$ and $b_{2}^{0}\in\re^{k\times l}$, we iteratively solve for $k=1,2,\ldots$

\begin{align}
v^{k+1}&=\argmin_{v}\frac{\lambda_{1}}{2}\norm{2}{b_{1}^{k}+\nabla v-w^{k}}^{2}+\frac{\lambda_{2}}{2}\norm{2}{b_{2}^{k}+v-\widetilde{v}^{k}}^{2}
\label{sub1}\\
\widetilde{v}^{k+1}&=\argmin_{\widetilde{v}\geq 0}\frac{1}{2}\int\frac{(g-\widetilde{v})^{2}}{g}+\frac{\lambda_{2}}{2}\norm{2}{b_{2}^{k}+v^{k+1}-\widetilde{v}}^{2}\label{sub2}\\
w^{k+1}&=\argmin_{w}\beta\norm{1}{w}+\frac{\lambda_{1}}{2}\norm{2}{b_{1}^{k}+\nabla v^{k+1}-w}^{2}\label{sub3}\\
b_{1}^{k+1}&=b_{1}^{k}+\nabla w^{k+1}-v^{k+1}\\
b_{2}^{k+1}&=b_{2}^{k}+v^{k+1}-\widetilde{v}^{k+1}
\end{align}

Note that, as before, in the solution of \eqref{image_subproblem} a simple backprojection of the sinogram is used and we set $\lambda_{1}=\lambda_{2}=1$. Moreover, since we do not apply any positivity constraint on the image as it is done in the full algorithm used in section \ref{petrecon}, we might observe small negative values in the reconstructed images presented in the following.

\begin{figure}[h!]
\begin{center}
\begin{subfigure}[h]{3.7cm}
                \centering 
                \caption{$\beta=10^{-3}$}               
                \includegraphics[scale=0.2]{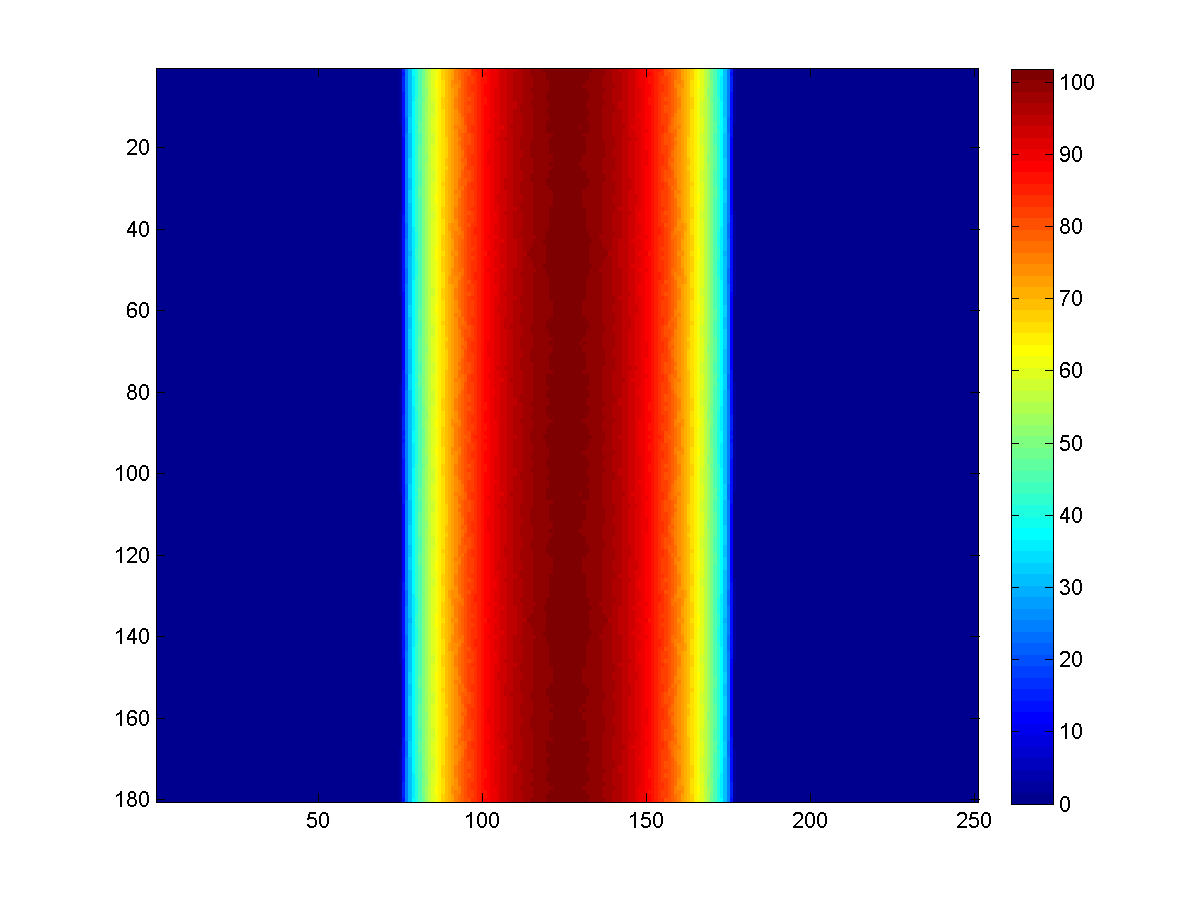}
\end{subfigure}%
\begin{subfigure}[h]{3.7cm}
                \centering
                \caption{$\beta$=20}  
 \includegraphics[scale=0.2]{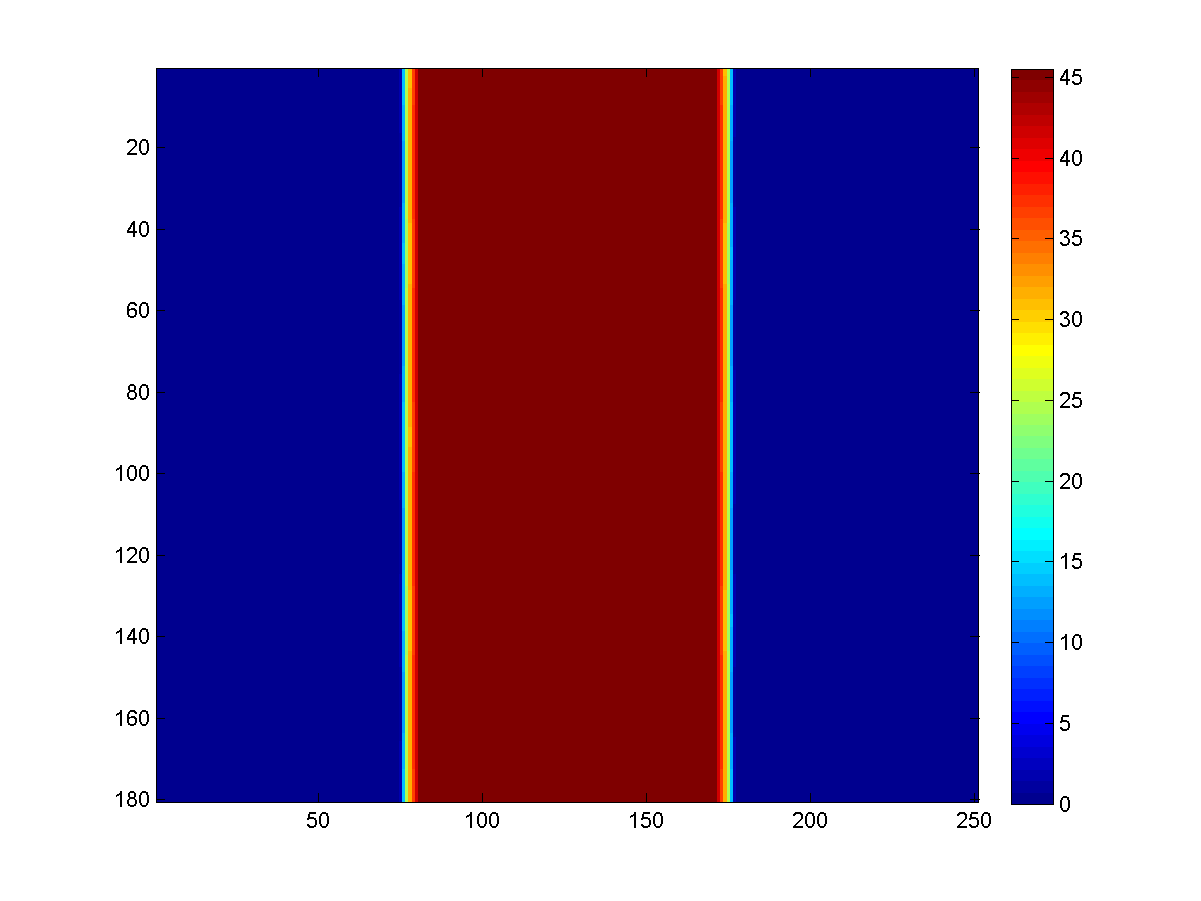}
\end{subfigure}
\begin{subfigure}[h]{3.7cm}
                \centering
                \caption{$\beta$=45}  
 \includegraphics[scale=0.2]{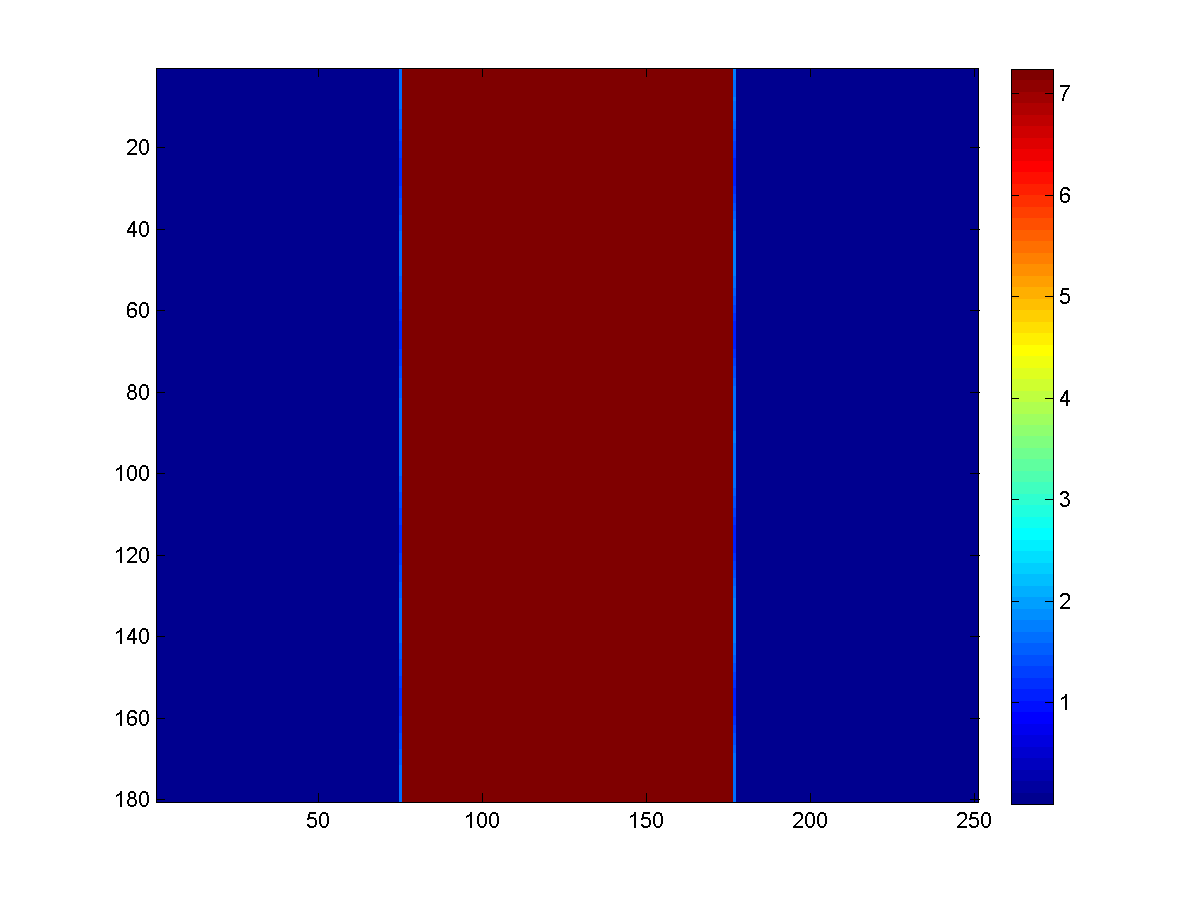}
\end{subfigure}
\begin{subfigure}[h]{3.7cm}
                \centering
                \caption{$\beta$=50.5}             
                 \includegraphics[scale=0.2]{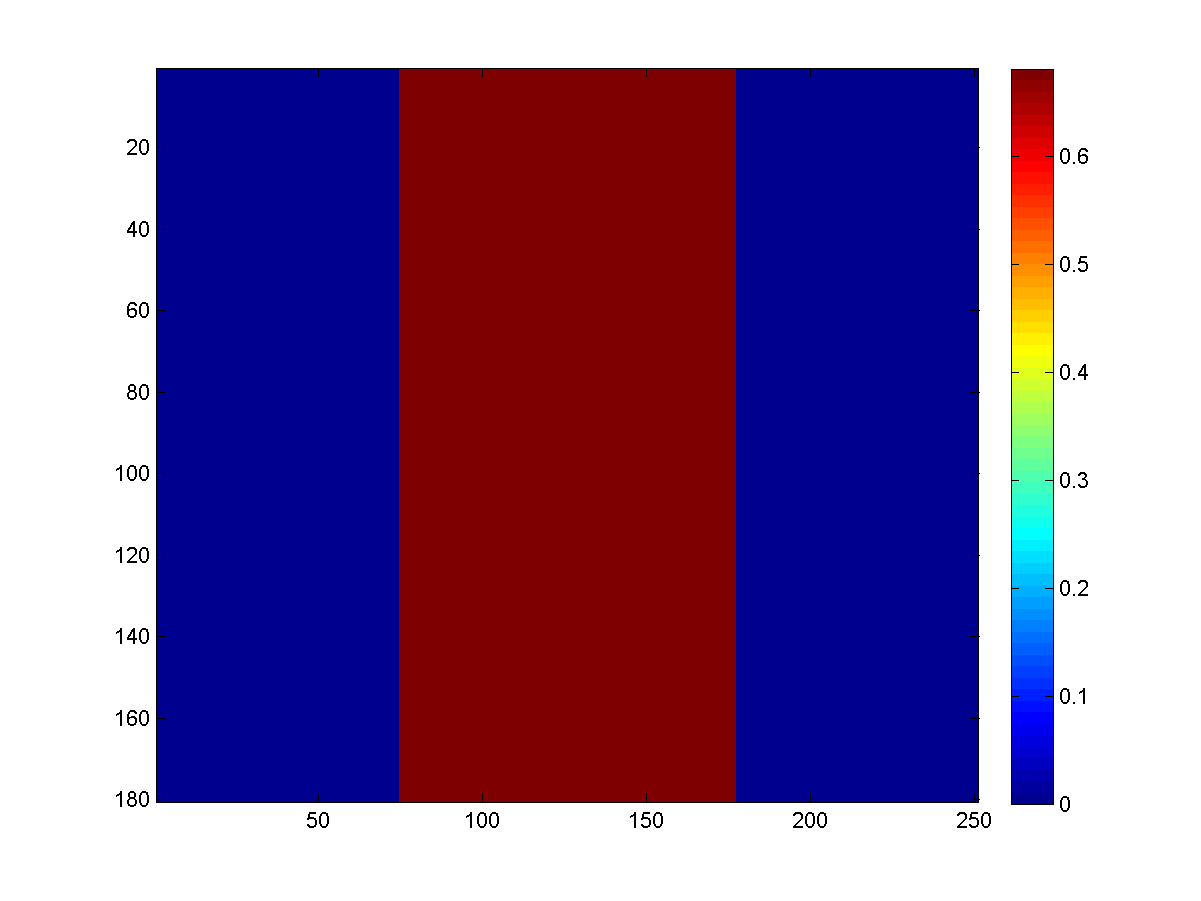}
\end{subfigure}\\
\begin{subfigure}[h]{3.7cm}
                \centering 
                \includegraphics[scale=0.2]{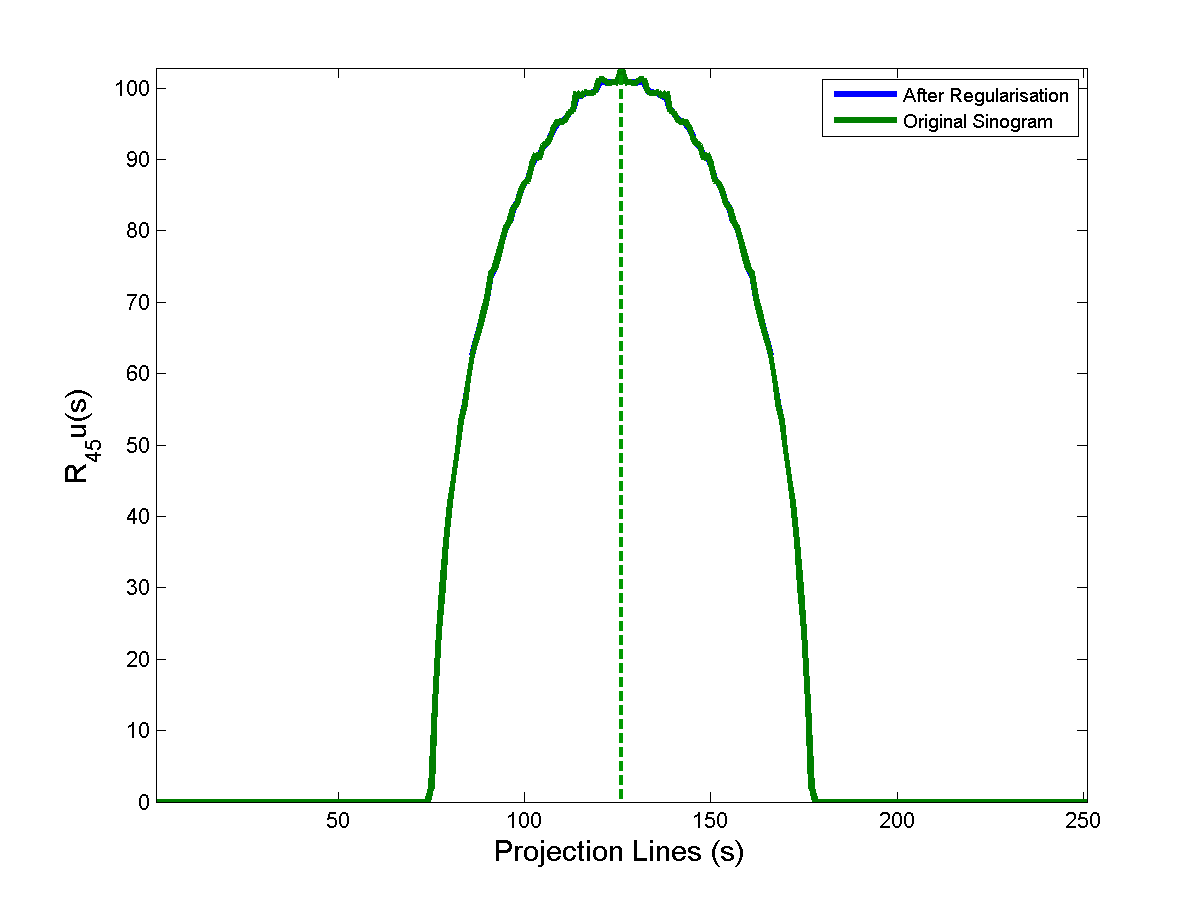}
\end{subfigure}
\begin{subfigure}[h]{3.7cm}
                \centering  
                \includegraphics[scale=0.2]{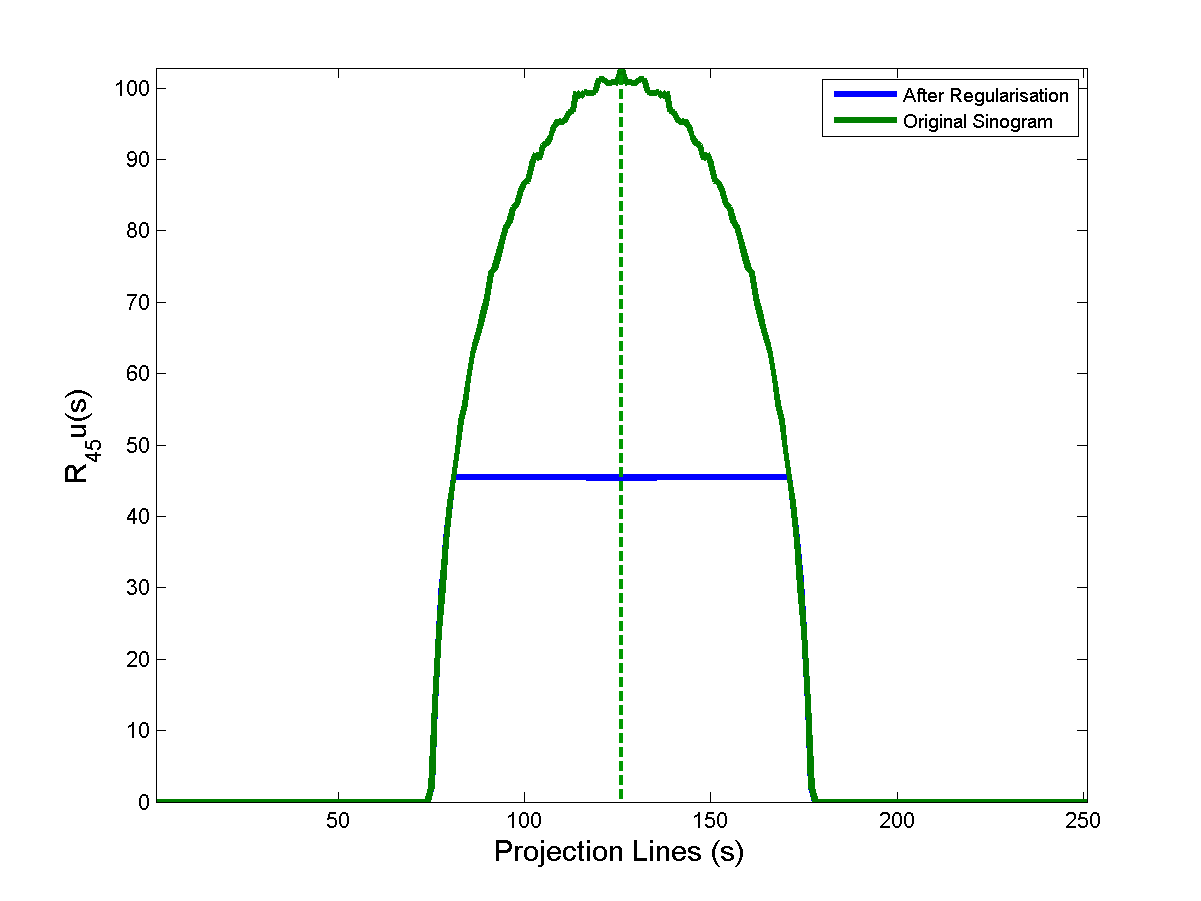}
\end{subfigure}
\begin{subfigure}[h]{3.7cm}
                \centering              
                \includegraphics[scale=0.2]{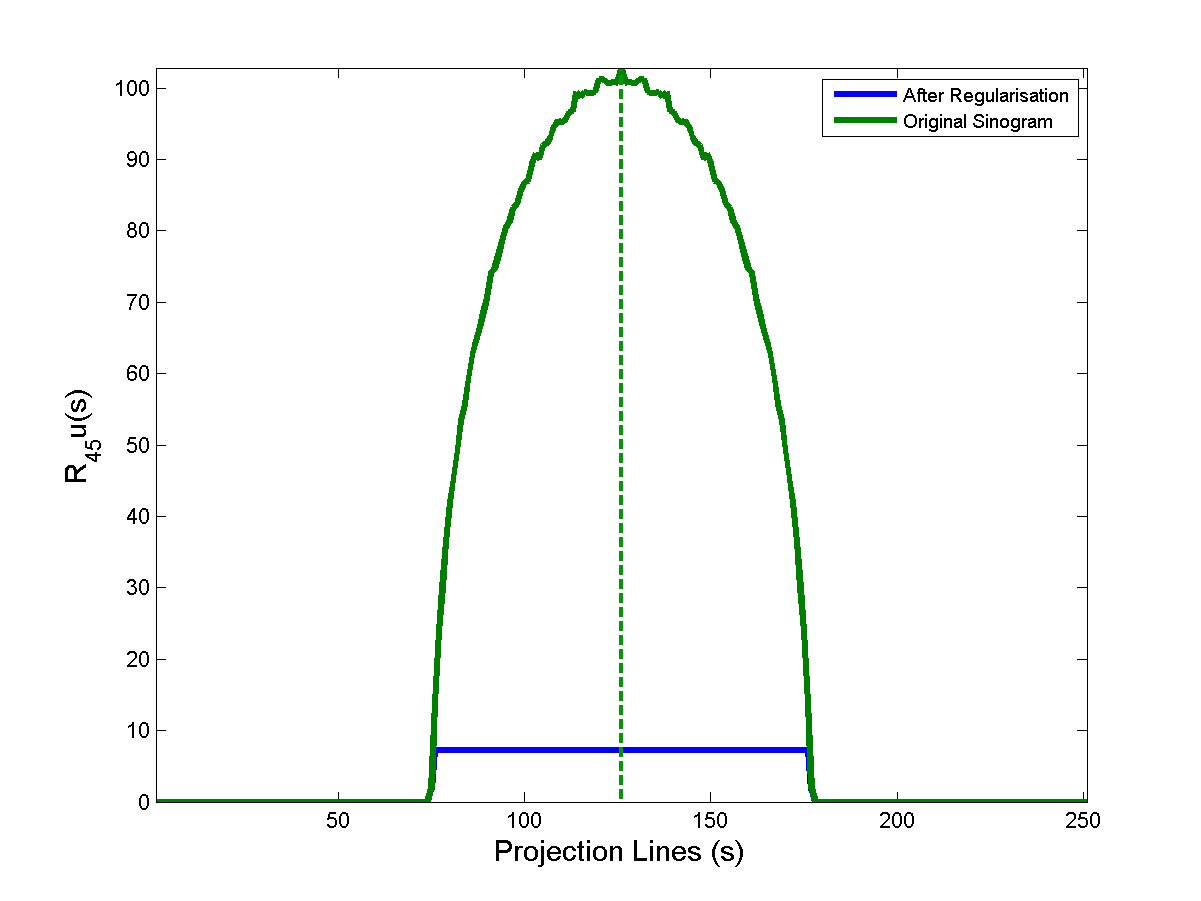}
\end{subfigure}
\begin{subfigure}[h]{3.7cm}
                \centering 
                 \includegraphics[scale=0.2]{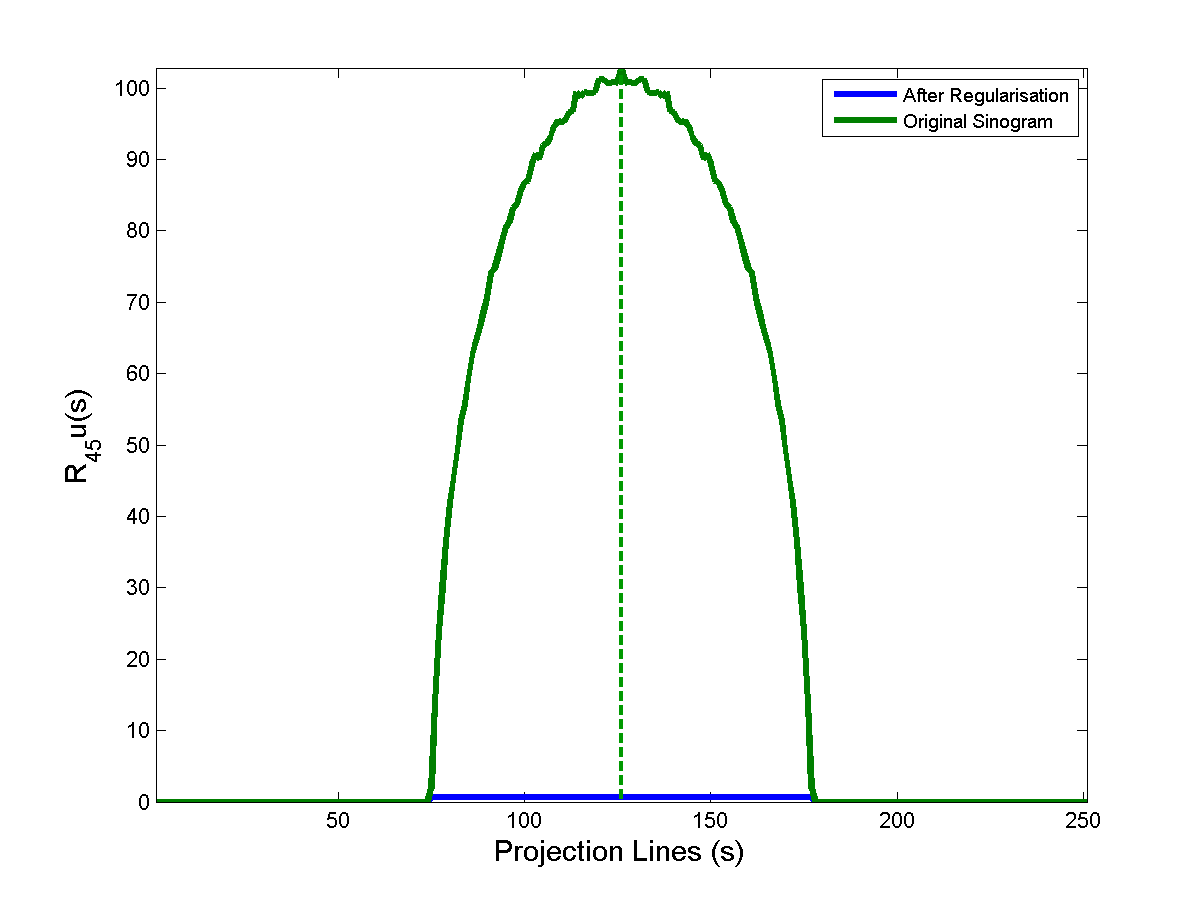}
\end{subfigure}\\
\begin{subfigure}[h]{3.7cm}
                \centering              
                \includegraphics[scale=0.2]{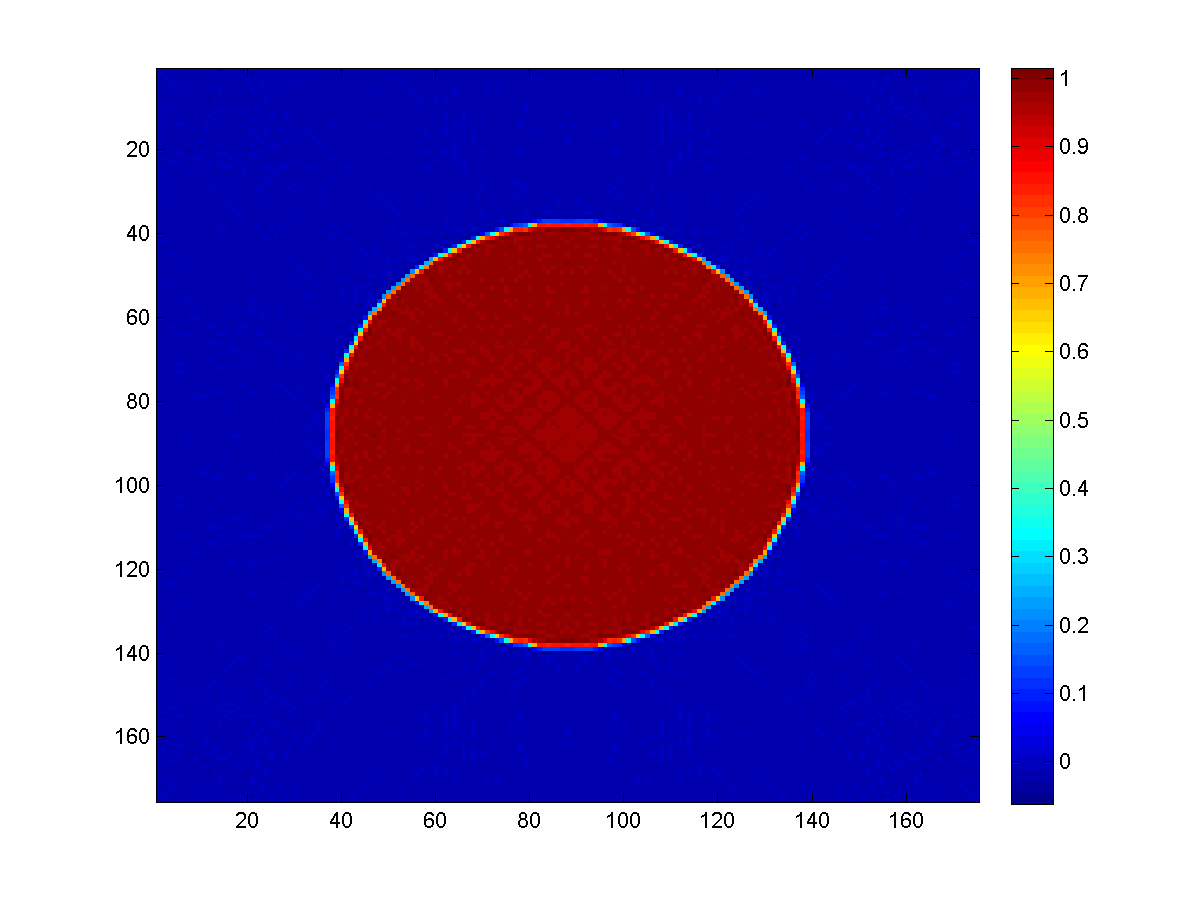}
\end{subfigure}
\begin{subfigure}[h]{3.7cm}
                \centering          
                \includegraphics[scale=0.2]{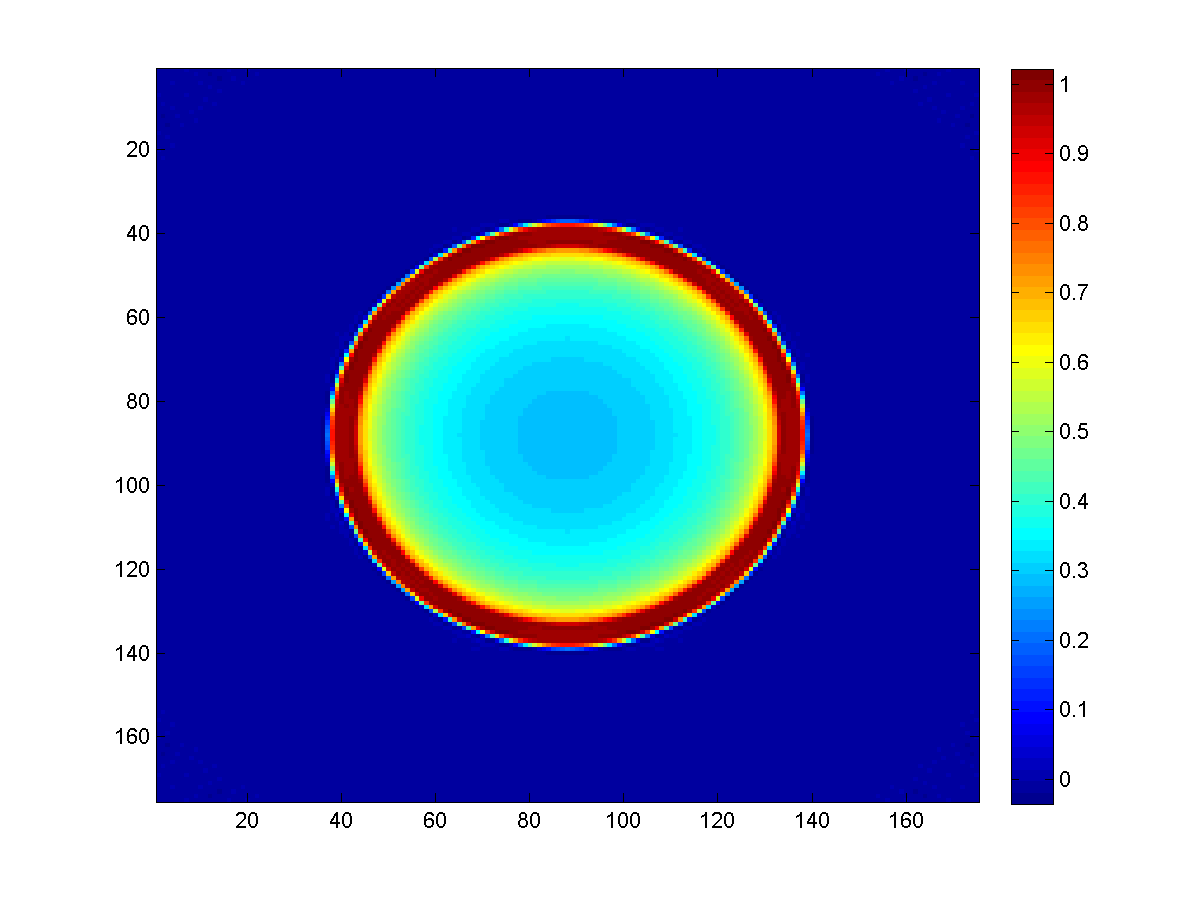}
\end{subfigure}
\begin{subfigure}[h]{3.7cm}
                \centering                            
                 \includegraphics[scale=0.2]{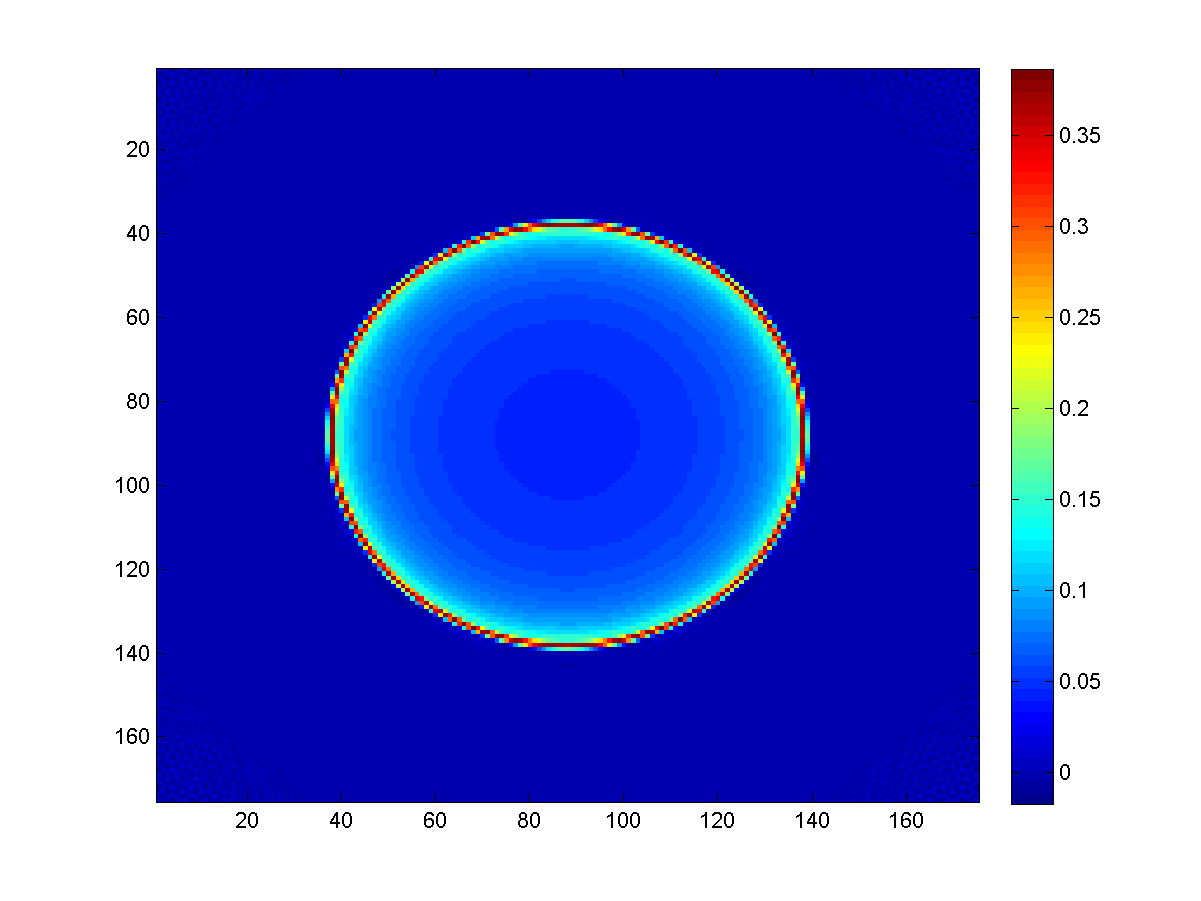}
\end{subfigure}
\begin{subfigure}[h]{3.7cm}
                \centering                           
                 \includegraphics[scale=0.2]{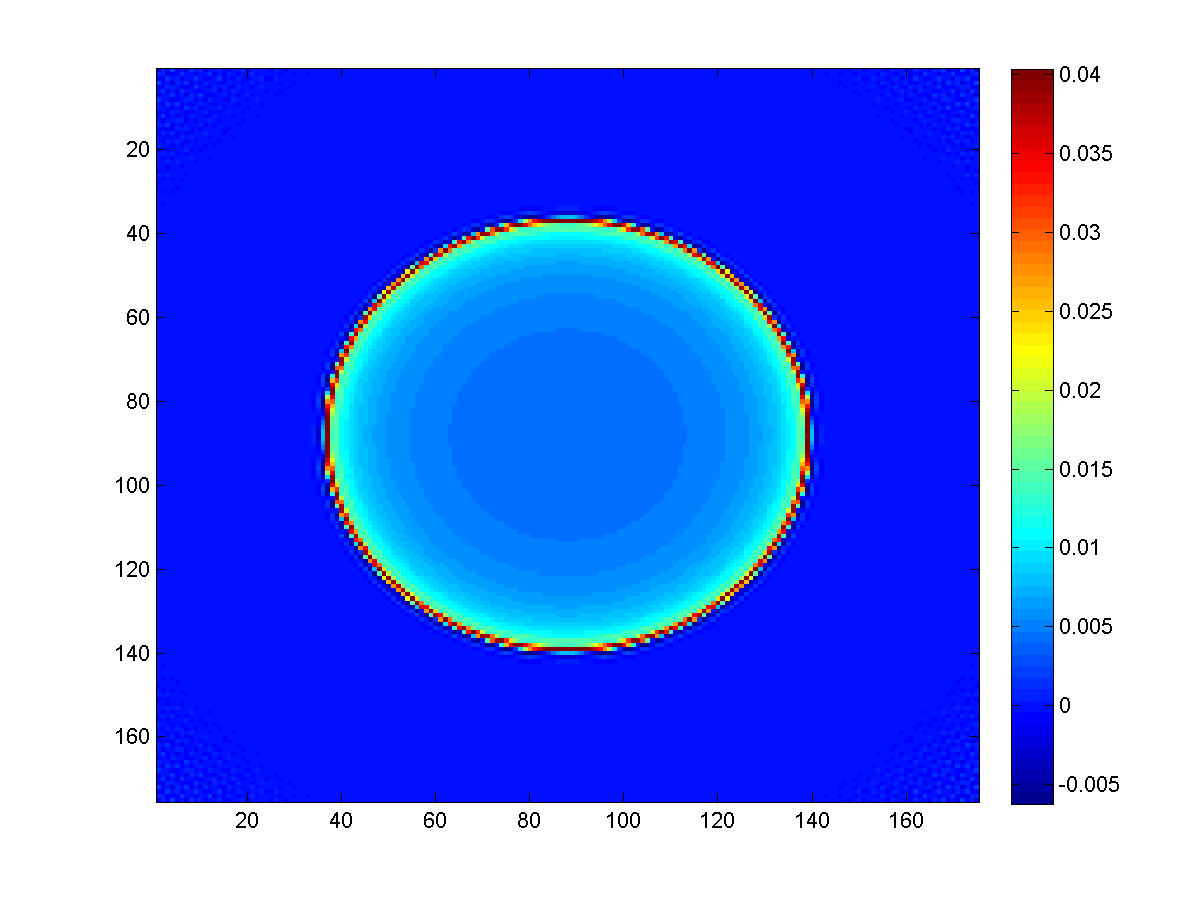}
\end{subfigure}\\
\begin{subfigure}[h]{3.7cm}
                 \centering                            
                  \includegraphics[scale=0.2]{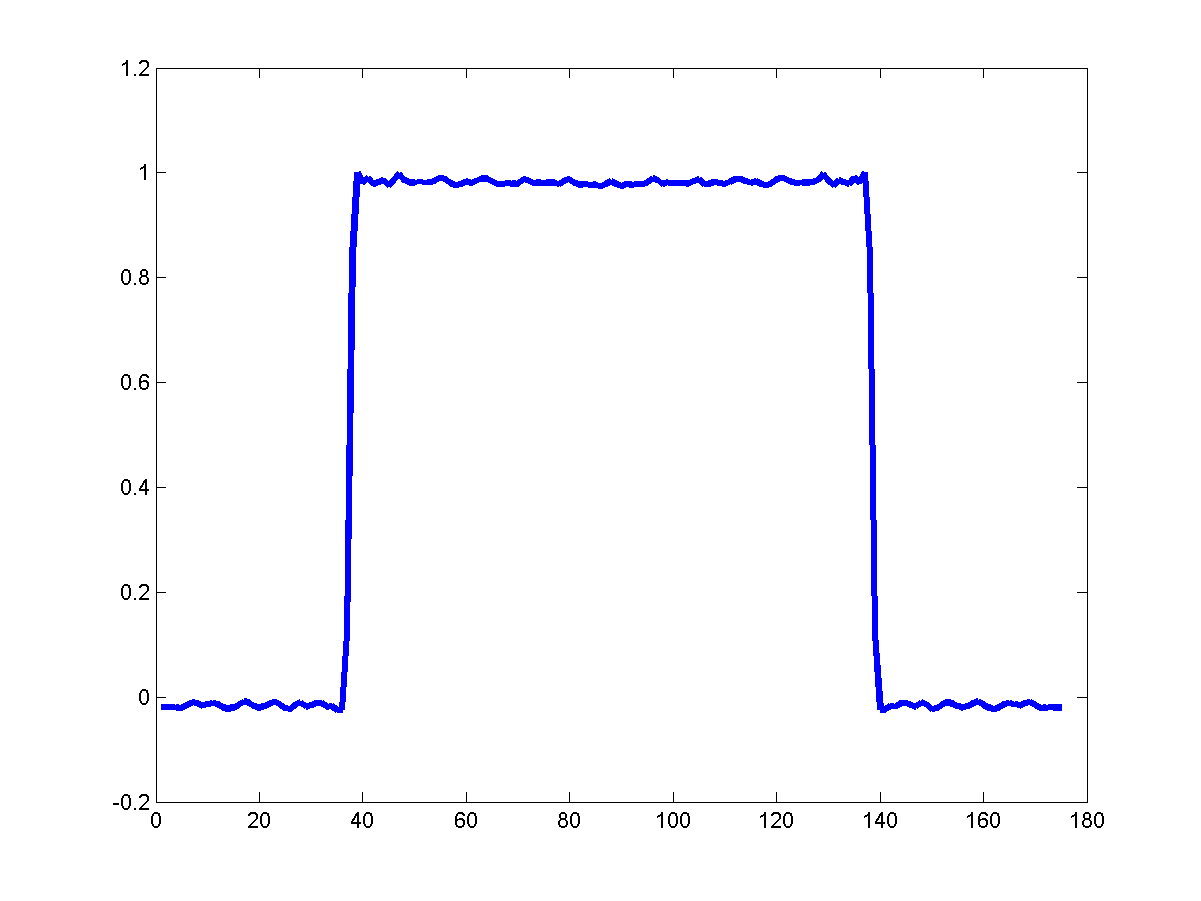}
\end{subfigure}
\begin{subfigure}[h]{3.7cm}
                 \centering                            
                  \includegraphics[scale=0.2]{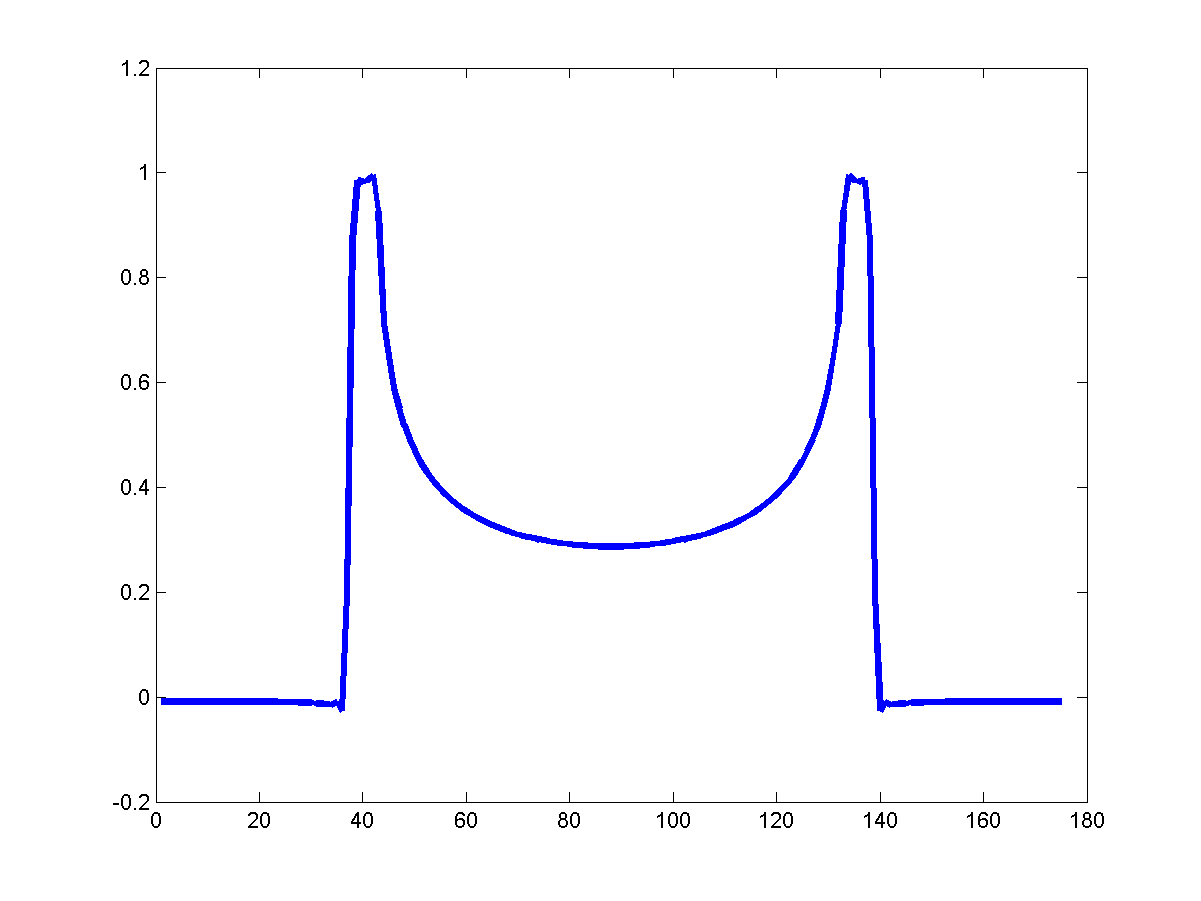}
\end{subfigure}
\begin{subfigure}[h]{3.7cm}
                 \centering                            
                  \includegraphics[scale=0.2]{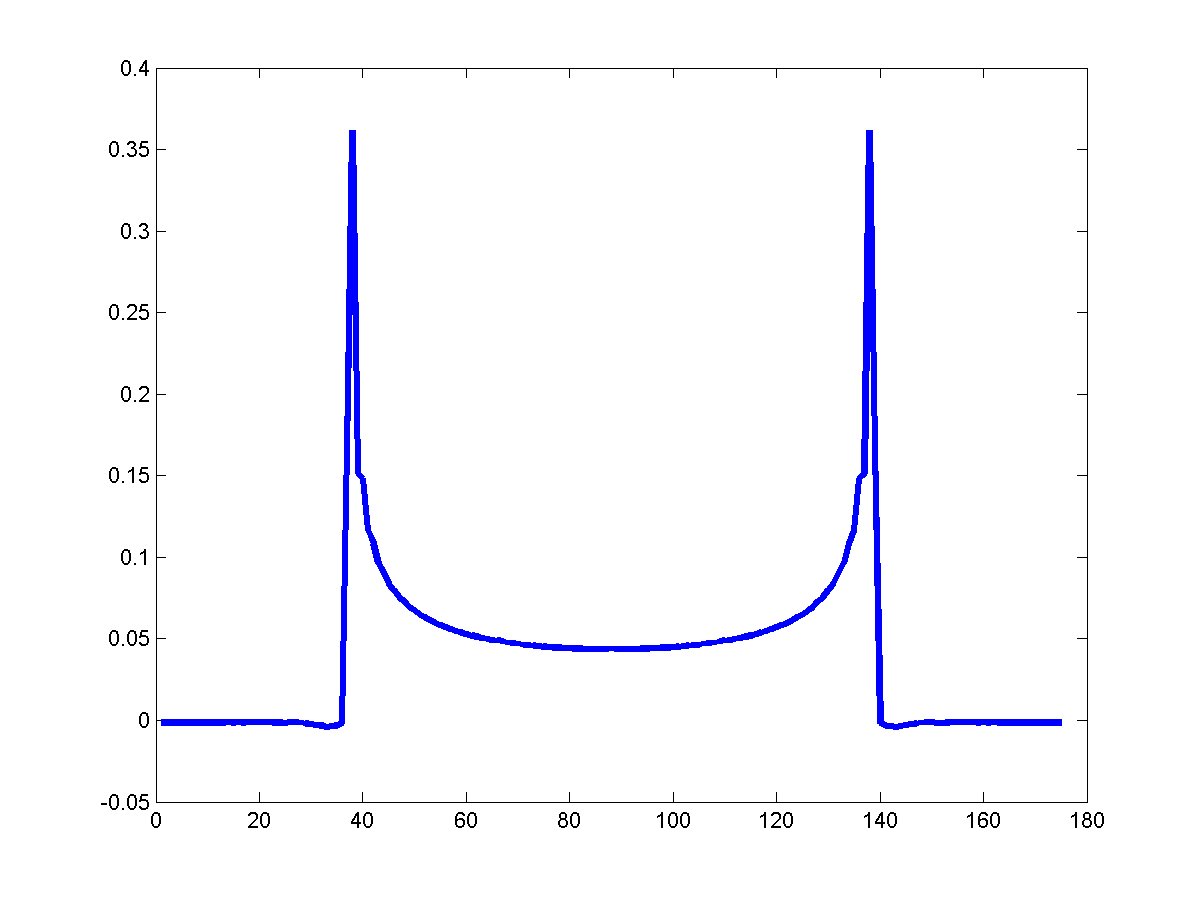}
\end{subfigure}
\begin{subfigure}[h]{3.7cm}
                 \centering                             
                 \includegraphics[scale=0.2]{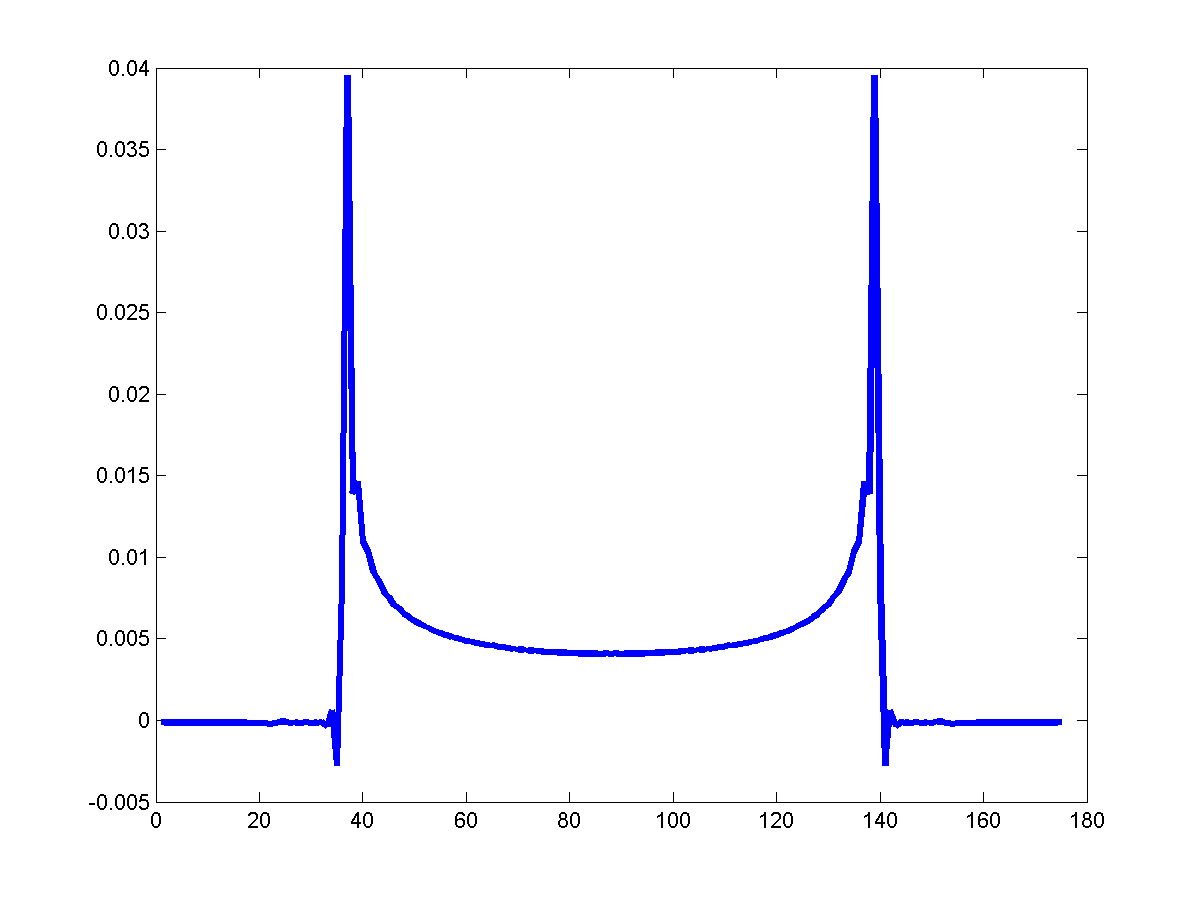}
\end{subfigure}
\caption{Sinogram regularisation with different values of $\beta$ (first row) and the corresponding filtered backprojected images (third row). The second row represents a $45^{o}$ comparison of the original sinogram and the sinogram after regularisation. The highest value of the sinogram is $102.8$. The fourth row represents the middle line profiles of the reconstructed images in the third row.}
\label{sxima11b}
\end{center}
\end{figure}

First, we consider image functions with radial symmetry such as in section \ref{explicit} equation \eqref{candidate}. Figure \ref{sxima11b} shows the numerically computed regularised sinograms and corresponding images for an original image of a disc with radius $r=50.5$. Here we have used MATLAB's built-in function \textit{iradon} with a \textit{Ram-Lak} filter and \textit{spline} interpolation to compute the FBP of the regularised sinogram. Moreover, Table \ref{table5}, shows the correspondence of the numerical solution with the analytic solution in section \ref{explicit} for three discs of radii $r=15.5, 30.5$ and $50.5$. Here, $\delta^{an}$ and $\delta^{num}$ denote the analytic and numerical $\delta$, respectively, in the expression of the regularised solution in \eqref{candidate}. As predicted from the computations in section \ref{explicit}, we see that with increasing regularisation parameter $\beta$ the regularised image more and more emphasises the boundary of the disc. 


\begin{table}[h!]
\begin{center}
\begin{tabular}{cc|c|c|c|c|c|c|c|c}
\cline{1-9}
\multicolumn{1}{ |c| }{\multirow{3}{*}{$r=15.5$} } &
\multicolumn{1}{ |c| }{$\beta$} & $10^{-3}$ & 0.1 & 1 & 5 & 10 & 15 & 15.5 &     \\ \cline{2-9}
\multicolumn{1}{ |c  }{}                        &
\multicolumn{1}{ |c| }{$\delta^{an}$} & 30.94 &	29.88	& 25.84	& 16.09	& 7.59 & 0.64	& 0.084	&     \\ \cline{2-9}
\multicolumn{1}{ |c  }{}                        &
\multicolumn{1}{ |c| }{$\delta^{num}$} & 31.32 &	29.76 &	25.71 &	15.96	& 7.37 &	0.67 &	0.37 &\\ \cline{1-9}
\multicolumn{1}{ |c| }{\multirow{3}{*}{$r=30.5$} } &
\multicolumn{1}{ |c| }{$\beta$} & $10^{-3}$ & 1 & 10 & 15 & 20 & 25 & 30.5 &     \\ \cline{2-9}
\multicolumn{1}{ |c  }{}                        &
\multicolumn{1}{ |c| }{$\delta^{an}$} & 60.93 &	59.6	& 31.35	& 22.37	& 14.46 & 7.27	& 0.09	&     \\ \cline{2-9}
\multicolumn{1}{ |c  }{}                        &
\multicolumn{1}{ |c| }{$\delta^{num}$} & 61.98 &	54.58 &	31.42 &	22.47	& 14.55 &	7.34 &	0.65 &\\ \cline{1-9}
\multicolumn{1}{ |c| }{\multirow{3}{*}{$r=50.5$} } &
\multicolumn{1}{ |c| }{$\beta$} & $10^{-3}$ & 1 & 10 & 20 & 30 & 45 & 50.5 &     \\ \cline{2-9}
\multicolumn{1}{ |c  }{}                        &
\multicolumn{1}{ |c| }{$\delta^{an}$} & 100.92 &	93.33	& 65.74	& 45.46	& 28.71 & 7.16	& 0.12	&     \\ \cline{2-9}
\multicolumn{1}{ |c  }{}                        &
\multicolumn{1}{ |c| }{$\delta^{num}$} & 101.83 &	93.26 &	65.75 &	45.41	& 28.82 &	7.24 &	0.68 &\\ \cline{1-9}
\end{tabular}
\end{center}
\caption{Comparison of analytic and numerical computations of sinogram regularisation for three test images of characteristic functions of circles with radii $r=15.5, 30.5$ and $50.5$. The parameters $\delta^{an}$ and $\delta^{num}$ denote the analytic and numerical $\delta$, respectively, in the expression of the regularised solution in \eqref{candidate}. Compare also Figure \ref{sxima11b} for regularised reconstructions for the circle with radius $r=50.5$.}
\label{table5}
\end{table}

%

Going beyond radial symmetry we consider three additional examples where the sinogram depends on the angle $\theta$. First, we simply consider the image that we used in the previous section in Figure \ref{sxima4} without adding additional noise to its sinogram. The effect of $\beta$ regularisation in this case is presented in Figure \ref{sxima13}. We see that as we increase $\beta$ we loose details in the image, starting again from the inner structure of the discs, while enhancing the boundaries of the objects. Here, the connection of the choice of $\beta$ with the radius of every circle is clearly visible. More precisely, for $\beta<r_2$ the boundary of the smaller circle is enhanced and for $r_2<\beta<r_1$ the small circle is lost and the boundary of the larger circle is enhanced.

\begin{figure}[h!]
\begin{center}
\begin{subfigure}[h]{3cm}
                \centering 
                \caption{$\beta=10^{-3}$}               
                \includegraphics[scale=0.15]{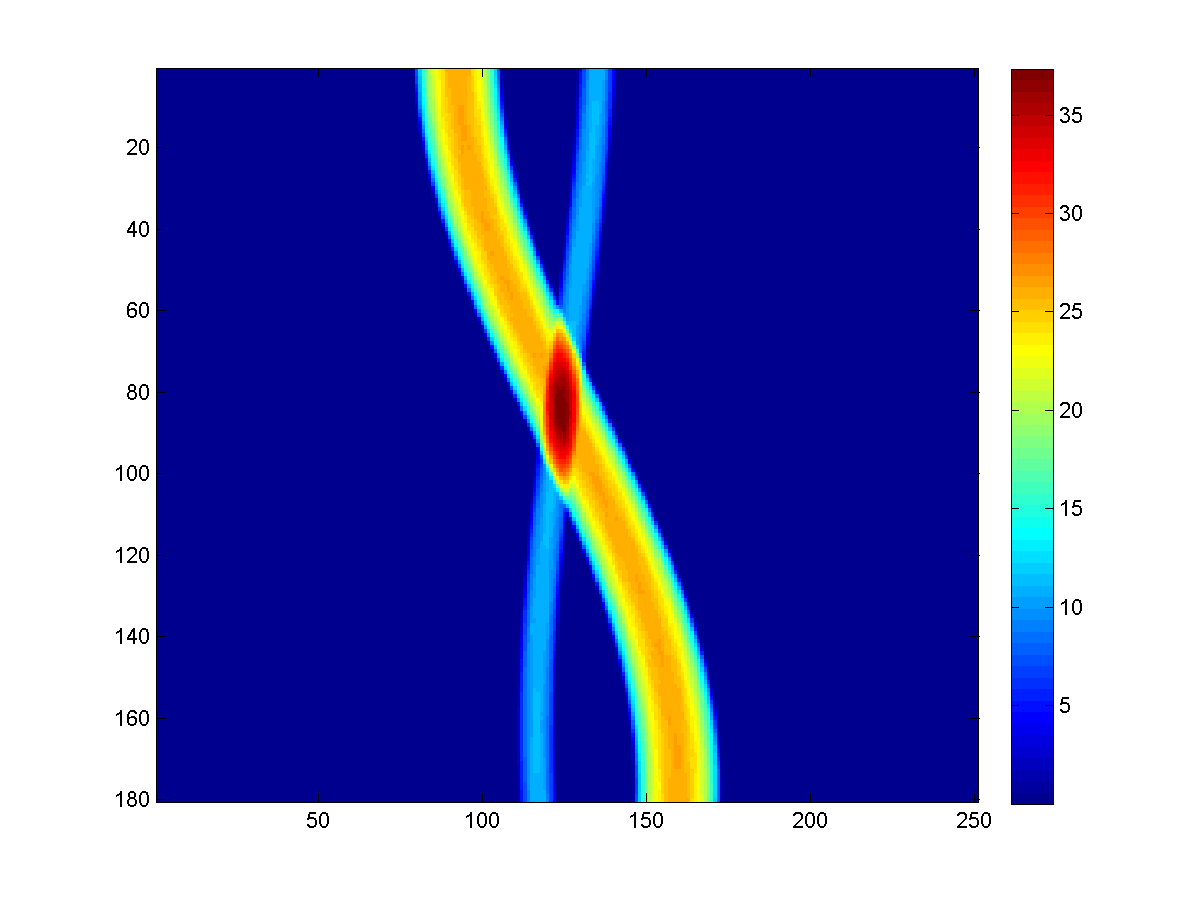}
\end{subfigure}%
\begin{subfigure}[h]{3cm}
                \centering
                \caption{$\beta$=1}  
                \includegraphics[scale=0.15]{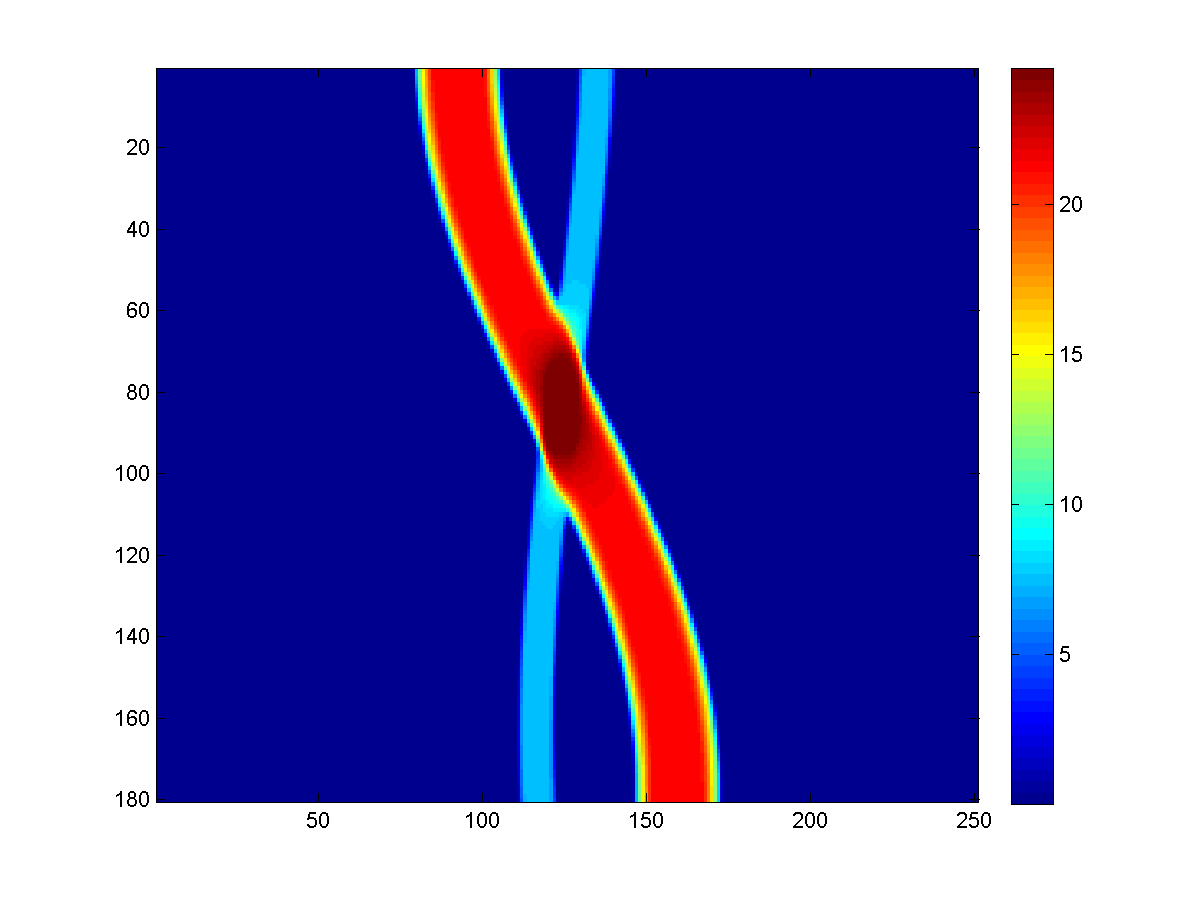}
\end{subfigure}
\begin{subfigure}[h]{3cm}
                \centering
                \caption{$\beta$=3}                
                \includegraphics[scale=0.15]{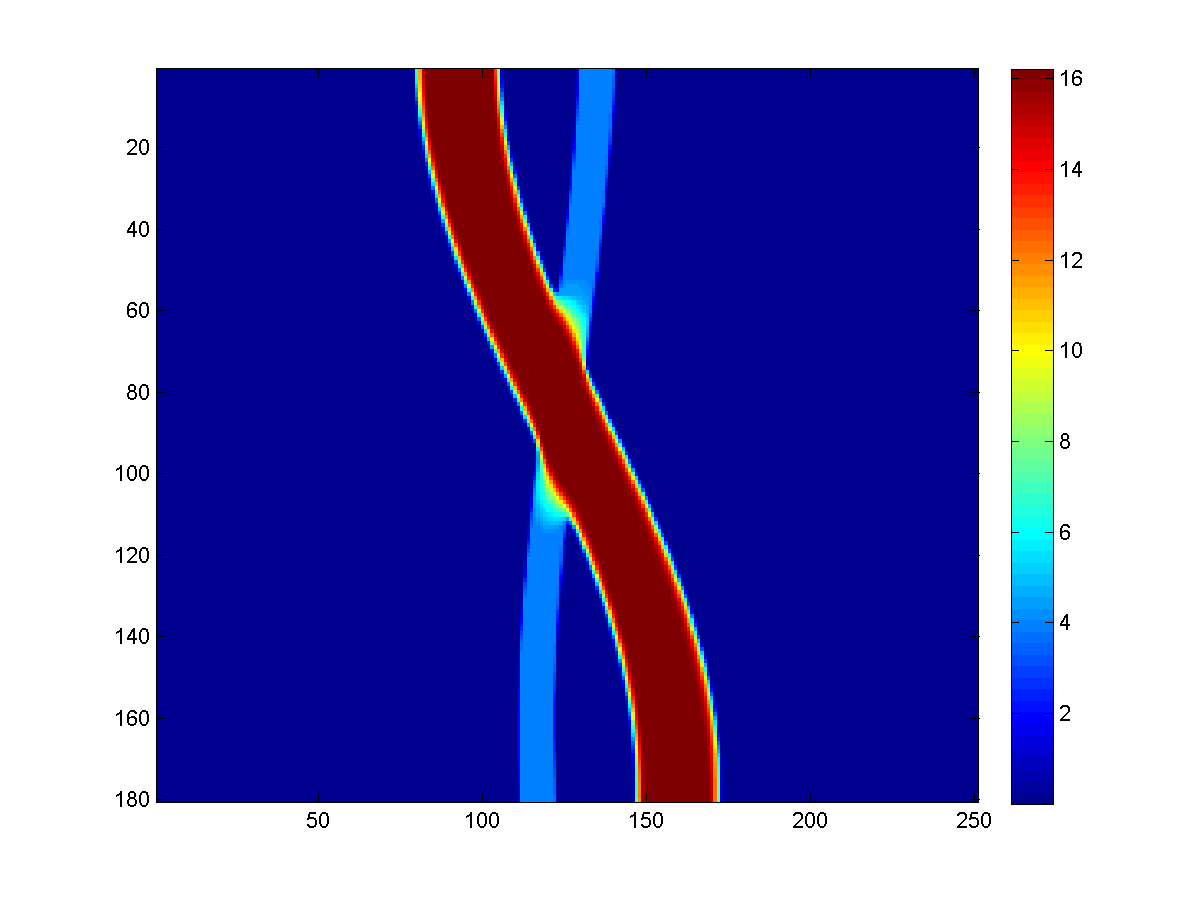}
\end{subfigure}
\begin{subfigure}[h]{3cm}
                \centering
                \caption{$\beta$=7}              
                 \includegraphics[scale=0.15]{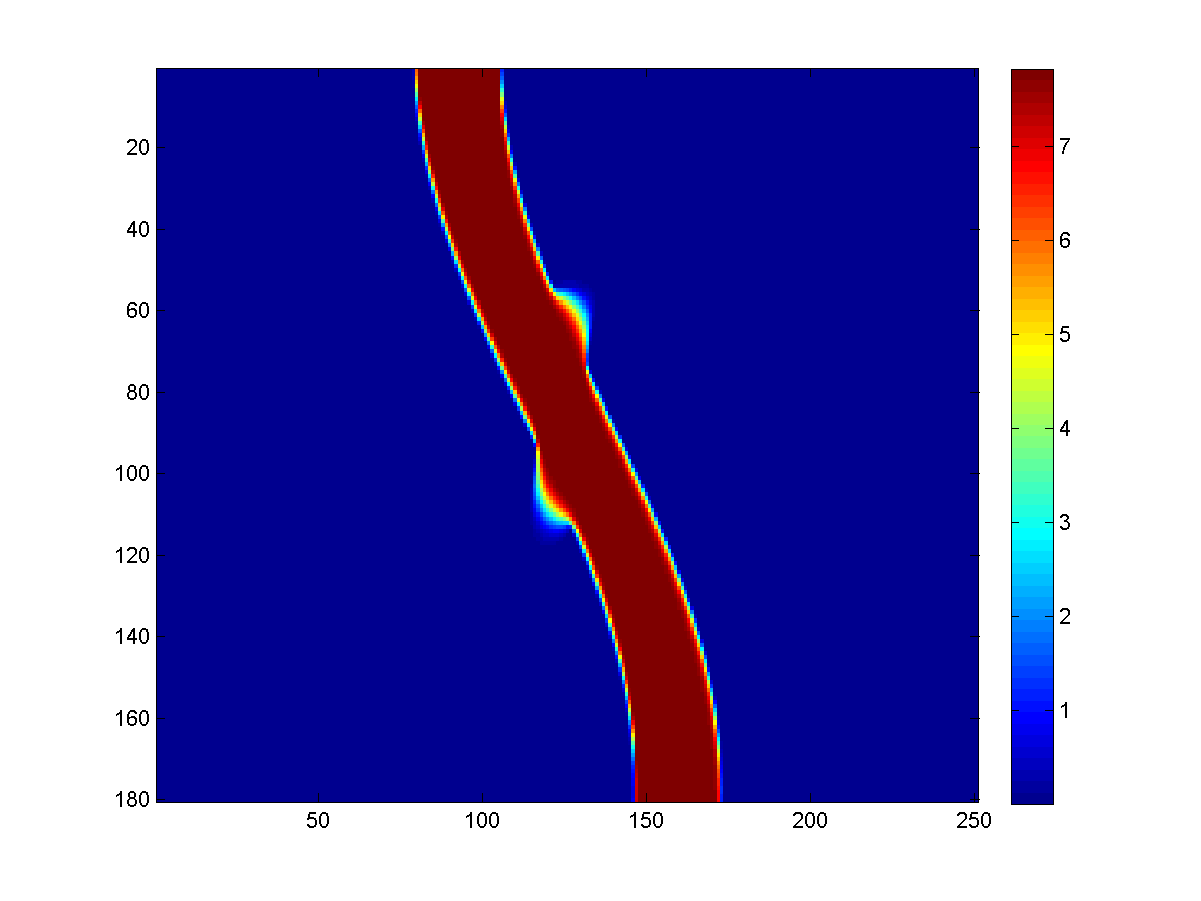}
\end{subfigure}
\begin{subfigure}[h]{3cm}
                \centering
                \caption{$\beta$=12}         
                \includegraphics[scale=0.15]{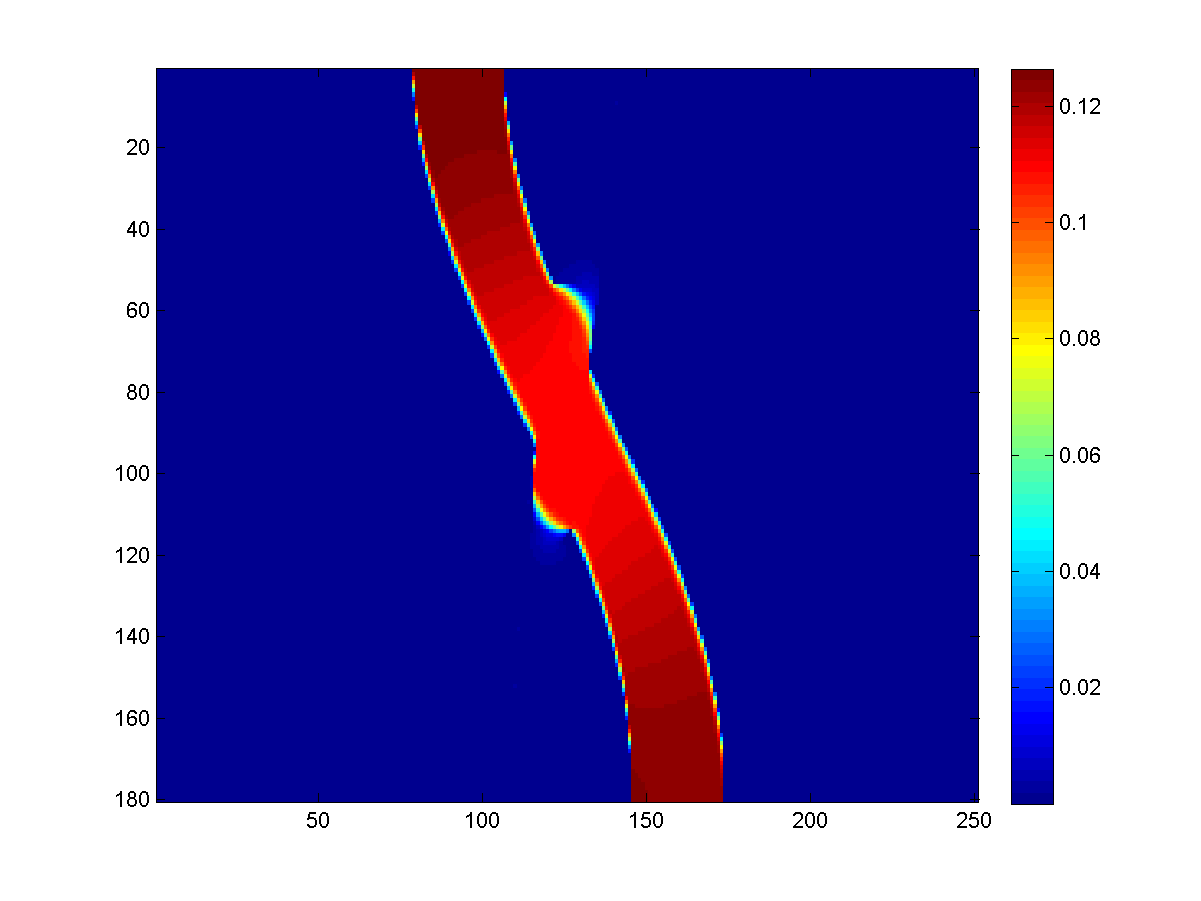}
\end{subfigure}\\
\begin{subfigure}[h]{3cm}
                \centering 
                \includegraphics[scale=0.15]{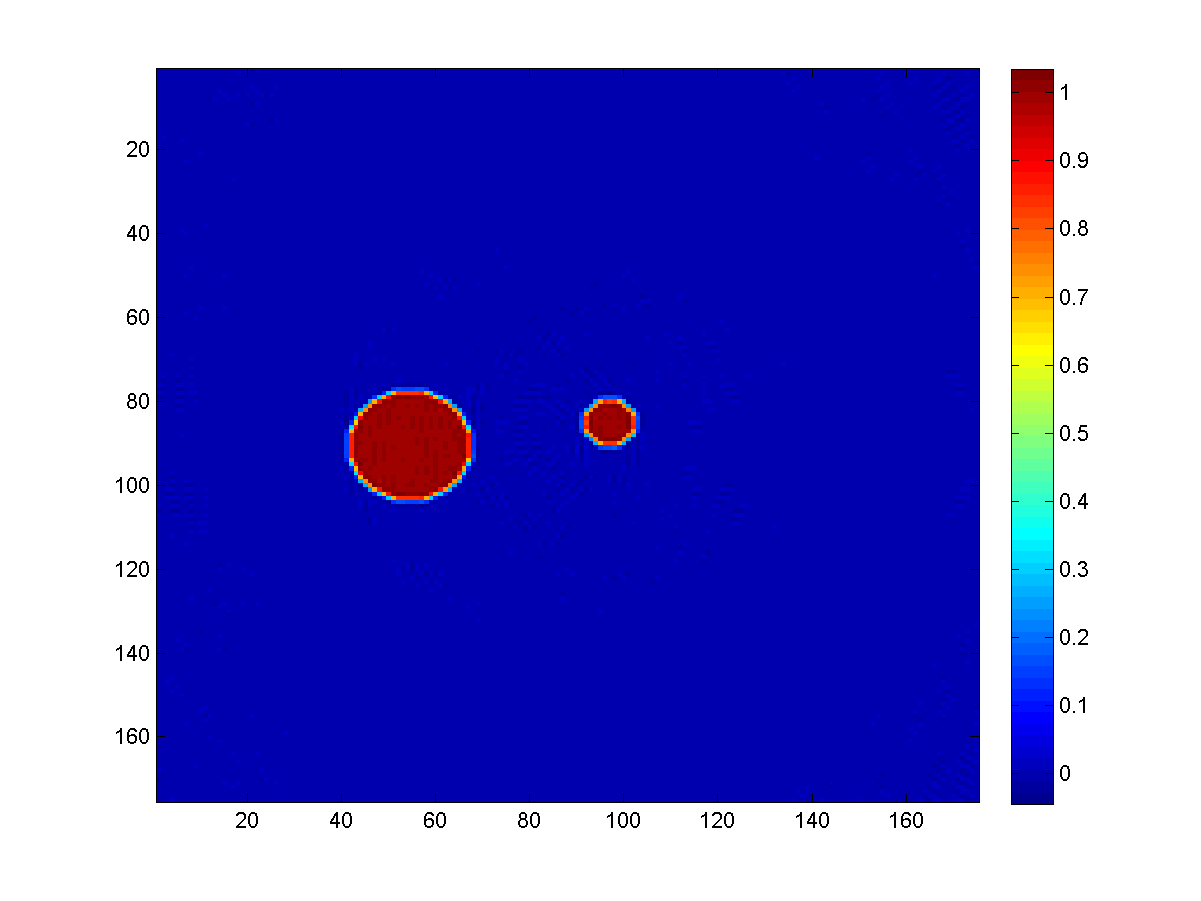}
\end{subfigure}
\begin{subfigure}[h]{3cm}
                \centering 
                \includegraphics[scale=0.15]{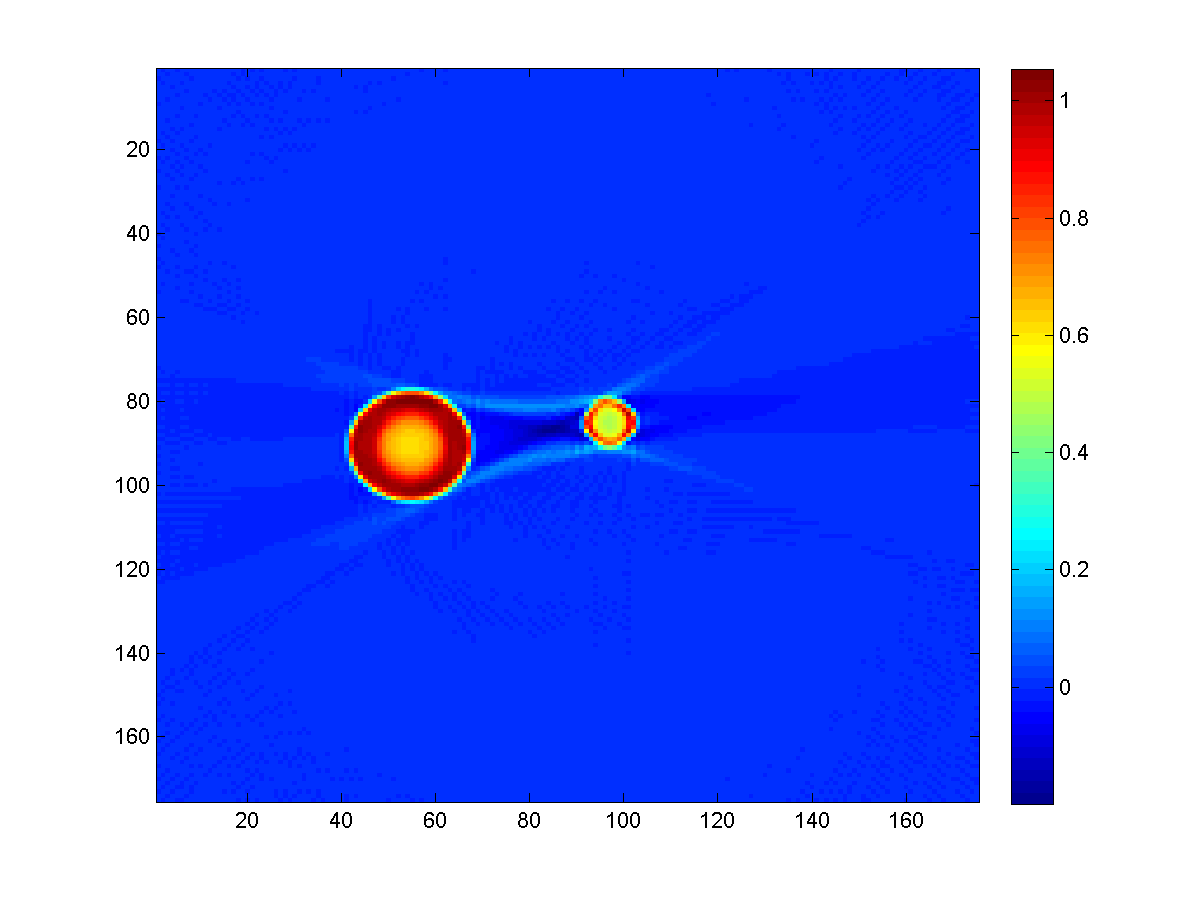}
\end{subfigure}
\begin{subfigure}[h]{3cm}
                \centering 
                \includegraphics[scale=0.15]{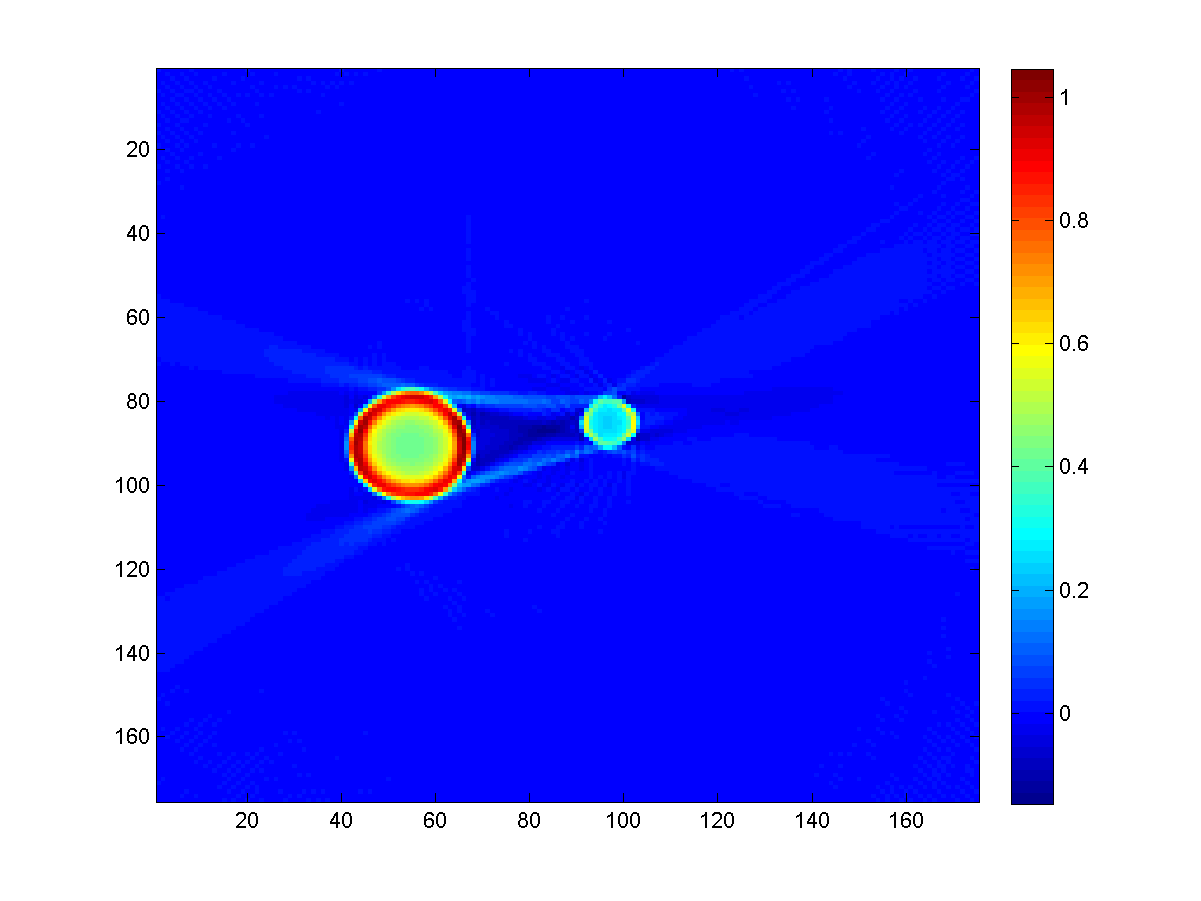}
\end{subfigure}
\begin{subfigure}[h]{3cm}
                \centering 
                \includegraphics[scale=0.15]{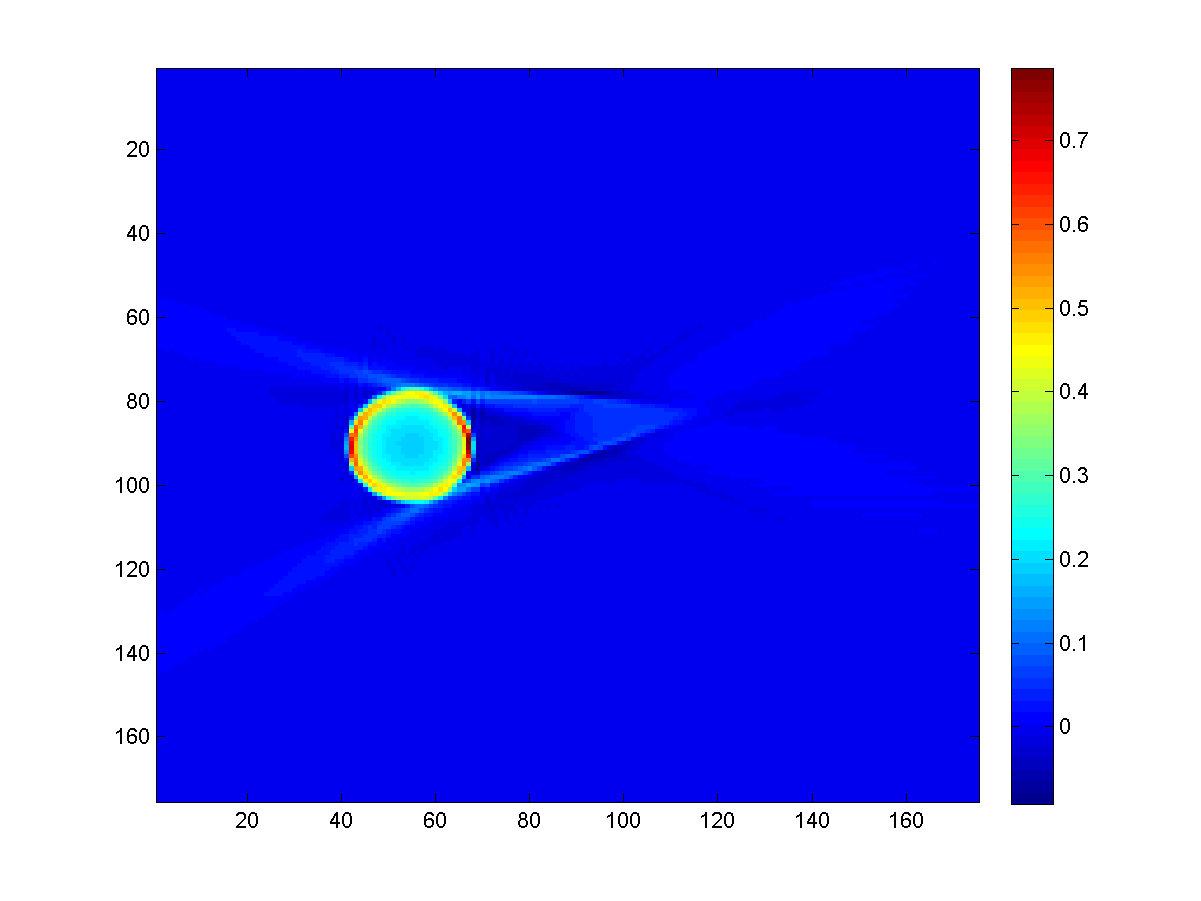}
\end{subfigure}
\begin{subfigure}[h]{3cm}
                \centering 
                \includegraphics[scale=0.15]{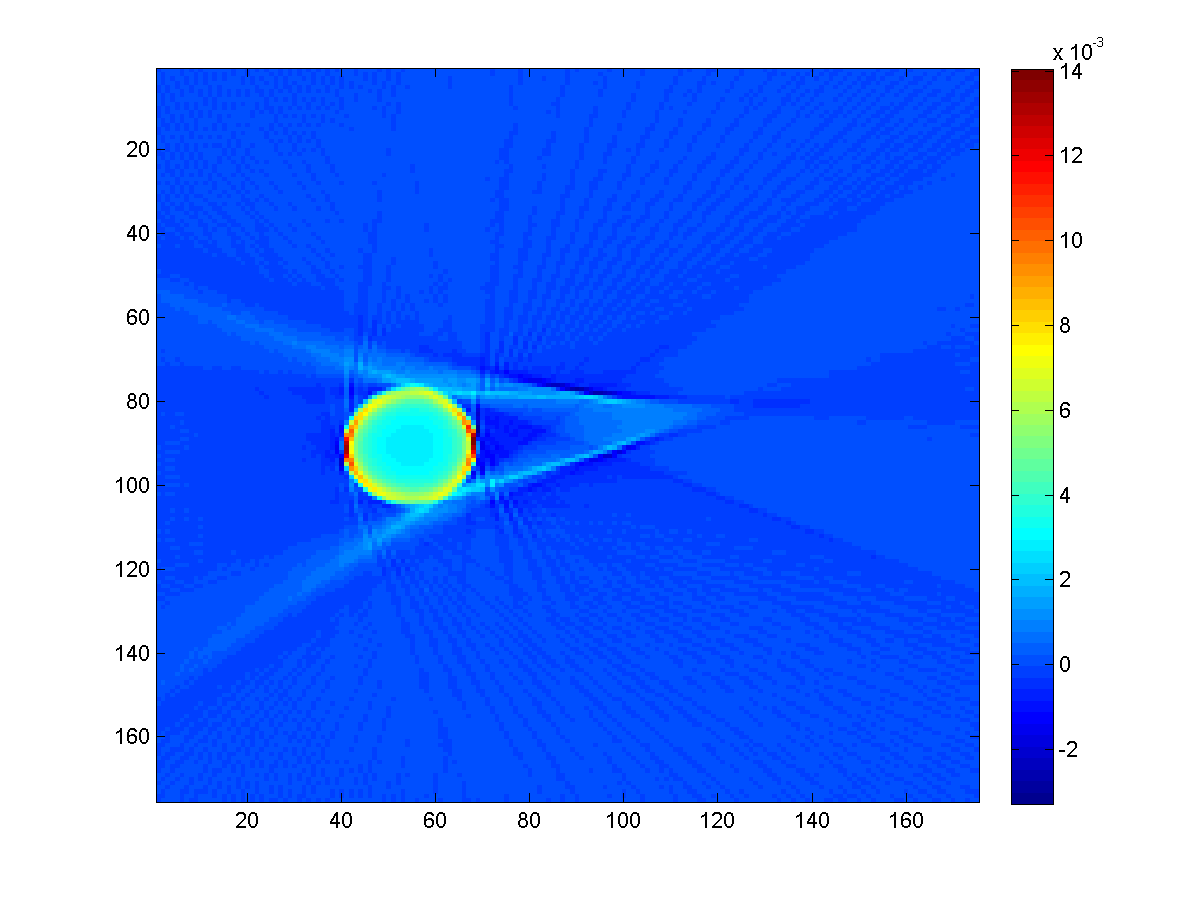}
\end{subfigure}
\caption{Sinogram regularisation with different values of $\beta$ and the corresponding filtered backprojected images using MATLAB's \textit{iradon} built-in function. The radii for the discs are $r_{1}=13$ and $r_{2}=5.5$.}
\label{sxima13}
\end{center}
\end{figure}

\begin{figure}[h!]
\begin{center}
\begin{subfigure}[h]{3cm}
                \centering  
               \includegraphics[scale=0.2]{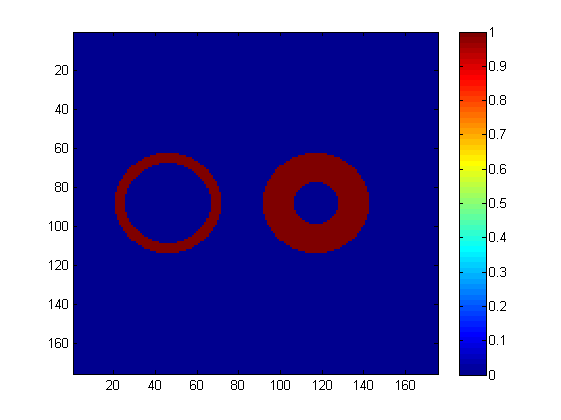}
               \caption{2 rings}
               \label{2rings_star:1}
\end{subfigure}
\begin{subfigure}[h]{3cm}
                \centering            
                  \includegraphics[scale=0.2]{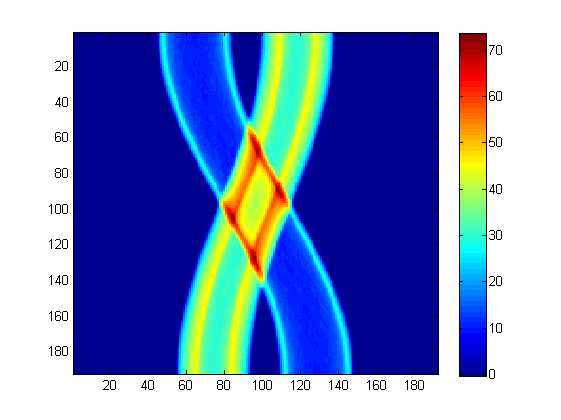}
                  \caption{Sinogram}  
                  \label{2rings_star:2}
\end{subfigure} 
\begin{subfigure}[h]{3cm}
                \centering 
               \includegraphics[scale=0.15]{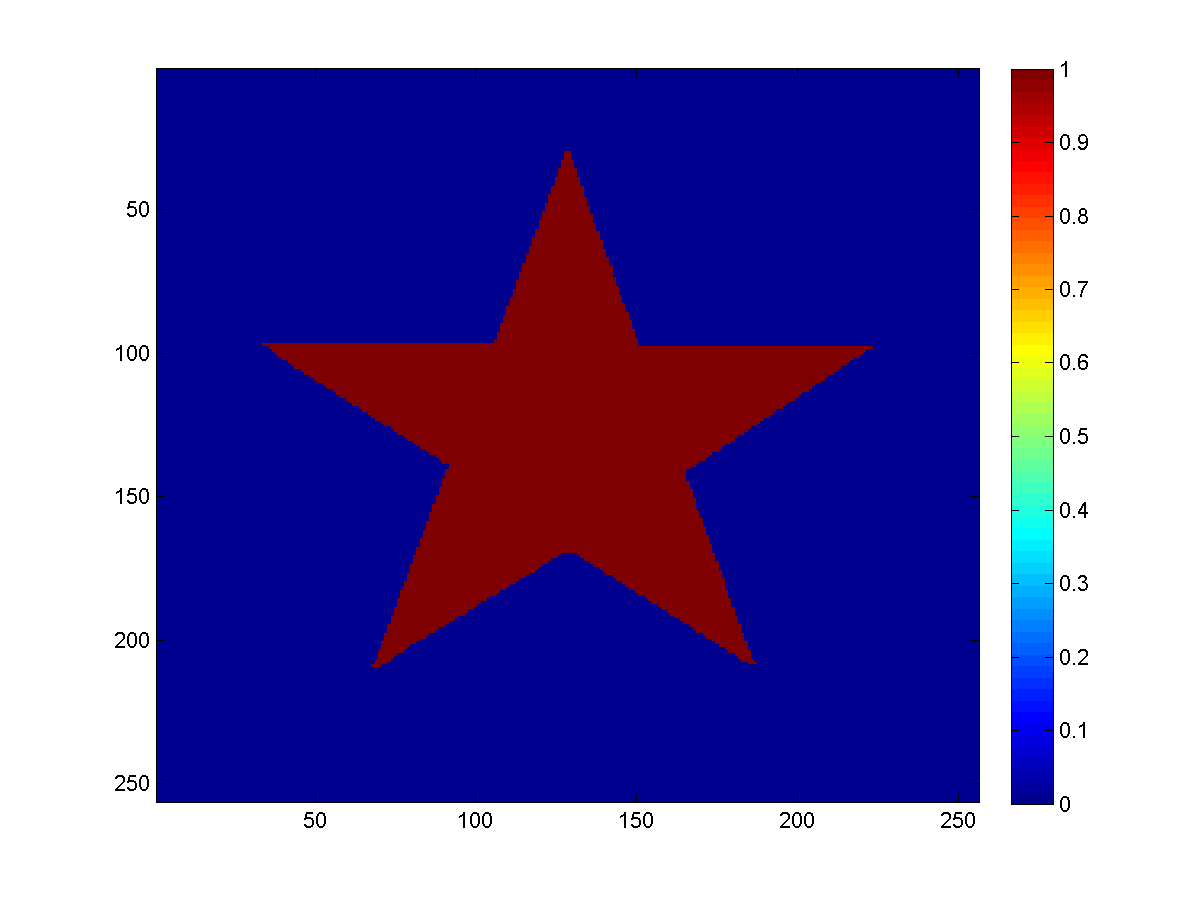}
               \caption{Star} 
               \label{2rings_star:3}
\end{subfigure}
\begin{subfigure}[h]{3cm}
                \centering          
                \includegraphics[scale=0.15]{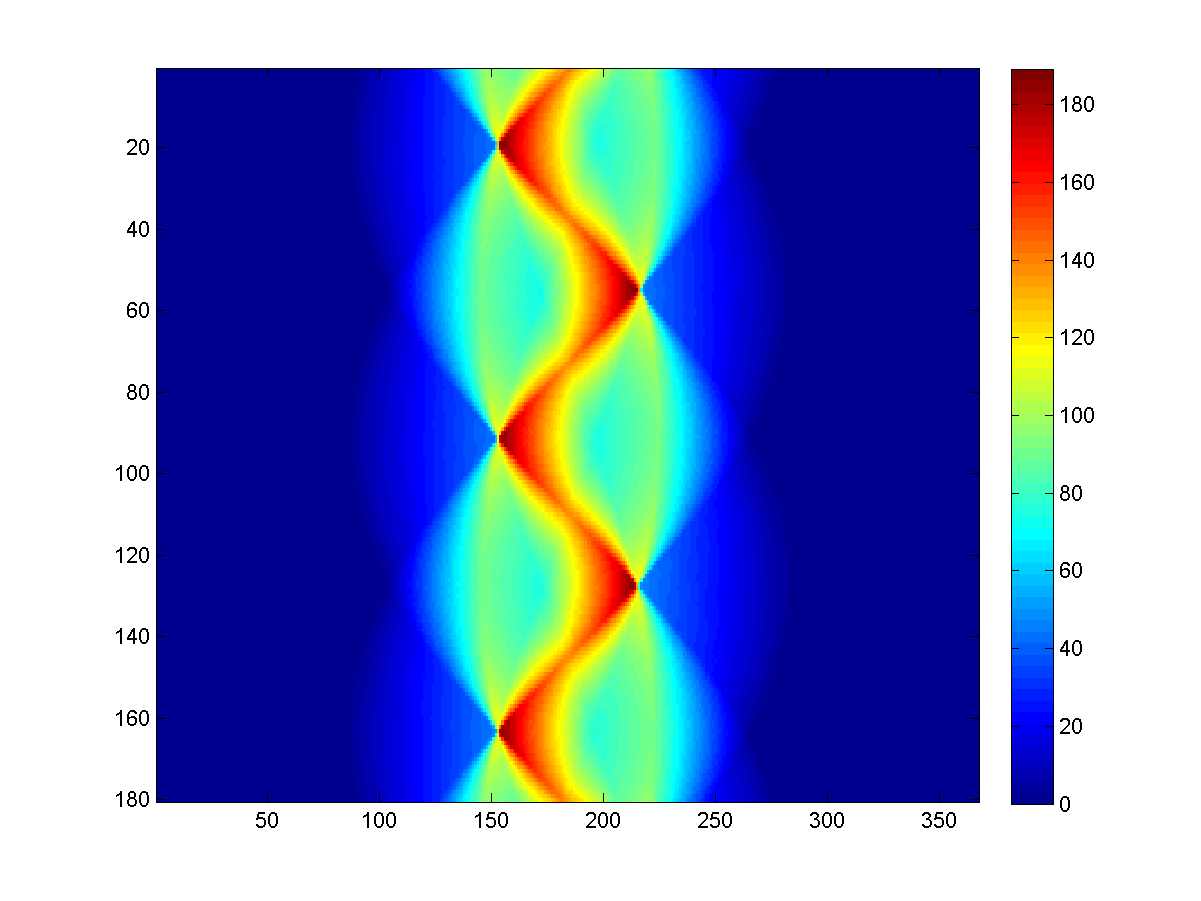}
                \caption{Sinogram}   
                \label{2rings_star:4} 
\end{subfigure}
\caption{2 rings with different annulus regions and its sinogram ((a) and (b)), star-shaped image of 5 points and its sinogram ((c) and (d)).}
\label{2rings_star}
\end{center}
\end{figure}

In Figure \ref{2rings_star}, we present two more test images. The first one is an image of two rings with the same outer radius but with different annulus regions, compare Figure \eqref{2rings_star:1}. A similar scale-space analysis as for the previous examples is carried out in Figure \ref{sxima16}. Additionally to the enhancement of the outer boundaries of the two rings we see that for increasing $\beta$ regularisation the reconstructed image approaches the convex hull of the two rings. This is even more apparent for the last example of a star-shaped object in Figure \eqref{2rings_star:3}. See Figure \ref{sxima16i} in particular.

\begin{figure}[h!]
\begin{center}
\begin{subfigure}[h]{3cm}
                \centering
                \caption{$\beta$=0.001}  
 \includegraphics[scale=0.15]{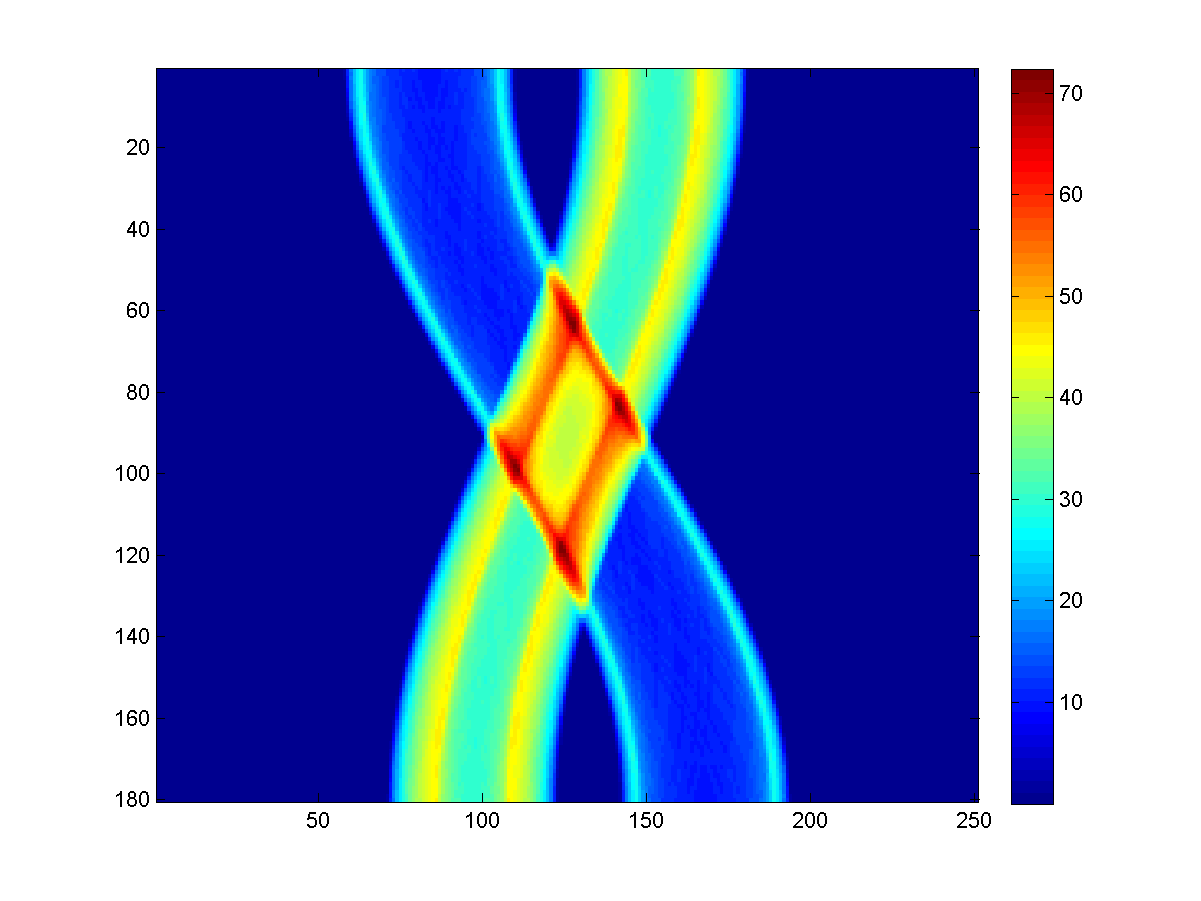}
\end{subfigure}
\begin{subfigure}[h]{3cm}
                \centering
                \caption{$\beta$=1}   \includegraphics[scale=0.15]{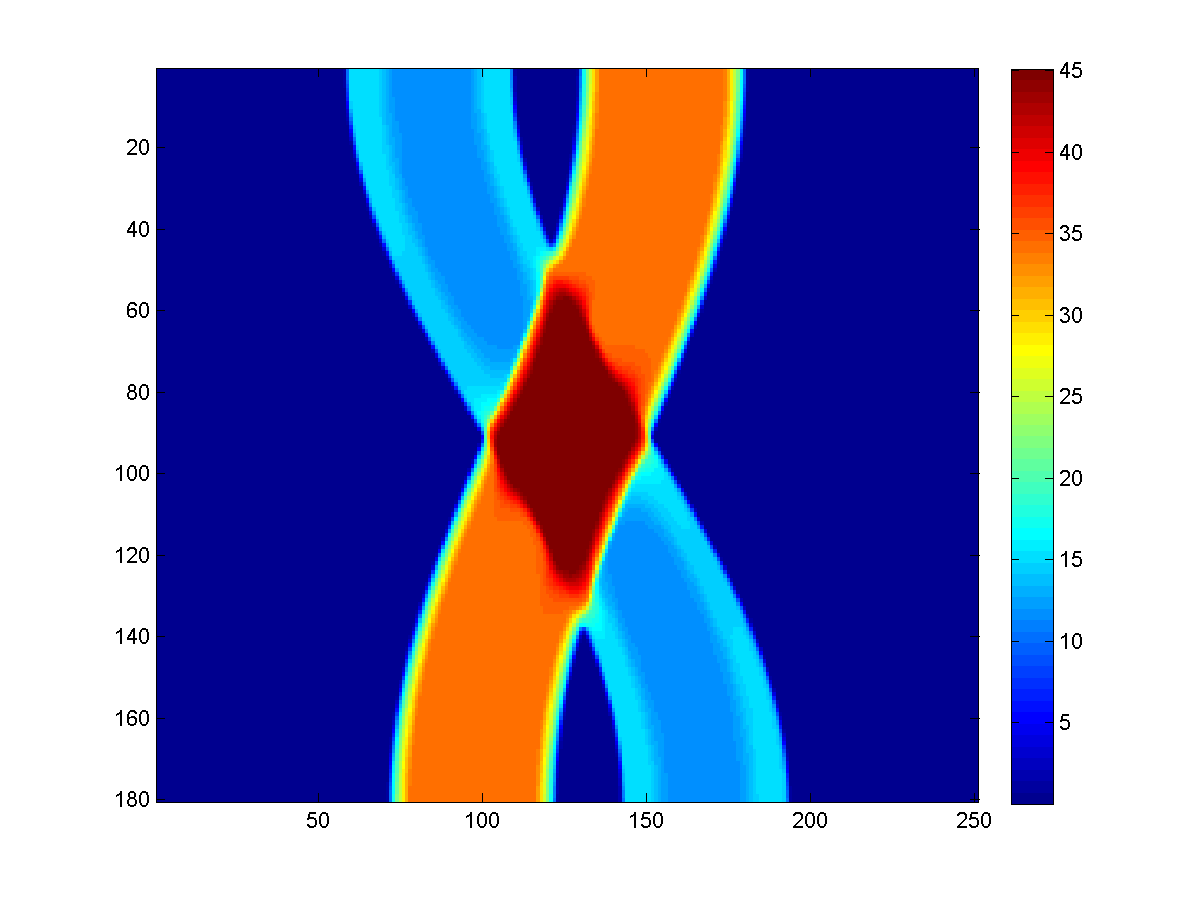}
\end{subfigure}
\begin{subfigure}[h]{3cm}
                \centering
                \caption{$\beta$=10}   \includegraphics[scale=0.15]{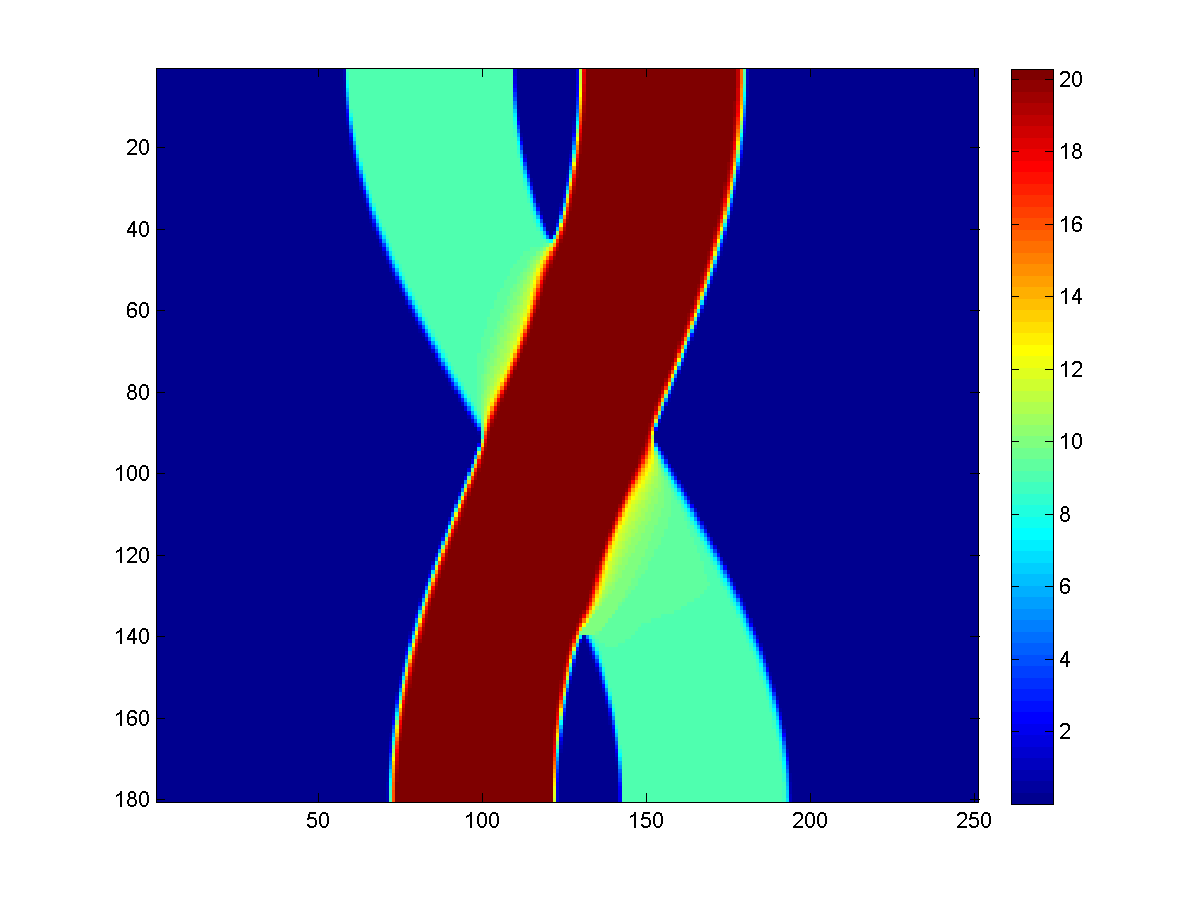}
\end{subfigure}
\begin{subfigure}[h]{3cm}
                \centering 
                \caption{$\beta$=25.5} \includegraphics[scale=0.15]{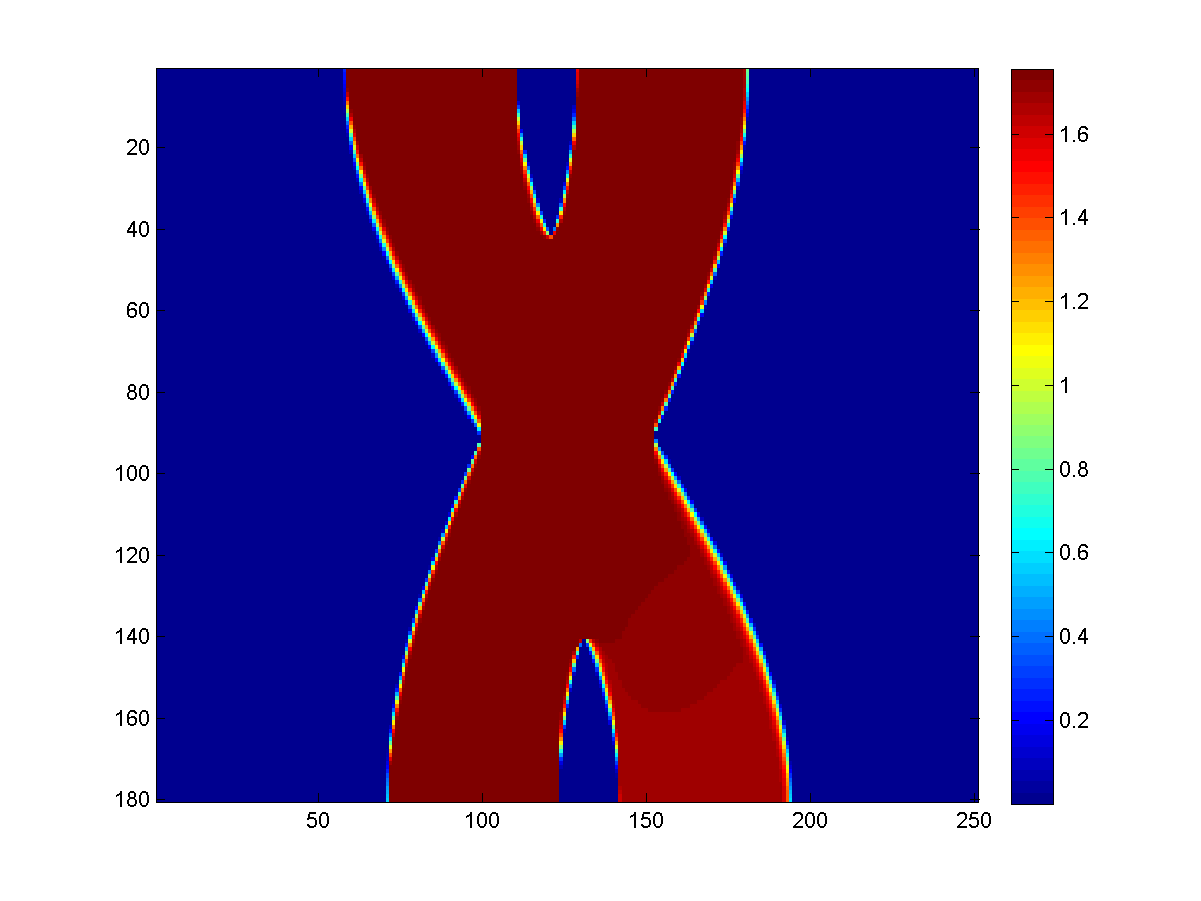}
\end{subfigure}\\
\begin{subfigure}[h]{3cm}
                \centering
 \includegraphics[scale=0.15]{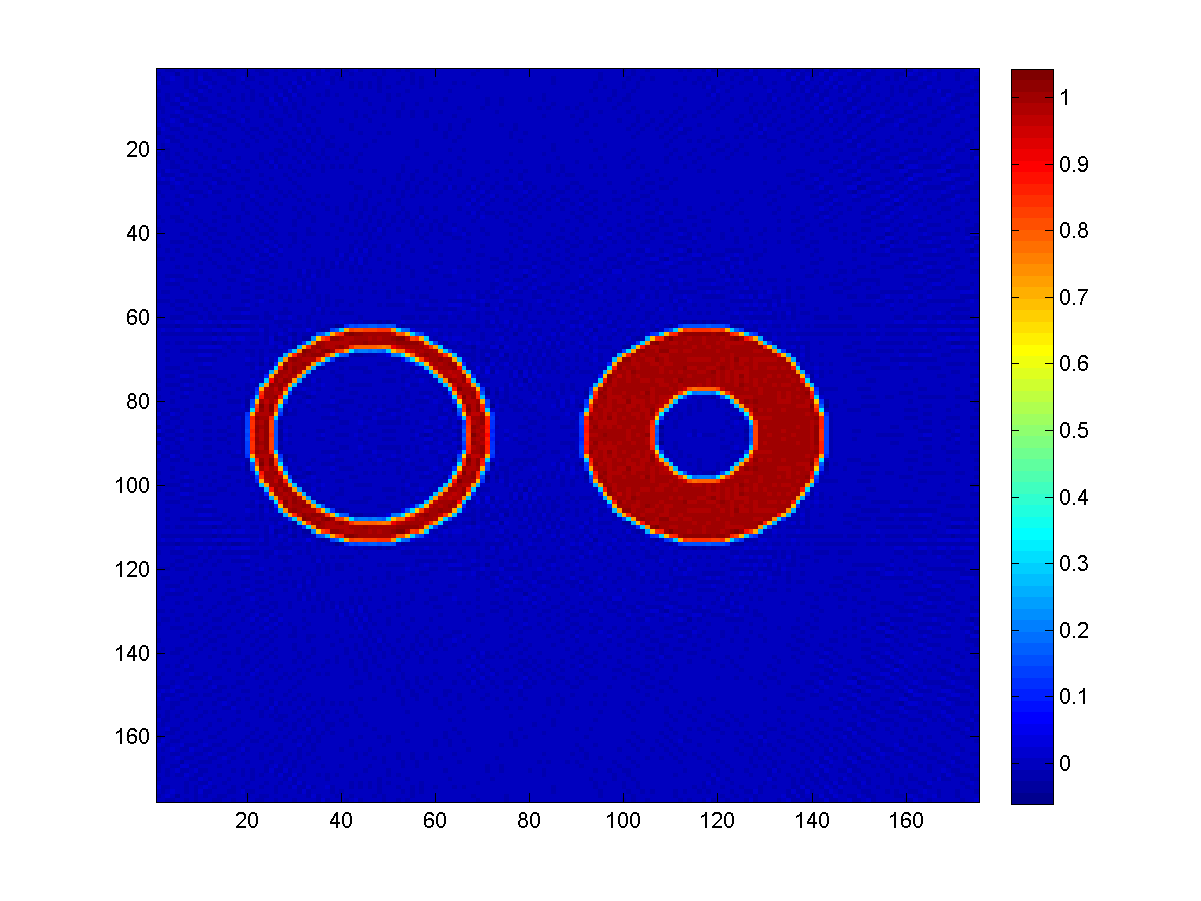}
\end{subfigure}
\begin{subfigure}[h]{3cm}
                \centering 
                \includegraphics[scale=0.15]{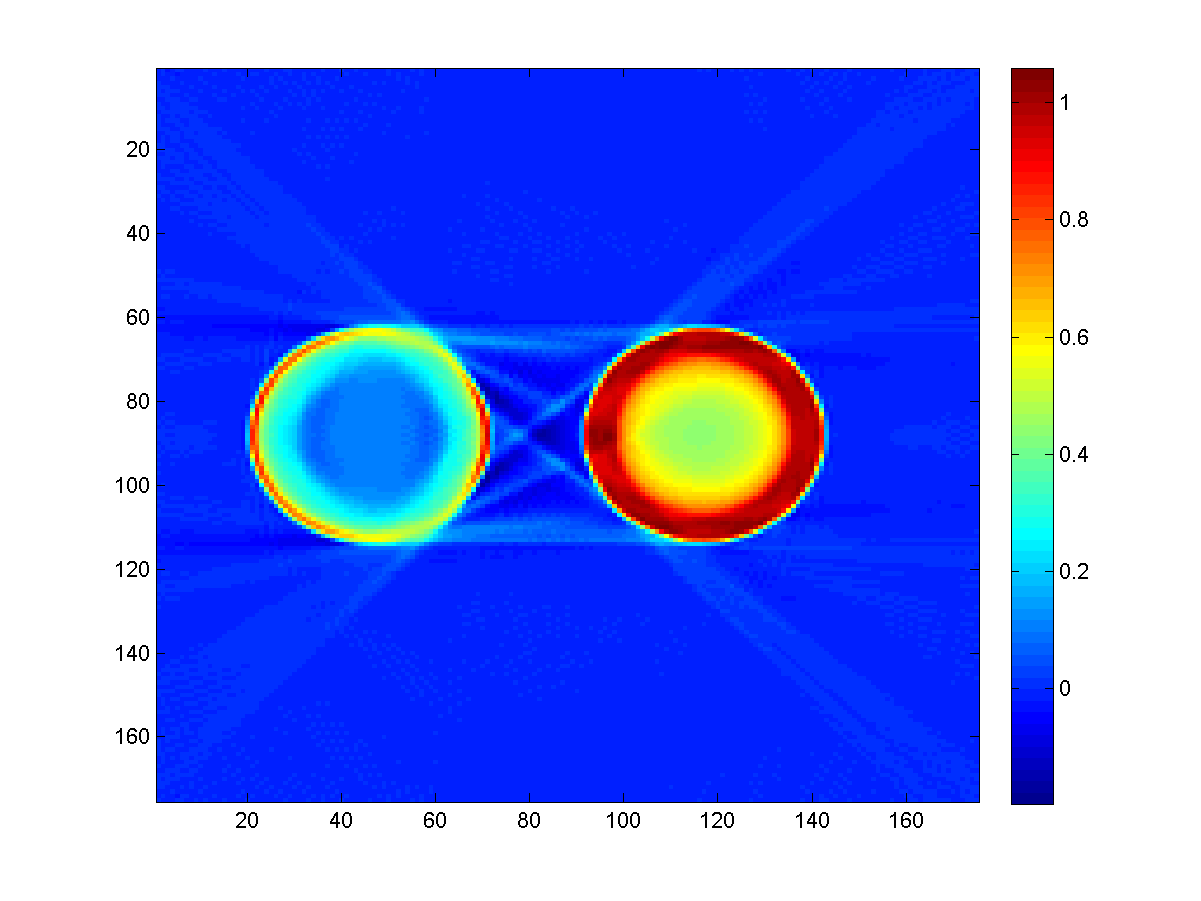}
\end{subfigure}
\begin{subfigure}[h]{3cm}
                \centering 
                \includegraphics[scale=0.15]{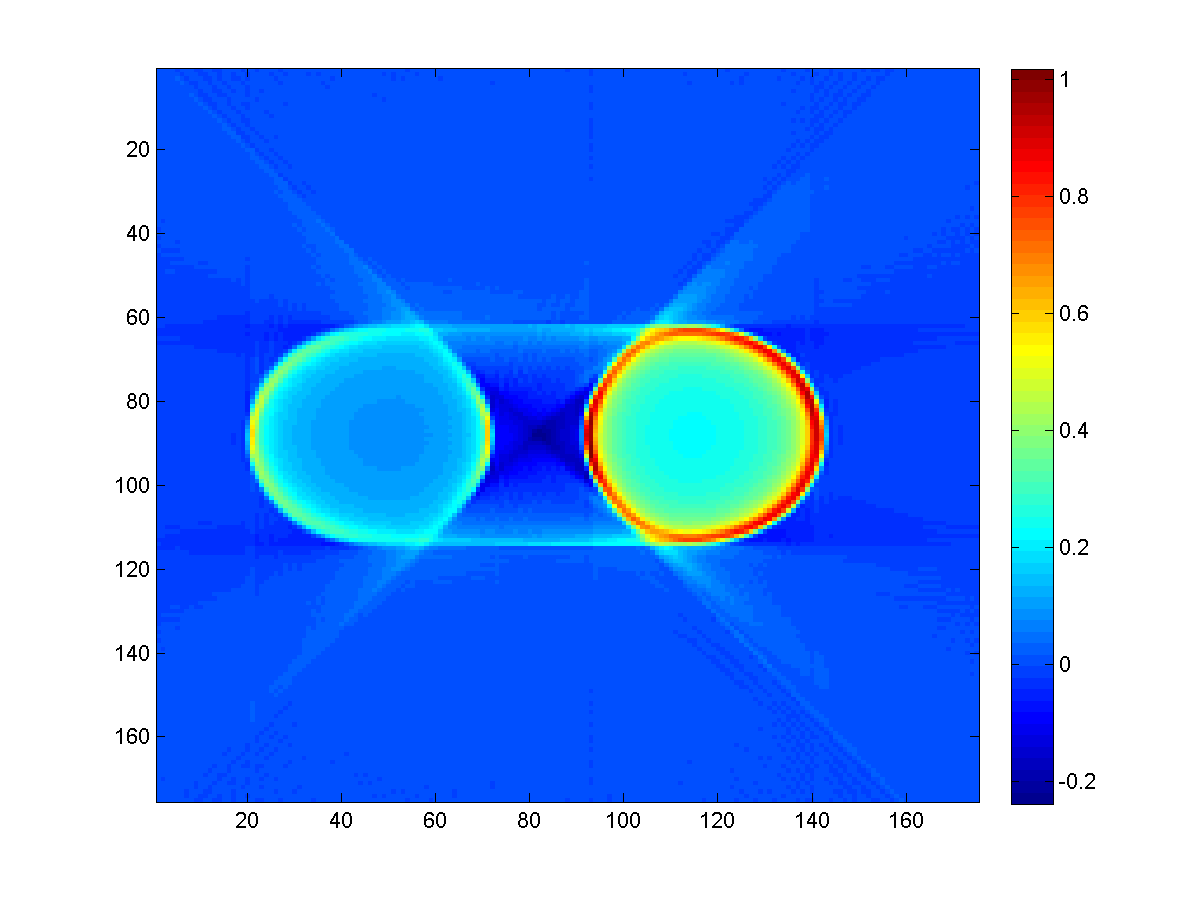}
\end{subfigure}
\begin{subfigure}[h]{3cm}
                \centering   \includegraphics[scale=0.15]{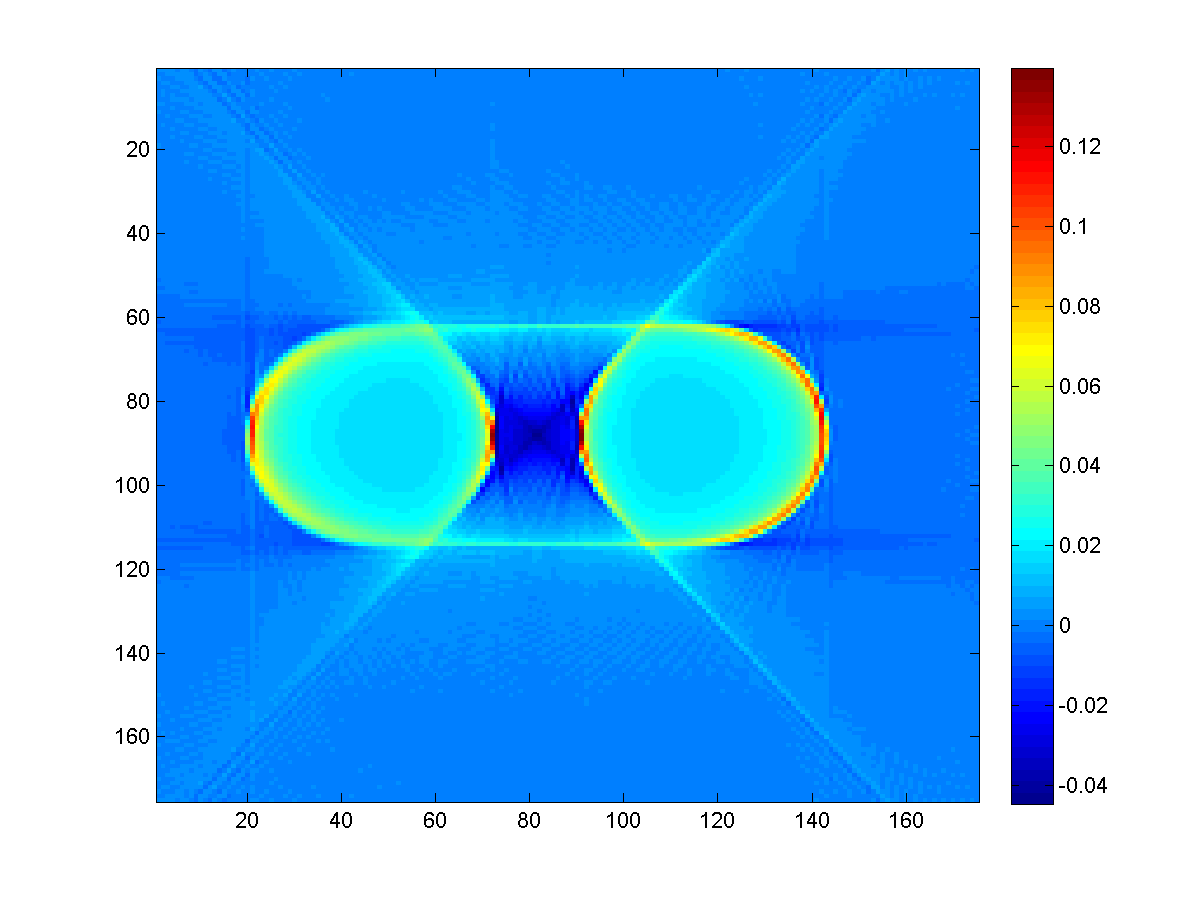}
\end{subfigure}
\caption{2 rings with different annulus regions: The outer radius for both rings is r=25.5 and the inner radii are $r_{1}=21$ and $r_{2}=11$. For figures (a)-(d), we present the sinogram regularisation for increasing values of $\beta$ with the corresponding filtered backprojected using MATLAB's \textit{iradon} built-in function.}
\label{sxima16}
\end{center}
\end{figure}

\begin{figure}[h!]
\begin{center}
\begin{subfigure}[h]{3cm}
                \centering
                \caption{$\beta$=0.001}  
                \includegraphics[scale=0.15]{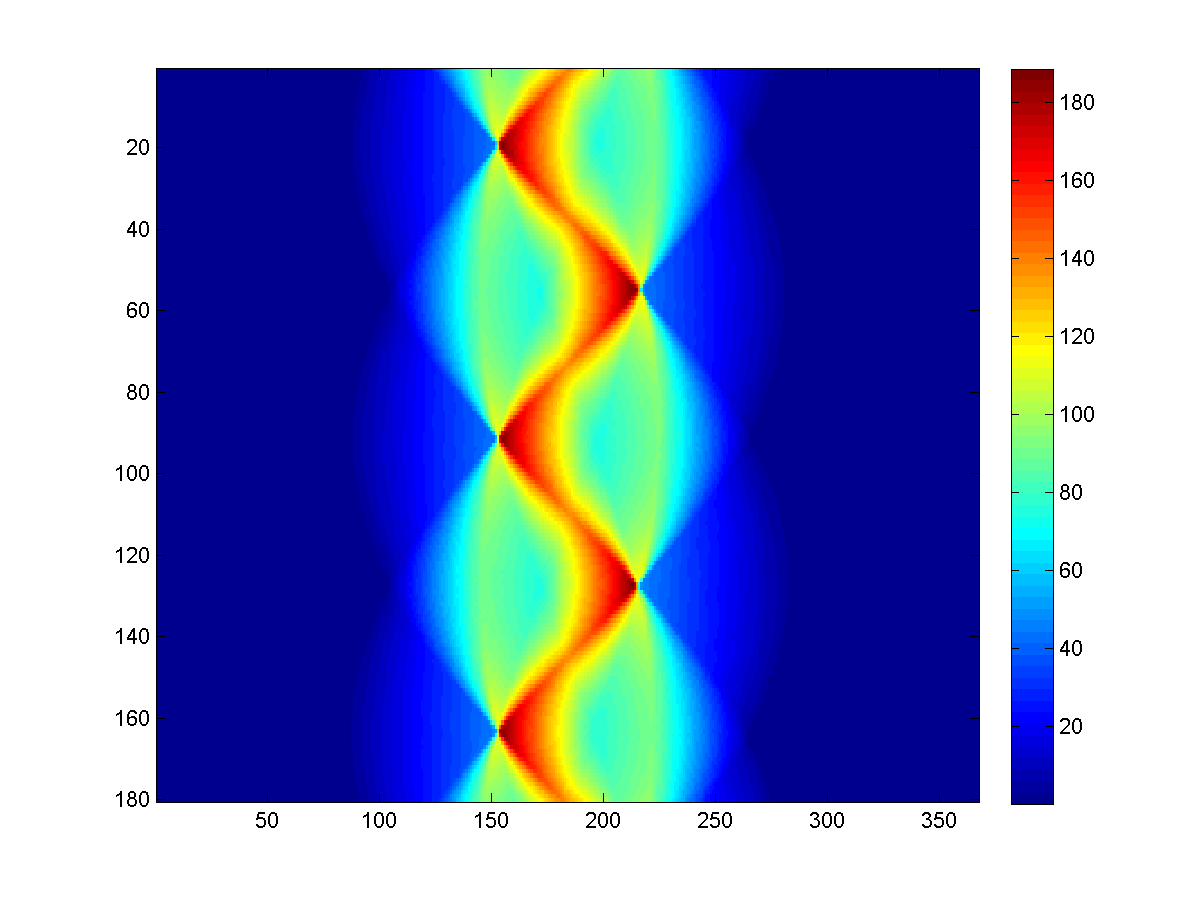}
\end{subfigure}
\begin{subfigure}[h]{3cm}
                \centering
                \caption{$\beta$=0.1}   \includegraphics[scale=0.15]{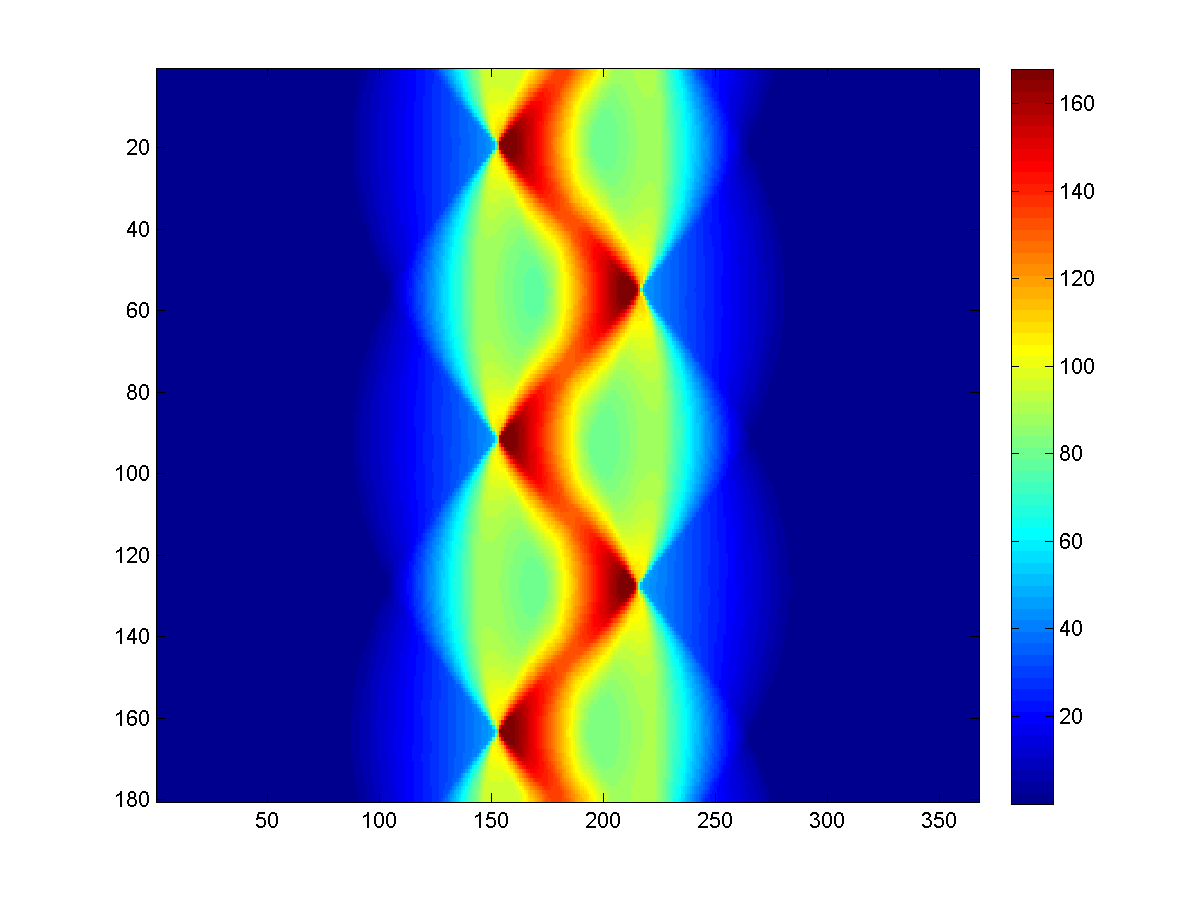}
\end{subfigure}
\begin{subfigure}[h]{3cm}
                \centering
                \caption{$\beta$=1}   \includegraphics[scale=0.15]{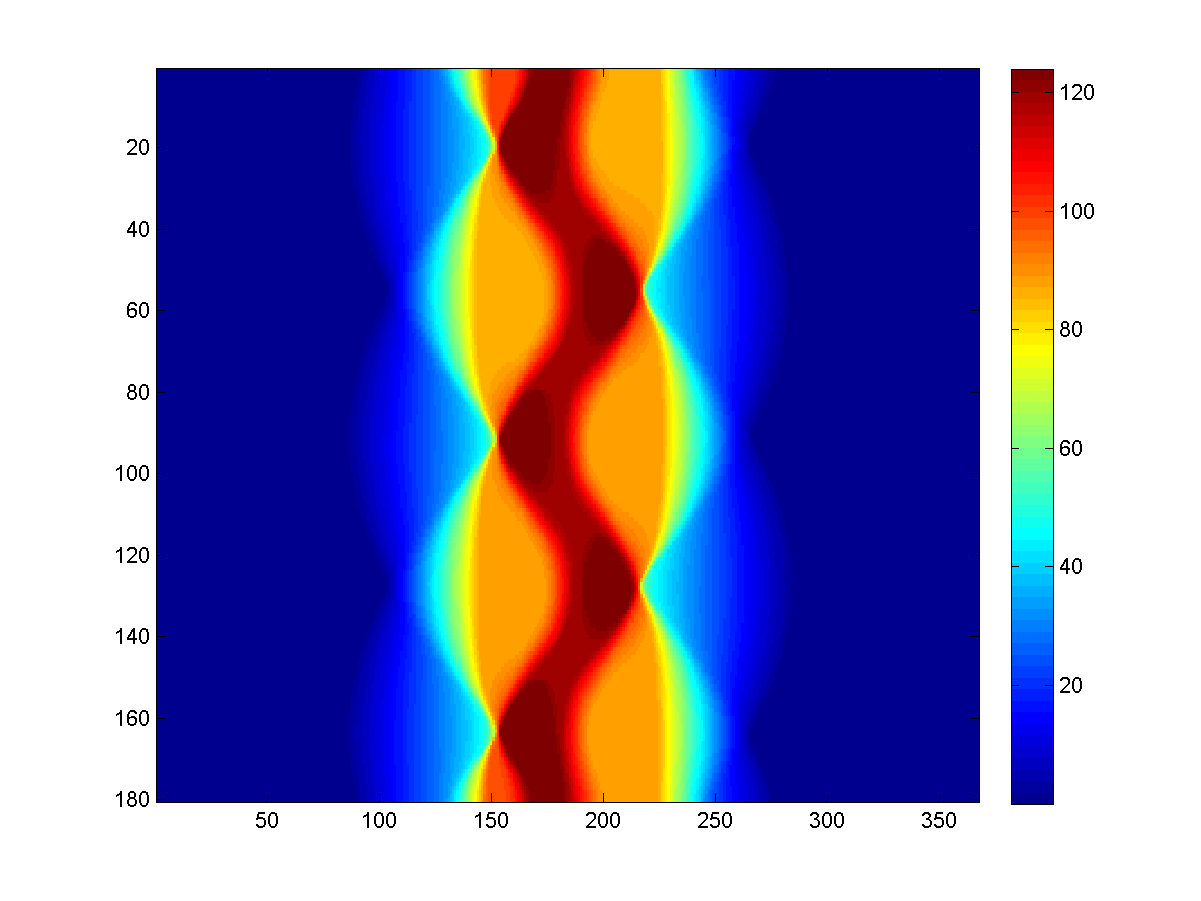}
\end{subfigure}
\begin{subfigure}[h]{3cm}
                \centering 
                \caption{$\beta$=10} \includegraphics[scale=0.15]{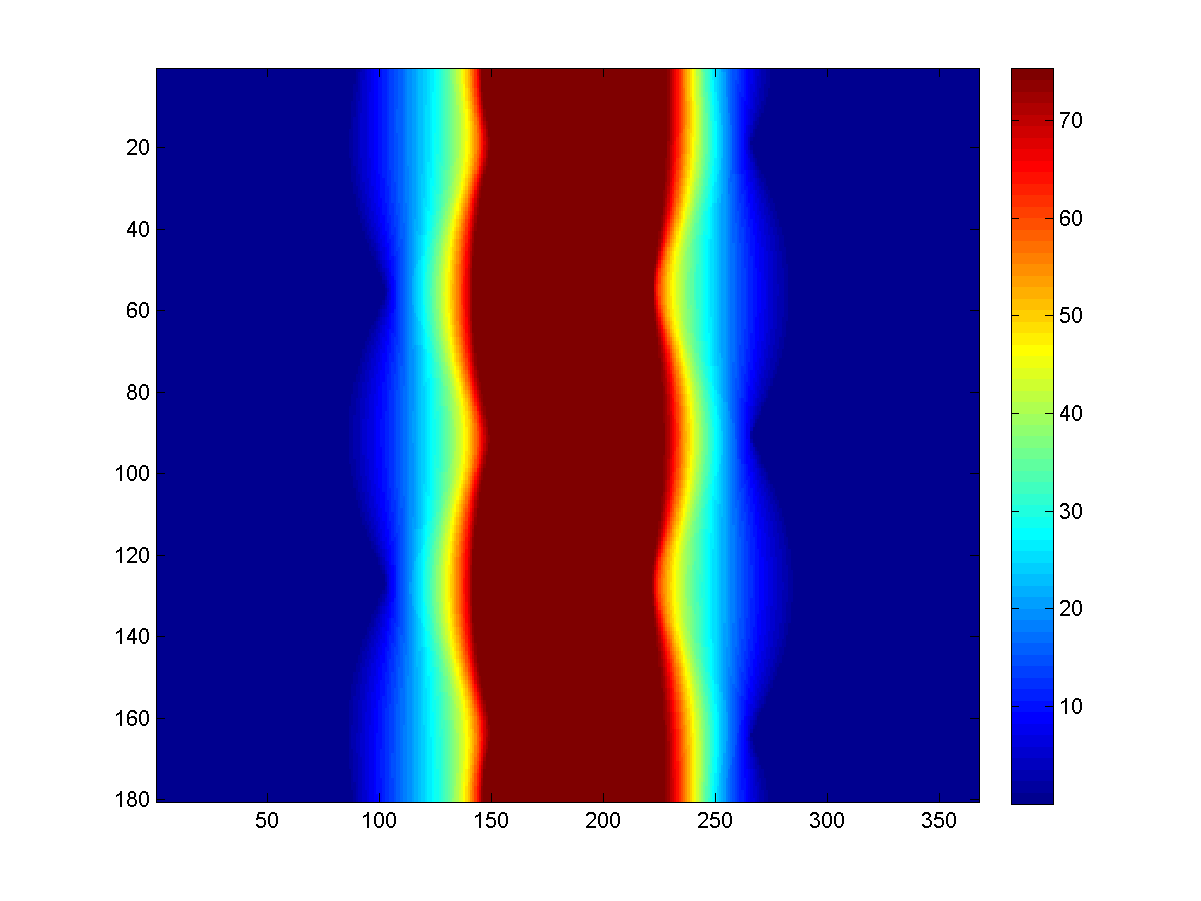}
\end{subfigure}
\begin{subfigure}[h]{3cm}
                \centering 
                \caption{$\beta$=50} \includegraphics[scale=0.15]{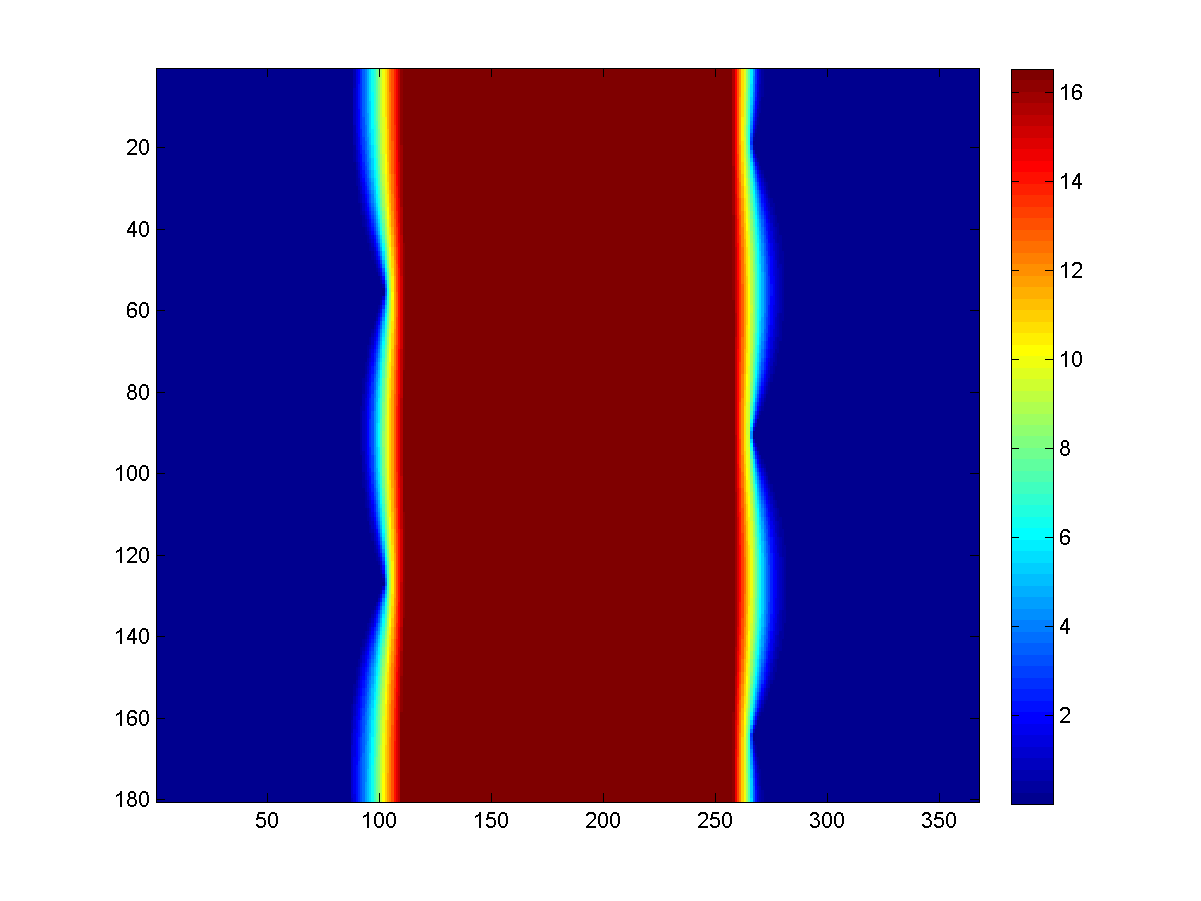}
\end{subfigure}
\begin{subfigure}[h]{3cm}
                \centering
 \includegraphics[scale=0.15]{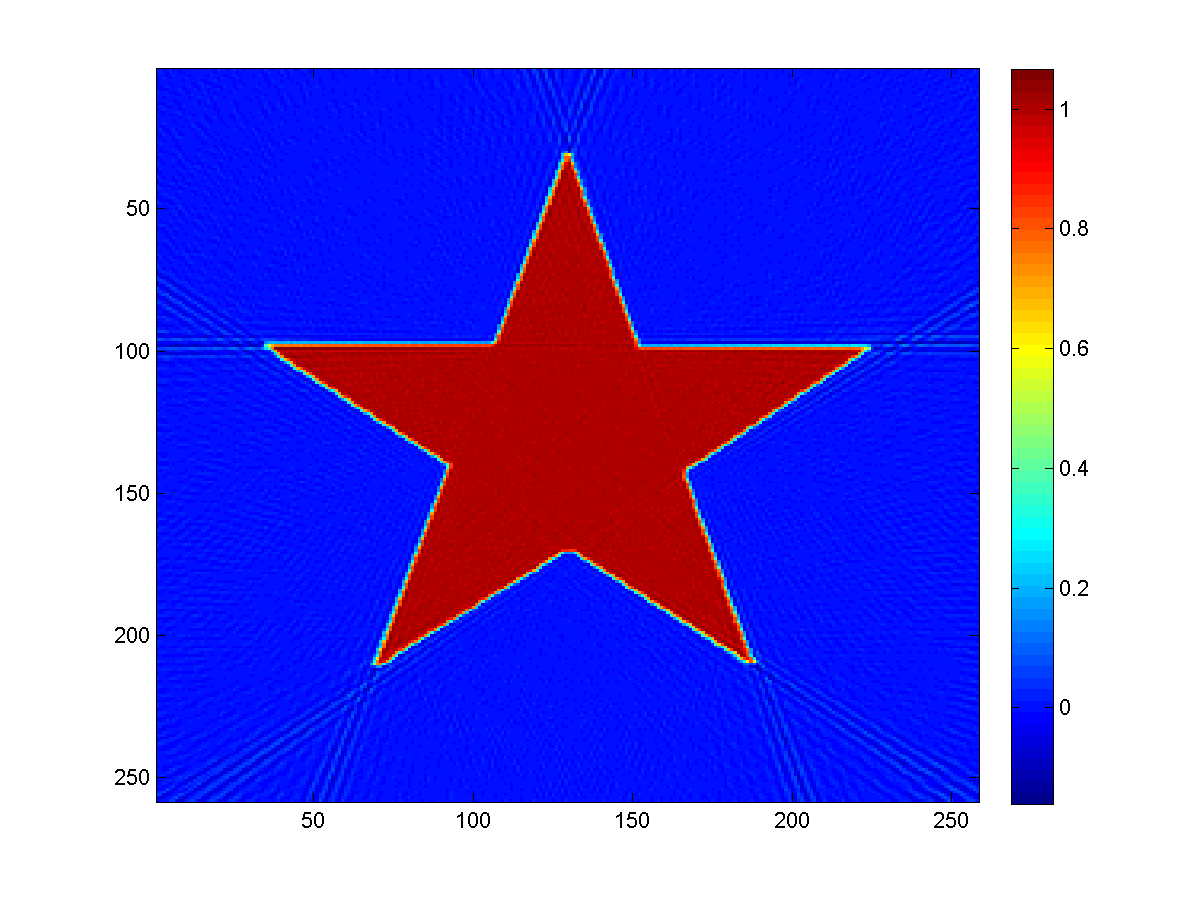}
\end{subfigure}
\begin{subfigure}[h]{3cm}
                \centering 
                \includegraphics[scale=0.15]{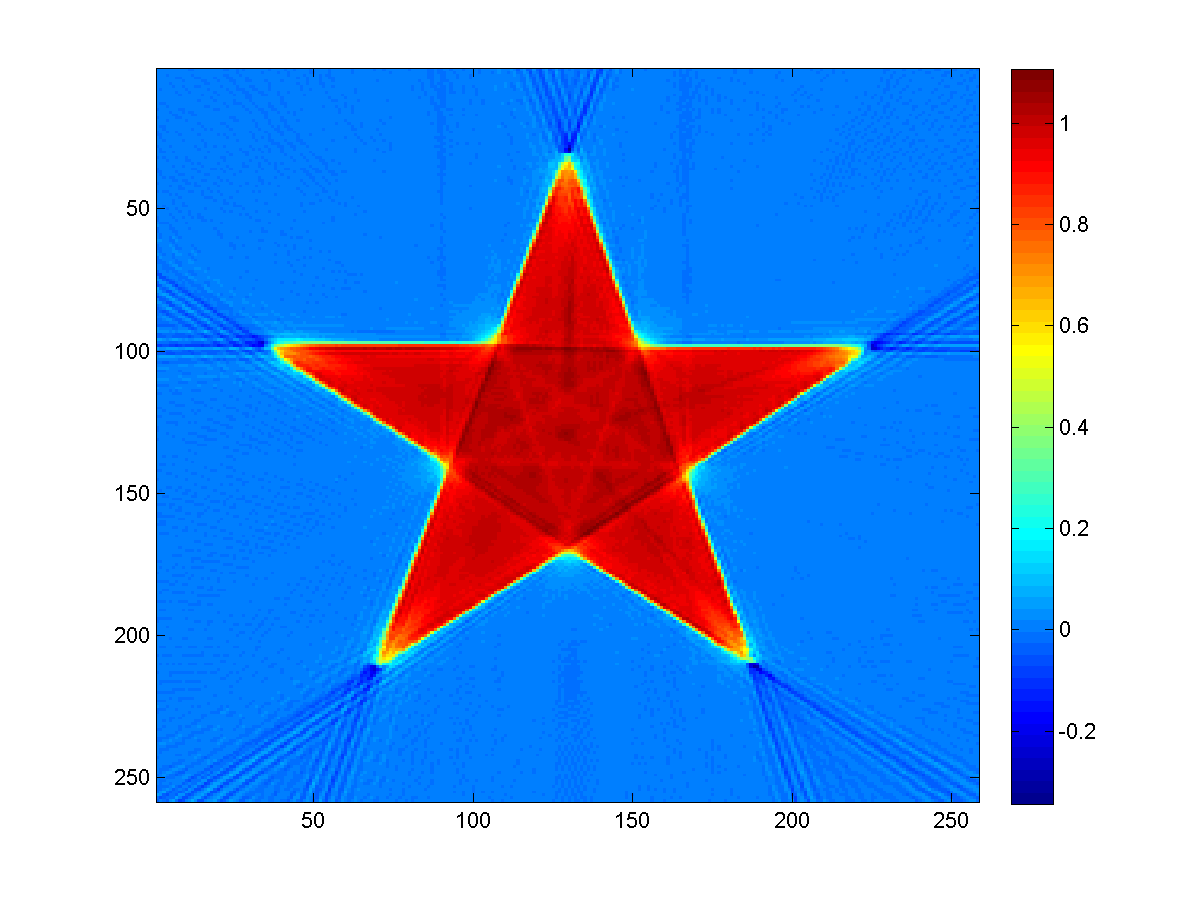}
\end{subfigure}
\begin{subfigure}[h]{3cm}
                \centering 
                \includegraphics[scale=0.15]{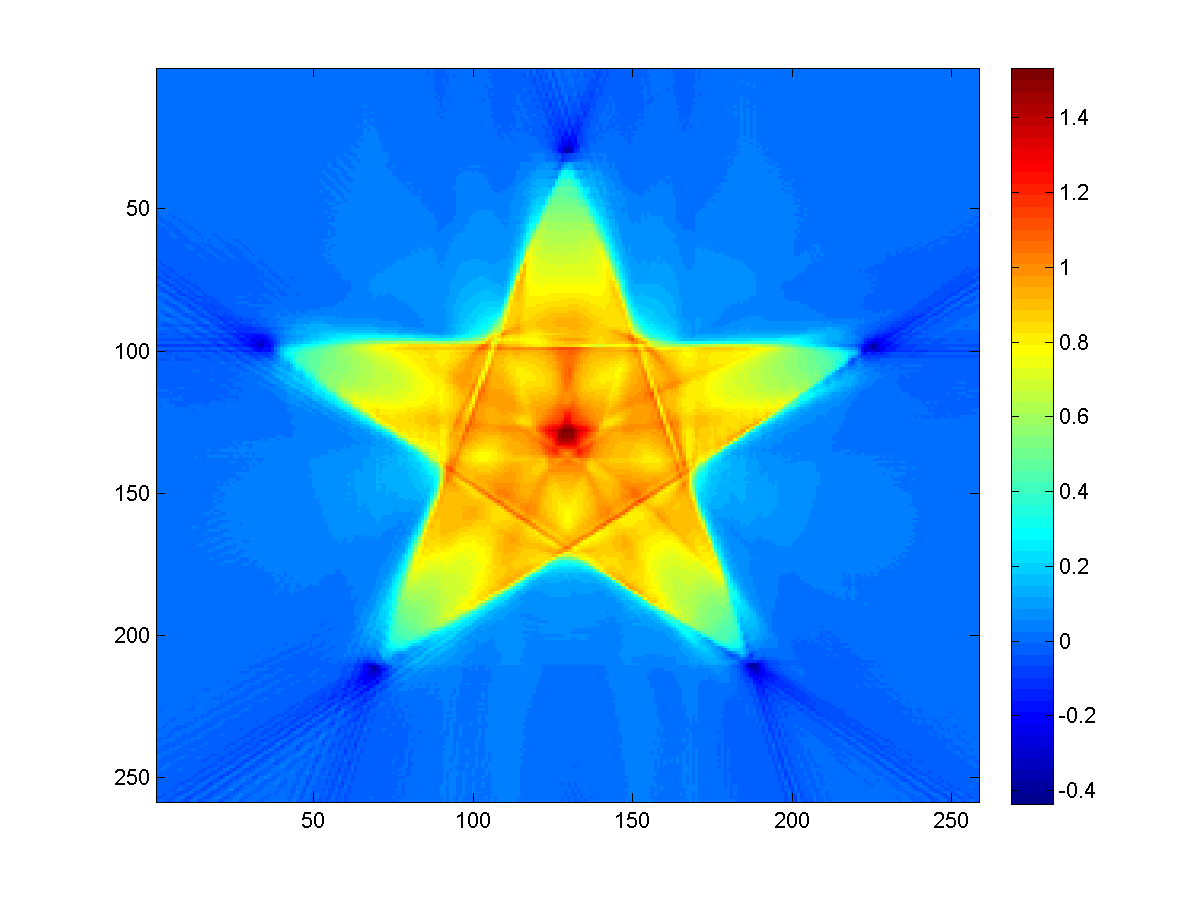}
\end{subfigure}
\begin{subfigure}[h]{3cm}
                \centering   \includegraphics[scale=0.15]{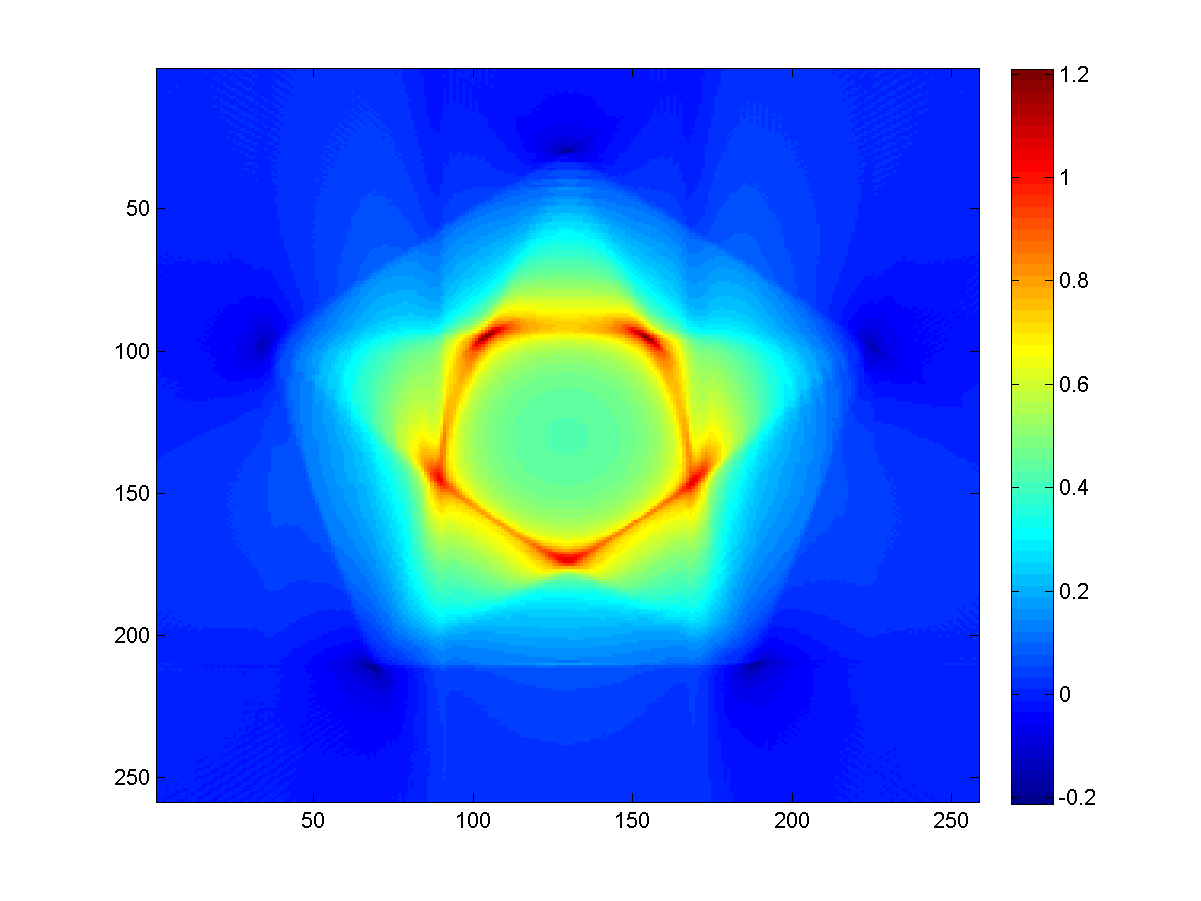}
\end{subfigure}
\begin{subfigure}[h]{3cm}
                \centering   \includegraphics[scale=0.15]{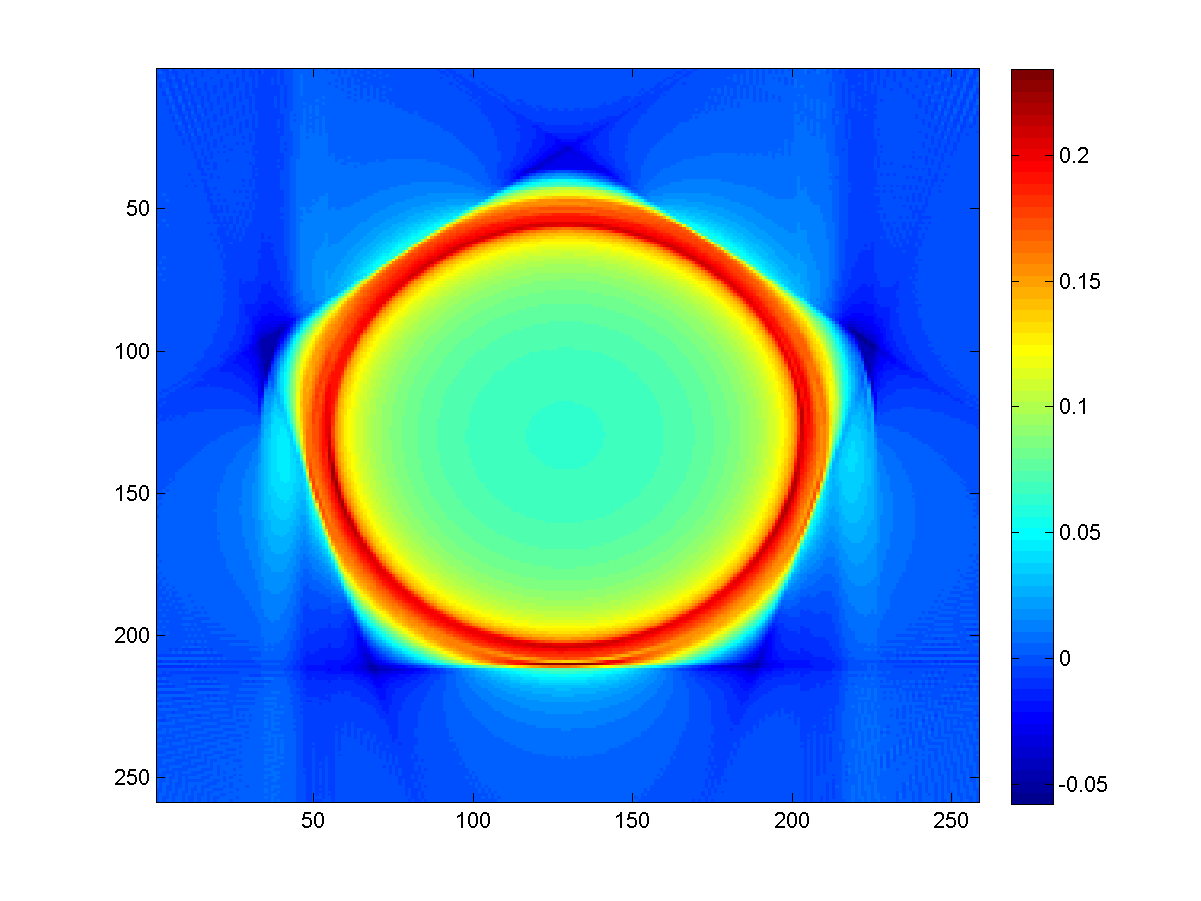}
\end{subfigure}
\caption{Star-shaped image of 5 corners: In Figures (a)-(d), we present the sinogram regularisation for increasing values of $\beta$ with the corresponding filtered backprojected image using MATLAB's built-in function \textit{iradon}.}
\label{sxima16i}
\end{center}
\end{figure}

The conclusion of this section is the motivation for the next section at the same time. Analysing the effect of total variation regularisation on the sinogram by considering its scale space and its effect on the reconstructed image we have seen in Figure \ref{sxima11b} -- \ref{sxima16i} the potential use of this method for the enhancement and detection of object boundaries. As we will see in the next section, this effect can be exploited for enhancing thin structures in images obtained from Radon measurements.

\subsection{Thin Structure Reconstruction}
\label{thinrecon}

In what follows, we discuss how total variation regularisation of the sinogram can improve the quality of the reconstruction in comparison with pure total variation regularisation of the image in the presence of \textit{thin} structures in the image. Our first example is a thin rectangular frame in Figure \ref{sxima17}. Similarly as in section \ref{petrecon}, we start by finding an optimal value of $\alpha$ with $\beta=0$, in terms of SNR. Then, we select a range of $\alpha$ values close to this optimal one and we allow strictly positive values for $\beta$. The noise that is added on the sinogram, is generated by MATLAB's \textit{imnoise} routine, with a $10^{12}$ scaling factor, see the beginning of section \ref{numerical} for more explanation. The test image that is shown in Figure \ref{sxima17} has 50 pixels width and 100 pixels length and the rectangular frame has a width of 2 pixels. In Figure \ref{sxima18i}, we first present some of the results obtained with pure total variation regularisation on the image, that is when $\beta=0$. As we increase the $\alpha$ parameter, we observe that the best SNR corresponds to $\alpha$=5 with SNR=19.9764. 
That is because for small values of $\alpha$ we observe that the large-scale structure of the object is still intact, with the cost that noise is still present in the reconstructed image, see Figures \ref{sxima18i} (\subref{sxima18i:1})-(\subref{sxima18i:3}). However, with higher values of $\alpha$ noise is further eliminated but at the expense of a significant loss of contrast and some unpleasant artifacts along the boundaries of the frame, see Figure \ref{sxima18i} (\subref{sxima18i:4})-(\subref{sxima18i:6}).

\begin{figure}[h!]
\begin{center}
\begin{subfigure}[h]{5cm}
                \centering
                \includegraphics[scale=0.2]{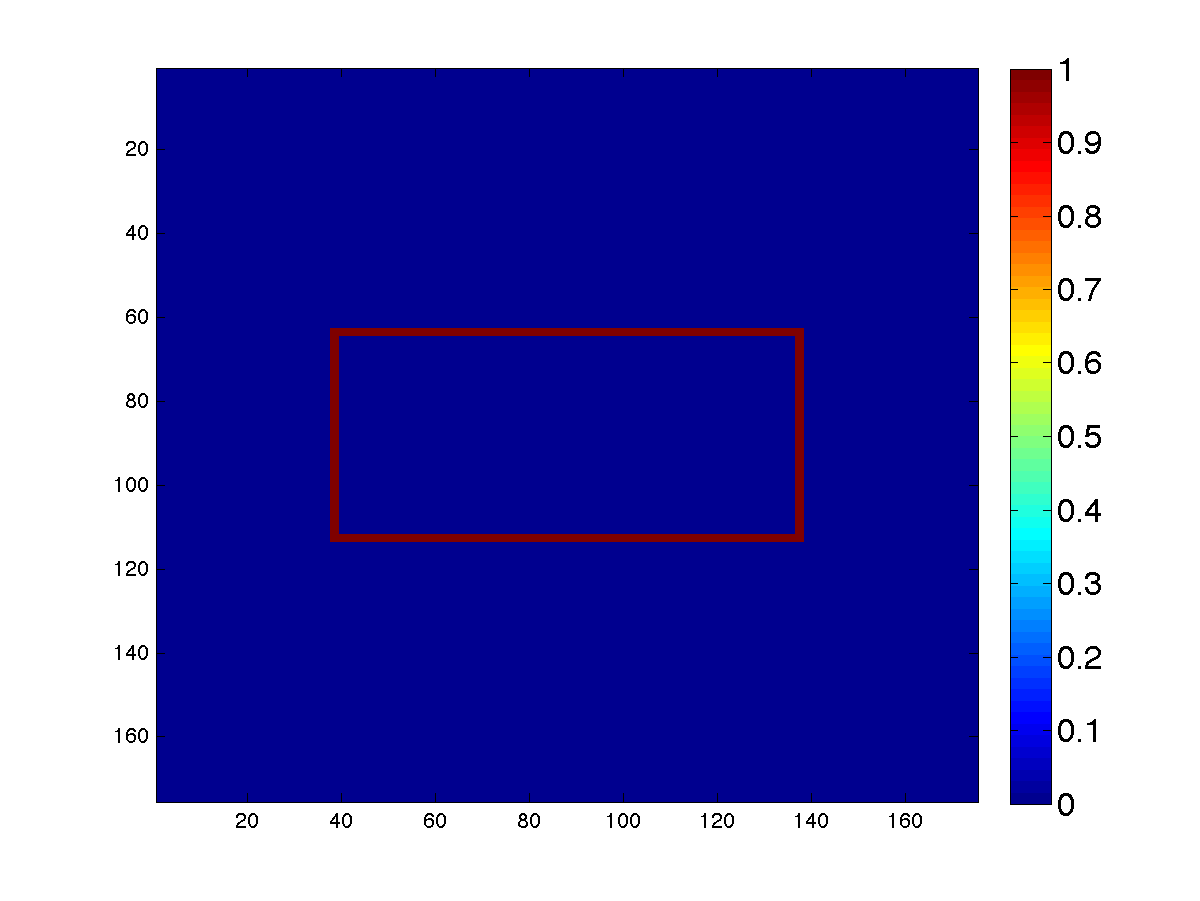}
		\caption{Thin Rectangle}
\end{subfigure}
\begin{subfigure}[h]{5cm}
                \centering
                \includegraphics[scale=0.2]{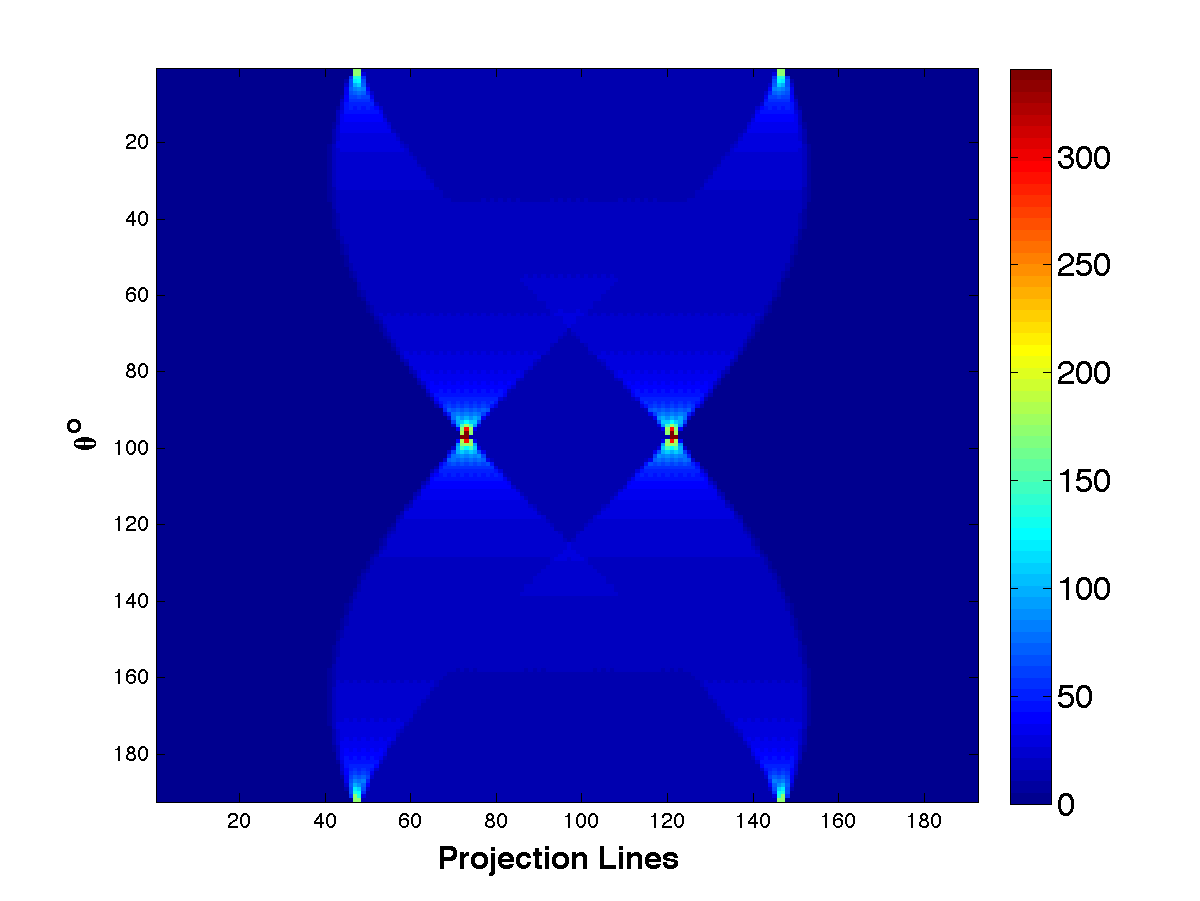}
		\caption{Noiseless Sinogram}
\end{subfigure}
\begin{subfigure}[h]{5cm}
                \centering
                \includegraphics[scale=0.2]{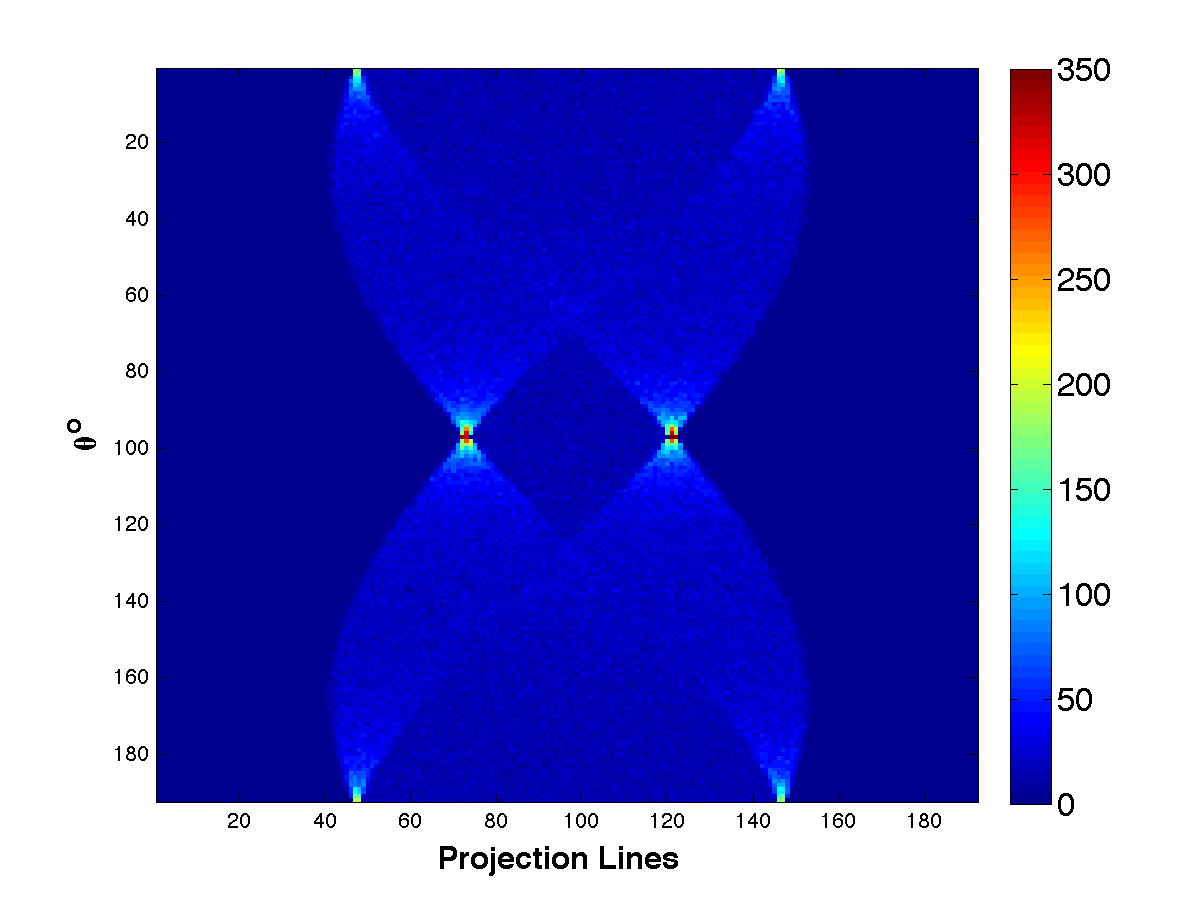}
		\caption{Low noise: SNR=14.9146}
\end{subfigure}
\caption{A thin rectangle of 50 pixels width and 100 pixels length with 2 pixels length on the boundaries. The corresponding noiseless and noisy sinograms with $10^{12}$ scaling factor in \it{imnoise}.}
\label{sxima17}
\end{center}
\end{figure}


\begin{figure}[h!]
\begin{center}
\begin{subfigure}[h]{5cm}
                \centering
                \includegraphics[scale=0.25]{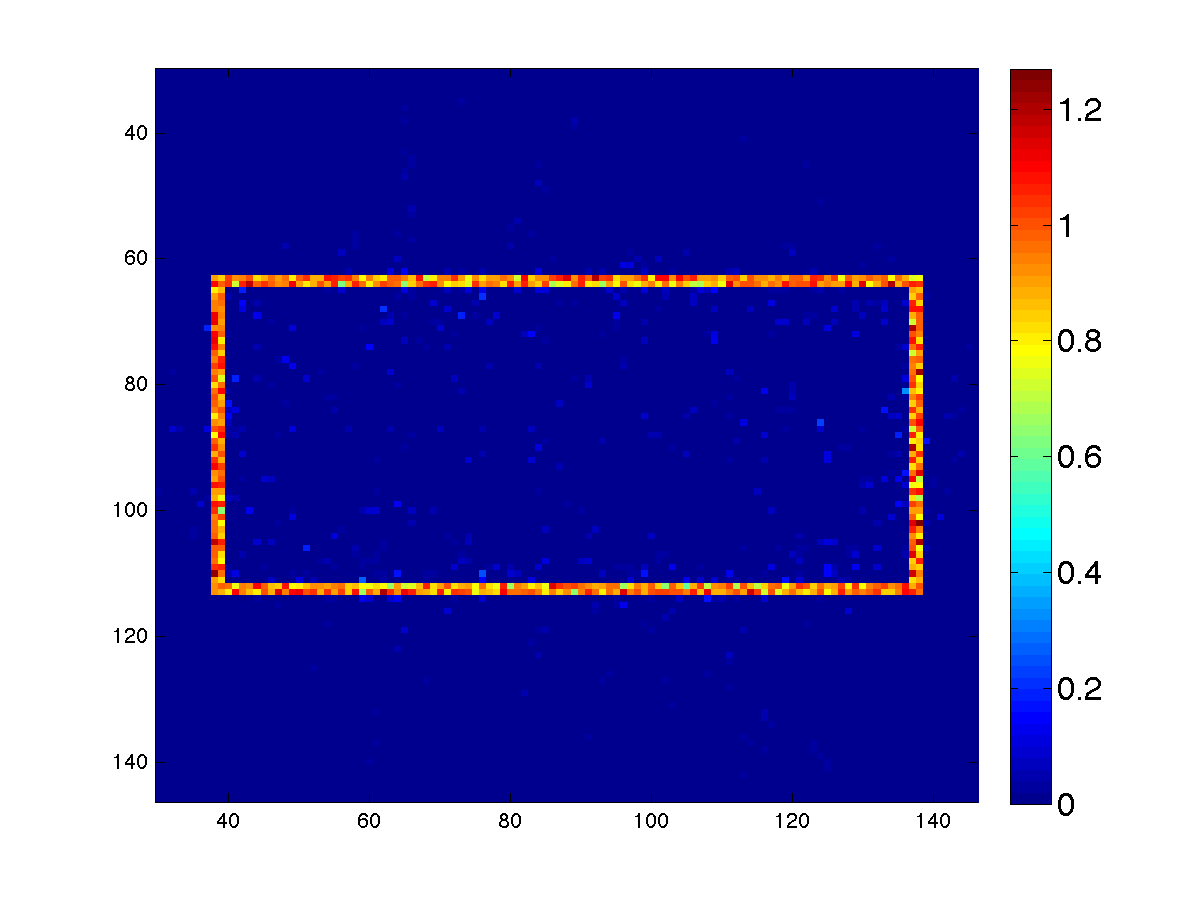}
		\caption{$\alpha=1$, $\beta=0$\\SNR=16.5399}
		\label{sxima18i:1}
\end{subfigure}
\begin{subfigure}[h]{5cm}
                \centering
                \includegraphics[scale=0.25]{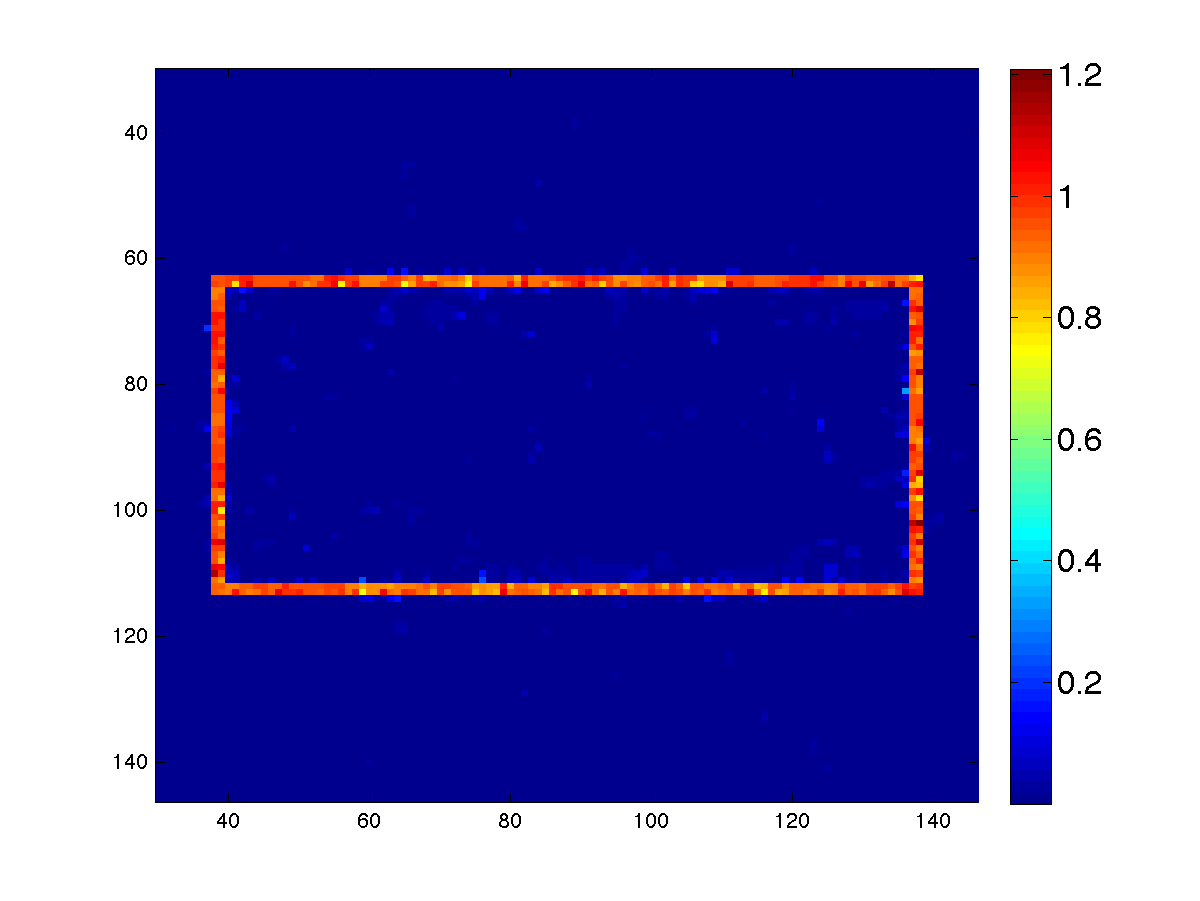}
		\caption{$\alpha=5$, $\beta=0$\\SNR=19.9764}
\end{subfigure}
\begin{subfigure}[h]{5cm}
                \centering
                \includegraphics[scale=0.25]{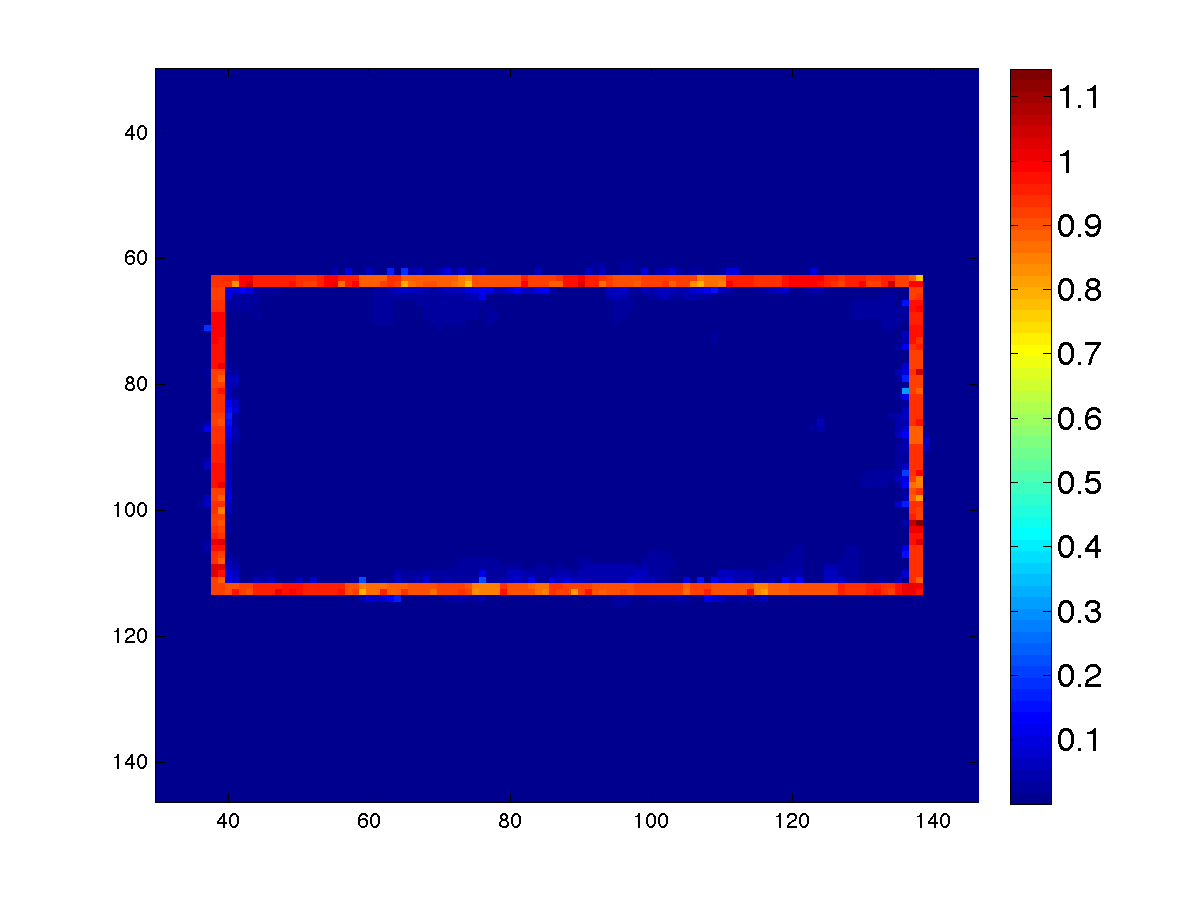}
		\caption{$\alpha=10$, $\beta=0$\\SNR=19.8700}
		\label{sxima18i:3}
\end{subfigure}
\begin{subfigure}[h]{5cm}
                \centering
                \includegraphics[scale=0.25]{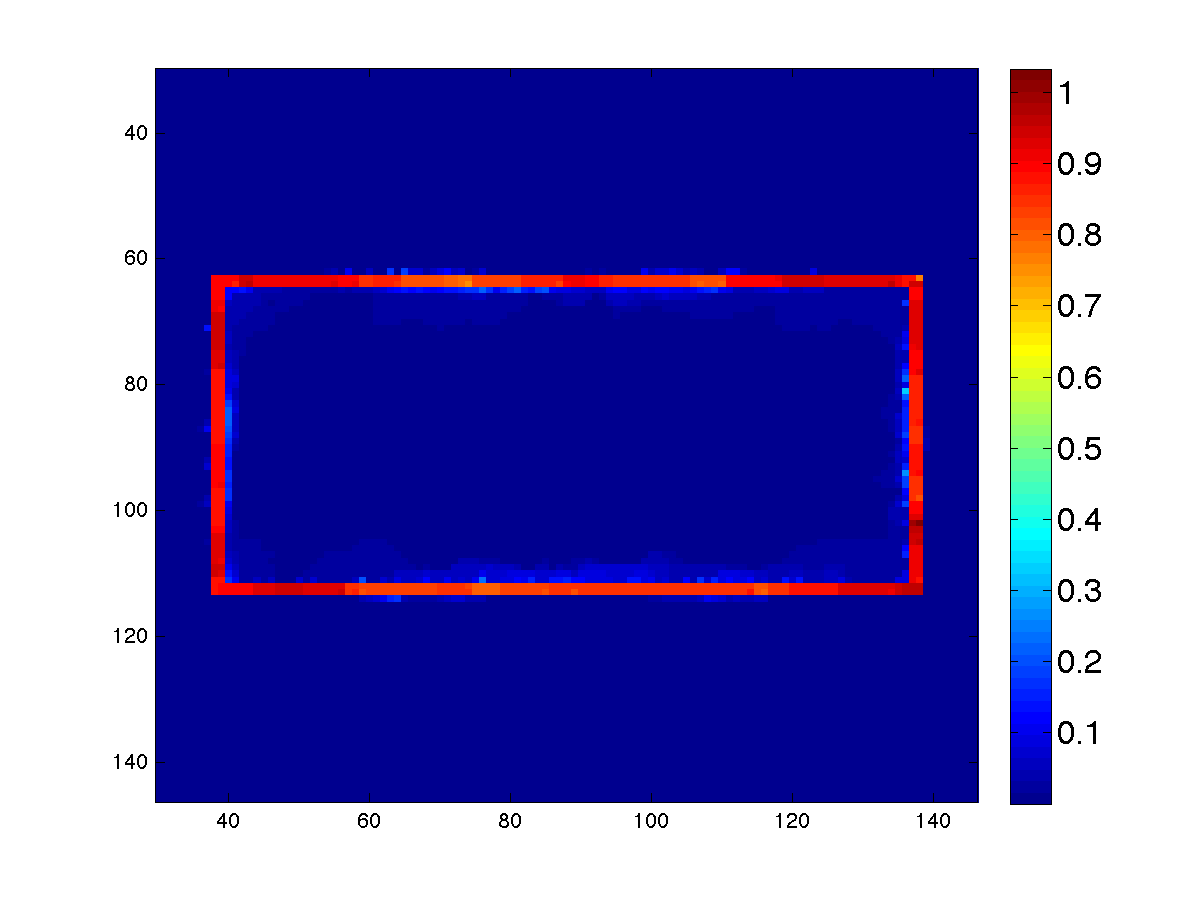}
		\caption{$\alpha=20$, $\beta=0$\\SNR=16.2678}
		\label{sxima18i:4}
\end{subfigure}
\begin{subfigure}[h]{5cm}
                \centering
                \includegraphics[scale=0.25]{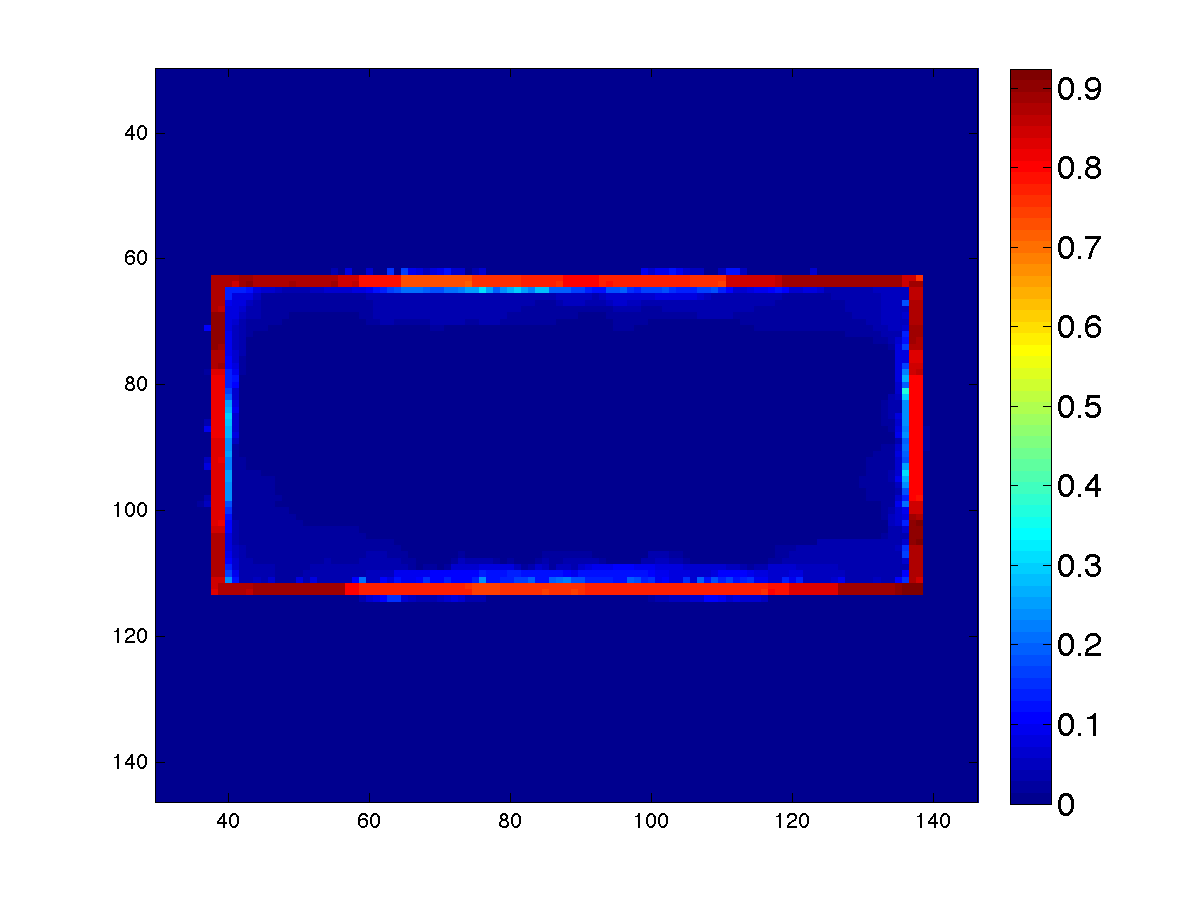}
		\caption{$\alpha=30$, $\beta=0$\\SNR=13.054}
\end{subfigure}
\begin{subfigure}[h]{5cm}
                \centering
                \includegraphics[scale=0.25]{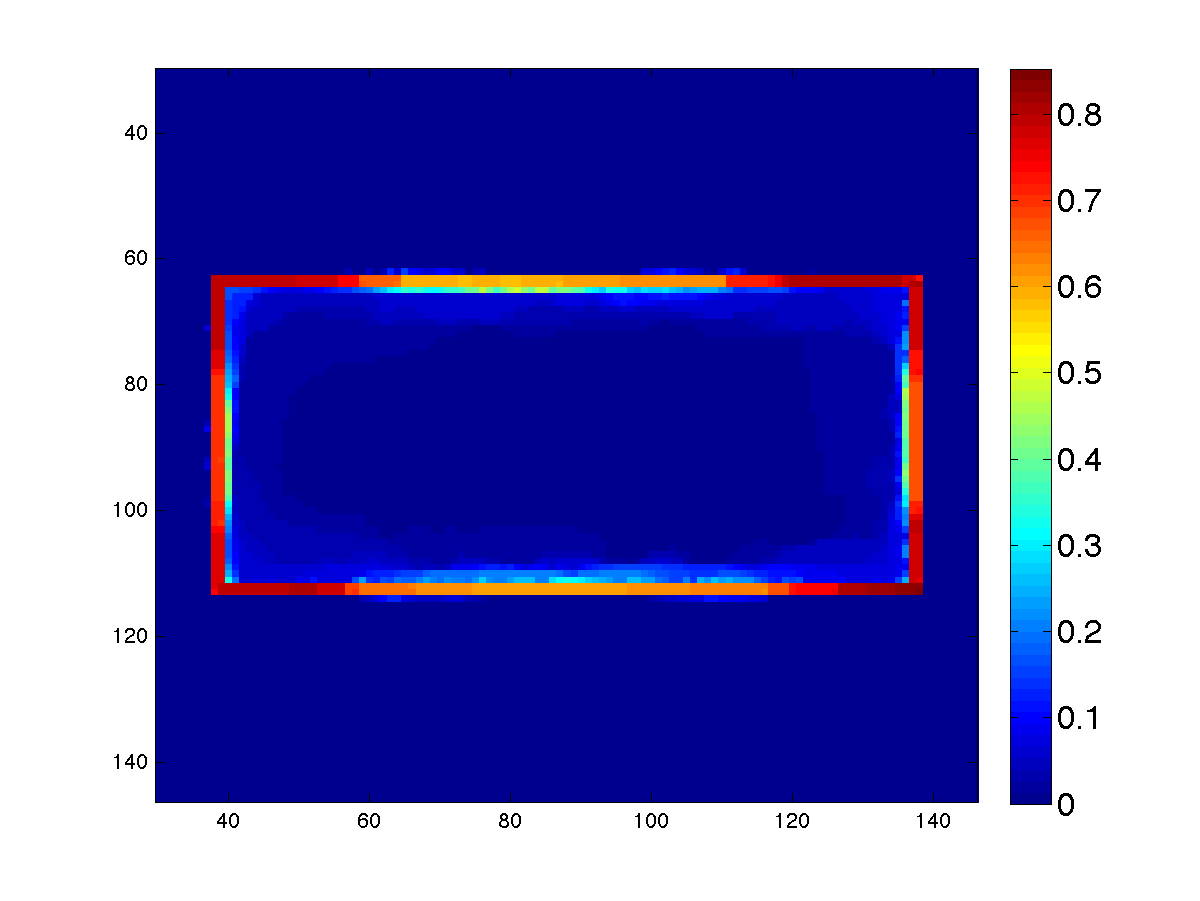}
		\caption{$\alpha=50$, $\beta=0$\\SNR=8.6456}
		\label{sxima18i:6}
\end{subfigure}
\caption{Thin Rectangle: Reconstruction without total variation regularisation on the sinogram and different parameters of $\alpha$. }
\label{sxima18i}
\end{center}
\end{figure}

\begin{table}[h!]
\begin{center}
\begin{tabular}{|c|c|c|c|c|c|c|}
\cline{3-7}
\multicolumn{2}{c|}{} & \multicolumn{5}{c|}{$\beta$} \\
\cline{3-7}
\multicolumn{2}{c|}{} & 0 & 0.005 & 0.01 & 0.05 & 0.1  \\
\hline
\multirow{7}{*}{$\alpha$} & 2 & 17.6798 & 18.0078 & 19.7238 & \textbf{24.5981} & 24.2978  \\
\cline{2-7}
& 3 & 18.6444 & 18.9855 & 20.6460 & 23.9028 & 24.2647 \\
\cline{2-7}
& 4 & 19.4269 & 19.7539 & 21.5305 & 23.9178 & 23.2860  \\
\cline{2-7}
& \textbf{5} & \textbf{19.9764} & 20.2979 & 21.7962 & 23.6466  & 22.8525   \\
\cline{2-7}
& 6 & 20.2583 & 20.5771 & 21.9057 & 23.2213 & 22.3440  \\
\cline{2-7}
& \textbf{7} & \textbf{20.4471} & 20.8665 & 21.8372 & 22.7554 & 21.8147  \\
\cline{2-7}
& 8 & 20.3511 & 20.3276 & 20.9859 & 22.2391 & 21.2477  \\
\cline{1-7}
\end{tabular}
\end{center}
\caption{Thin Rectangle: SNR with $\beta\neq0$.}
\label{thin_rec_SNR_not_0}
\end{table}

\begin{figure}[h!]
\begin{center}
\begin{subfigure}[h]{6cm}
                \centering
                \includegraphics[scale=0.3]{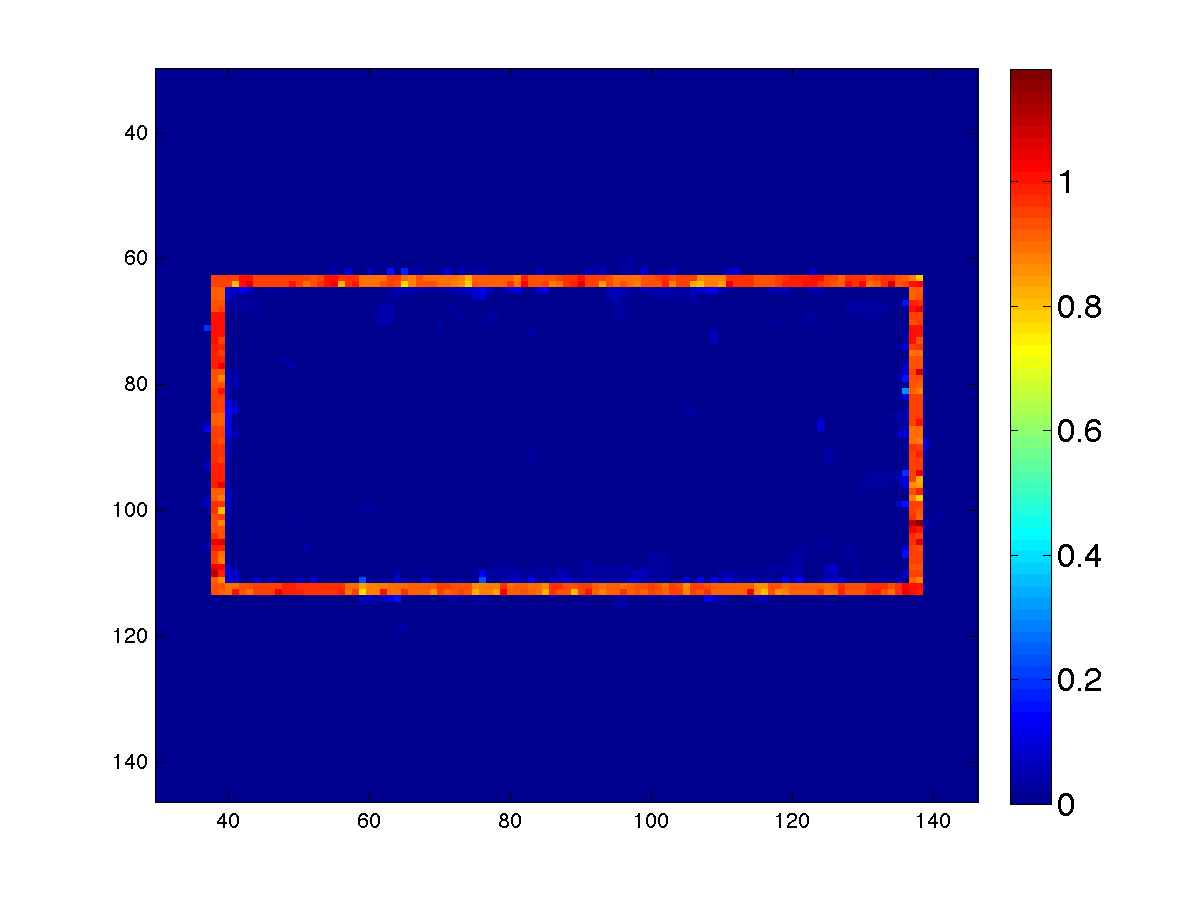}
								\caption{$\alpha=7$, $\beta=0$\\SNR=20.4471}
\end{subfigure}
\begin{subfigure}[h]{6cm}
                \centering
                \includegraphics[scale=0.3]{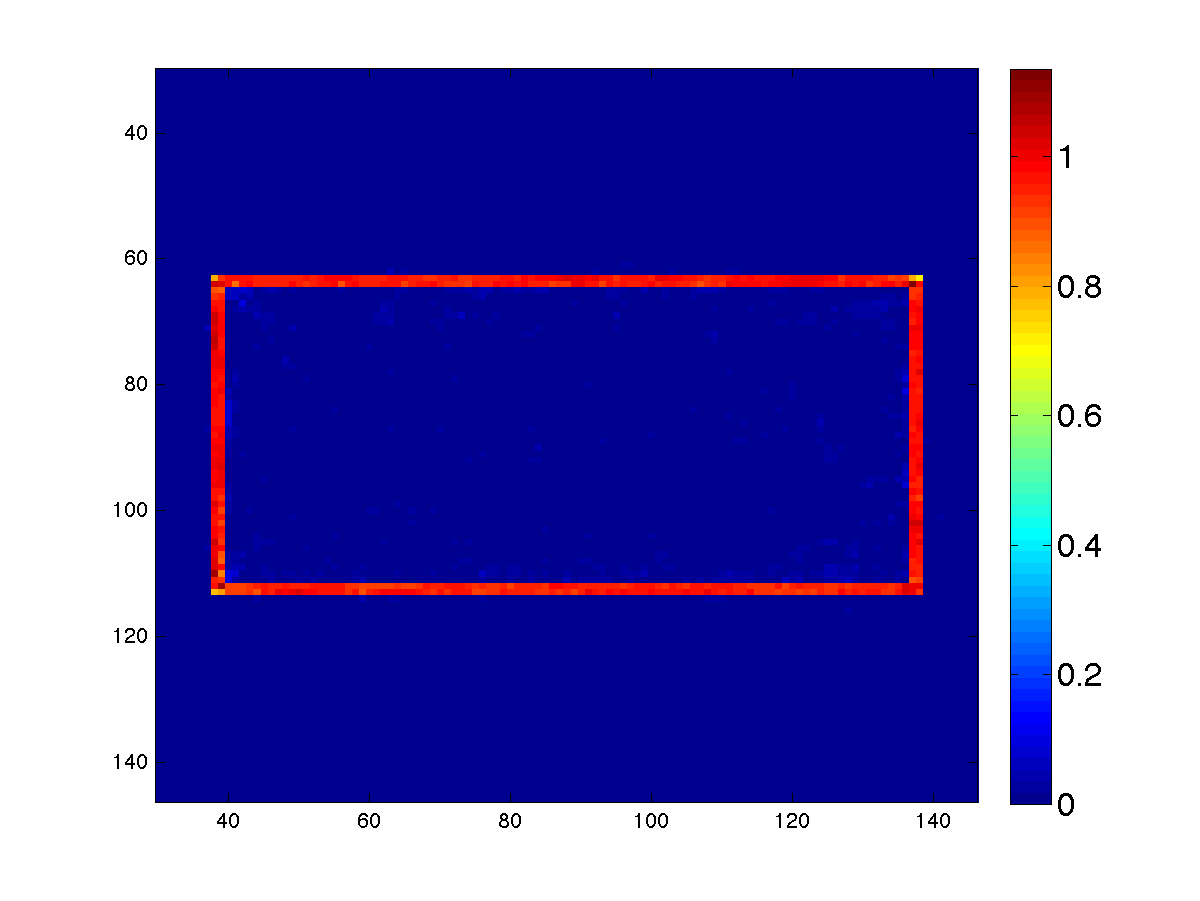}
								\caption{$\alpha=2$, $\beta=0.05$\\SNR=24.5981}
\end{subfigure}
\caption{Thin Rectangle: Best reconstructions with and without total variation regularisation on the sinogram as reported in Table \ref{thin_rec_SNR_not_0}.}
\label{sxima18ii}
\end{center}
\end{figure}

If we switch on total variation regularisation on the sinogram, that is taking $\beta >0$, we obtain results which are greatly improved both in terms of the SNR of the reconstructed images but also -- visually -- in terms of finding the right balance of eliminating the noise and accurately preserving the thin structures, see Table \ref{thin_rec_SNR_not_0} and Figure \ref{sxima18ii}. This observation is confirmed by a second example of an image of two thin straight lines which cross, compare Figure \ref{sxima19a}. The width of the thin lines is 3 pixels. The length of the horizontal line is 121 pixels and of the vertical line is 100 pixels. The noise, added on the sinogram, is generated with the same scaling factor of $10^{12}$ as before. Again, we observe that for positive values of $\beta$, we obtain much better reconstructions with almost all noise eliminated while keeping the boundaries of the thin structures intact, see Figure \ref{sxima19}. 

\begin{figure}[h!]
\begin{center}
\begin{subfigure}[h]{5cm}
                \centering
                \includegraphics[scale=0.25]{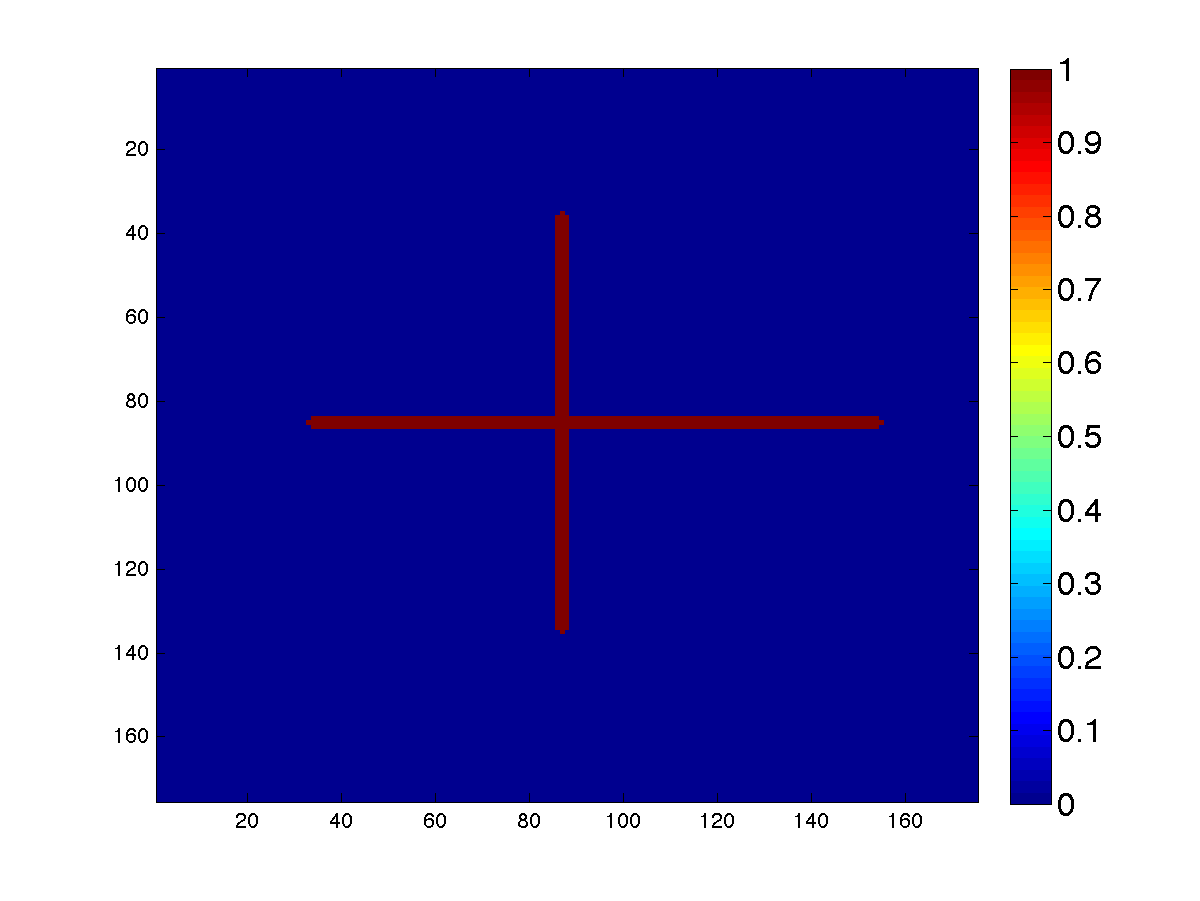}
		\caption{Cross}
\end{subfigure}
\begin{subfigure}[h]{5cm}
                \centering
                \includegraphics[scale=0.25]{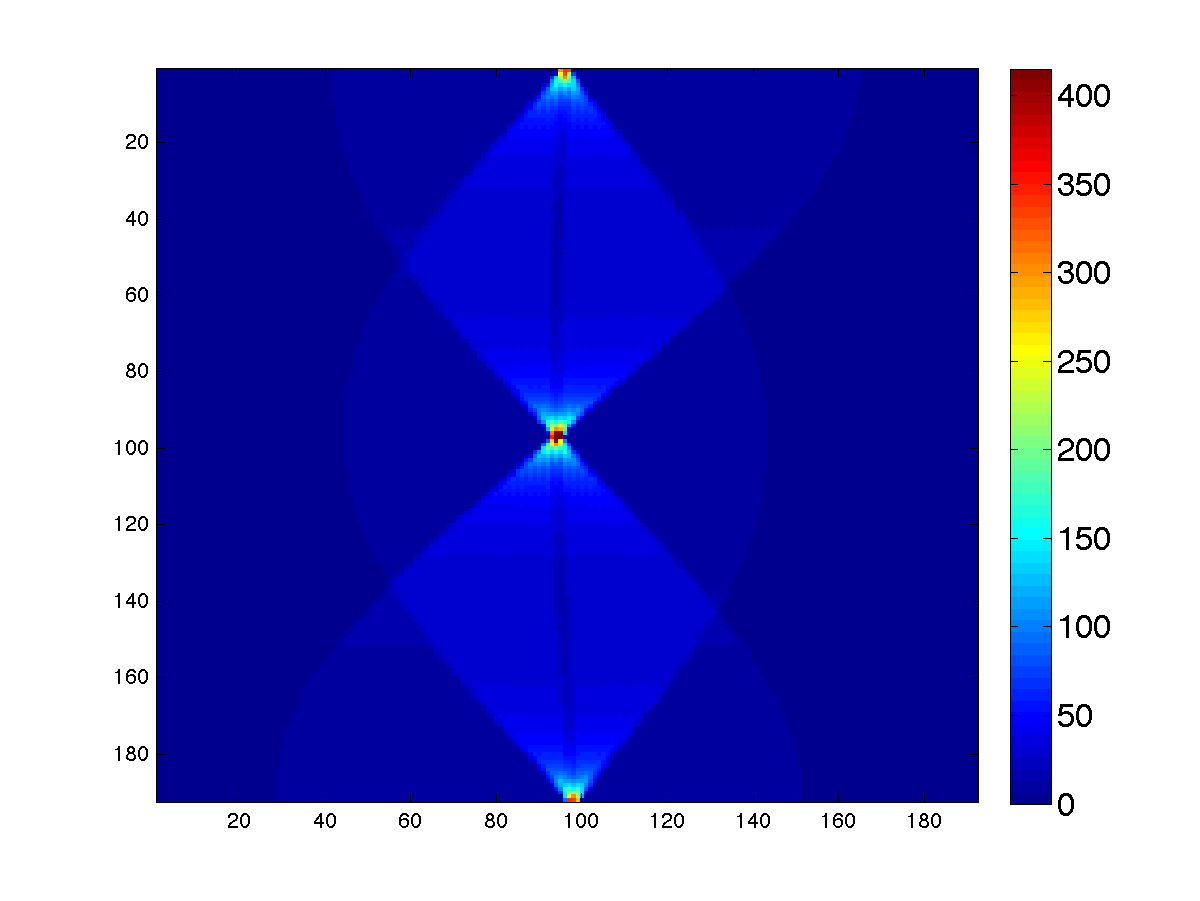}
		\caption{Sinogram}					
\end{subfigure}
\begin{subfigure}[h]{5cm}
                \centering
                \includegraphics[scale=0.25]{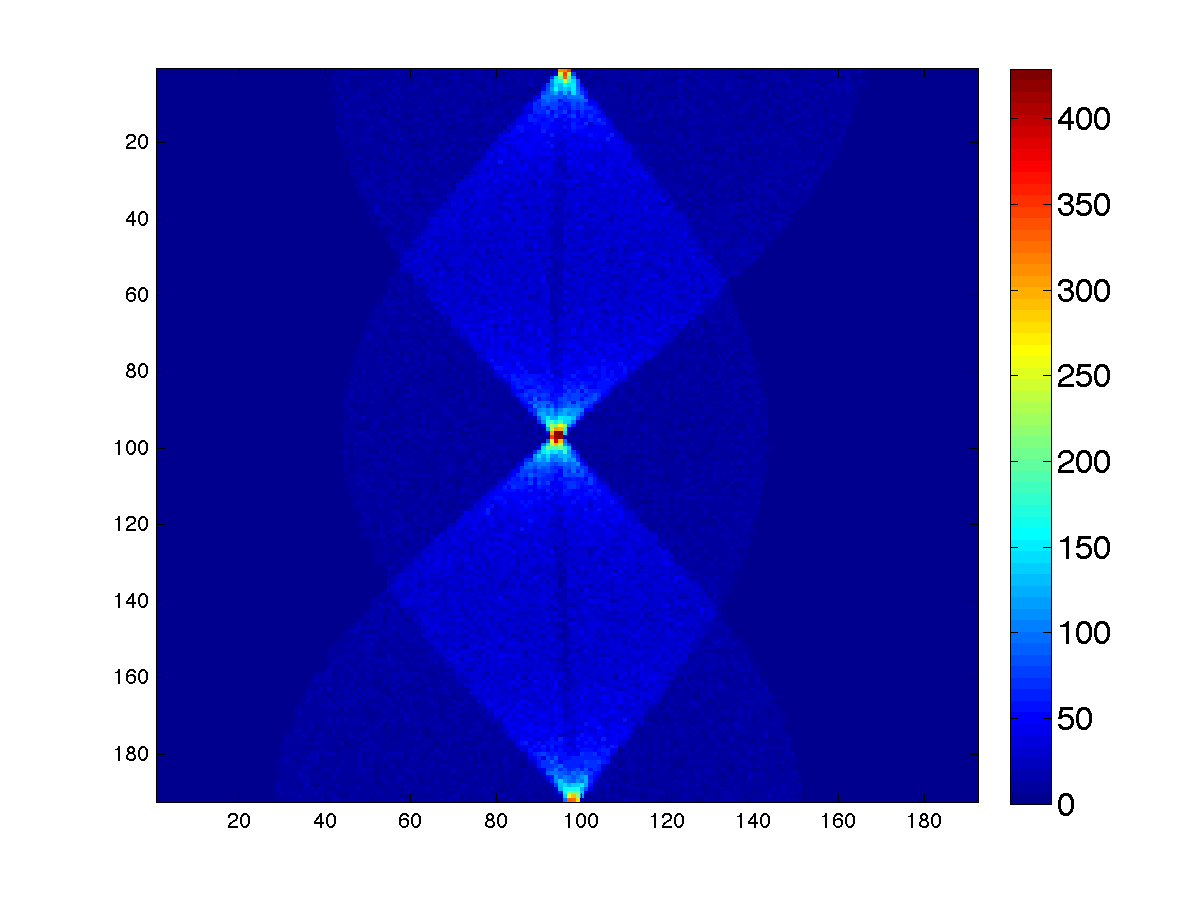}
		 \caption{SNR=16.1538}					
\end{subfigure}
\end{center}
\caption{Test image of two thin crossing lines and its noiseless and noisy sinograms respectively.}
\label{sxima19a}
\end{figure}

\begin{figure}[h!]
\begin{center}
\begin{subfigure}[h]{7cm}
                \centering
                \includegraphics[scale=0.35]{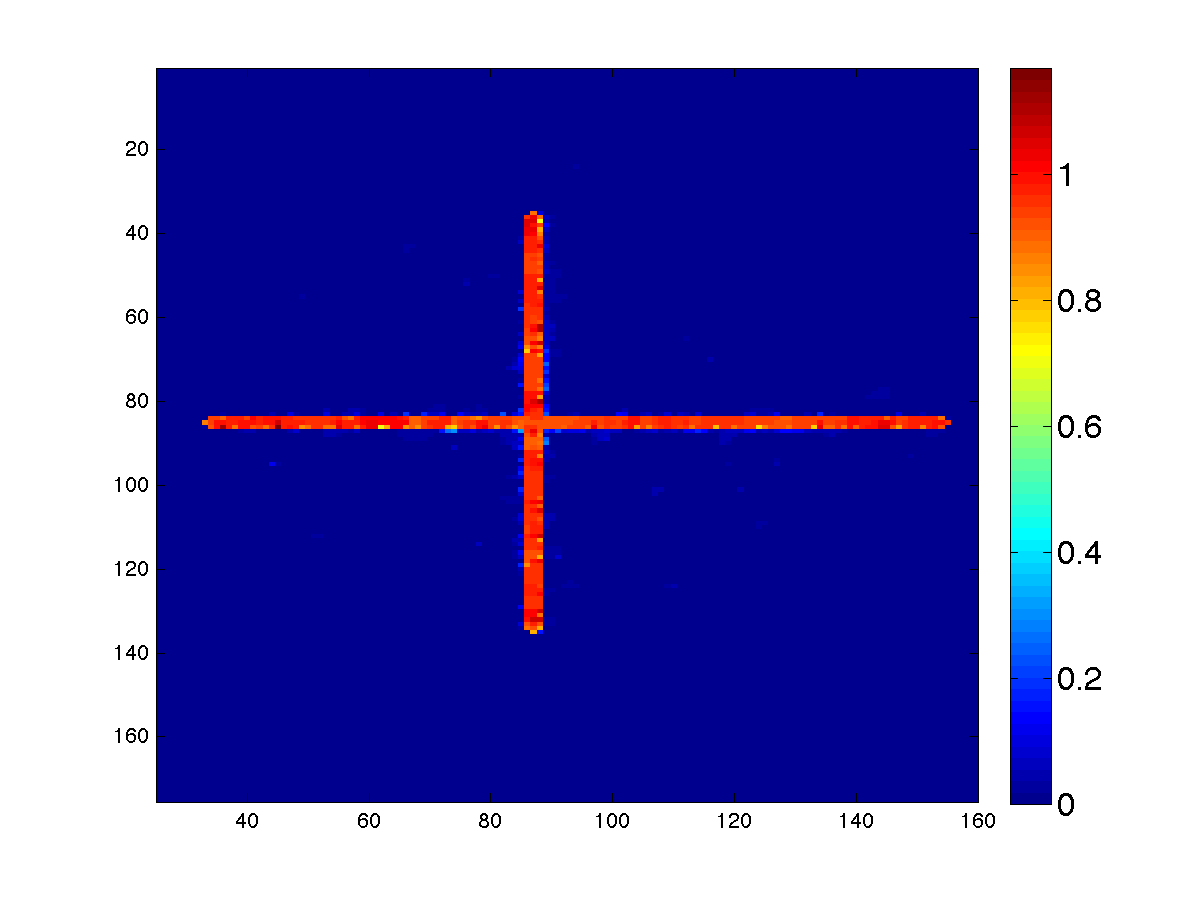}
								\caption{$\alpha=7$, $\beta=0$\\SNR=20.6859}
\end{subfigure}
\begin{subfigure}[h]{7cm}
                \centering
                \includegraphics[scale=0.35]{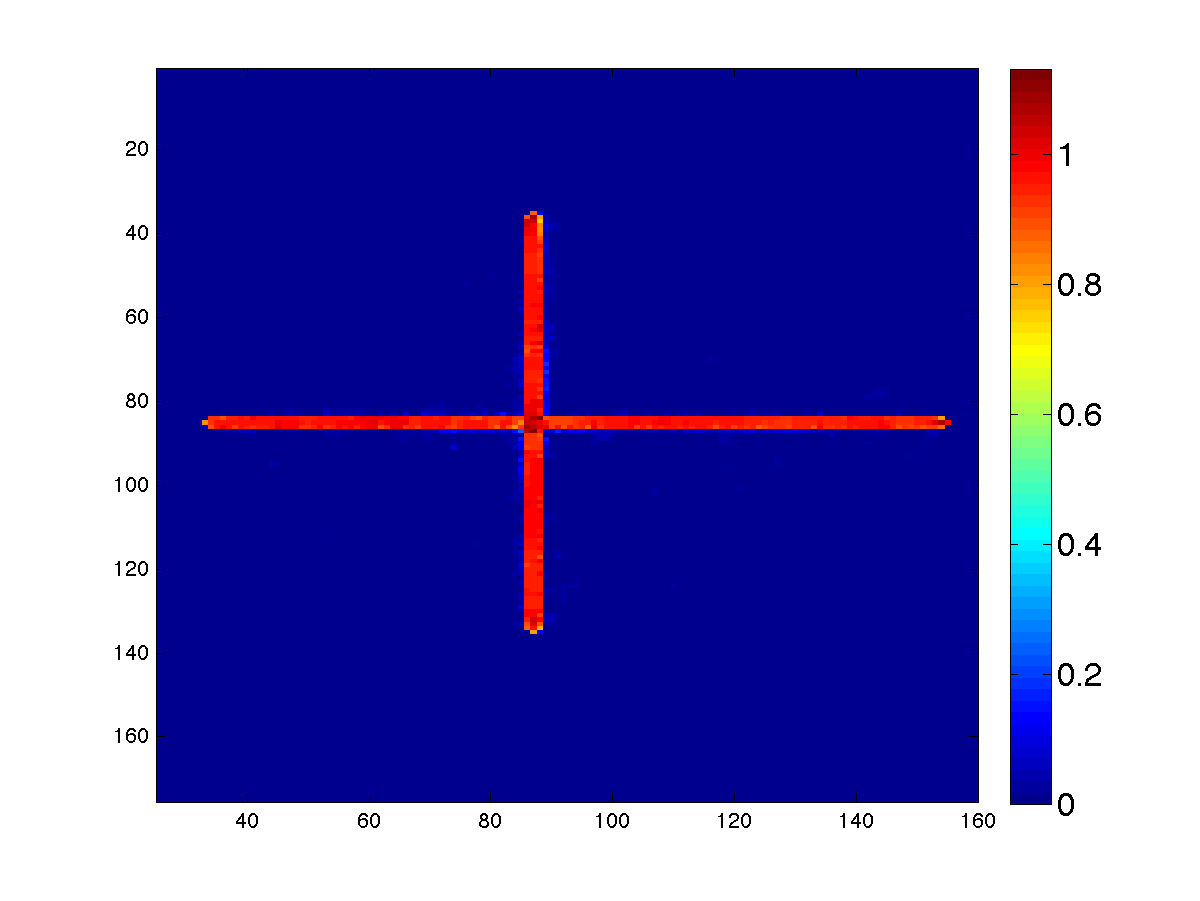}
								\caption{$\alpha=5$, $\beta=0.05$\\SNR=22.8333}
\end{subfigure}
\end{center}
\caption{Reconstruction for the noisy sinogram in Figure \ref{sxima19a} that correspond to the best SNR for both cases of $\beta$.}
\label{sxima19}
\end{figure}

%



\begin{figure}[h!]
                \centering
                \includegraphics[scale=0.4]{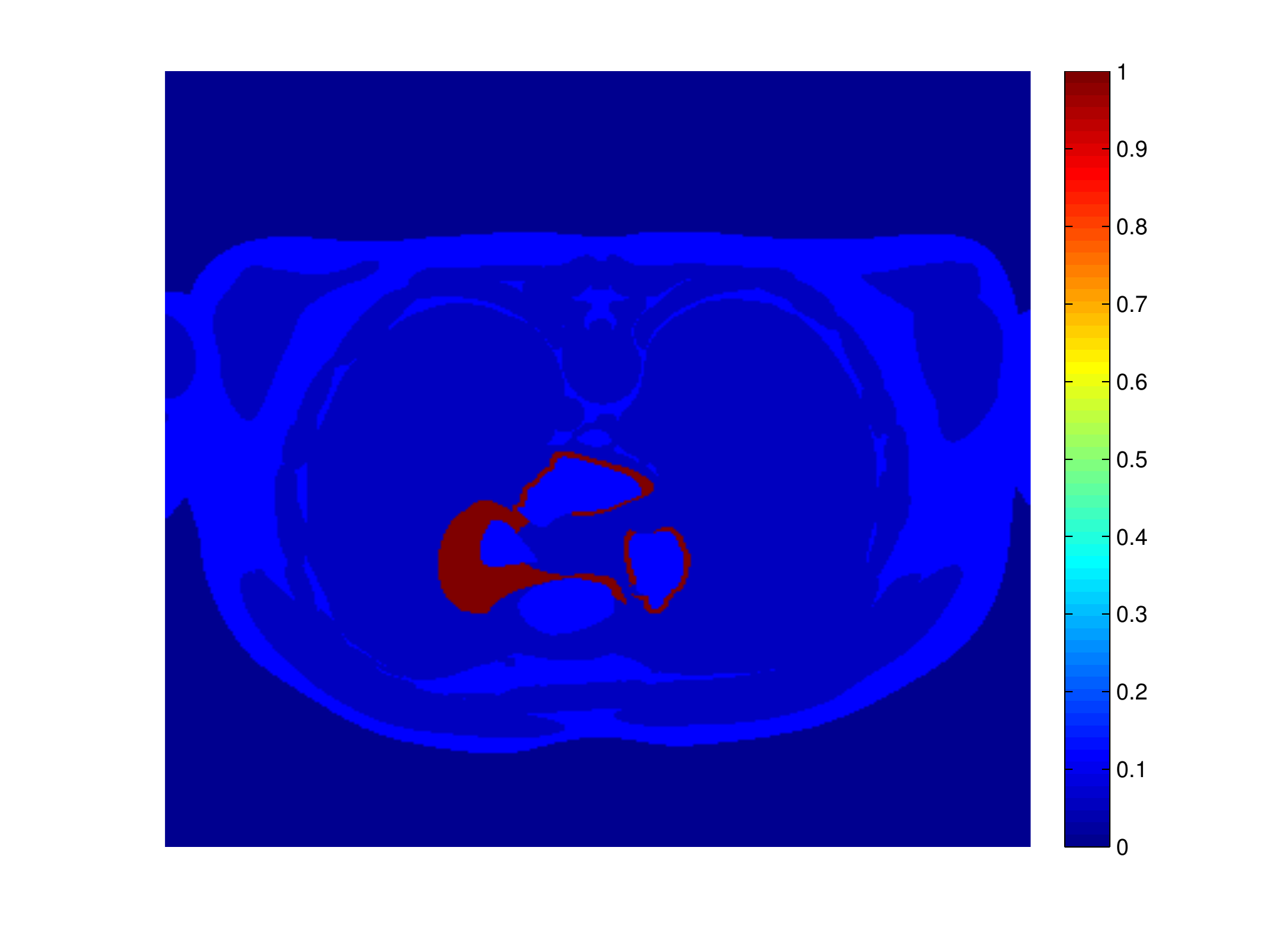}
                \caption{XCAT phantom}
                 \label{xcat1}
\end{figure}

We also apply our method to a more realistic PET phantom for visualising activity of the human heart. The XCAT phantom is a 3D phantom. For our purpose we used one z-slice through the centre of the phantom which represents a transverse plane view of the human body, see Figure \eqref{xcat1}. In particular, we can see the activity of the heart through the myocardium (the muscle surrounding the heart) in red. We focus on regions where thin structures are observed, see Figure \eqref{xcat2:right}-\eqref{xcat2:upper} and add the usual level of Poisson noise to their corresponding sinograms, see Figure  \eqref{xcat2:right_sinogram}-\eqref{xcat2:upper_sinogram}. In Figures \eqref{xcat:right1}-\eqref{xcat:upper2} we present our best reconstructions for these two different data-regions in terms of the SNR values for both cases of with and without sinogram regularisation. It is obvious that the best reconstructions are achieved when there is no regularisation on the sinogram. That is because for increasing values of $\beta$ a smoothing on the originally blocky boundaries is enforced and hence the SNR value is reduced. Indeed, as we show in the following experiments this is only true if the initial data that we start our experiments with is of low resolution and the thin structures have \emph{blocky} instead of smooth boundaries. If we change our experiment to the consideration of a high resolution version of the XCAT phantom with thin structures as in Figures \eqref{xcat2:right}-\eqref{xcat2:upper} but with medically more realistic smooth boundaries, the positive effect of the TV sinogram regularisation can be observed. As it is expected, regularising only on the image space creates a rather unpleasant staircasing effect along the boundaries which is clearly eliminated when we combine the regularisation on both spaces, see Figure \ref{xcat4}. Indeed, a significant increase of the SNR when turning on the TV regularisation on the sinogram ($\beta>0$) can be observed. 


\begin{figure}[h!]
\begin{center}
\begin{subfigure}[h]{7cm}
                \centering
                \includegraphics[scale=0.4]{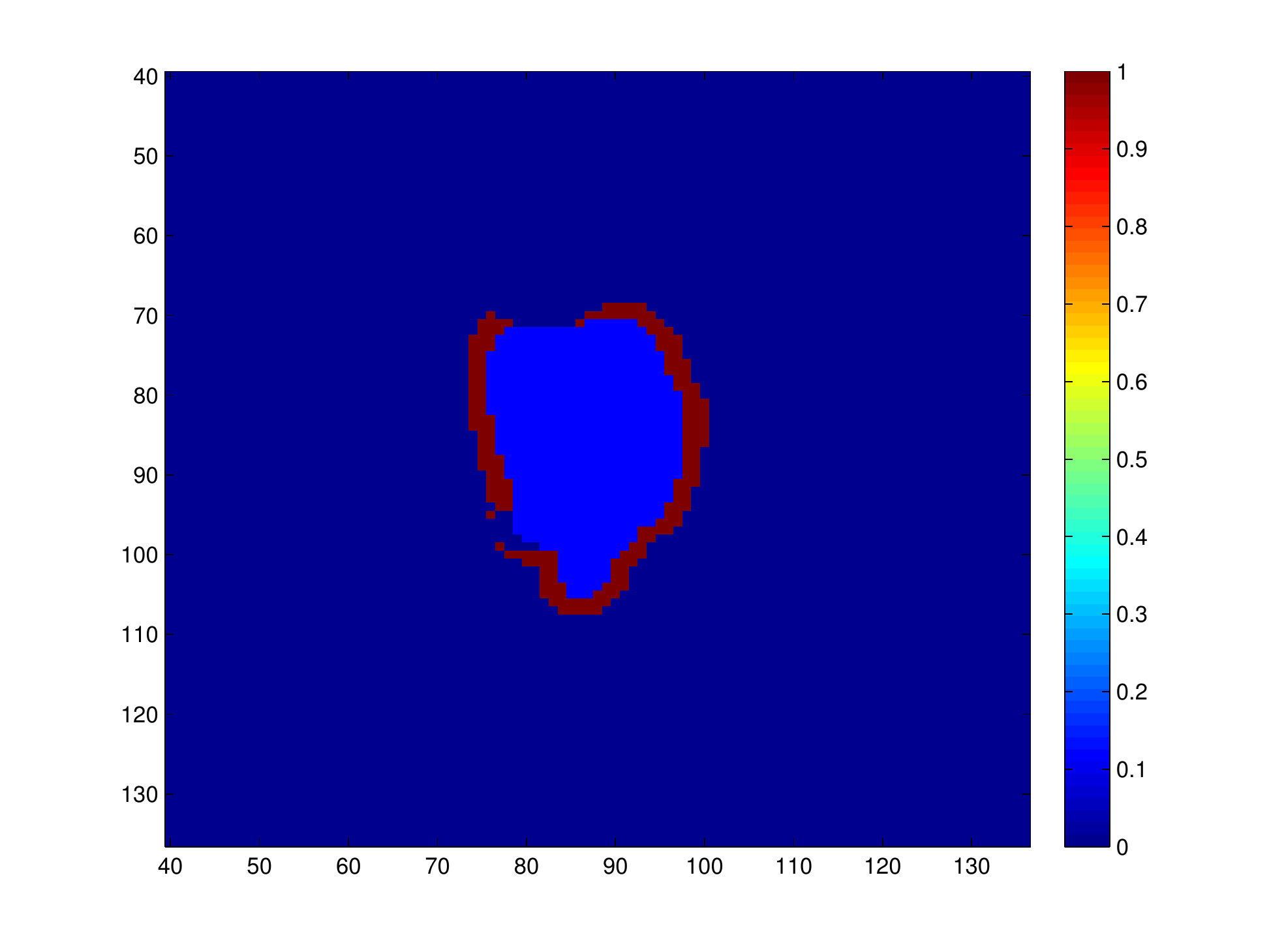}
		\caption{Zoom in}	
		\label{xcat2:right}				
\end{subfigure}
\begin{subfigure}[h]{7cm}
                \centering
                \includegraphics[scale=0.4]{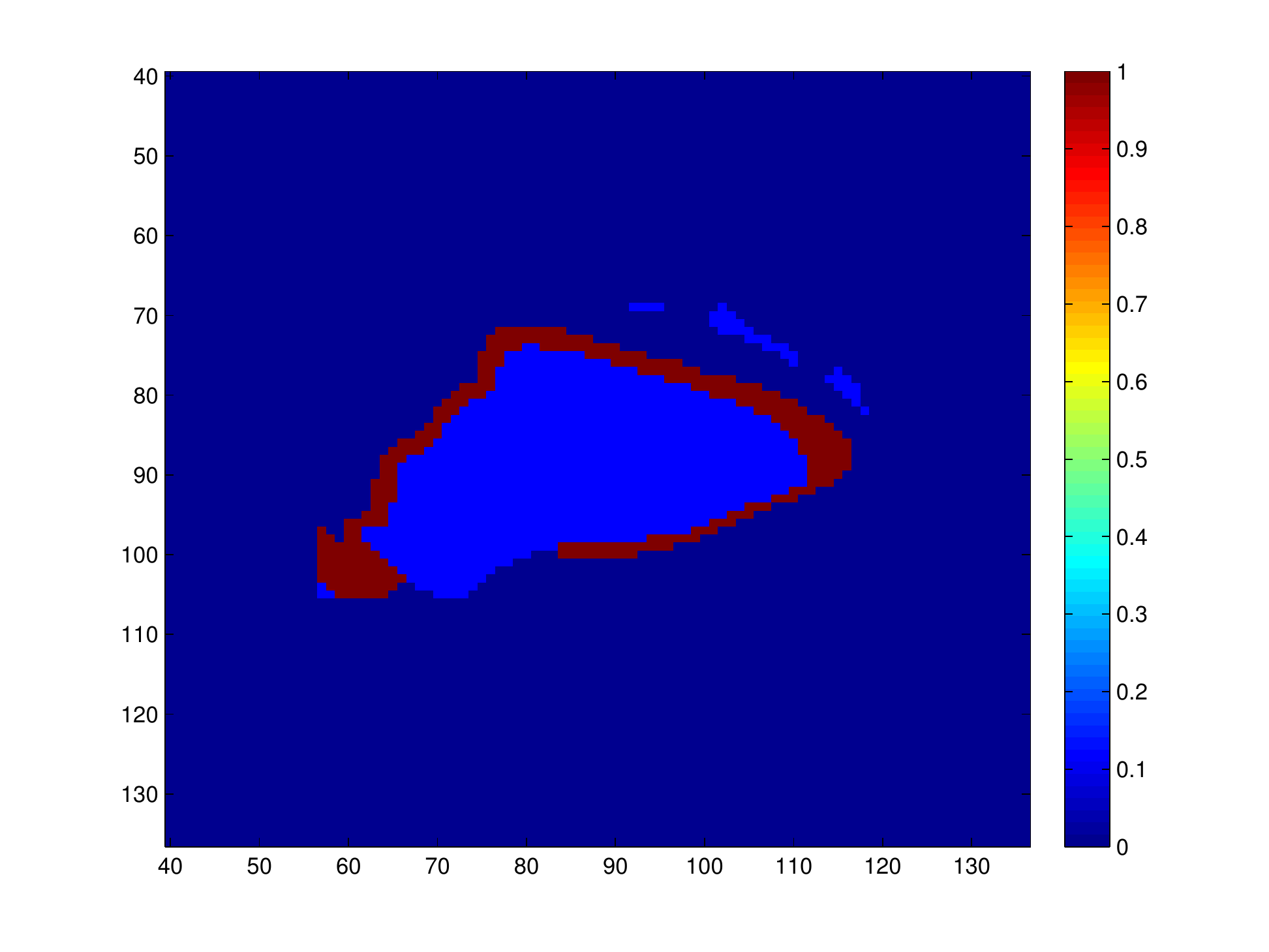}
		\caption{Zoom in}	
		\label{xcat2:upper}						
\end{subfigure}\\
\begin{subfigure}[h]{7cm}
                \centering
                 \includegraphics[scale=0.4]{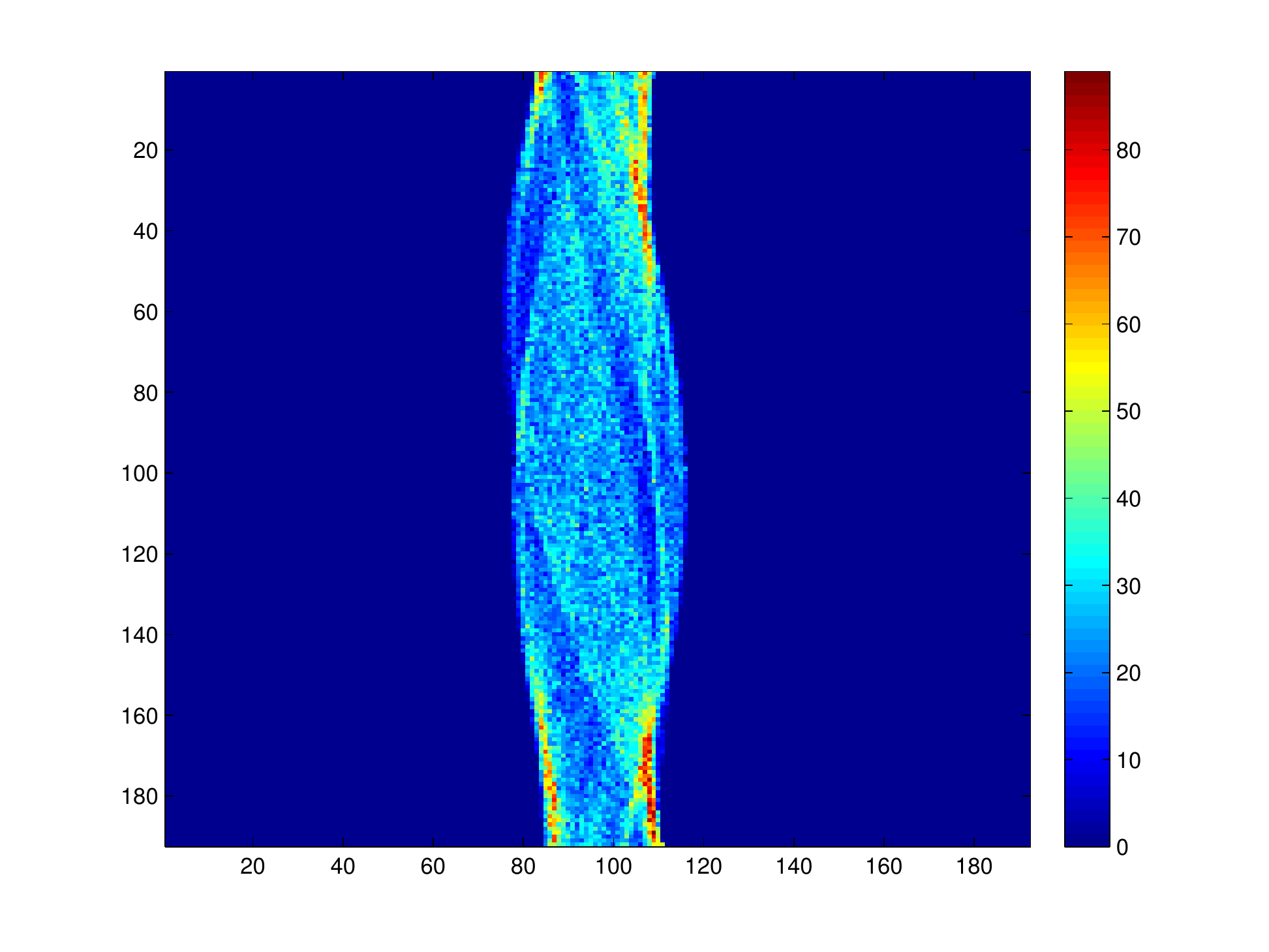}
		\caption{Noisy sinogram of (a)}
		\label{xcat2:right_sinogram}
\end{subfigure}
\begin{subfigure}[h]{7cm}
                \centering
                \includegraphics[scale=0.4]{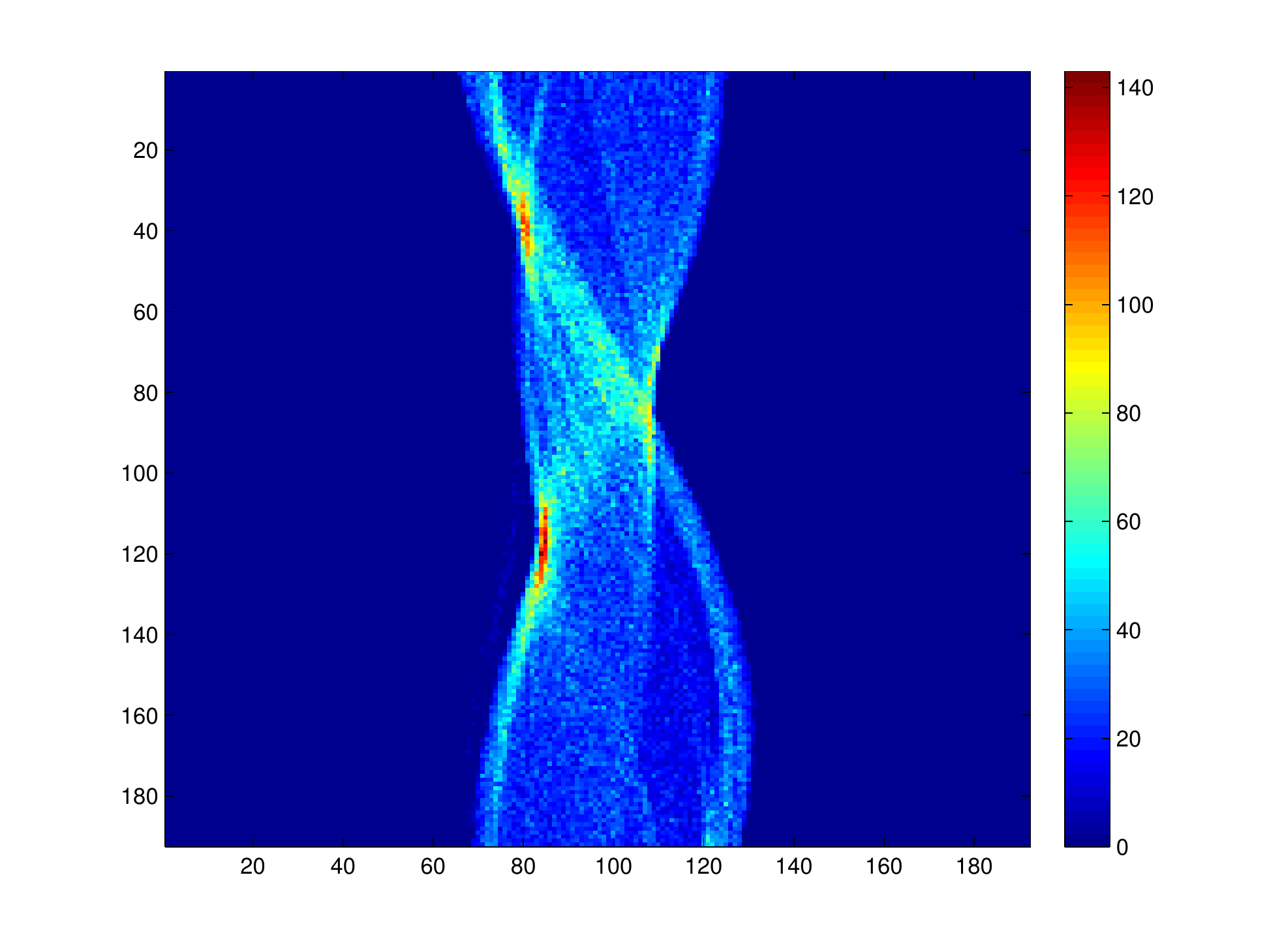}
		\caption{Noisy sinogram of (b)}
		\label{xcat2:upper_sinogram}					
\end{subfigure}
\caption{Selected regions of the XCAT phantom with the corresponding noisy sinograms.}
\label{xcat2}
\end{center}
\end{figure}

\begin{figure}[h!]
\begin{center}
\begin{subfigure}[h]{3.5cm}
                \centering
                \includegraphics[scale=0.2]{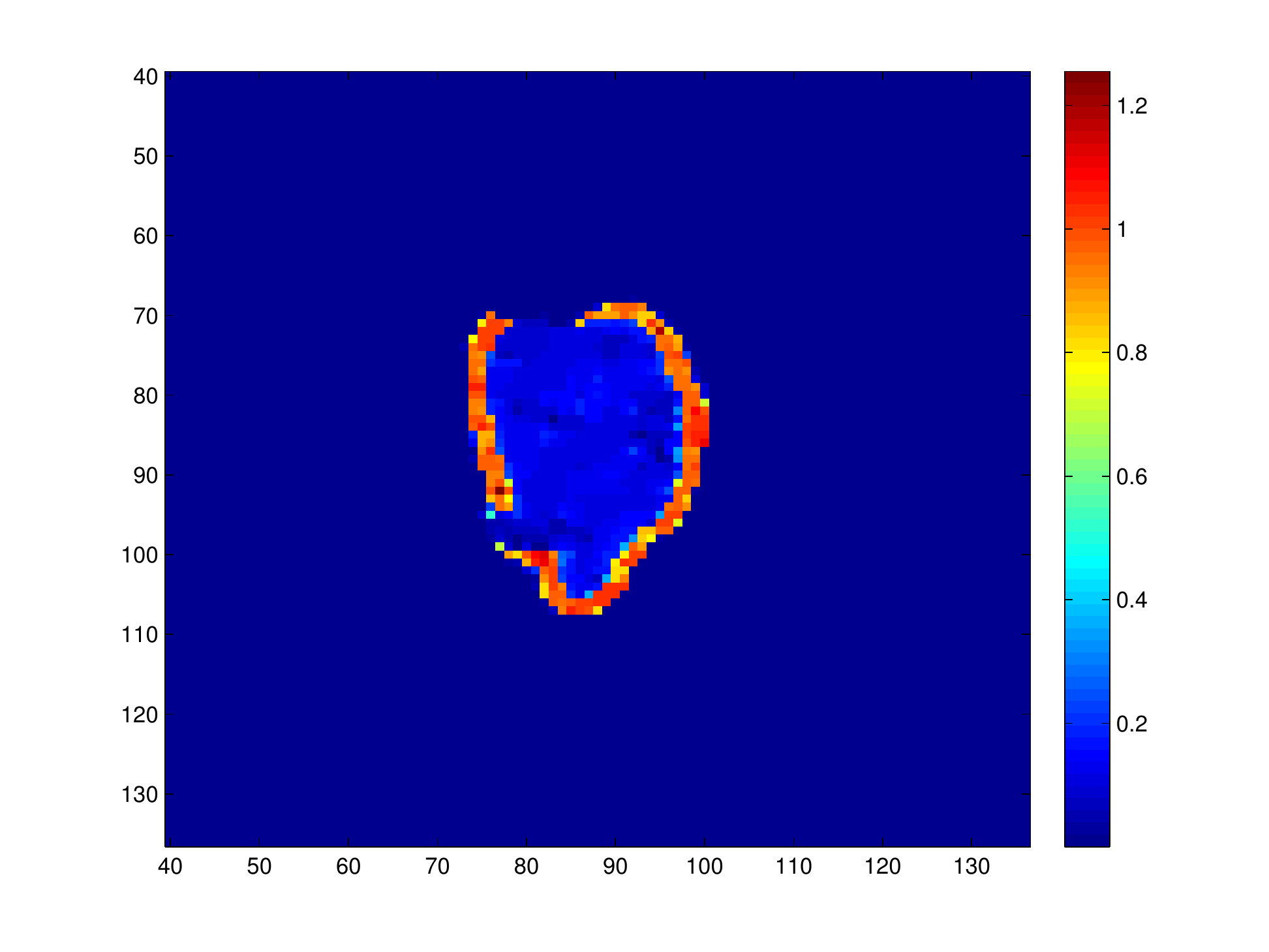}
		\caption{$\alpha=5$, $\beta=0$\\ SNR=17.49887}	
		\label{xcat:right1}				
\end{subfigure}
\begin{subfigure}[h]{3.5cm}
                \centering
                \includegraphics[scale=0.2]{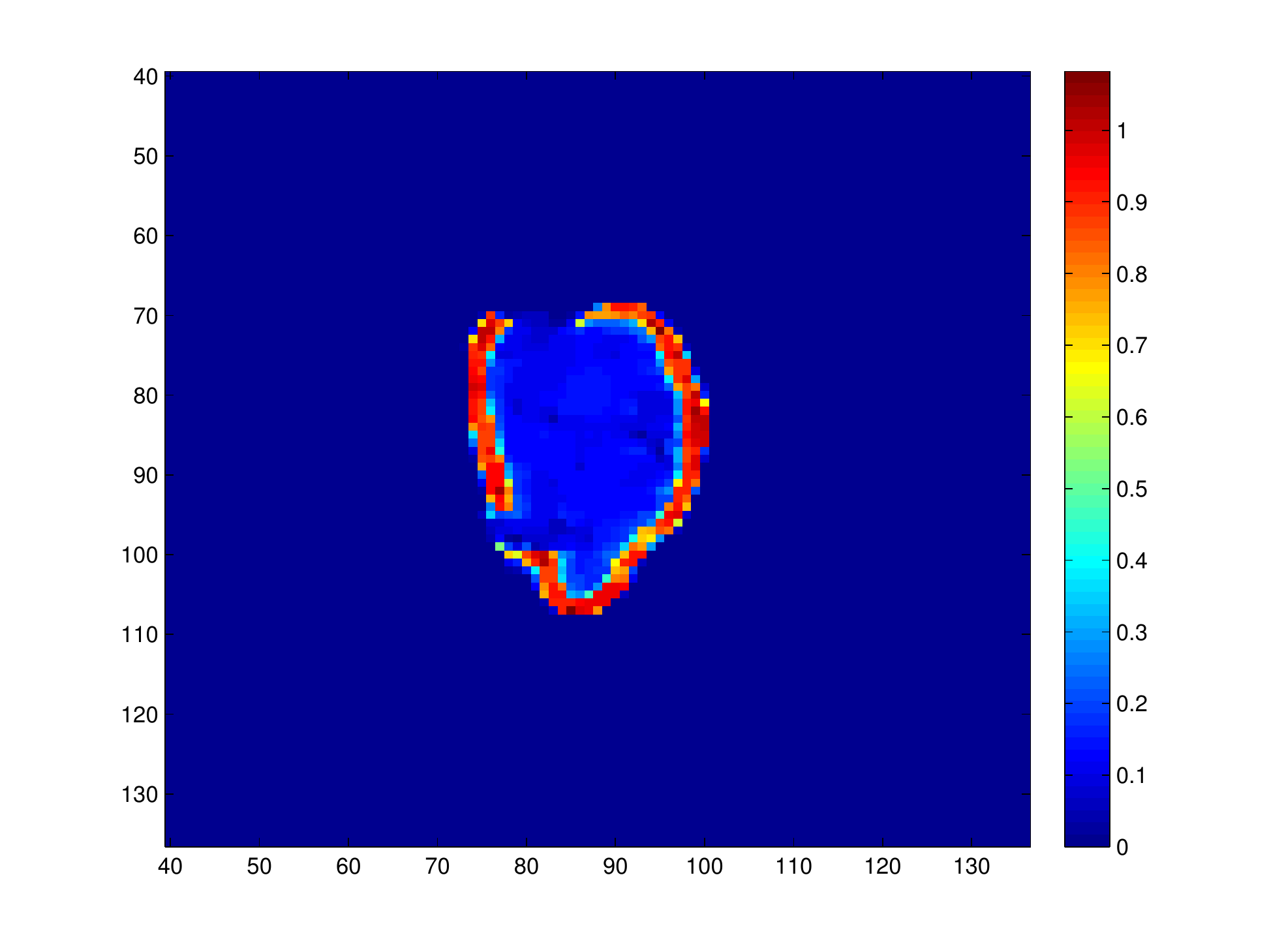}
		\caption{$\alpha=5$, $\beta=0.05$\\SNR=13.4267}	
		\label{xcat:right2}						
\end{subfigure}
\begin{subfigure}[h]{3.5cm}
                \centering
                 \includegraphics[scale=0.2]{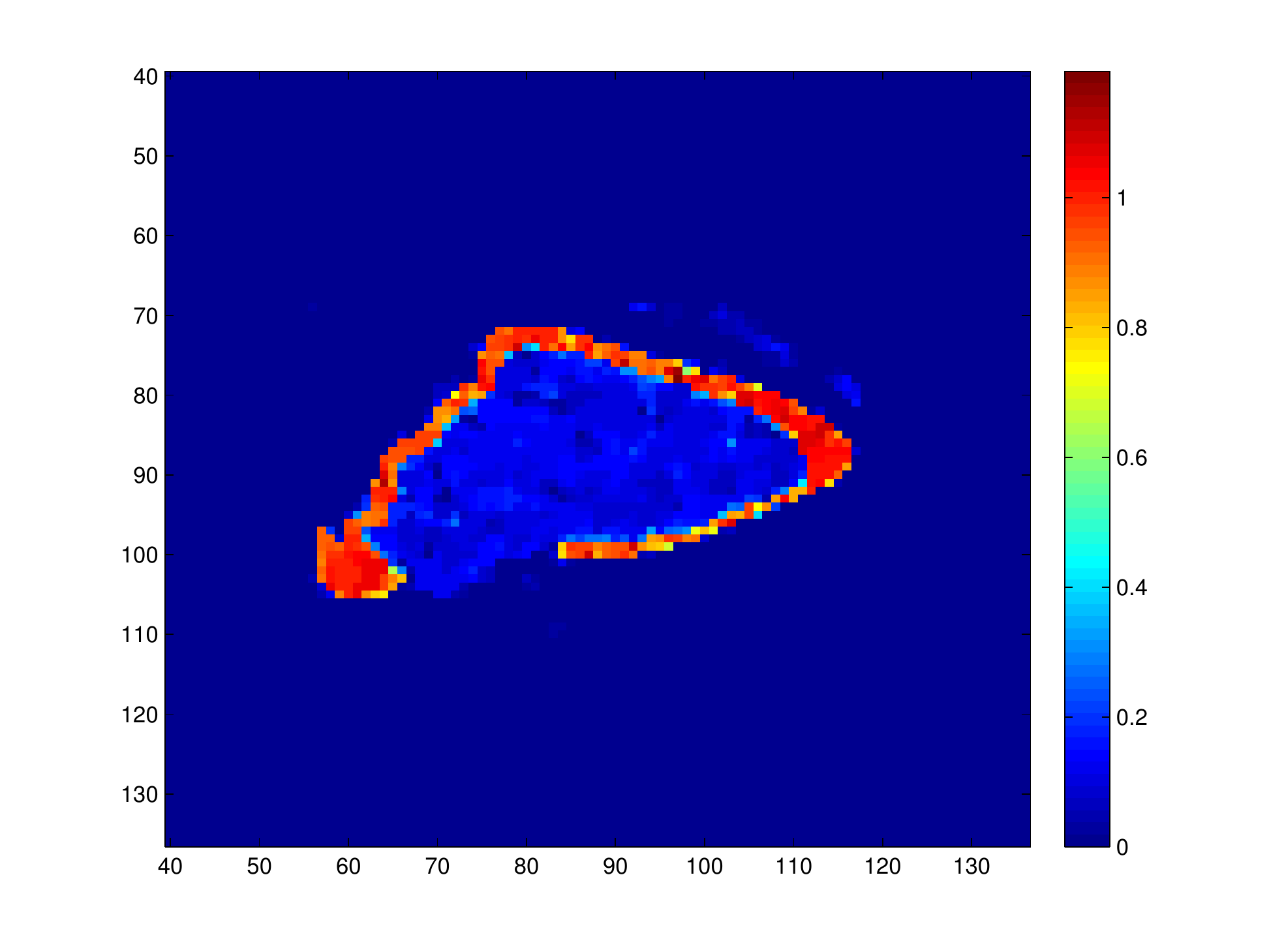}
		\caption{$\alpha=4$, $\beta=0$\\SNR=16.8721}		
		\label{xcat:upper1}
\end{subfigure}
\begin{subfigure}[h]{3.5cm}
                \centering
                \includegraphics[scale=0.2]{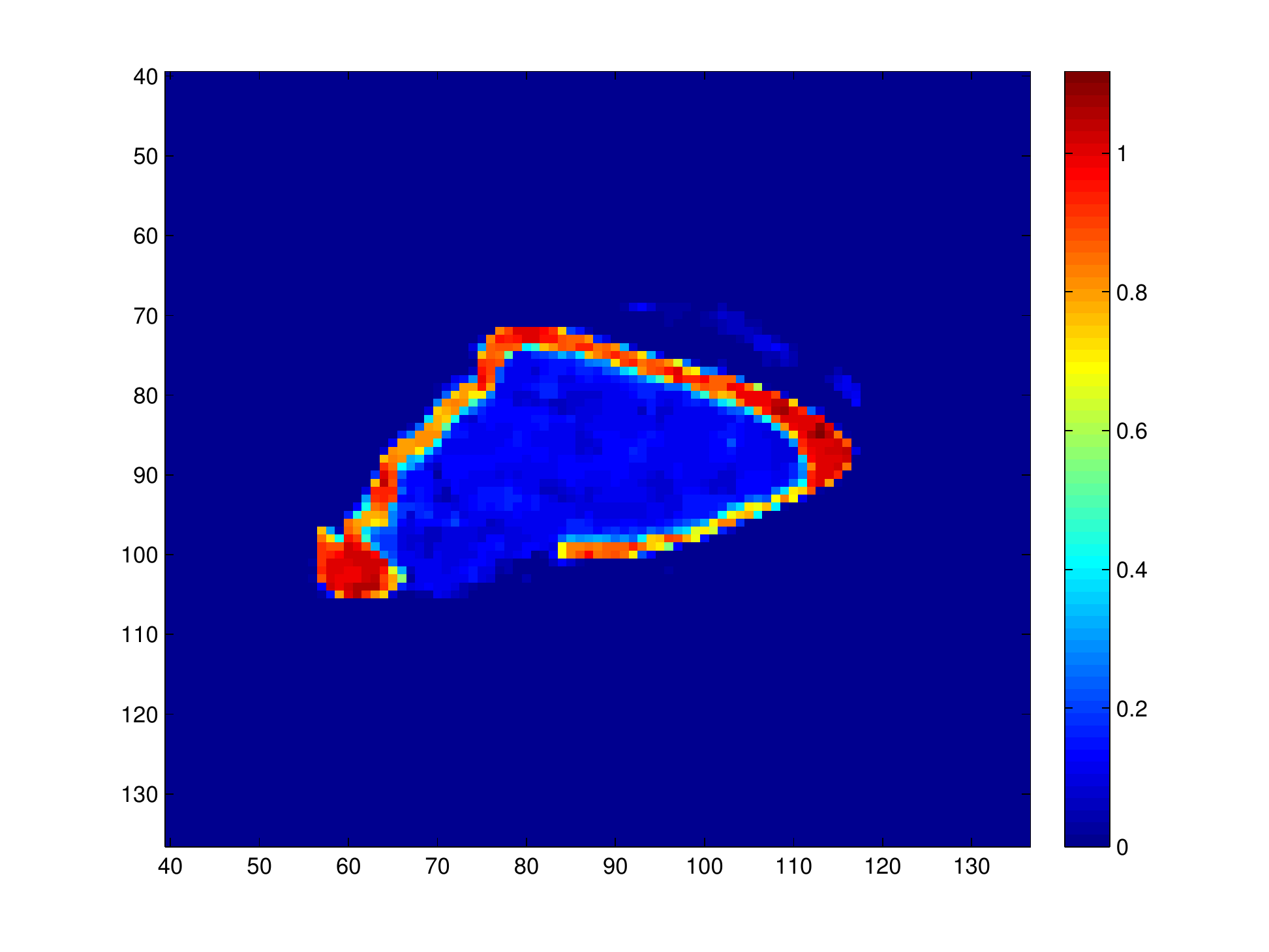}
		\caption{$\alpha=4$, $\beta=0.05$\\SNR=13.1124}	
		\label{xcat:upper2}					
\end{subfigure}
\caption{Reconstructions with and without total variation regularisation on the details of the XCAT sinogram in Figure \ref{xcat2}.}
\label{xcat3}
\end{center}
\end{figure}

\begin{figure}[h!]
\begin{center}
\begin{subfigure}[h]{7cm}
                \centering
                \includegraphics[scale=0.4]{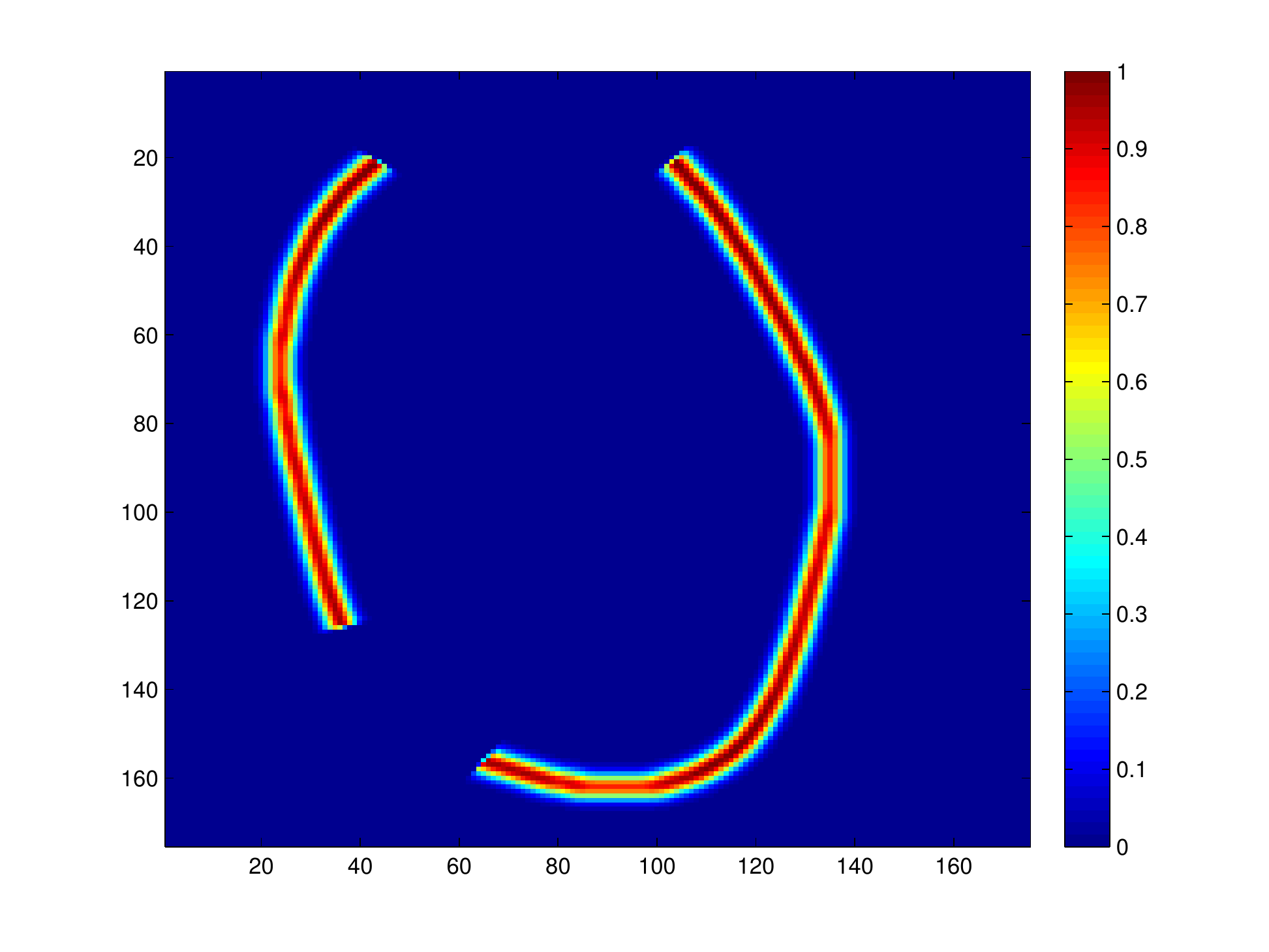}
		\caption{Smooth version of Figure \eqref{xcat2:right}}	
		\label{xcat3:right1_smooth}				
\end{subfigure}
\begin{subfigure}[h]{7cm}
                \centering
                \includegraphics[scale=0.4]{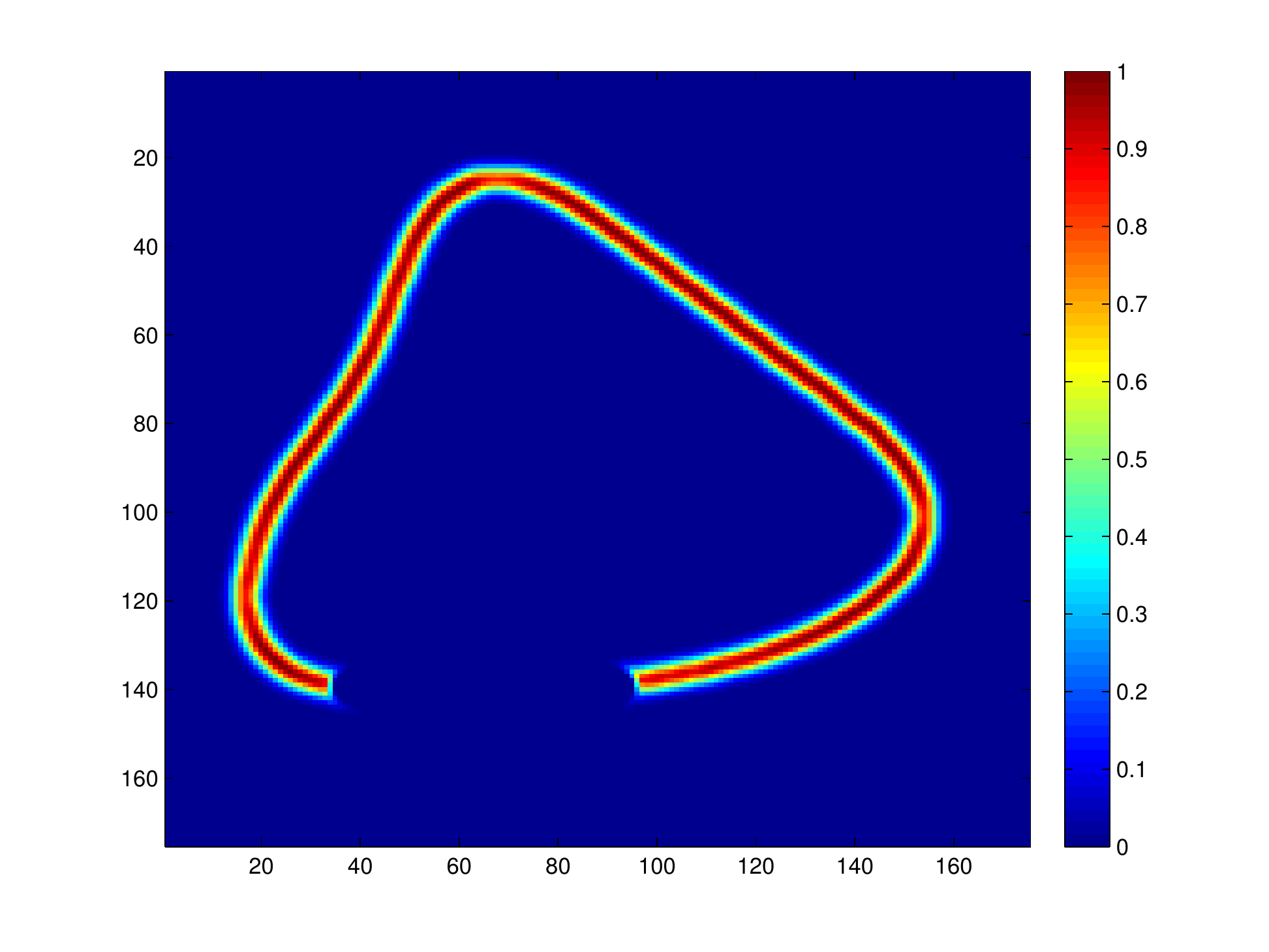}
		\caption{Smooth version of Figure \eqref{xcat2:upper}}	
		\label{xcat3:right1_smooth}						
\end{subfigure}\\
\begin{subfigure}[h]{7cm}
                \centering
                \includegraphics[scale=0.4]{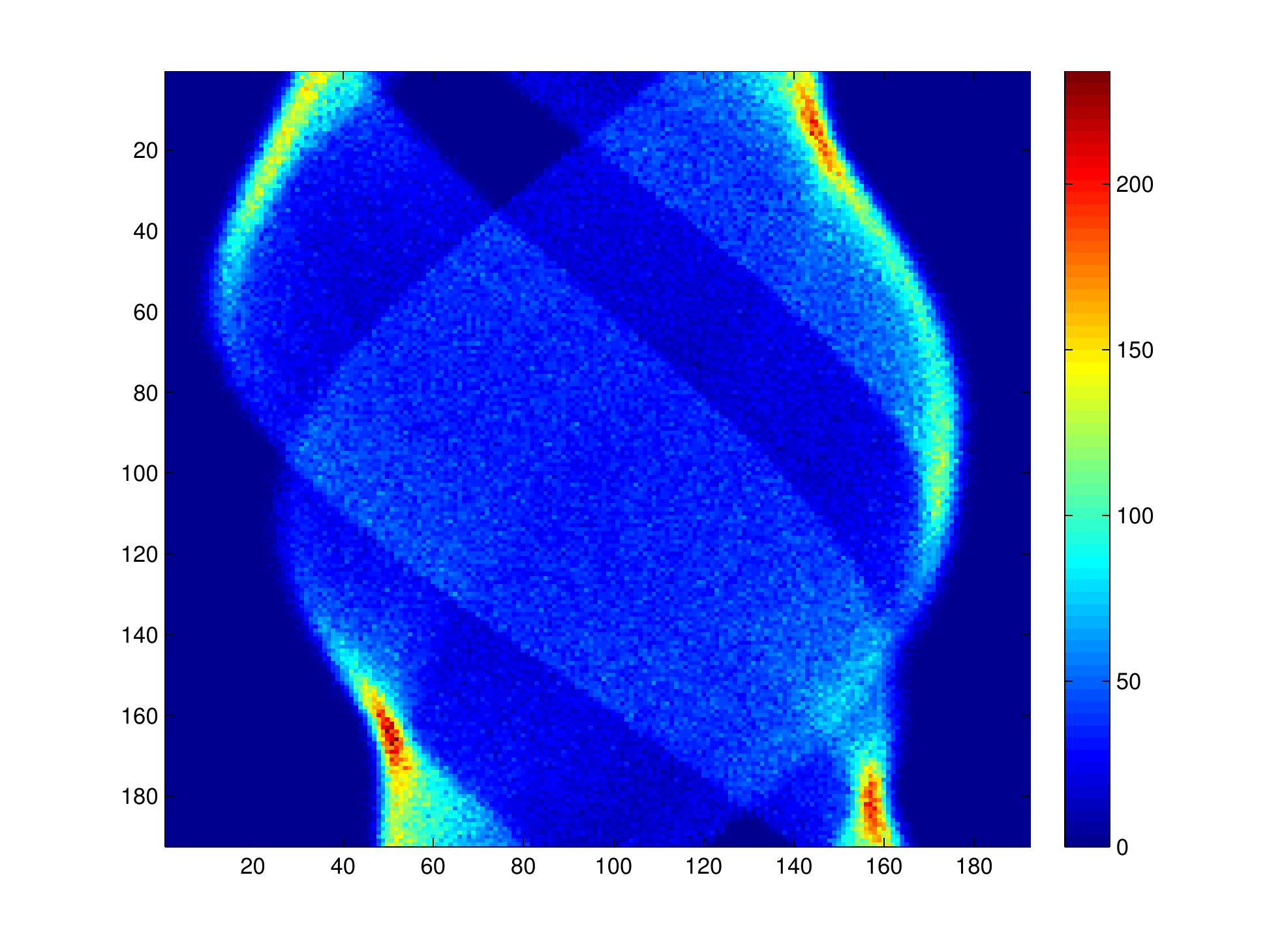}
		\caption{Noisy sinogram of (a)}		
		\label{xcat3:right1_smooth_sin}
\end{subfigure}
\begin{subfigure}[h]{7cm}
                \centering
                \includegraphics[scale=0.4]{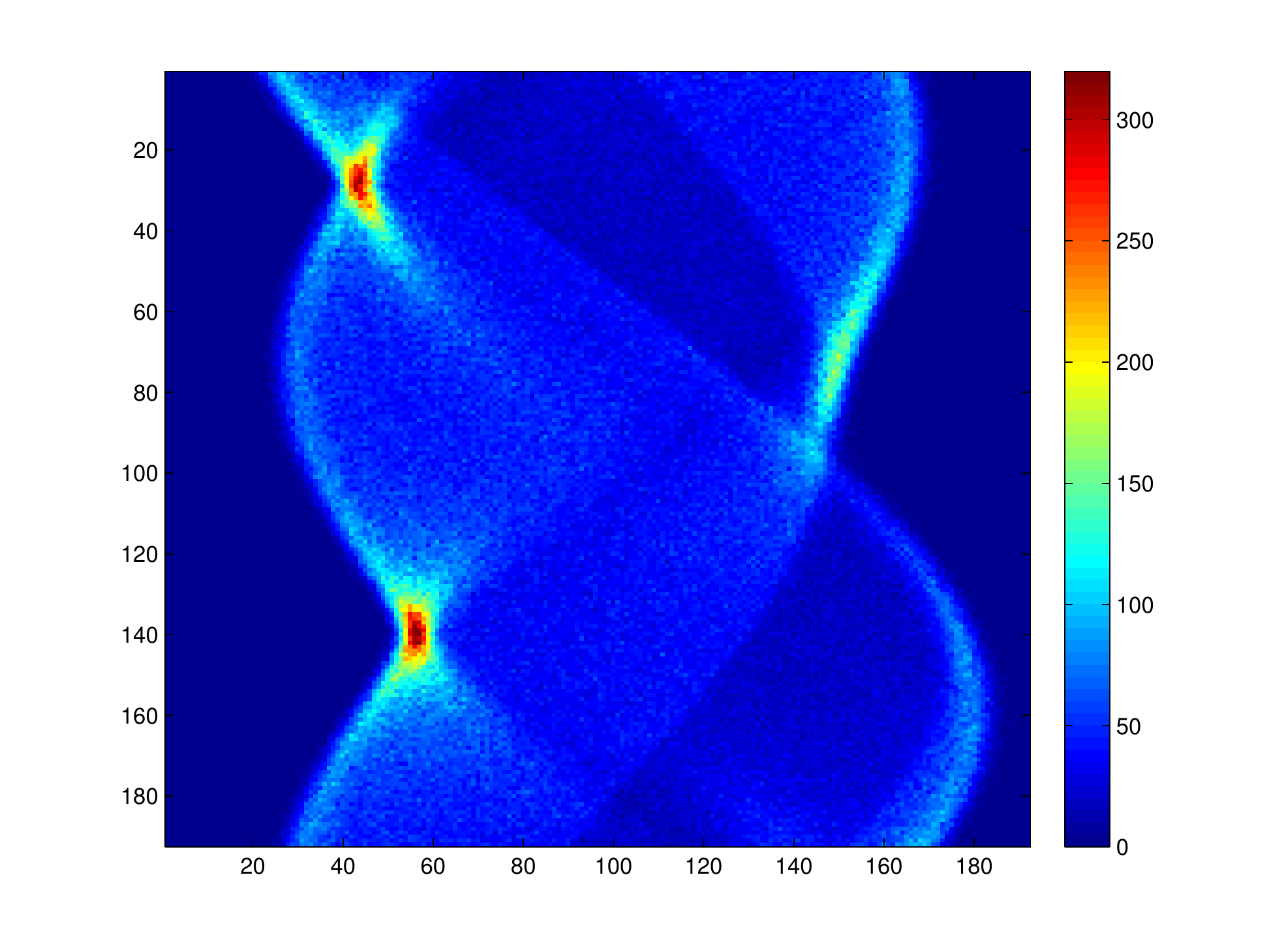}
		\caption{Noisy sinogram of (b)}
		\label{xcat3:right1_smooth_sin}				
\end{subfigure}
\caption{High resolution XCAT: smooth versions of Figures \eqref{xcat2:right}-\eqref{xcat2:upper} and their noisy sinograms}
\label{xcat3}
\end{center}
\end{figure}

\begin{figure}[h!]
\begin{center}
\begin{subfigure}[h]{7cm}
                \centering
                \includegraphics[scale=0.4]{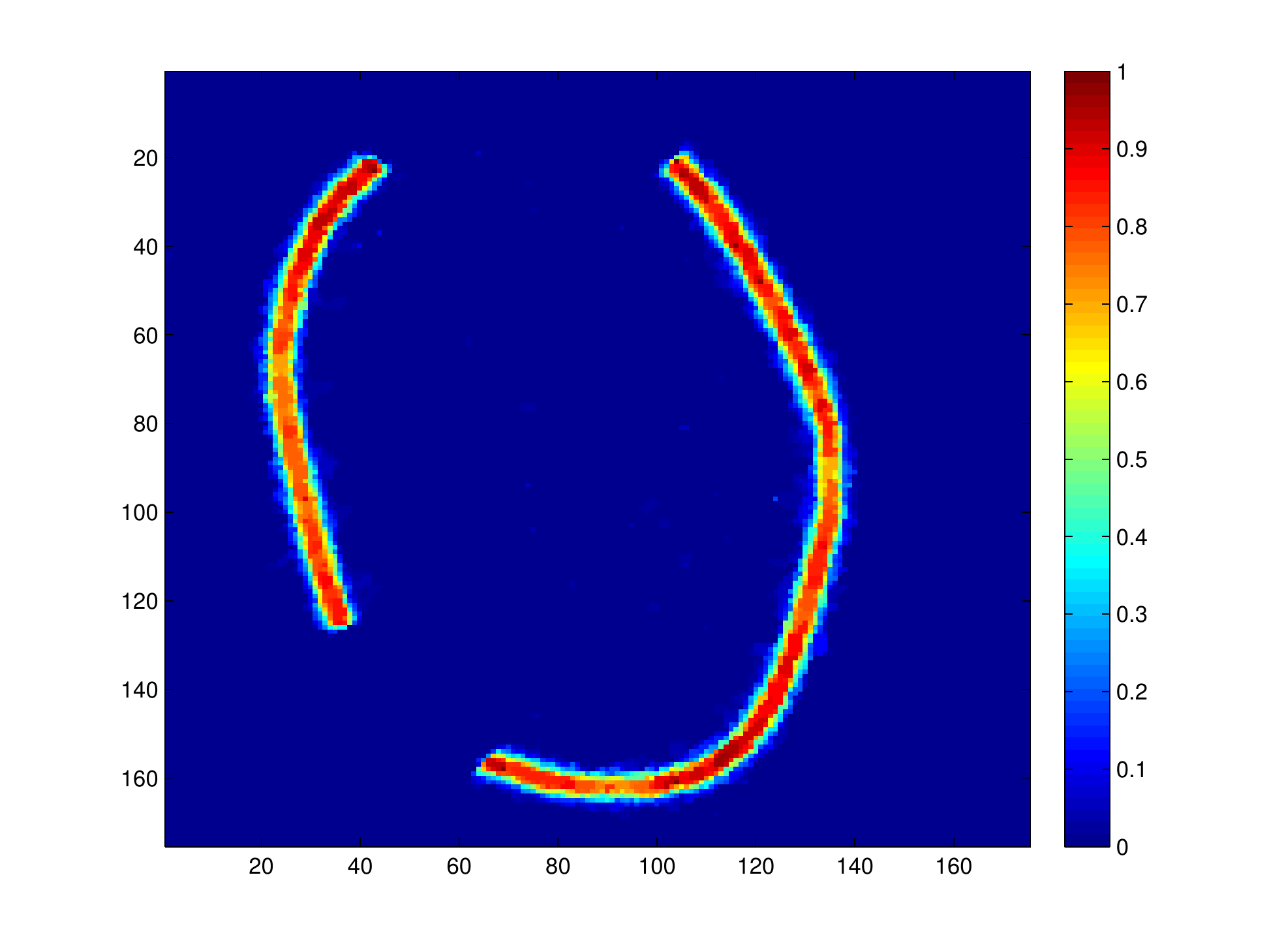}
		\caption{$\alpha=6$,  $\beta=0$\\SNR=17.7647}	
		\label{xcat4:right1}				
\end{subfigure}
\begin{subfigure}[h]{7cm}
                \centering
                 \includegraphics[scale=0.4]{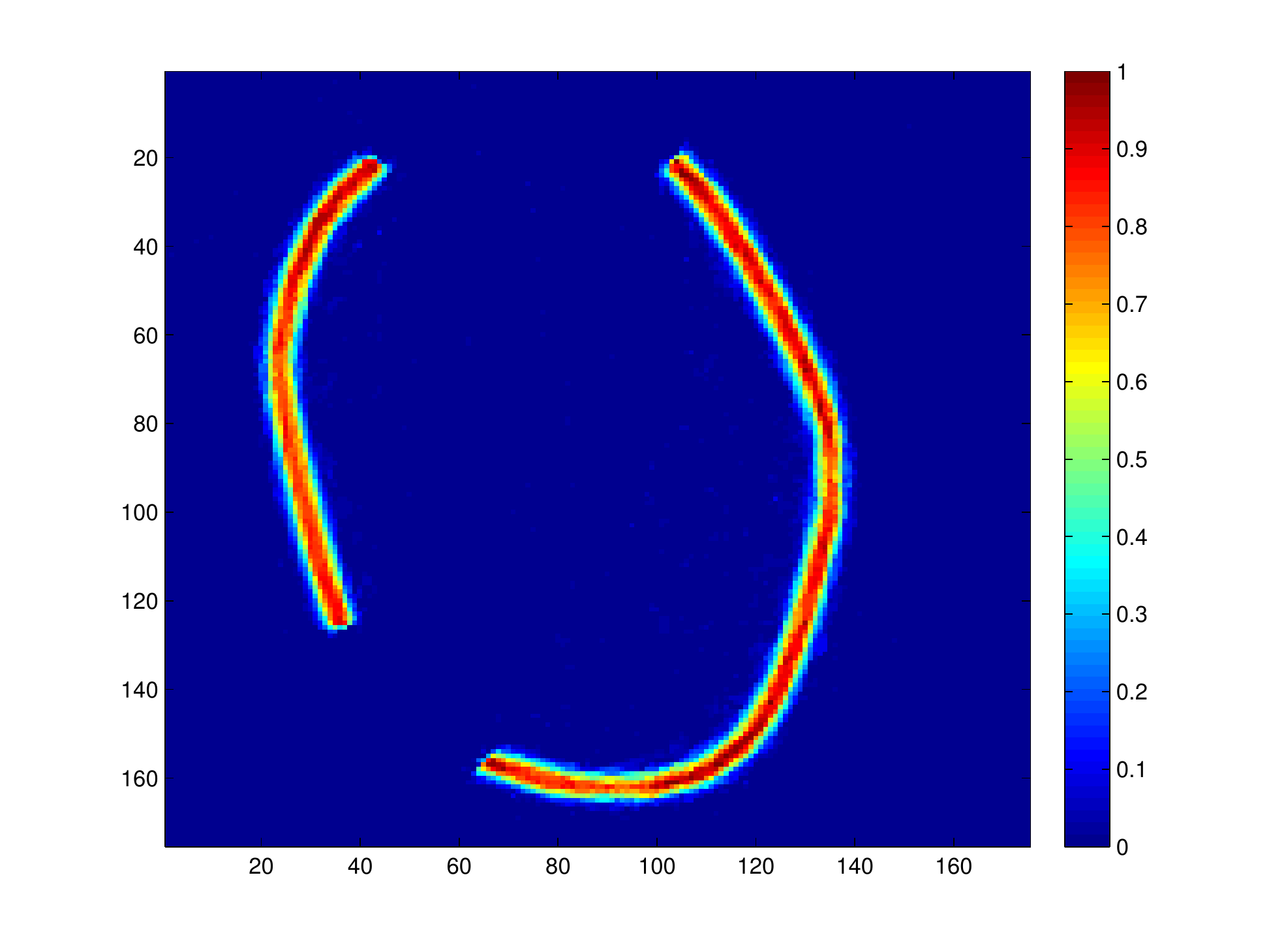}
		\caption{$\alpha=2$, $\beta=0.05$\\SNR=19.5103}	
		\label{xcat4:right2}						
\end{subfigure}\\
\begin{subfigure}[h]{7cm}
                \centering
                \includegraphics[scale=0.4]{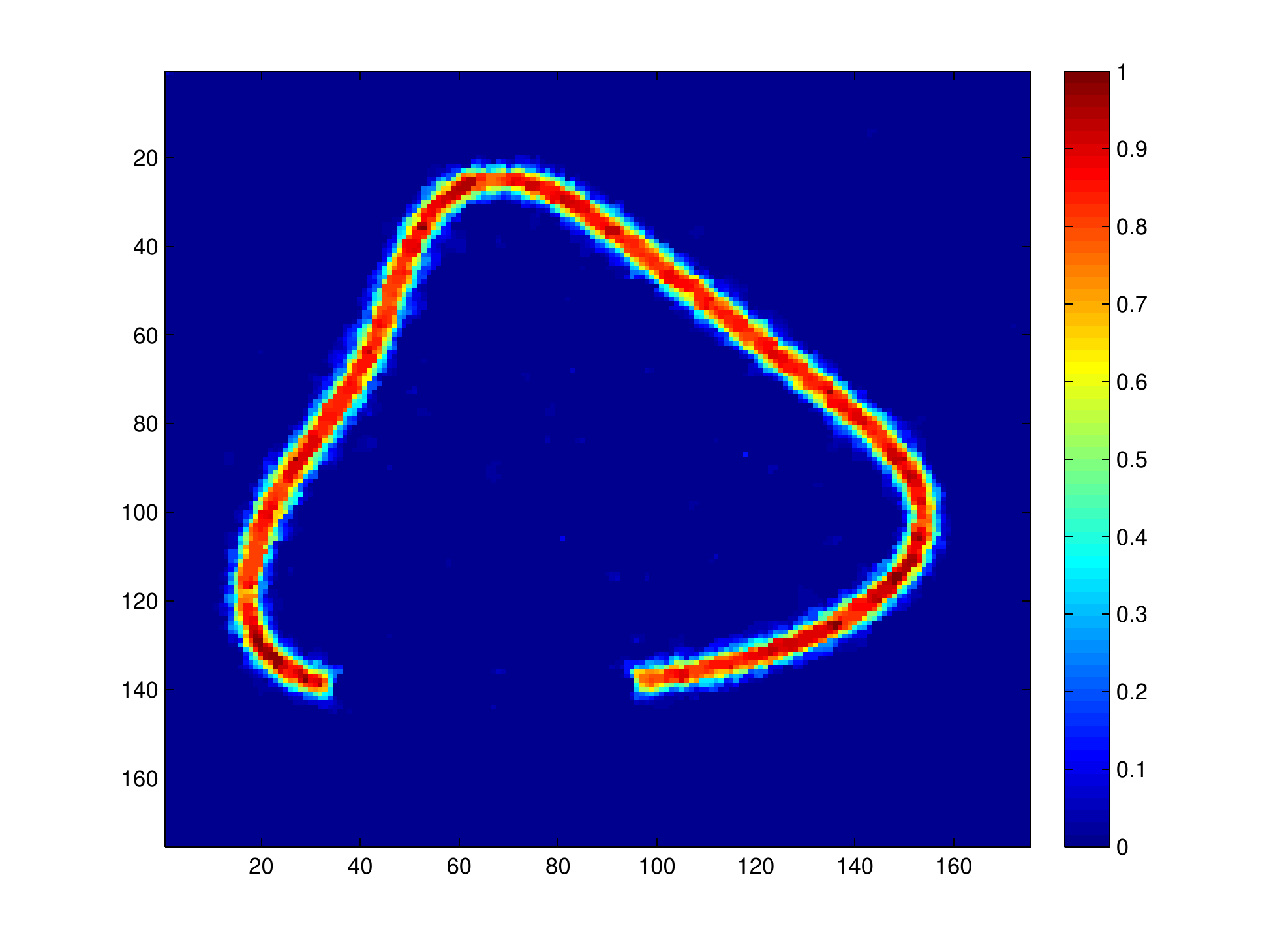}
		\caption{$\alpha=5$, $\beta=0$\\SNR=17.3795}		
		\label{xcat4:upper1}
\end{subfigure}
\begin{subfigure}[h]{7cm}
                \centering
                 \includegraphics[scale=0.4]{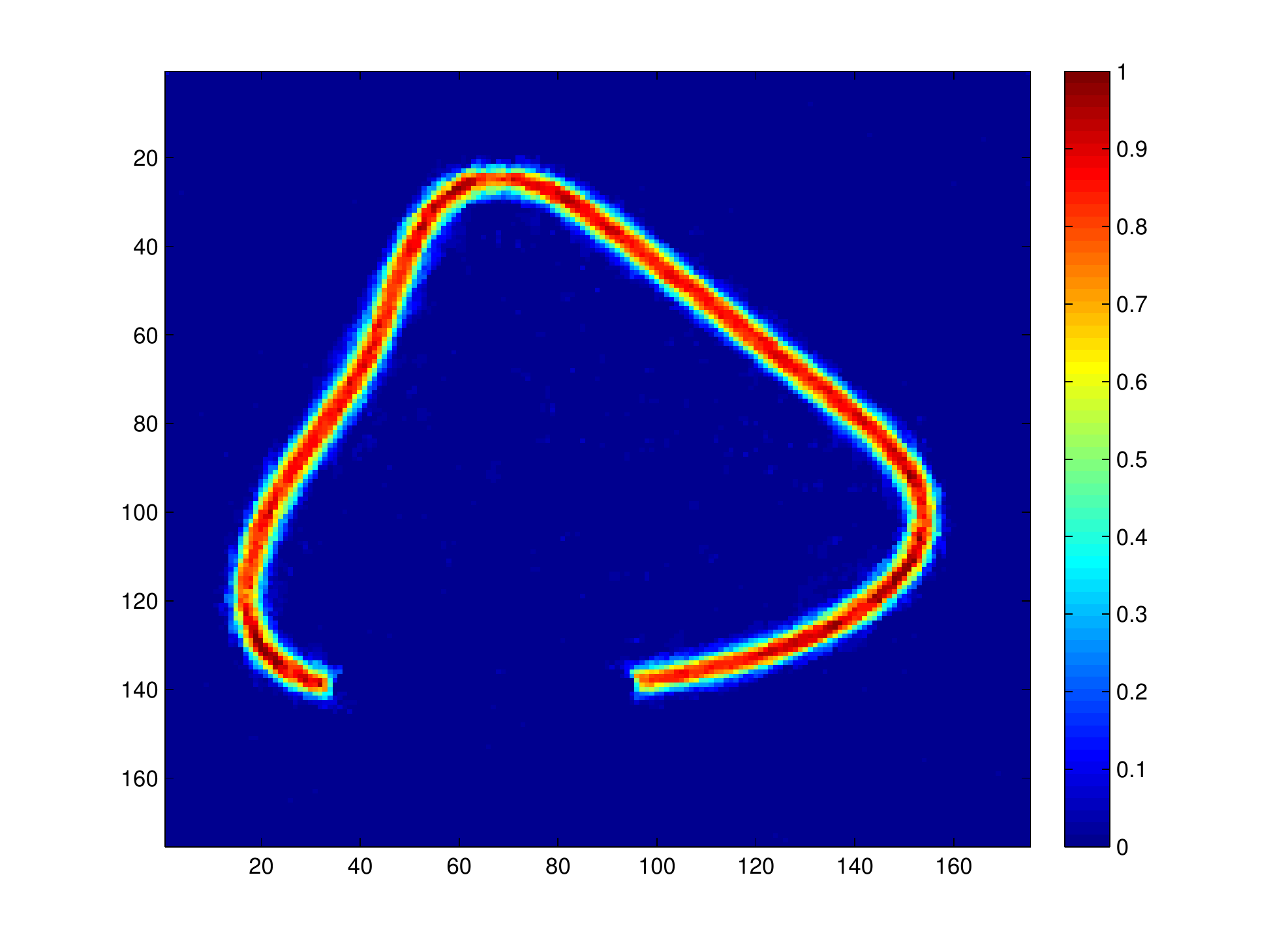}
		\caption{$\alpha=2$, $\beta=0.05$\\SNR=19.1820}	
		\label{xcat4:upper2}				
\end{subfigure}
\caption{Reconstructions for the structures in Figure \ref{xcat3} with and without total variation regularisation on the sinogram. Smoothing along the boundaries is achieved when sinogram regularisation is active, resulting in a significant improvement of the SNR.}
\label{xcat4}
\end{center}
\end{figure}

\section{Conclusion}

We present a combined approach of total variation regularisation of both the image and the sinogram for PET reconstruction. We prove existence, uniqueness and stability results for our proposed model with an additional error analysis through Bregman distance. Our explicit reconstruction of total variation regularisation, directly on the sinogram space, provides us with a new insight on how PET reconstruction could be improved and in which cases. 

We compute an optimal solution of the weighted-ROF model for a sinogram of disc in $\mathbb R^2$ and find analytically the corresponding solution on the image space via the Radon transform.  The weighted L$^{2}$ fidelity behaves as an approximation of the Poisson noise model given by the Kullback-Leibler divergence and allows us to find a crucial relation between the regularising parameter $\beta$ and the support of our object.  This connection could be verified numerically when appropriate values of $\beta$ are chosen to be close to the radius r and tend to approximate the boundaries or the convex hull of the reconstructed object. Hence, a combined penalisation on both the image and the sinogram space leads us to an enhancement and detection of object boundaries, specifically for images where \emph{thin structures} are present.

In real PET data thin structures will only make up parts of the image which will in general consist of small and larger scale objects as well as background. Our experiments for the cropped thin structures of the XCAT phantom in Figure \ref{xcat1} suggest TV regularisation on a targeted local Radon transform instead of the full Radon transform that allows to increase the regularisation on the sinogram in regions with thin structures.

\ack
This work has been financially supported by the King Abdullah University of Science and Technology (KAUST) Award No. KUK-I1-007-43, the EPSRC first grant EP/J009539/1, the Royal Society International Exchange Award Nr. IE110314, and the German Science Foundation (DFG) through the Collaborative Research Centre SFB 656 subproject B2 and Cells-in-Motion Cluster of Excellence (EXC 1003 – CiM), University of M\"unster. This work has been carried out while the second author was with the Institute of Computational Mathematics, University of M\"unster.

\section*{References}
\bibliographystyle{iopstyle}
\bibliography{refs}

\begin{thebibliography}{10}
\expandafter\ifx\csname urlstyle\endcsname\relax
  \providecommand{\doi}[1]{doi:\discretionary{}{}{}#1}\else
  \providecommand{\doi}{doi:\discretionary{}{}{}\begingroup
  \urlstyle{rm}\Url}\fi

\bibitem{Acar}
Acar R and Vogel C~R 1994 Analysis of bounded variation penalty methods for
  ill-posed problems.
\newblock \emph{\IP} \textbf{10}  1217--1229

\bibitem{Ambrosio}
Ambrosio L, Fusco N and Pallara D 2000 \emph{Functions of Bounded Variation and
  Free Discontinuity Problems}.
\newblock Oxford Science Publications

\bibitem{Carola}
Barbano P~E, Fokas A and Sch\"{o}nlieb C~B 2011 Alternating regularisation in
  measurement- and image space for {PET} reconstruction.
\newblock \emph{Proc. Int. Conf. Sampta Singapore}

\bibitem{bardsley2010theoretical}
Bardsley J~M 2010 A theoretical framework for the regularization of poisson
  likelihood estimation problems.
\newblock \emph{Inverse Probl. Imaging} \textbf{4}  11--17

\bibitem{bardsley2010hierarchical}
Bardsley J~M, Calvetti D and Somersalo E 2010 Hierarchical regularization for
  edge-preserving reconstruction of {PET} images.
\newblock \emph{\IP} \textbf{26}  035010

\bibitem{Martin}
Benning M and Burger M 2011 Error estimates for general fidelities.
\newblock \emph{Electronic transactions on Numerical Analysis} \textbf{38}
  44--68

\bibitem{benning2012ground}
Benning M and Burger M 2014 Ground states and singular vectors of convex
  variational regularization methods.
\newblock \emph{Methods and Applications of Analysis} To appear

\bibitem{Berg}
Bergounioux M and Tr\'elat E 2010 A variational method using fractional order
  {H}ilbert spaces for tomographic reconstruction of blurred and noised binary
  images.
\newblock \emph{J. Functional Analysis} \textbf{259}  2296--2332

\bibitem{Bregman}
Bregman L 1967 The relaxation of finding the common points of convex sets and
  its application to the solution of problems in convex programming.
\newblock \emph{USSR Comput. Math. Math. Phys.} \textbf{7}  200--217

\bibitem{Brune}
Brune C, Burger M, Sawatzky A, K\"osters T and W\"ubbeling F 2009
  Forward-{B}ackward {EM-TV} methods for inverse problems with {P}oisson noise.
\newblock \emph{Preprint}   5133 -- 5137

\bibitem{Brune1}
Brune C, Sawatzky A and Burger M 2010 Primal and dual {B}regman methods with
  application to optical nanoscopy.
\newblock \emph{International Journal of Computer Vision, 2010}

\bibitem{Martin1}
Burger M and Osher S 2007 Convergence rates of convex variational
  regularization.
\newblock \emph{Institute of physics publishing} \textbf{27}  257--263

\bibitem{caselles2007discontinuity}
Caselles V, Chambolle A and Novaga M 2007 The discontinuity set of solutions of
  the {TV} denoising problem and some extensions.
\newblock \emph{Multiscale modeling and simulation} \textbf{6}  879--894

\bibitem{Cham}
Chambolle A 2004 An algorithm for total variation minimisation and
  applications.
\newblock \emph{J. Math. Imaging. Vis} \textbf{20}  89--97

\bibitem{Chan}
Chan T and Shen J 2005 \emph{Image Processing and Image Analysis, Variational,
  PDE, Wavelet and Stochastic Methods}.
\newblock SIAM, Philadelphia

\bibitem{Ekeland}
Ekeland I and T\'emam R 1999 \emph{Convex Analysis and Variational Problems}.
\newblock SIAM, Philadelphia

\bibitem{esser2010general}
Esser E, Zhang X and Chan T 2010 A general framework for a class of first order
  primal-dual algorithms for convex optimization in imaging science.
\newblock \emph{SIAM J. Imaging Sci.} \textbf{3}  1015--1046

\bibitem{Pascal}
Getreuer P 2012 Rudin-{O}sher-{F}atemi total variation denoising using split
  {B}regman.
\newblock \emph{Image Processing Online}

\bibitem{Osher}
Goldstein T and Osher S 2009 The split {B}regman algorithm method for l1
  regularized problems.
\newblock \emph{SIAM J. Img. Sci.} \textbf{2}  323--343

\bibitem{Hertle}
Hertle A 1983 Continuity of the {R}adon transform and its inverse on
  {E}uclidean space.
\newblock \emph{Mathematische Zeitschrift} \textbf{184}  165--192

\bibitem{Le}
Le T, Chartrand R and Asaki T~J 2007 A variational approach to reconstructing
  images corrupted by poisson noise.
\newblock \emph{Journal Math Imaging and Vision} \textbf{27}  257--263

\bibitem{Fokas}
Marinakis V, Fokas A~S and Iserles A 2006 Reconstruction algorithm for single
  photon emission computed tomography and its numerical implementation.
\newblock \emph{Journal of the Royal Society Interface} \textbf{6}  45--54

\bibitem{And}
Markoe A 2006 \emph{Analytic Tomography}.
\newblock Cambridge University Press

\bibitem{Natt1}
Natterer F 2001 \emph{The Mathematics of computerized tomography in Applied
  Mathematics}.
\newblock SIAM, Philadelphia

\bibitem{Osher1}
Osher S, Burger M, Goldfarb D, Xu J and Yin W 2005 An iterative regularization
  method for total variation-based image restoration.
\newblock \emph{Multiscale Model. Simul.} \textbf{4}  460--489

\bibitem{Poul}
Poularikas A 2010 \emph{Transforms and application handbook, Third Edition}.
\newblock CRC Press

\bibitem{Prince}
Prince J~L and Willsky A~S 1990 A geometrical projection-space reconstruction
  algorithm.
\newblock \emph{Linear Algebra and its Applications} \textbf{130}  151--191

\bibitem{Rudin}
Rudin L~I, Osher S and Fatemi E 1992 Nonlinear total variation based noise
  removal algorithms.
\newblock \emph{Physica D: Nonlinear Phenomena} \textbf{60}  259--268

\bibitem{AlexPhD}
Sawatzky A 2011 {Ph.D thesis: (Non Local) Total variation in medical imaging}

\bibitem{Alex1}
Sawatzky A, Brune C, M{\"u}ller J and Burger M 2009 Total variation processing
  of images with poisson statistics.
\newblock \emph{Computer Analysis of Images and Patterns}. Springer, 533--540

\bibitem{Alex}
Sawatzky A, Brune C, W{\"u}bbeling F, Kosters T, Sch{\"a}fers K and Burger M
  2008 Accurate {EM-TV} algorithm in {PET} with low {SNR}.
\newblock \emph{Nuclear Science Symposium Conference Record, 2008. NSS '08.
  IEEE}   5133 -- 5137

\bibitem{setzer2009split}
Setzer S 2009 Split {B}regman algorithm, {D}ouglas-{R}achford splitting and
  frame shrinkage.
\newblock \emph{Scale space and variational methods in computer vision}.
  Springer, 464--476

\bibitem{setzer2011operator}
Setzer S 2011 Operator splittings, {B}regman methods and frame shrinkage in
  image processing.
\newblock \emph{International Journal of Computer Vision} \textbf{92}  265--280

\bibitem{Thirion}
Thirion J~P 1991 A geometric alternative to computed tomography.
\newblock \emph{Technical report 1463, INRIA}

\bibitem{Luminita}
Vese L 2001 A study in the {BV} {S}pace of a {D}enoising-{D}eblurring
  {V}ariational problem.
\newblock \emph{Applied Mathematics and Optimization} \textbf{44}  131--161

\bibitem{WA04}
Wernick M~N and Aarsvold J~N 2004 \emph{{Emission Tomography - The Fundamentals
  of PET and SPECT}}.
\newblock Elsevier Inc.

\bibitem{willett2010poisson}
Willett R~M, Harmany Z~T and Marcia R~F 2010 Poisson image reconstruction with
  total variation regularization.
\newblock \emph{Image Processing (ICIP), 2010 17th IEEE International
  Conference on}. IEEE, 4177--4180

\end{thebibliography}

\end{document}